\documentclass[parskip=never,abstract=true]{scrartcl}

\usepackage[utf8]{inputenc}
\usepackage{csquotes}

\usepackage[T1]{fontenc}

\usepackage{lmodern}

\usepackage{yfonts}
\usepackage{textcomp}
\usepackage{setspace}
\usepackage{fullpage}

\usepackage{extarrows}
\usepackage{etoolbox}
\usepackage{xparse}

\usepackage{mathtools}
\usepackage{amsthm,thmtools}
\usepackage{dsfont}
\let\mathbb\mathds
\usepackage{amssymb}
\usepackage[scr]{rsfso} 
\usepackage{bm}
\usepackage{euscript}
\usepackage{amsbsy}
\DeclareMathAlphabet\mathbfcal{OMS}{cmsy}{b}{n}

\usepackage{microtype}
\usepackage{authblk}
\usepackage{tikz}
\usetikzlibrary{cd,quotes}
\usetikzlibrary{decorations.markings}
\usetikzlibrary{decorations.pathreplacing,calligraphy}
\usepackage{pgfplots}
\pgfplotsset{compat=1.13}
\usepackage{booktabs}
\usepackage{float}
\usepackage{environ}
\usepackage{xcolor,soul}
\usepackage[bookmarks=false]{hyperref} 
\hypersetup{%
  colorlinks,%
  linkcolor={red!60!black},%
  citecolor={blue!50!black},%
  urlcolor={blue!80!black}%
}
\usepackage{enumitem,cleveref}
\usepackage{eso-pic}

\usepackage[normalem]{ulem}

\def\on{\operatorname}

\def\bS{\textbf{MB}}
\def\MS{\textbf{MS}}
\def\sS{\textbf{S}}

\def\OO{\mathbb{O}}

\def\Fun{\on{Fun}}

\def\Cat{\on{Cat}}

\def\iCat{\EuScript{C}\!\on{at}}

\def\Hom{\on{Hom}}

\def\Nsc{\on{N}^{\mathbf{sc}}}
\def\Nsm{\on{N}^{\mathbf{ms}}}

\def\Nms{\on{N}^{\on{ms}}}

\def\Set{\on{Set}}
\def\Map{\on{Map}}
\def\scsSet{{\on{Set}_{\Delta}^{\mathbf{sc}}}}

\def\mbsSet{{\Set_\Delta^{\mathbf{mb}}}}

\def\sc{{\on{sc}}}

\def\ev{\on{ev}}

\def\Sqe{\mathbf{Sq}_{\on{E}}}

\declaretheoremstyle[bodyfont=\itshape,notefont=\bfseries]{abellanA}
\declaretheoremstyle[notefont=\bfseries]{abellanB}

\declaretheorem[style=abellanA,numberwithin=section,name={Theorem}]{theorem}

\declaretheorem[style=abellanA,numberlike=theorem,name={Lemma}]{lemma}
\declaretheorem[style=abellanB,numberlike=theorem,name={Definition}]{definition}
\declaretheorem[style=abellanB,numberlike=theorem,name={Definition-Lemma}]{definition-lemma}
\declaretheorem[style=abellanB,numberlike=theorem,name={Remark}]{remark}

\declaretheorem[style=abellanA,numberlike=theorem,name={Proposition}]{proposition}

\declaretheorem[style=abellanB,numbered=no,name={Notation}]{notation}
\declaretheorem[style=abellanA,numberlike=theorem,name={Corollary}]{corollary}

\newtheorem{thm*}{Theorem}

\newtheorem*{prop*}{Proposition}
\newtheorem{cor*}{Corollary}

\docsvlist{theorem,lemma,definition,remark,proposition,example,corollary}

\let\leq\leqslant
\let\geq\geqslant
\let\epsilon\varepsilon
\let\isom\simeq

\newcommand*\tensor{\otimes}

\newcommand*\btensor{\overline{\tensor}}
\newcommand*\bodot{\overline{\odot}}

\newcommand*\mathblank{\mathord{-}}

\DeclareMathOperator*\colim{colim}

\def\Cat{\on{Cat}}
\def\scr{\EuScript}

\usepackage[bbgreekl]{mathbbol}

\let\emptyset\varnothing

\makeatletter
\newcommand{\fixed@sra}{$\vrule height 2\fontdimen22\textfont2 width 0pt\rightarrow{}$}
\newcommand{\shortarrowup}[1]{%
  \mathrel{\text{\rotatebox[origin=c]{65}{\fixed@sra}}}
}
\newcommand{\shortarrowdown}[1]{%
  \mathrel{\text{\rotatebox[origin=c]{250}{\fixed@sra}}}
}
\makeatother

\newcommand{\upslash}{\!\shortarrowup{1}}

\def\lra{\longrightarrow}
\def\lla{\longleftarrow}

\def\llra{\def\arraystretch{.1}\begin{array}{c} \lra \\ \lla \end{array}}

\def\op{{\on{op}}}

\makeatletter
\tikzcdset{
  open/.code     = {\tikzcdset{hook, circled};},
  open'/.code    = {\tikzcdset{hook', circled};},
  circled/.code  = {\tikzcdset{markwith = {\draw (0,0) circle (.375ex);}};},
  markwith/.code ={
    \pgfutil@ifundefined%
    {tikz@library@decorations.markings@loaded}%
    {\pgfutil@packageerror{tikz-cd}{You need to say %
        \string\usetikzlibrary{decorations.markings} to use arrows with markings}{}}{}%
    \pgfkeysalso{/tikz/postaction = {
        /tikz/decorate,
        /tikz/decoration={markings, mark = at position 0.5 with {#1}}}
    }
  },
}
\makeatother

\usetikzlibrary{calc}
\usetikzlibrary{decorations.pathmorphing}

\tikzset{curve/.style={settings={#1},to path={(\tikztostart)
    .. controls ($(\tikztostart)!\pv{pos}!(\tikztotarget)!\pv{height}!270:(\tikztotarget)$)
    and ($(\tikztostart)!1-\pv{pos}!(\tikztotarget)!\pv{height}!270:(\tikztotarget)$)
    .. (\tikztotarget)\tikztonodes}},
    settings/.code={\tikzset{quiver/.cd,#1}
        \def\pv##1{\pgfkeysvalueof{/tikz/quiver/##1}}},
    quiver/.cd,pos/.initial=0.35,height/.initial=0}

\tikzset{tail reversed/.code={\pgfsetarrowsstart{tikzcd to}}}
\tikzset{2tail/.code={\pgfsetarrowsstart{Implies[reversed]}}}
\tikzset{2tail reversed/.code={\pgfsetarrowsstart{Implies}}}
\tikzset{no body/.style={/tikz/dash pattern=on 0 off 1mm}}

\makeatletter
\DeclareSymbolFont{lettersA}{U}{txmia}{m}{it}

\DeclareRobustCommand*{\varmathbb}[1]{\gdef\F@ntPrefix{m@thbbch@r}%
  \@EachCharacter #1\@EndEachCharacter}

\long\def\DoLongFutureLet #1#2#3#4{%
  \def\@FutureLetDecide{#1#2\@FutureLetToken
    \def\@FutureLetNext{#3}\else
    \def\@FutureLetNext{#4}\fi\@FutureLetNext}
  \futurelet\@FutureLetToken\@FutureLetDecide}
\def\DoFutureLet #1#2#3#4{\DoLongFutureLet{#1}{#2}{#3}{#4}}
\def\@EachCharacter{\DoFutureLet{\ifx}{\@EndEachCharacter}%
  {\@EachCharacterDone}{\@PickUpTheCharacter}}
\def\m@keCharacter#1{\csname\F@ntPrefix#1\endcsname}
\def\@PickUpTheCharacter#1{\m@keCharacter{#1}\@EachCharacter}
\def\@EachCharacterDone \@EndEachCharacter{}

\DeclareMathSymbol{\m@thbbch@rA}{\mathord}{lettersA}{129}
\DeclareMathSymbol{\m@thbbch@rB}{\mathord}{lettersA}{130}
\DeclareMathSymbol{\m@thbbch@rC}{\mathord}{lettersA}{131}
\DeclareMathSymbol{\m@thbbch@rD}{\mathord}{lettersA}{132}
\DeclareMathSymbol{\m@thbbch@rE}{\mathord}{lettersA}{133}
\DeclareMathSymbol{\m@thbbch@rF}{\mathord}{lettersA}{134}
\DeclareMathSymbol{\m@thbbch@rG}{\mathord}{lettersA}{135}
\DeclareMathSymbol{\m@thbbch@rH}{\mathord}{lettersA}{136}
\DeclareMathSymbol{\m@thbbch@rI}{\mathord}{lettersA}{137}
\DeclareMathSymbol{\m@thbbch@rJ}{\mathord}{lettersA}{138}
\DeclareMathSymbol{\m@thbbch@rK}{\mathord}{lettersA}{139}
\DeclareMathSymbol{\m@thbbch@rL}{\mathord}{lettersA}{140}
\DeclareMathSymbol{\m@thbbch@rM}{\mathord}{lettersA}{141}
\DeclareMathSymbol{\m@thbbch@rN}{\mathord}{lettersA}{142}
\DeclareMathSymbol{\m@thbbch@rO}{\mathord}{lettersA}{143}
\DeclareMathSymbol{\m@thbbch@rP}{\mathord}{lettersA}{144}
\DeclareMathSymbol{\m@thbbch@rQ}{\mathord}{lettersA}{145}
\DeclareMathSymbol{\m@thbbch@rR}{\mathord}{lettersA}{146}
\DeclareMathSymbol{\m@thbbch@rS}{\mathord}{lettersA}{147}
\DeclareMathSymbol{\m@thbbch@rT}{\mathord}{lettersA}{148}
\DeclareMathSymbol{\m@thbbch@rU}{\mathord}{lettersA}{149}
\DeclareMathSymbol{\m@thbbch@rV}{\mathord}{lettersA}{150}
\DeclareMathSymbol{\m@thbbch@rW}{\mathord}{lettersA}{151}
\DeclareMathSymbol{\m@thbbch@rX}{\mathord}{lettersA}{152}
\DeclareMathSymbol{\m@thbbch@rY}{\mathord}{lettersA}{153}
\DeclareMathSymbol{\m@thbbch@rZ}{\mathord}{lettersA}{154}

\makeatother

\makeatletter
\newcommand{\myitem}[1]{%
  \item[#1]\protected@edef\@currentlabel{#1}%
}
\makeatother 

\def\CSeg{\on{CSeg}_{\Delta}\left(\bcat{C}\!\on{at}_{(\infty,1)}\right)}
\def\CSeglax{\on{CSeg}^{\mathcal{I}\!\on{oplax}}_{\Delta}\left(\bcat{C}\!\on{at}_{(\infty,1)}\right)}
\def\EDbl{\on{EDbl}(\bcat{C}\!\on{at}_{(\infty)})}
\def\EDblnuc{\on{EDbl}^{\on{nuc}}(\bcat{C}\!\on{at}_{(\infty)})}
\def\ESeg{\on{ESeg}_\Delta\left(\bcat{C}\!\on{at}_{(\infty,1)}\right)}
\def\ESegnuc{\on{ESeg}^{\on{nuc}}_\Delta\left(\bcat{C}\!\on{at}_{(\infty,1)}\right)}
\def\EGlob{\mathbf{Gl}_{\on{E}}}
\def\iCatt{\bcat{C}\!\on{at}_{(\infty,1)}}

\def\bcat{\varmathbb}

\title{Comparing lax functors of $(\infty,2)$-categories}
\author{Fernando Abellán}
\date{}
\begin{document}
  \maketitle
  \begin{abstract}
    In this work, we study oplax normalised functors of $(\infty,2)$-categories. Our main theorem is a comparison between the notion of oplax normalised functor of scaled simplicial sets due to Gagna-Harpaz-Lanari and the corresponding notion in the setting of complete Segal objects in $(\infty,1)$-categories studied by Gaitsgory and Rozenblyum. As a corollary, we derive that the Gray tensor product of $(\infty,2)$-categories as defined by Gaitsgory-Rozenblyum is equivalent to that of Gagna-Harpaz-Lanari. 

    Moreover, we construct an $(\infty,2)$-categorical variant of the quintet functor of Ehresmann, from the $(\infty,2)$-category of $(\infty,2)$-categories to the $(\infty,2)$-category of double $(\infty,1)$-categories and show that it is fully faithful.
    
    As a key technical ingredient, given $(\bcat{C},E)$ an $(\infty,2)$-category equipped with a collection of morphisms and a functor of $(\infty,2)$-categories $f:\bcat{C}\xlongrightarrow{} \bcat{D}$, we construct a right adjoint to the restriction functor $f^*$ from the $(\infty,2)$-category of functors $\bcat{D} \xlongrightarrow{} \bcat{C}\!\on{at}_{(\infty,2)}$ and natural transformations to the  $(\infty,2)$-category of functors $\bcat{C} \xlongrightarrow{} \bcat{C}\!\on{at}_{(\infty,2)}$ and partially lax (according to $E$) natural transformations. We apply this new technology of partially lax Kan extensions to the study of complete Segal objects in $(\infty,1)$-categories and double $(\infty,1)$-categories which allows us to define the notion of an \emph{enhanced} Segal object (resp. \emph{enhanced} double $(\infty,1)$-category), the former yielding yet another model for the theory of $(\infty,2)$-categories.
  \end{abstract}
  \tableofcontents 
 \section{Introduction} 
  While originally intended, (\cite{QuiModel}, \cite{BVog}), as a language to model certain mathematical structures native to homotopy theory, $(\infty,n)$-categories have evolved into subject of its own interest. The ability of category theory to identify the \emph{characterizing property} of a mathematical object, coupled with the homotopical point of view has led to the understanding of mathematical constructions such as K-Theory (\cite{UnivK}) and Topological Quantum Field Theories (\cite{Cobor}), in terms of homotopy-coherent universal properties. 

  It does not take a long time wandering in the realm of $(\infty,1)$-categories until we start asking ourselves questions of the form:
 \begin{itemize}
   \item How can we characterise an $(\infty,1)$-category $\scr{C}$ according to its behaviour with respect to other categories?
   \item Given a diagram $\{\scr{C}_i\}_{i \in I}$ of $(\infty,1)$-categories and functors, what is the “free category“ on that diagram?
   \item Given an $(\infty,1)$-category $\scr{C}$ which satisfies a universal property in $\scr{C}\!\on{at}_{(\infty,1)}$, the $(\infty,1)$-category of $(\infty,1)$-categories. How does this universal property behave with respect to natural transformations of functors?
 \end{itemize}
 To give a satisfactory answer to these questions requires us to accept that $\scr{C}\!\on{at}_{(\infty,1)}$ is just the one-dimensional core of an $(\infty,2)$-category $\bcat{C}\!\on{at}_{(\infty,1)}$ which is the natural habitat for a wider range of universal properties. This is nothing new, as mathematicians working on category theory have known for a long time that many familiar categorical constructions, such as adjunctions, monads and monoidal categories, can be naturally understood in the 2-categorical language. However, it is not until recently that we have started to see a systematic use of the 2-dimensional theory of  $(\infty,n)$-categories being applied in contexts such as homotopy theory (\cite{LurieGoodwillie}) and derived geometry (\cite{GaitsgoryRozenblyum}, \cite{Strat}). Moreover, there is a rising interest in studying invariants of higher-categorical structures such as the Balmer spectrum (\cite{Balmer})  of a symmetric monoidal stable $(\infty,1)$-category. From a 2-categorical perspective, we can understand stable $(\infty,1)$-categories as categorified analogues of abelian groups and their symmetric monoidal counterparts as categorified commutative rings. This vision is currently been pursued by the authors of \cite{CDW} with the goal of developing a categorified theory of homological algebra. To fully realize this picture though, a highly developed theory of $(\infty,2)$-categories is necessary.

  Just in the same way as strict categories have higher dimensional generalizations known as $n$-categories, their homotopical counterparts can be organized into an infinite tower

  \[
  \Bigl\{ \text{spaces}\Bigr\} \subset \Bigl\{ (\infty,1)\text{-categories} \Bigr\}\subset  \Bigl\{ (\infty,2)\text{-categories} \Bigr\}\subset  \Bigl\{ (\infty,3)\text{-categories} \Bigr\}\subset \cdots
 \]
In this work, we will climb two steps in this higher-categorical ladder and establish several foundational results in the study of $(\infty,2)$-categories. More precisely:
\begin{enumerate}
  \item We compare the notions of oplax normalised functors across different models for $(\infty,2)$-categories and prove that they coincide in \autoref{intro:thmlax}. 
  \item In \autoref{intro:gray}, we derive an equivalence between the corresponding definitions of the Gray tensor product in the models studied.
  \item We show in \autoref{introcor:gray} that the Gray tensor product of complete Segal objects (or 2-fold complete Segal spaces) in $(\infty,1)$-categories given in \cite{GaitsgoryRozenblyum} defines a closed monoidal structure.
  \item We construct a novel model for $(\infty,2)$-categories which we call \emph{enhanced Segal objects} in $(\infty,1)$-categories by equipping Segal objects with additional lax data and show in \autoref{intro:enhanced}, that the $(\infty,2)$-category of enhanced segal objects $\ESeg$ models the $(\infty,2)$-category of $(\infty,2)$-categories, oplax normalised functors and opglobular natural transformations.
  \item We apply our enhancement construction to double $(\infty,1)$-categories to produce a theory of \emph{enhanced double} $(\infty,1)$-categories. 
  \item In \autoref{intro:sqe}, we construct an enhanced variant of Ehresmann quintet construction and show that it defines fully faithful functor of $(\infty,2)$-categories  
  \[
     \Sqe: \bcat{C}\!\on{at}_{(\infty,2)} \xlongrightarrow{} \on{EDbl}(\bcat{C}\!\on{at}_{(\infty,1)})
   \] 
   with values in enhanced double $(\infty,1)$-categories thus answering another question in \cite{GaitsgoryRozenblyum}.
   \item We lay the foundations of a general theory of lax Kan extensions in \autoref{intro:rightlaxkan}.
\end{enumerate}

\subsection{Oplax normalised functors of \texorpdfstring{$(\infty,2)$-}-categories and the Gray tensor product}
 One of the most fundamental  advantages of the 2-dimensional theory is that it allows us to define the notion of oplax (resp. lax) functors. Given a pair of $(\infty,2)$-categories $\bcat{C}$ and $\bcat{D}$, an oplax functor $F:\bcat{C} \xlongrightarrow{} \bcat{D}$ sends $n$-morphisms in $\bcat{C}$ to $n$-morphisms in $\bcat{D}$ but only respects composition of $1$-morphisms up to a coherent choice of 2-morphism $F(g \circ f)\xRightarrow{} F(g) \circ F(f)$. By reversing the the order of the 2-morphism $F(g \circ f)\xLeftarrow{} F(g) \circ F(f)$ one obtains the notion of a lax functor. In this work, we will only look at those oplax functors which are \emph{normalised}. More precisely, we say that $F$ is normalised if the condition below holds.
 \begin{itemize}
   \item[$\mathbf{N})$] Given a composite $g \circ f$ in $\bcat{C}$ such that \emph{either} $f$ or $g$ is an equivalence, then it follows that the canonical 2-morphism $F(g \circ f)\xRightarrow{\simeq} F(g) \circ F(f)$ is invertible in $\bcat{D}$.
 \end{itemize}

  In homotopy theory, this notion appears naturally when dealing with monoidal $(\infty,1)$-categories where a normalised oplax monoidal functor is given by a map between the cartesian fibrations, 
 \[
   \begin{tikzcd}
     \scr{C}^\tensor \arrow[rr,"f"] \arrow[dr,swap,"p"] & & \scr{D}^{\tensor} \arrow[dl,"q"]\\
     &  \Delta & 
    \end{tikzcd}
 \]
 representing the monoidal $(\infty,1)$-categories, such that:
 \begin{itemize}
  \item[i)] The map $f$ commutes with the projection maps, i.e. $q \circ f= p$.
  \item[ii)] Given a cartesian edge $u:x \xlongrightarrow{} y $ in $\scr{C}^\tensor$ such that $p(u):p(x) \xlongrightarrow{} p(y)$  is either a surjection in $\Delta$ or the inclusion of a subinterval $p(x)=[n] \subseteq [k]=p(y)$ of $[k]$, then it follows that $f(u)$ defines a cartesian edge\footnote{It is common in homotopy theory to deal with non-normalised monoidal functors in which case one drops the condition regarding surjective maps in $\Delta$.} in $\scr{D}^\tensor$.
 \end{itemize}

  More generally, one can define $(\infty,2)$-categories as functors $\scr{X}_\bullet:\Delta^\op \xlongrightarrow{} \bcat{C}\!\on{at}_{(\infty,1)}$ such that:
 \begin{enumerate}
   \item The functor satisfies the Segal condition. That is, given $n \geq 2$ the monotone maps $i_{\alpha}:[1] \xlongrightarrow{} [n]$ for $0\leq \alpha \leq n-1$ given by the inclusion $\{\alpha,\alpha+1\} \xlongrightarrow{} [n]$ induce an equivalence of $(\infty,1)$-categories
   \[
      \scr{X}_n \xlongrightarrow{\simeq} \scr{X}_1 \times_{\scr{X}_0}\scr{X}_1 \times  \cdots \times_{\scr{X}_0} \scr{X}_1.
    \] 
   \item The $(\infty,1)$-category $\scr{X}_0$ is an $\infty$-groupoid.
   \item The underlying Segal space $\scr{X}^\simeq_\bullet$ is complete (\autoref{def:completesegalobj}). 
 \end{enumerate}
 We call such functors \emph{complete Segal objects} in $\bcat{C}\!\on{at}_{(\infty,1)}$.

 Let us denote by $\mathcal{I}\subset \Delta$ the class of \emph{idle morphisms} which consists in surjections and interval inclusions. Applying the Grothendieck construction, we can define an oplax normalised functor of $(\infty,2)$-categories as a commutative diagram of cartesian fibrations representing complete Segal objects
 \[
   \begin{tikzcd}
     \scr{X} \arrow[rr,"f"] \arrow[dr,swap,"p"] & & \scr{Y} \arrow[dl,"q"]\\
     &  \Delta & 
    \end{tikzcd}
 \]
 such that map $f$ preserves cartesian edges lying over idle morphisms in $\Delta$. By considering 2-morphisms given by commutative diagrams
 \[
   \begin{tikzcd}
     \scr{X}\times \Delta^1 \arrow[rr,"f"] \arrow[dr,swap,"p\circ \pi"] & & \scr{Y} \arrow[dl,"q"]\\
     &  \Delta & 
    \end{tikzcd}
 \]
 where $\pi: \scr{X}\times \Delta^1 \xlongrightarrow{} \scr{X}$ is the obvious projection, we obtain an $(\infty,2)$-category $\on{CSeg}_\Delta(\bcat{C}\!\on{at}_{(\infty,1)})^{\mathcal{I}\!\on{oplax}}$ whose objects are $(\infty,2)$-categories and where the morphisms are oplax normalised functors. This way of looking at oplax normalised functors is the perspective chosen by the authors in \cite{GaitsgoryRozenblyum}, where  $(\infty,2)$-categories play a key role in their study of derived algebraic geometry.

 The notion of a complete Segal object in $(\infty,1)$-categories (or 2-fold complete Segal space in the original terminology in \cite{Unicity}) is an internalization of the notion of a Segal space to $(\infty,1)$-categories. This approach is a very succesful way of inductively defining $(\infty,n)$-categories as certain category objects (i.e., complete Segal objects) in $(\infty,n-1)$-categories. 

 There are however, other models for $(\infty,n)$-categories which are of more “external“ nature. An important example of these family of models is given by $n$-complicial sets \cite{Verity}, cubical sets \cite{comical}, quasicategories for $n=1$ and scaled simplicial sets \cite{LurieGoodwillie} for $n=2$. By external, we mean models following the following pattern:
 \begin{itemize}
   \item A set of objects and a set of $n$-morphisms is postulated for $n\geq 1$  as \emph{merely combinatorial data}. For example, in the case of quasicategories, our combinatorial data is modelled by a simplicial set $X:\Delta^\op \xlongrightarrow{}\on{Set}$.
   \item A set of axioms is imposed to this data, usually in the form of lifting properties of model-categorical flavour which guarantees that our combinatorial objects admit a homotopical interpretation.  Again, in the quasicategorical model this is given by the inner-horn lifting criterion which we diagrammatically depict as
   \[
     \begin{tikzcd}
       \Lambda^n_i \arrow[d] \arrow[r]  & X \\
       \Delta^n. \arrow[ur,dotted] &
     \end{tikzcd}
   \]
 \end{itemize}

  It is believed that these two ways of viewing $(\infty,n)$-categories are equivalent (\cite{Loubaton}) but to this day no complete general proof is available.

  In our case of interest ($n=2$) results of Bergner, Gagna, Harpaz, Lanari, Lurie and Rezk show that these points of view agree. In this paper, we further extend this equivalence by showing that we can compare the notion of oplax normalised functor across these two families of models for $(\infty,2)$-categories.
  \subsubsection{ Oplax normalised functors of scaled simplicial sets}

  We will not be working with $2$-complicial sets but with the closely related model of \emph{scaled simplicial} sets introduced by Lurie in \cite{LurieGoodwillie}. Let us point out that it is a theorem of Gagna, Harpaz and Lanari \cite{GHL_Equivalence} that both models are equivalent. 

  We define a scaled simplicial set to be a pair $(X,T_X)$ consisting in a simplicial set $X$ equipped with a chosen collection of triangles $T_X \subset X_2$ (called the \emph{thin} 2-simplices) containing every degenerate triangle. We view the collection $T_X$ as representing weakly commuting triangles  in $X$. A map of scaled simplicial sets $f:(X,T_X) \xlongrightarrow{} (Y,T_Y)$ is simply a map of simplicial sets such that $f(T_X) \subseteq T_Y$. 

  In \cite{LurieGoodwillie}, Lurie defines an $\infty$-bicategory (or a fibrant scaled simplicial set by \cite[Theorem 5.1]{GHL_Equivalence}) as an scaled simplicial set satisfying certain lifting properties which are generalizations of the usual inner horn lifting criterion for quasicategories. Given fibrant scaled simplicial sets $(\bcat{C},T_{\bcat{C}})$ and $(\bcat{D},T_{\bcat{D}})$ we can make the following definition:
  \begin{itemize}
    \item An oplax normalised functor $F: \bcat{C} \xlongrightarrow{} \bcat{D}$ is given by a map of the underlying simplicial sets such that given $\sigma \in T_{\bcat{C}}$ such that either $d_0(\sigma)$ or $d_2(\sigma)$ is an equivalence, where $d_i(\sigma) \in \bcat{C}_1$ for $i \in \{0,2\}$, then it follows that $F(\sigma) \in T_{\bcat{D}}$.
  \end{itemize}
  Unpacking this definition, we see that this is essentially the point of view adopted when we first introduced the notion of an oplax normalised functor. Even more is true, we can understand the set of $n$-simplices of $(\bcat{C},T_{\bcat{C}})$ in terms of oplax normalised functors as we demonstrate below.

  \begin{itemize}
    \item Let $(\bcat{C},T_{\bcat{C}})$ be a fibrant scaled simplicial set.  We define a subcollection $M_{\bcat{C}} \subset T_{\bcat{C}}$ consisting in those $\sigma \in T_{\bcat{C}}$ such that either $d_0(\sigma)$ or $d_2(\sigma)$ is an equivalence. It then follows that a map of scaled simplicial sets
    \[
      F:(\bcat{C},M_{\bcat{C}}) \xlongrightarrow{} (\bcat{D},T_{\bcat{D}})
    \]
    is precisely the data of an oplax normalised functor.
    \item Let $\Delta^n_\sharp:=(\Delta^n,\sharp)$ denote an $n$-simplex where all triangles are thin. This marked simplicial set is fibrant and represents the ordinary poset $[n]=\{0<1<\cdots <n\}$. It follows that the collection $M_{\Delta^n_\sharp}$ consists precisely in the degenerate triangles of $\Delta^n$. We denote $\Delta^n_\flat=(\Delta^n,M_{\Delta^n_\sharp})$.
    \item As a consequence of the Yoneda lemma we see that given a fibrant scaled simplicial set $(\bcat{C},T_{\bcat{C}})$, the data of an n-simplex $\bcat{C}_n$ corresponds precisely to a map $\Delta^n_\flat \xlongrightarrow{} \bcat{C}$. In order words, 
    \[
       \bcat{C}_n=\{F: [n] \xlongrightarrow{} \bcat{C} \enspace | \enspace \text{ F is an oplax normalised functor}\}
    \]  
    \item In the case $n=2$ we can unpack this definition to see that a 2-simplex in $\bcat{C}$ is comprised of  the data:
    \begin{itemize}
    \item Three morphisms $u: x \xlongrightarrow{} y$ and $v: y \xlongrightarrow{} z$ and $w: x \xlongrightarrow{} z$ in $\bcat{C}$.
    \item A 2-morphism  $w \xRightarrow{} v \circ u$.
  \end{itemize}
   In particular, a triangle is thin if the associated oplax normalised functor is an actual functor which happens if and only if the corresponding 2-morphism  $w \xRightarrow{\simeq} v \circ u$ is invertible.
  \end{itemize}

    We can now state the main theorem which appears in the document as \autoref{thm:rearranged}

   \begin{thm*}\label{intro:thmlax}
     Let $\bcat{B}\!\on{icat}_\infty^{\mathbf{oplax}}$ be the $(\infty,2)$-category of $(\infty,2)$-categories (modelled as fibrant marked-scaled simplicial sets), oplax normalised functors and opglobular natural transformations\footnote{See  \autoref{def:globularGray} and \autoref{def:globfunctorcat}.}. Then there exists an equivalence of $(\infty,2)$-categories
     \[
       \mathbf{G}:\bcat{B}\!\on{icat}_\infty^{\mathbf{oplax}} \xlongrightarrow{\simeq } \on{CSeg}_\Delta(\bcat{C}\!\on{at}_{(\infty,1)})^{\mathcal{I}\!\on{oplax}}
     \]
   \end{thm*}
 
\subsubsection{The Gray tensor product}
 Let $F,G: \bcat{C} \xlongrightarrow{}\bcat{D}$ be a pair of functors of $(\infty,2)$-categories. When defining what a natural transformation should be in the 2-categorical world, there is no good reason to always require that given a morphism $u:c \xlongrightarrow{} c'$ the associated square
 \[\begin{tikzcd}
  {F(c)} & {F(c')} \\
  {G(c)} & {G(c')}
  \arrow["{F(u)}", from=1-1, to=1-2]
  \arrow[from=1-1, to=2-1]
  \arrow["{G(u)}"', from=2-1, to=2-2]
  \arrow[from=1-2, to=2-2]
  \arrow[Rightarrow, shorten >=6pt, from=1-2, to=2-1]
\end{tikzcd}\]
commutes, i.e. the diagonal 2-morphism is invertible. Informally speaking, a general notion of morphism between functors should be given by:
\begin{itemize}
  \item For every object $c \in \bcat{C}$ a 1-morphism $\alpha_c: F(c) \xlongrightarrow{} G(c)$.
  \item For every morphism $u:c \xlongrightarrow{} c'$ a 2-morphism $\alpha_u:\alpha_{c'} \circ F(u) \xRightarrow{} G(u) \circ \alpha_c$.
  \item For every 2-morphism a commuting diagram obtained by pasting certain 2-morphisms.
\end{itemize}
This definition yields the notion of an oplax natural transformation, where its lax counterpart is obtained by reversing the order of the 2-morphisms appearing in the definition. Part of the rapid development that the theory of $(\infty,2)$-categories has enjoyed in the recent years comes from the fact that we have provided good ways of realising the notion of (op)lax natural transformation in the homotopy-coherent world:
\begin{itemize}
  \item We have now, due to the effort of several authors (\cite{Berman},\cite{GHL_LaxLim}, \cite{Strat}, \cite{AScof}) a well-developed theory of partially lax limits in $(\infty,2)$-categories. Thanks to this new range of universal properties it was discovered by Linskens, Nardin and Pol (\cite{Global}) that the $(\infty,1)$-category of global spectra can be characterised as a partially lax limit.
\end{itemize}
Recall that an ordinary natural transformation between a pair of functors is given by a map out of the cartesian product $\alpha: \bcat{C} \times [1] \xlongrightarrow{} \bcat{D}$. The insight of Gray (\cite{Grayoriginal}) was to define a lax analogue of the cartesian product, which is now known as the Gray tensor product and denoted by “$\tensor$“. Equipped with this construction one sees that an oplax natural transformation corresponds to a functor 
\[
  \bcat{C} \tensor[1] \xlongrightarrow{} \bcat{D}.  
\]
Let us point out that the Gray tensor product is not commutative, as a map out of $[1] \tensor \bcat{C}$ corresponds to a lax natural transformation.

Gray's original definition has been leveraged to the homotopy-coherent setting (\cite{Verity},\cite{GHL_Gray}, \cite{GaitsgoryRozenblyum},\cite{graycampion}) and characterized by the following universal property:
\begin{itemize}
  \item Let $\bcat{C}$, $\bcat{D}$ and $\bcat{A}$ be $(\infty,2)$-categories. Then the data of a functor $\bcat{C}\tensor \bcat{D}\xlongrightarrow{}\bcat{A}$ corresponds precisely to the data of an oplax normalised functor 
  \[
    F:\bcat{C} \times \bcat{D} \xlongrightarrow{} \bcat{A}
  \]
  out of the cartesian product such that:
  \begin{enumerate}
    \item For every object $c \in \bcat{C}$, we have a (non-lax) functor $F(c,\mathblank): \bcat{C} \xlongrightarrow{} \bcat{A}$ and similarly for $F(\mathblank,d)$ for $d \in \bcat{D}$.
    \item Let $u:c \xlongrightarrow{} c'$ and $v: d \xlongrightarrow{} d'$ be a pair of morphisms in $\bcat{C}$ and $\bcat{D}$ and consider the composition $(u,v)=(\on{id}_{c'},v) \circ (u,\on{id}_d)$ in $\bcat{C} \times \bcat{D}$. Then it follows that the canonical 2-morphism $F(u,v) \xRightarrow{\simeq} F(\on{id}_{c'},v) \circ F(u,\on{id}_d)$ is invertible
  \end{enumerate}
\end{itemize}
It is possible to work only with this universal property in setting of complete Segal objects in $(\infty,1)$-categories (see \cite{GaitsgoryRozenblyum}) but the lack of a good workable model makes difficult to prove certain important properties that the Gray tensor product enjoys.

In the setting of scaled simplicial sets however, the Gray tensor product becomes very explicit as seen in the work of Gagna, Harpaz and Lanari. Exploiting the combinatorial nature of their definitions, it is proven in \cite{GHL_Gray} that the Gray tensor product of scaled simplicial sets is associative and preserves $\infty$-categorical colimits in each variable (see Propopsition 2.2 and Corollary 2.15 \cite{GHL_Gray}) and moreover defines a closed monoidal structure on $\bcat{C}\!\on{at}_{(\infty,2)}$. As a consequence 
of \autoref{intro:thmlax} we can derive (see \autoref{thm:maingray}, \autoref{cor:GHRprop}) the next results.
\begin{thm*}\label{intro:gray}
  The equivalence of $(\infty,2)$-categories given in \autoref{intro:thmlax},
   \[
     \mathbf{G}:  \bcat{B}\!\on{icat}_\infty^{\mathbf{oplax}} \xlongrightarrow{\simeq } \on{CSeg}_\Delta(\bcat{C}\!\on{at}_{(\infty,1)})^{\mathcal{I}\!\on{oplax}}
   \]
is compatible with the Gray tensor product. More precisely, given $(\infty,2)$-categories $\bcat{C}$ and $\bcat{D}$ it follows that there exists an equivalence (natural in $\bcat{C}$ and $\bcat{D}$)
\[
  \mathbf{G}(\bcat{C} \tensor \bcat{D}) \simeq \mathbf{G}(\bcat{C}) \tensor \mathbf{G}(\bcat{D}).
\]

\end{thm*}
\begin{cor*}\label{introcor:gray}
  The Gray tensor product defines a closed monoidal structure in $\on{CSeg}_\Delta(\bcat{C}\!\on{at}_{(\infty,1)})$. In particular, it is associative and commutes with $\infty$-categorical colimits in each variable.
\end{cor*}

\subsection{The enhancement construction}
The main insight of this paper and the reason that we are able to prove the theorems above is that there exists an interpolating model for $(\infty,2)$-categories receiving two equivalences
\[
  \begin{tikzcd}
    K_* : \on{CSeg}_\Delta(\bcat{C}\!\on{at}_{(\infty,1)})^{\mathcal{I}\!\on{oplax}} \arrow[r,"\simeq"] & \ESeg & \arrow[l,swap,"\simeq"] \bcat{B}\!\on{icat}_\infty^{\mathbf{oplax}}: \EGlob.
  \end{tikzcd}
\]
which we call \emph{enhanced Segal objects} in $\bcat{C}\!\on{at}_{(\infty,1)}$.

Let $\bcat{X}_\bullet: \Delta^\op \xlongrightarrow{} \bcat{C}\!\on{at}_{(\infty,1)}$. We say that $\bcat{X}_\bullet$ satisfies the \emph{oplax Segal condition} if:
\begin{enumerate}
  \item For every $n\geq 2$, the canonical map 
  \[
    L_n: \bcat{X}_n \xlongrightarrow{}\bcat{X}_1 \times_{\bcat{X}_0}\bcat{X}_1 \times  \cdots \times_{\bcat{X}_0} \bcat{X}_1.
    \] 
   admits a fully faithful right adjoint.
   \item Let $\scr{X}_n$ for $n\geq 0$ be the full subcategory of $\bcat{X}_n$ determined by condition 1 above, where we are making the convention $\bcat{X}_i=\scr{X}_i$ for $i=0,1$. Then it follows that for every $u:[n] \xlongrightarrow{} [k]$ in $\Delta$ the induced map $\bcat{X}(u)$ preserves the corresponding subcategories thus inducing a functor
   \[
     \scr{X}_\bullet : \Delta^\op  \xlongrightarrow{} \bcat{C}\!\on{at}_{(\infty,1)}
   \]
   which satisfies the Segal condition. We call this functor the \emph{nucleus} of $\bcat{X}_\bullet$.
\end{enumerate}
Using \cite[Theorem 5.3.5]{mates} (or more generally the results in \cite{Agh}) we can take the mate of the natural transformation $\scr{X}_\bullet \xRightarrow{} \bcat{X}_\bullet$ which is an oplax natural transformation $\bcat{X}_\bullet \xRightarrow{}\scr{X}_\bullet$ given pointwise by the left adjoint $L_n$. We can then observe that this oplax natural transformation is natural with respect to the idle morphism $\mathcal{I}\subset \Delta$ and thus we obtain an $\mathcal{I}$oplax natural transformation $\bcat{X}_\bullet \xRightarrow{}\scr{X}_\bullet$.

Let us consider a functor given by the composite
\[
  T_n:(\Delta_{/[n]})^\op \xlongrightarrow{} \Delta^\op \xlongrightarrow{\scr{X}_\bullet} \bcat{C}\!\on{at}_{(\infty,1)}
\]
where the first map is the canonical projection of the comma category $\Delta_{/[n]}$. We now turn our attention to the $\mathcal{I}$oplax natural transformation $\bcat{X}_\bullet \xRightarrow{}\scr{X}_\bullet$ and note that:
\begin{itemize}
  \item For every object $u:[k] \xlongrightarrow{} [n]$ in $(\Delta_{/[n]})^\op $ we have a functor
  \[
     \bcat{X}_n \xlongrightarrow{\bcat{X}(u)} \bcat{X}_k \xlongrightarrow{} \scr{X}_k=T_n(u).
   \] 
   where the final map is induced by the $\mathcal{I}$oplax natural transformation $\bcat{X}_\bullet \xRightarrow{}\scr{X}_\bullet$. 
   \item For every commutative diagram
   \[
     \begin{tikzcd}
       {[}k{]} \arrow[rr,"f"] \arrow[dr,swap,"u"] && {[}\ell{]} \arrow[dl,"v"] \\
       & {[}n{]} &
     \end{tikzcd}
   \]
   in $(\Delta_{/[n]})^\op$, we obtain a laxly commuting triangle 
  \[\begin{tikzcd}
  & {\bcat{X}_n} \\
  {\bcat{X}_{\ell}} && {\bcat{X}_k} \\
  {\scr{X}_\ell} && {\scr{X}_k}
  \arrow[from=1-2, to=2-1]
  \arrow[from=2-1, to=2-3]
  \arrow[from=1-2, to=2-3]
  \arrow[from=2-1, to=3-1]
  \arrow[from=3-1, to=3-3]
  \arrow[from=2-3, to=3-3]
  \arrow[Rightarrow, shorten <=18pt, shorten >=18pt, from=2-3, to=3-1]
\end{tikzcd}\]
which actually commutes whenever $f$ is idle.
\item In particular, we see that the data above assembles into an $\mathcal{I}$oplax natural transformation $\bcat{X}_n \xRightarrow{} T_n$ where we abuse notation and identify $\bcat{X}_n$ with its corresponding constant functor.
\end{itemize}

Given a functor $\bcat{X}_\bullet: \Delta^\op \xlongrightarrow{} \bcat{C}\!\on{at}_{(\infty,1)}$ we say that $\bcat{X}_\bullet$ is an \emph{enhanced Segal object} if:
\begin{enumerate}
  \item The functor $\bcat{X}_\bullet$ satisfies the oplax Segal condition.
  \item The $\mathcal{I}$oplax natural transformation $\bcat{X}_n \xRightarrow{} T_n$ exhibits $\bcat{X}_n$  as the $\mathcal{I}oplax$ limit (see \cite{GHL_LaxLim} for a definition of partially lax limits) of $T_n$.
  \item The nucleus $\scr{X}_\bullet$ of  $\bcat{X}_\bullet$ is a complete Segal object in $(\infty,1)$-categories. 
 \end{enumerate}

 We finally define $\ESeg$ as the $(\infty,2)$-category whose objects are given by enhanced Segal objects, morphisms are \emph{just} natural transformations between these functors and where the 2-morphisms are given by modifications. Let us point out that there is no special reason to require $\bcat{X}_0$ to be a space and we might aswell consider a slight variation of condition 3, which takes the form:
 \begin{itemize}
   \item[3'.]  The nucleus $\scr{X}_\bullet$ of  $\bcat{X}_\bullet$ is a double $(\infty,1)$-category (see \autoref{def:dblinftycat}).
 \end{itemize}
In this situation we obtain the notion of an \emph{enhanced double $(\infty,1)$-category}. In a totally analogous way as before, we assemble the collection of enhanced double $(\infty,1)$-categories into a $(\infty,2)$-category which we denote as $\on{EDbl}\left(\bcat{C}\!\on{at}_{(\infty,1)} \right)$. 

 The relevance of our definitions can be justified through the following result (\autoref{thm:enhanced}).
 \begin{thm*}\label{intro:enh}
There exists a commutative diagram of $(\infty,2)$-categories
   \[
  \begin{tikzcd}
     \on{CSeg}_\Delta(\bcat{C}\!\on{at}_{(\infty,1)})^{\mathcal{I}\!\on{oplax}} \arrow[d] \arrow[r,"\simeq"] & \ESeg  \arrow[d] \\
     \on{Dbl}\left(\bcat{C}\!\on{at}_{(\infty,1)} \right)^{\mathcal{I}\!\on{oplax}} \arrow[r,"\simeq"]  & \on{EDbl}\left(\bcat{C}\!\on{at}_{(\infty,1)} \right)
  \end{tikzcd}
\]
where the vertical maps are fully faithful and the horizontal morphisms are equivalences.
 \end{thm*}
  This means, that in the enhanced picture a natural transformation already encodes the data of an oplax normalised functor. Moreover, under this equivalence the honest functors in the right-hand side are enconded by  \emph{nuclear} maps:
\begin{itemize}
  \item A map $f:\bcat{X}_\bullet \xRightarrow{} \bcat{Y}_\bullet$ is called nuclear if it induces a morphism $\scr{X}_\bullet \xRightarrow{} \scr{Y}_\bullet$ between the corresponding nuclei.
\end{itemize}
The theorem above then yields a nuclear version.
\begin{cor*}\label{intro:enhanced}
    The commutative diagram  given in \autoref{intro:enh} restricts to a diagram
    \[
     \begin{tikzcd}
       \CSeg \arrow[r,"\simeq"] \arrow[d] & \ESegnuc \arrow[d] \\
       \on{Dbl}(\bcat{C}\!\on{at}_{(\infty,1)}) \arrow[r,"\simeq"] & \EDblnuc
     \end{tikzcd}
    \]
  where the vertical morphisms are fully faithful and horizontal morphisms are equivalences of $(\infty,2)$-categories.
  \end{cor*}

  We expect that these new models will be helpful when dealing with categorical constructions whose universal properties are defined in terms of lax functors. In future work, we will use the enhanced point of view to study:
  \begin{itemize}
    \item The span double $(\infty,1)$-category of an $(\infty,1)$-category $\scr{C}$ and show that it satisfies the expected universal property.
    \item The Morita $(\infty,2)$-category of a monoidal $(\infty,1)$-category $\scr{C}$ where also expect to be able to produce a conceptual proof of its universal property as studied in \cite{Morita}.
  \end{itemize}

\subsubsection{Partially lax right Kan extensions}
Let $f:\bcat{C} \xlongrightarrow{} \bcat{D}$ be a functor of $(\infty,2)$-categories. Given a collection of 1-morphisms $E$ in $\bcat{C}$ we can construct a functor via restriction
\[
  f^*: \Fun(\bcat{D},\bcat{C}\!\on{at}_{(\infty,1)}) \xlongrightarrow{} \Fun(\bcat{C},\bcat{C}\!\on{at}_{(\infty,1)})^{\on{Elax}}
\]
where the right-hand side denotes the $(\infty,2)$-category of functors, partially lax natural transformations according $E$ and modifications. Here by partially lax we mean those lax natural transformations such that for every morphism $u \in E$ the corresponding square commutes. At the technical core of this paper (\autoref{thm:rightlaxkan}), lies the theory of partially lax Kan extensions.
\begin{thm*}\label{intro:rightlaxkan}
  Let $(\bcat{C},E)$ be an $(\infty,2)$-category equipped with a collection of edges containing the equivalences and let $f:\bcat{C} \xlongrightarrow{} \bcat{D}$. Then there exists an adjunction of $(\infty,2)$-categories.
  \[
    f^*: \Fun(\bcat{D},\bcat{C}\!\on{at}_{(\infty,2)}) \llra   \Fun(\bcat{C},\bcat{C}\!\on{at}_{(\infty,2)})^{\on{Elax}}: f_*
  \]
  and similarly for the Eoplax case.
\end{thm*}
To produce this adjunction we use \autoref{thm:laxGroth} to pass to a fibrational model where the description of the functor $f_*$ can be explicitly obtained. We would like to point out that the general theory of partially lax extensions (including left Kan extensions) will soon appear in \cite{Kan}. 

Applying \autoref{intro:rightlaxkan} to the identity functor $\Delta^\op \xlongrightarrow{} \Delta^\op$ and the collection of edges $E=\mathcal{I}$ we obtain a functor
\[
  K_*: \Fun(\Delta^\op,\bcat{C}\!\on{at}_{(\infty,1)})^{\mathcal{I}\!\on{oplax}} \xlongrightarrow{} \Fun(\Delta^\op,\bcat{C}\!\on{at}_{(\infty,1)})
\]
whose restriction to $ \on{CSeg}_\Delta(\bcat{C}\!\on{at}_{(\infty,1)})^{\mathcal{I}\!\on{oplax}} $ induces the equivalence given in \autoref{intro:enh}.

\subsection{The enhanced square and globular functors}
Let $\bcat{C}$ be an $(\infty,2)$-category and let $\Fun^{\mathbf{opgr}}(\Delta^n_\flat,\bcat{C})^{\leq 1}$ be the $(\infty,1)$-category such that:
\begin{itemize}
  \item The objects are oplax normalised functors $[n] \xlongrightarrow{} \bcat{C}$.
  \item The morphisms are oplax natural transformations among those  oplax functors. 
  \item The 2-morphisms are given by modifications.
\end{itemize}
We further define $\Fun^{\mathbf{opgl}}(\Delta^n_\flat,\bcat{C})^{\leq 1}$ as the subcategory containing the same set of objects but whose morphisms are given by those oplax natural transformations $f \xRightarrow{} g$ such that for every $i \in [n]$, the component at $c$, $f(c) \xlongrightarrow{\simeq} g(c)$ is an equivalence. We call such oplax natural transformations: opglobular natural transformations.

We can assemble these definitions into functors 
\[
  \Sqe(\bcat{C})_\bullet: \Delta^\op \xlongrightarrow{} \bcat{C}\!\on{at}_{(\infty,1)}, \enspace \enspace [n] \mapsto \Sqe(\bcat{C})_n:=\Fun^{\mathbf{opgr}}(\Delta^n_\flat,\bcat{C})^{\leq 1}
\]
\[
  \EGlob(\bcat{C})_\bullet: \Delta^\op \xlongrightarrow{} \bcat{C}\!\on{at}_{(\infty,1)}, \enspace \enspace [n] \mapsto \EGlob(\bcat{C})_n:=\Fun^{\mathbf{opgl}}(\Delta^n_\flat,\bcat{C})^{\leq 1}
\]
which further organise themselves into functors of $(\infty,2)$-categories
\[
  \Sqe: \bcat{B}\!\on{icat}_\infty \xlongrightarrow{} \on{EDbl}\left(\bcat{C}\!\on{at}_{(\infty,1)} \right), \enspace \enspace \EGlob: \bcat{B}\!\on{icat}^{\mathbf{oplax}}_\infty \xlongrightarrow{} \ESeg.
\]
Here we are denoting by $\bcat{B}\!\on{icat}_\infty $  the $(\infty,2)$-category of $(\infty,2)$-categories, functors and natural transformations and by $\bcat{B}\!\on{icat}^{\mathbf{oplax}}_\infty $ the $(\infty,2)$-category of $(\infty,2)$-categories, oplax normalised functors and opglobular natural transformations. We will refer to $\Sqe$ as the enhanced square functor and to $\EGlob$ as the enhanced globular functor.

As we already mentioned, in \autoref{thm:maineglob} we prove that the enhanced globular functor is an equivalence of $(\infty,2)$-categories. Our main theorem  (\autoref{thm:sqefullyfaithful}) regarding the enhanced square functor is the following.

\begin{thm*}\label{intro:sqe}
  The enhanced square functor $\Sqe: \bcat{B}\!\on{icat}_\infty \xlongrightarrow{} \on{EDbl}(\bcat{C}\!\on{at}_{(\infty,1)})$ is fully faithful.
\end{thm*}

We consider another functor $\mathbf{Sq}:\bcat{B}\!\on{icat}_\infty \xlongrightarrow{} \on{Dbl}\left(\bcat{C}\!\on{at}_{(\infty,1)} \right)$ which sends every $(\infty,2)$-category $\bcat{C}$ to the nucleus of $\Sqe(\bcat{C})$. This functor is known as Ehresmann's quintet construction \cite{quintet} and it was claimed in \cite{GaitsgoryRozenblyum} that it was a fully faithful embedding of $(\infty,2)$-categories into double $(\infty,1)$-categories.

In \autoref{cor:ordinarysqe}, we use the previous result to settle this claim.

\begin{cor*}
  The ordinary square functor $\mathbf{Sq}: \bcat{B}\!\on{icat}_\infty  \xlongrightarrow{} \on{Dbl}(\bcat{C}\!\on{at}_{(\infty,1)})$ is fully faithful. Moreover, the composite
    \[
        \bcat{B}\!\on{icat}_\infty  \xlongrightarrow{\mathbf{Sq}} \on{Dbl}(\bcat{C}\!\on{at}_{(\infty,1)})\xlongrightarrow{} \on{Dbl}\left(\bcat{C}\!\on{at}_{(\infty,1)}\right)^{\mathcal{I}\!\on{oplax}}
       \]   
     where the second functor is the obvious inclusion is also fully faithful.  
\end{cor*}

The author was surprised to discover that an oplax normalised functor between $\mathbf{Sq}(\bcat{C}) \xRightarrow{}\mathbf{Sq}(\bcat{D})$ must be in fact, an honest functor. This points to the possibility of classifying the essential image of $\Sqe$ as certain enhanced double $(\infty,1)$-categories satisfying rigidity properties with respect to lax functors.

\subsubsection{Adjoints to the enhanced square functor}
In the strict 2-categorical world, it follows from the work of Ehresmann that the functor $\mathbf{Sq}$ admits a right adjoint $\mathbf{Comp}$, which is known as the companion functor. In \autoref{prop:coendleftquillen} we construct (at level of $(\infty,1)$-categories) a left adjoint to the functor $\Sqe$
\[
  \Sqe:\bcat{B}\!\on{icat}_\infty  \llra \on{EDbl}(\bcat{C}\!\on{at}_{(\infty,1)}): \int^{\mathbf{opgr}}_E
\]
but we \emph{do not} pursue the construction of an enhanced version of the functor $\mathbf{Comp}$. The study of the square-companion adjunction is the subject of upcoming work by Ozornova, Rovelli and Ruit \cite{comp} where an independent proof of the fully faithfulness of the functor $\mathbf{Sq}$ is derived as a corollary from the adjunction $\mathbf{Sq} \dashv \mathbf{Comp}$.

 \subsection{Organization of the paper} 
 In Section 2 we review the main definitions of the theory of marked-scaled (resp. scaled) simplicial sets and complete Segal objects in $(\infty,1)$-categories. Exploiting the comparison of models provided by Lurie in \cite{LurieGoodwillie} we prove in \autoref{thm:model} that the functor $\mathbf{Gl}$ that sends a fibrant marked-scaled simplicial set $\bcat{C}$ to the nucleus of $\EGlob(\bcat{C})$ defines an equivalence of models for $(\infty,2)$-categories.

In Section 3 we prove some key combinatorial results (\autoref{thm:bigcombinatorial}) which are fundamental to the proofs of of fully faithfulness of our functors $\Sqe$ and $\EGlob$. 

Section 4 is devoted to the study of partially lax (right) Kan extensions and the theory of enhanced Segal objects (resp. enhanced double $(\infty,1)$-categories).

In Section 5 and Section 6 we study the functors $\Sqe$ and $\EGlob$ and provide proofs of the theorems stated in this section.

\subsection{Acknowledgements}
I would like to thank R. Haugseng for giving me the idea of working on this project and W. Stern for his helpful suggestions regarding the exposition of the paper which greatly improved this document.

\newpage

 \section{Preliminaries}
 In this section we will mainly gather the main definitions of the theory of scaled simplicial sets (resp. marked-scaled) and the theory of complete Segal objects in $(\infty,1)$-categories as presented by Lurie in \cite{LurieGoodwillie}. In \autoref{thm:compari} we prove a new comparison result between both models for $(\infty,2)$-categories.

 \subsection{Scaled simplicial sets}

\begin{definition}\label{def:simp}
  We define the simplex category $\Delta$, as the full subcategory of the 1-category of categories $\on{Cat}$, spanned by the linearly ordered sets $[n]=\{0<1<2<\dots<n\}$.
\end{definition}

\begin{definition}\label{def:scaledsimplicialset}
  A scaled simplicial set $(X,T_X)$ consists in a simplicial set $X$ together with a collection of $2$-simplices (also called triangles) $T_X$ which contains \emph{every degenerate} triangle. We call the elements of $T_X$ the \emph{thin triangles} of $X$. A morphism of scaled simplicial sets $f:(X,T_X) \xlongrightarrow{} (Y,T_Y)$ is a map of simplicial sets $f: X \xlongrightarrow{} Y$ such that $f(T_X)\subseteq T_Y$. We denote the corresponding category of scaled simplicial sets by $\on{Set}_\Delta^{\mathbf{sc}}$.
\end{definition}

\begin{remark}
  When no confusion shall arise we will omit the pair notationand simply denote the scaled simplicial  $(X,T_X)$ set as  $X$.
\end{remark}

\begin{notation}
  Given a simplicial set $A$ we have two canonical ways of viewing it as a scaled simplicial set:
  \begin{itemize}
    \item We define $A_\flat=(A,\flat)$ where $\flat$ is the collection consisting only in the degenerate triangles of $A$.
    \item We define $A_\sharp=(A,\sharp)$ where $\sharp$ is the collection consisting in every triangle of $A$.
  \end{itemize}
\end{notation}

  \begin{definition}\label{def:scanodyne}
  The set of \emph{generating scaled anodyne maps} \(\sS\) is the set of maps of scaled simplicial sets consisting of:
  \begin{enumerate}
    \myitem{S1)}\label{item:anodyne-inner} the inner horns inclusions
    \[
    \bigl(\Lambda^n_i,\{\Delta^{\{i-1,i,i+1\}}\}\bigr)\xlongrightarrow{} \bigl(\Delta^n,\{\Delta^{\{i-1,i,i+1\}}\}\bigr)
    \quad , \quad n \geq 2 \quad , \quad 0 < i < n ;
    \]
    \myitem{S2)}\label{i:saturation} the map 
    \[
    (\Delta^4,T)\xlongrightarrow{} (\Delta^4,T\cup \{\Delta^{\{0,3,4\}}, \ \Delta^{\{0,1,4\}}\}),
    \]
    where we define
    \[
    T\overset{\text{def}}{=}\{\Delta^{\{0,2,4\}}, \ \Delta^{\{ 1,2,3\}}, \ \Delta^{\{0,1,3\}}, \ \Delta^{\{1,3,4\}}, \ \Delta^{\{0,1,2\}}\};
    \]
    \myitem{S3)}\label{item:anodyne_outer} the set of maps
    \[
    \Bigl(\Lambda^n_0\coprod_{\Delta^{\{0,1\}}}\Delta^0,\{\Delta^{\{0,1,n\}}\}\Bigr)\xlongrightarrow{} \Bigl(\Delta^n\coprod_{\Delta^{\{0,1\}}}\Delta^0,\{\Delta^{\{0,1,n\}}\}\Bigr)
    \quad , \quad n\geq 3.
    \]
  \end{enumerate}
  A general map of scaled simplicial set is said to be \emph{scaled anodyne} if it belongs to the weakly saturated closure of \(\sS\).
\end{definition}

\begin{definition}
  A scaled simplicial set $X$ is said to be an $\infty$-bicategory if it has the right lifting property again the scaled of scaled anodyne maps in \autoref{def:scanodyne}. In this case, we view the 2-simplices of $T_X$ as the collection of commuting triangles.
\end{definition}

\begin{definition}
  We denote by $\Cat_\Delta^+$ the category of $\on{Set}_\Delta^+$-enriched categories  (i.e. categories enriched in marked simplicial sets). We note that we can view the category of (strict) 2-categories $2\Cat$ as a full subcategory of $\Cat_\Delta^+$ by applying the nerve functor Hom-wise and marking the equivalences in each mapping category.
\end{definition}

\begin{definition}\label{def:OI}
  Let $I$ be a linearly ordered finite set. We define a $2$-category $\OO^{I}$ as
  follows
  \begin{itemize}
    \item the objects of $\OO^I$ are the elements of $I$,
    \item the category $\mathbb{O}^{I}(i,j)$ of morphisms between objects $i,j \in I$ is defined
    as the poset of finite sets $S \subseteq I$ such that $\min(S)=i$ and $\max(S)=j$
    ordered by inclusion,
    \item the composition functors are given, for $i,j,l\in I$, by
    \[
    \mathbb{O}^{I}(i,j) \times \mathbb{O}^{I}(j,l) \xlongrightarrow{} \mathbb{O}^{I}(i,l), \quad (S,T) \mapsto S \cup T.
    \]
  \end{itemize}
  When $I=[n]$, we denote $\OO^I$ by $\OO^n$. We will refer to the 2-categories $\OO^n$ for $n\geq 0$ as the (2-dimensional) orientals following the original terminology of \cite{street}.
\end{definition}

\begin{definition}\label{def:rigidification}
  The map
  \[
  \begin{tikzcd}
    \Delta \arrow[r,"\mathfrak{C}"] & {\Cat_\Delta^+} 
  \end{tikzcd}
  \]
  which sends $[n]$ to $\OO^n$ gives us a cosimplicial object in $\Cat_\Delta^+$. We can moreover send the thin 2-simplex $\Delta^2_\sharp$ to $\mathfrak{C}[\Delta^2]$ equipped with maximally-marked mapping spaces. The usual machinery of nerve and realization then gives us adjoint functors  
  \[
  \begin{tikzcd}
    \mathfrak{C}^{\sc}: &[-3em] \scsSet \arrow[r,shift left] & \Cat_\Delta^+ \arrow[l,shift left] &[-3em] :\Nsc 
  \end{tikzcd}
  \]
  which we will call the \emph{scaled nerve} and \emph{scaled rigidification}. 
\end{definition}

\begin{theorem}\label{thm:scaledmodel}
  There is a left proper, combinatorial model structure on $\scsSet$ with
  \begin{itemize}
    \item[W)] The weak equivalences are the morphisms $f:A\xlongrightarrow{} B$ such that $\mathfrak{C}^{\sc}[f]:\mathfrak{C}^{\sc}[A] \xlongrightarrow{} \mathfrak{C}^{\sc}[B]$ is an equivalence in $\Cat_\Delta^+$. 
    \item[C)] The cofibrations are the monomorphisms. 
  \end{itemize}
  Moreover, the fibrant objects in this model structure are the $\infty$-bicategories, and the adjunction 
  \[
  \begin{tikzcd}
    \mathfrak{C}^{\sc}: &[-3em] \scsSet \arrow[r,shift left] & \Cat_\Delta^+ \arrow[l,shift left] &[-3em] :\Nsc 
  \end{tikzcd}
  \]
  is a Quillen equivalence.
\end{theorem}
\begin{proof}
  This is \cite[Thm A.3.2.4]{LurieGoodwillie}. The characterization of fibrant objects is \cite[Thm 5.1]{GHL_Equivalence}.
\end{proof}

\begin{definition}
  We say that a map of scaled simplicial sets is a \emph{bicategorical equivalence} if it is a weak equivalence in the model structure given in \autoref{thm:scaledmodel}. Similarly, call the fibrations in the model structure of scaled simplicial sets \emph{bicategorical fibrations}.
\end{definition}

\begin{definition}\label{def:funsc}
  Given a pair of scaled simplicial sets $X,Y$ we denote by $\on{Fun}^{\mathbf{sc}}(X,Y)$ the scaled simplicial set determined by the universal property
  \[
    \Hom_{\on{Set}_\Delta^{\mathbf{sc}}}(K,\on{Fun}^{\mathbf{sc}}(X,Y))\isom  \Hom_{\on{Set}_\Delta^{\mathbf{sc}}}(K \times X, Y)
  \]
  where $K\times X$ denotes the Cartesian product of scaled simplicial sets.
\end{definition}

\begin{definition}\label{def:mappingcat}
  Let $\bcat{C}$ be an $\infty$-bicategory. Given an object $y \in \bcat{C}$ we define a scaled simplicialset $\bcat{C}_{\upslash y}$  as follows. The data of an $n$-simplices $\Delta^n \xlongrightarrow{} \bcat{C}_{\upslash y}$ is given by a map $\sigma:\Delta^{n+1} \xlongrightarrow{} \bcat{C}$ such that $\sigma(n+1)=y$. The inclusion $d_{n+1}: \Delta^n \xlongrightarrow{} \Delta^{n+1}$ induces a map 
  \[
    \pi:\bcat{C}_{\upslash y} \xlongrightarrow{} \bcat{C}
  \]
  which we use to declare a triangle in $\bcat{C}_{\upslash y}$ to be thin if and only if its image unde $\pi$ is thin in $\bcat{C}$. It follows from \cite[Prop. 2.33]{GHL_Equivalence} that the fibre of $\pi$ at an object $x \in \bcat{C}$ is a model for $\bcat{C}(x,y)$, the mapping $(\infty,1)$-category.
\end{definition}

\subsubsection{Marked-scaled simplicial sets}
For practical reasons it will be convenient to work with another model for $(\infty,2)$-categories consisting in scaled simplicial sets equipped with a collection of distinguished edges which model the equivalences inside an $(\infty,2)$-category. The results exposed here can be found in Section 3.1 in \cite{AbGroth}.
\begin{definition}\label{def:msSet}
  A marked-scaled simplicial set denoted by $(X,E_X,T_X)$ is given by
  \begin{enumerate}
    \item A simplicial set $X$.
    \item A collection of edges $E_X \subseteq X_1$ which contains \emph{all degenerate} edges. We refer to the elements of this collection as \emph{marked} edges.
    \item A collection of triangles $T_X \subseteq X_2$ which contains $\emph{all degenerate}$ triangles. We refer to the elements of this collection as \emph{thin} triangles.
  \end{enumerate}
  A morphism of marked-scaled simplicial sets $f:(X,E_X,T_X) \xlongrightarrow{} (Y,E_Y,T_Y)$ is given by a map of simplicial sets such that $f(E_X) \subseteq E_Y$ and $f(T_X) \subseteq T_Y$. We denote by $\on{Set}_\Delta^{\mathbf{ms}}$ the category of marked-scaled simplicial sets.
\end{definition}

\begin{definition}\label{def:msanodyne}
  The set of \emph{generating marked-scaled anodyne maps} $\MS$ is the set of maps of marked-scaled simplicial sets consisting of
  \begin{enumerate}
    \myitem{(M1)}\label{ms:innerhorn} The inner horn inclusions 
    \[
    \bigl(\Lambda^n_i,\flat, \{\Delta^{\{i-1,i,i+1\}}\}\bigr)\xlongrightarrow{} \bigl(\Delta^n,\flat, \{\Delta^{\{i-1,i,i+1\}}\}\bigr).
    \quad , \quad n \geq 2 \quad , \quad 0 < i < n ;
    \]
    
    \myitem{(M2)}\label{ms:wonky4} The map 
    \[
    (\Delta^4,\flat, T)\xlongrightarrow{} (\Delta^4,\flat, T\cup \{\Delta^{\{0,3,4\}}, \ \Delta^{\{0,1,4\}}\}),
    \]
    where we define
    \[
    T\overset{\text{def}}{=}\{\Delta^{\{0,2,4\}}, \ \Delta^{\{ 1,2,3\}}, \ \Delta^{\{0,1,3\}}, \ \Delta^{\{1,3,4\}}, \ \Delta^{\{0,1,2\}}\};
    \]
    \myitem{(M3)}\label{ms:2coCartesianmorphs} The set of maps
    \[
    \Bigl(\Lambda^n_0,\{\Delta^{\{0,1\}}\}, \{ \Delta^{\{0,1,n\}} \}\Bigr) \xlongrightarrow{} \Bigl(\Delta^n,\{\Delta^{\{0,1\}}\}, \{ \Delta^{\{0,1,n\}} \}\Bigr) \quad , \quad n \geq 2.
    \]
    
    \myitem{(M4)}\label{mb:2coCartliftsExist} The inclusion of the initial vertex
    \[
    \Bigl(\Delta^{0},\sharp,\sharp \Bigr) \xlongrightarrow{} \Bigl(\Delta^1,\sharp,\sharp \Bigr).
    \]
    \myitem{(MS1)}\label{ms:composeacrossthin} The map
    \[
    \Bigl(\Delta^2,\{\Delta^{\{0,1\}}, \Delta^{\{1,2\}}\},\sharp \Bigr) \xlongrightarrow{} \Bigl(\Delta^2,\sharp,\sharp \Bigr).
    \]
    
    \myitem{(ME)}\label{ms:equivalences} For every Kan complex $K$, the map
    \[
    \Bigl( K,\flat,\sharp  \Bigr) \xlongrightarrow{} \Bigl(K,\sharp, \sharp\Bigr).
    \]
    Which requires that every equivalence is a marked morphism.
  \end{enumerate}
  A map of $\MS$ simplicial sets is said to be \MS-anodyne if it belongs to the weakly saturated closure of \MS.
\end{definition}

\begin{remark}
  Observe that $(X,E_X,T_X)$ has the right lifting property against the class of $\MS$-anodyne morphisms if and only if $X$ is an $\infty$-bicategory, $E_X$ is the collection of equivalences in $X$ and $T_X$ is the collection of thin triangles. Consequently, we might call such marked-scaled simplicial sets $\infty$-bicategories as well.
\end{remark}

\begin{remark}\label{rem:saturation}
  Let $0<i<3$ and let $U_i$ be the collection of all triangles in $\Delta^3$ except the face skipping the vertex $i$. 
  \begin{enumerate}
    \myitem{(MSI)}\label{ms:saturation} Then the map,
    \[
      (\Delta^3,\flat,U_i) \xlongrightarrow{} (\Delta^3,\flat,\sharp)
    \]
  \end{enumerate}
  fits into a pushout diagram
  \[
    \begin{tikzcd}
      (\Delta^4,\flat, T) \arrow[d] \arrow[r] & (\Delta^3,\flat,U_i) \arrow[d] \\
      (\Delta^4,\flat, T\cup \{\Delta^{\{0,3,4\}}, \ \Delta^{\{0,1,4\}}\}) \arrow[r] & (\Delta^3,\flat,\sharp)
    \end{tikzcd}
  \]
  and is consequently $\mathbf{MS}$-anodyne.  
\end{remark}

\begin{definition}\label{def:cofms}
  We say that a morphism of marked-scaled simplicial sets is a cofibration if its underlying map of simplicial sets is a cofibration. The class of cofibration is generated by the following families of maps:
\begin{enumerate}
  \myitem{(C1)}\label{cof:c1} The maps $(\partial \Delta^n,\flat,\flat) \xlongrightarrow{} (\Delta^n,\flat,\flat)$ for $n\geq 0$.
  \myitem{(C2)}\label{cof:c2} The map $(\Delta^1,\flat,\sharp) \xlongrightarrow{} (\Delta^1,\sharp,\sharp)$.
  \myitem{(C3)}\label{cof:c3} The map $(\Delta^2,\flat,\flat) \xlongrightarrow{} (\Delta^2,\flat,\sharp)$.
\end{enumerate}

\end{definition}

\begin{proposition}\label{prop:msCart}
  Let $f:X \xlongrightarrow{}  Y$ be a cofibration in $\on{Set}_\Delta^{\mathbf{ms}}$ and $g: A \xlongrightarrow{}  B$ be an $\MS$-anodyne morphism. Then the pushout-product
  \[
    X \times B \coprod_{X \times A} Y \times A \xlongrightarrow{} Y \times B
  \]
  is again $\MS$-anodyne.
\end{proposition}

\begin{corollary}
  Given a marked-biscaled simplicial set $X$ which has the right lifting property against the class of $\MS$-anodyne morphisms it follows that $\on{Fun}^{\mathbf{ms}}(A,X)$ (compare to \autoref{def:funsc}) has the right lifting property against the class $\MS$-anodyne morphisms for every $A \in \on{Set}_\Delta^{\mathbf{ms}}$.
\end{corollary}

\begin{definition}
  A morphism of marked-scaled simplicial sets $i:A \xlongrightarrow{}  B$ is said to be a \emph{weak equivalence} if for every $\infty$-bicategory $X$ the induced map
  \[
    i^*: \on{Fun}^{\mathbf{ms}}(B,X) \xlongrightarrow{\simeq} \on{Fun}^{\mathbf{ms}}(A,X)
  \]
  is a bicategorical equivalence.
\end{definition}

\begin{remark}\label{rem:mstensor}
  Given a marked simplicial set $K \in \on{Set}_\Delta^+$ and $X \in \on{Set}_\Delta^{\mathbf{ms}}$ we define the tensor $K \times X:=I(K) \times X$ where $I(K)=(K,E_K,\sharp)$. Similarly, we define the cotensor $X^{K}:= \on{Fun}^{\mathbf{ms}}(I(K),X)$.
\end{remark}

\begin{theorem}\label{thm:markedscaledmodel}
  There exists a left proper combinatorial simplicial model category structure on $\on{Set}_\Delta^{\mathbf{ms}}$, which is characterized uniquely by the following properties:
  \begin{itemize}
    \item[C)] A morphism $f:X \xlongrightarrow{}  Y$ in  $\on{Set}_\Delta^{\mathbf{ms}}$ us a cofibration if and only if $f$ induces a monomorphism on the underlying simplicial set.
    \item[F)] An object $X$ in $\on{Set}_\Delta^{\mathbf{ms}}$ is fibrant if and only if it was the right lifting property against the class of marked-anodyne morphisms.
  \end{itemize}
  Moreover the tensor and cotensor in \autoref{rem:mstensor} equipps $\on{Set}_\Delta^{\mathbf{ms}}$ with the structure of a $\on{Set}_\Delta^+$-enriched model category.
\end{theorem}

\begin{theorem}\label{thm:scms}
  The functor $L: \on{Set}_\Delta^{\mathbf{sc}} \xlongrightarrow{}  \on{Set}_\Delta^{\mathbf{ms}}$ sending a scaled simplicial set $(X,T_X)$ to the marked-scaled simplicial set $(X,\flat,T_X)$ is a left Quillen equivalence.
\end{theorem}

\subsubsection{The Gray tensor product of marked-scaled simplicial sets}
In this section we collect the main definitions of the Gray tensor product of $(\infty,2)$-categories as developed in \cite{GHL_Gray} in addition to its universal property in terms of oplax normalised functors. We will adapt the original definitions in \cite{GHL_Gray} to the setting of marked-scaled simplicial sets and derive the expected properties of the Gray tensor product in the marked-scaled world. We would to point out that our definitions are essentially the same as the Gray tensor product of 2-complicial sets given in \cite{Verity}.

\begin{definition}\label{def:gray}
  Let $X,Y \in \on{Set}_\Delta^{\mathbf{ms}}$ we define the \emph{Gray tensor product} $X \tensor Y \in \on{Set}_\Delta^{\mathbf{ms}}$ as follows:
  \begin{enumerate}
   \item The underlying simplicial set of $X \tensor Y$ is given by the usual Cartesian product of the underlying simplicial sets $X \times Y$.
   \item A 2-simplex $\sigma: \Delta^2 \xlongrightarrow{}  X \tensor Y$ is declared to be thin in the Gray tensor product if and only if the following holds:
   \begin{itemize}
      \item[i)] The simplex $\sigma$ is thin in \emph{both} $X$ and $Y$.
      \item[ii)] The the restriction of $\sigma$ to $\Delta^{\{0,1\}}$ is marked in $Y$ \emph{or} the restriction of $\sigma$ to $\Delta^{\{1,2\}}$ is marked in $X$.
    \end{itemize} 
   \item An edge in $\Delta^1:e \xlongrightarrow{}  X \tensor Y$ is declared to be marked if it is marked in \emph{both} $X$ and $Y$. 
  \end{enumerate}
\end{definition}

\begin{remark}\label{rem:failurecolimit}
  The Gray tensor product given in \autoref{def:gray} is associative but it \emph{fails} to preserve colimits in each variable. To illustrate this, we look at the pushout diagram of marked-scaled simplicial 
  \[
    \begin{tikzcd}
      (\Delta^1,\flat, \sharp) \arrow[d] \arrow[r,"d_0"] & (\Delta^2,\flat,\sharp) \arrow[d] \\
      (\Delta^1,\sharp,\sharp) \arrow[r] & (\Delta^2,\Delta^{\{1,2\}},\sharp)
    \end{tikzcd}
  \]
 and consider the Gray tensor product $(\Delta^2,\Delta^{\{1,2\}},\sharp) \tensor (\Delta^1,\flat,\sharp)$. Then it follows that we have a map
  \[
    \iota: (\Delta^1,\sharp, \sharp) \tensor  (\Delta^1,\flat, \sharp)  \coprod_{ (\Delta^1,\flat, \sharp) \tensor (\Delta^1,\flat,\sharp)} (\Delta^2,\flat,\sharp) \tensor (\Delta^1,\flat,\sharp) \xlongrightarrow{} (\Delta^2,\Delta^{\{1,2\}},\sharp) \tensor (\Delta^1,\flat,\sharp)
  \]
  which in is an isomorphism of the underlying marked simplicial sets but where the scalings differ. Fortunately, this failure is just cosmetic as seen in the next proposition.
\end{remark}

\begin{proposition}\label{prop:bothgraycoincide}
  Let $\overline{\tensor}:\on{Set}_\Delta^{\mathbf{ms}} \times \on{Set}_\Delta^{\mathbf{ms}} \xlongrightarrow{} \on{Set}_\Delta^{\mathbf{ms}}$ be the colimit preserving functor whose values restricted the marked simplicial sets $(\Delta^n,\flat,\flat), (\Delta^1,\sharp,\sharp)$ and $(\Delta^2,\flat,\sharp)$ coincides with the Gray tensor product $\tensor$ given in \autoref{def:gray}. Then there exists a natural transformation $\xi:( \mathblank \overline{\tensor} \mathblank) \xRightarrow{}( \mathblank \tensor \mathblank )$ which is a pointwise trivial cofibration of marked-scaled simplicial sets.
\end{proposition}
\begin{proof}
  Given marked-scaled simplicial sets $(A,E_A,T_A)$ and $(B,E_B,T_B)$ and a 2-simplex $\sigma:(\sigma_A,\sigma_B):\Delta^2 \xlongrightarrow{} A \overline{\tensor} B$ it follows that $\sigma$ is thin precisely when:
  \begin{itemize}
    \item Both 2-simplices $\sigma_A$ and $\sigma_B$ are thin in $A$ and $B$ respectively.
    \item The edge $\sigma_A( 1 \to 2)$ belongs to $E_A$ and the triangle $\sigma_A$ is degenerate \emph{or} the edge $\sigma_B(0 \to 1)$ belongs to $E_B$ and triangle $\sigma_B$ is degenerate. 
  \end{itemize}
 Since both $A \overline{\tensor} B$ and $A \tensor B$ only differ on the scaling, we see that we have a canonical map $A \overline{\tensor} B \xlongrightarrow{} A \tensor B$ yielding the desired natural transformation.

 We will show that we can scale the remaining 2-simplices in $A \btensor B$ using pushouts along anodyne morphisms. Let $(\theta_A,\theta_B): \Delta^2_\sharp \xlongrightarrow{} A \tensor B$ be a thin 2-simplex and suppose that $\theta_A(1 \xlongrightarrow{} 2)$ belongs to $E_A$. We look at the map of marked-scaled simplicial sets $\rho=(s_2\circ \theta_A,s_1\circ \theta_B)$ which we represent as
 \[
   (a_0,b_0) \xlongrightarrow{} (a_1,b_1) \xlongrightarrow{} (a_2,b_1) \xlongrightarrow{} (a_2,b_2)
 \]
 We observe that the only faces of $\rho$ which are not already scaled in $A \btensor B$ are precisely $d_3(\rho)$ and $d_2(\rho)=\theta$. If we show that we can scale $d_3(\rho)$ the claim will follow after taking a pushout along a morphism of type \ref{ms:wonky4} in \autoref{def:msanodyne}. Finally, we consider another 3-simplex $\varphi$
 \[
   (a_0,b_0) \xlongrightarrow{} (a_1,b_0) \xlongrightarrow{} (a_1,b_1) \xlongrightarrow{} (a_2,b_1)
 \]
and observe that every face in $\varphi$ is already scaled in $A \btensor B$ except $d_1(\varphi)=d_3(\rho)$ and thus the claim holds by taking a pushout along a morphism of type \ref{ms:wonky4} in \autoref{def:msanodyne}. 

The remaining case where $\theta_B(0 \xlongrightarrow{} 1)$ belongs to $E_B$ is totally analogous and left as an exercise.
 
\end{proof}

\begin{lemma}\label{lem:0saturations}
  The following morphisms are $\mathbf{MS}$-anodyne:
  \begin{itemize}
    \item[i)] The map $(\Delta^3,\Delta^{\{0,1\}},U_0) \xlongrightarrow{}  (\Delta^3,\Delta^{\{0,1\}},\sharp)$ where $U_0$ contains all triangles in $\Delta^3$ except the 2-simplex $\Delta^{\{1,2,3\}}$.
    \item[ii)] The map $(\Delta^2, \{\Delta^{\{0,1\}},\Delta^{\{0,2\}}\},\sharp ) \xlongrightarrow{}  (\Delta^2,\sharp,\sharp)$.
  \end{itemize}
\end{lemma}
\begin{proof}
  We prove the case $i)$ and leave the remaining (easier) verification as an exercise. We use the small object argument to produce a factorization
  \[
    (\Delta^3,\Delta^{\{0,1\}},U_0)  \xlongrightarrow{i} (X,E_X,T_X) \xlongrightarrow{p} (\Delta^3,\Delta^{\{0,1\}},\sharp)
  \]
  where the first morphism is $\mathbf{MS}$-anodyne and the second morphism has the right lifting property (RLP) against the class of marked-scaled anodyne morphisms.

  It follows from the fact that $(\Delta^3,\Delta^{\{0,1\}},\sharp)$ is maximallly scaled that $(X,T_X)$ is an $\infty$-bicategory. We claim that we can produce the dotted arrow
  \[
    \begin{tikzcd}
      (\Delta^3,\Delta^{\{0,1\}},U_0) \arrow[r,"i"] \arrow[d] & (X,E_X,T_X) \arrow[d,"p"] \\
      (\Delta^3,\Delta^{\{0,1\}},\sharp) \arrow[r,"="] \arrow[ur,dotted] & (\Delta^3,\Delta^{\{0,1\}},\sharp)
    \end{tikzcd}
  \]
  As a consequence of the claim, we see that we can express the left-most vertical morphism in the diagram above as a retract of a $\mathbf{MS}$-anodyne morphism and thus the result will follow.

  To prove the claim, one only needs to show that the triangle $\Delta^{\{1,2,3\}}$ gets mapped under $i$ to a thin triangle. Since $(X,T_X)$ is an $\infty$-bicategory this ammounts to show that the associated 2-morphism in the mapping $(\infty,1)$-category $X(i(1),i(3))$ is an equivalence. The fact that $p$ has the RLP with respect to the morphisms of type \ref{ms:2coCartesianmorphs} shows that $i(0\xlongrightarrow{}  1)=e$ is a coCartesian edge in $X$. We conclude that a 2-morphism  $X(i(1),i(3))$ is invertible if and only if it is invertible after precomposing with $e$. The result follows.
\end{proof}

\begin{proposition}\label{prop:grpushout}
  Let $f: X \xlongrightarrow{}  Y$ be cofibration in $\on{Set}_\Delta^{\mathbf{ms}}$ and let $g:A \xlongrightarrow{} B$ be a $\mathbf{MS}$-anodyne morphism. Then the pushout-product maps
  \[
    f \wedge_{\on{gr}} g: X \tensor B \coprod\limits_{X \tensor A} Y \tensor A \xlongrightarrow{} Y \tensor B
  \]
  \[
    g \wedge_{\on{gr}} f : B \tensor X \coprod\limits_{A \tensor X} A \tensor Y \xlongrightarrow{} B \tensor Y
  \]

  are trivial cofibrations and the same holds by replacing $\tensor$ with $\btensor$
\end{proposition}
\begin{proof}
 For the sake of brevity we include only the case $f \wedge_{\on{gr}} g$ leaving the analogous computations for $g \wedge_{\on{gr}} f$ as an exercise. We will however comment some instances where both cases differ.

 Note that by \autoref{prop:bothgraycoincide} it suffices to verify the claim for only one of the versions of the Gray tensor product given. Since $\btensor$ preserves colimits it will be enough to check the claim on the set of generating cofibrations given in \autoref{def:cofms} and the set of generating $\textbf{MS}$-anodyne morphisms. Again, by \autoref{prop:bothgraycoincide} it will be enough to verify the claim for the generators using $\tensor$. We proceed by cases:
 \begin{itemize}
   \item \textbf{The cofibration is of type:} \ref{cof:c1}
   \begin{itemize}
      \item[\ref{ms:innerhorn}]
      Since both marked-scaled simplicial sets are minimally marked it will be enough to work with the underlying scaled simplicial sets. The claim now follows from Proposition 2.16 in \cite{GHL_Gray}.
      \item[\ref{ms:wonky4}]
      Same as above.
      \item[\ref{ms:2coCartesianmorphs}]
      Let $f:\partial \Delta^n \xlongrightarrow{} \Delta^n$ and $g:\Lambda^m_0 \xlongrightarrow{} \Delta^m$ (with the corresponding decorations). Then it is clear from the definition that every marked edge in $\Delta^n \tensor \Delta^m$ must factor through pushout in the definition of $f \wedge_{\on{gr}} g$. The claim follows mutatis mutandis from \cite[Proposition 2.16]{GHL_Gray} by applying the proof for the scaled anodyne map \ref{item:anodyne_outer} in \autoref{def:scanodyne} and replacing the degenerate edge by the marked edge.
      \item[\ref{mb:2coCartliftsExist}]
      Since $g$ is maximally marked our Gray tensor product degenerates to the usual Cartesian product and the result follows from \autoref{prop:msCart}.
      \item[\ref{ms:composeacrossthin}]
      Let $f:\partial \Delta^n \xlongrightarrow{} \Delta^n$ and observe that $f \wedge_{\on{gray}} g$ is an isomorphism if $n \geq 2.$ For the case $n=1$ one easily verifies that the map in question is in the weakly saturated class of morphisms of  type \ref{ms:saturation} in \autoref{rem:saturation}.
      \item[\ref{ms:equivalences}] 
      The associated map is an isomorphism on the underlying simplicial sets. Moreover, both marked-scaled simplicial sets have the same collection of marked edges. We will show that we can add the remaining scaled triangles using pushouts along $\mathbf{MS}$-anodyne morphisms. The only case we need to show is the case $f: \partial \Delta^1 \xlongrightarrow{} \Delta^1$ since in the remaining cases, the map $f\wedge_{\on{gr}} g$ is an isomorphism. 

      Let $g:K_\flat=(K,\flat,\sharp) \xlongrightarrow{} K_{\sharp}=(K,\sharp,\sharp)$. Given an edge $e: \Delta^1 \xlongrightarrow{} K$ we will show that we can scale all the triangles in the image of the map 
      \[
         \on{id} \tensor e : \Delta^1 \tensor \Delta^1 \xlongrightarrow{} \partial \Delta^0 \otimes K_{\sharp} \coprod\limits_{\partial \Delta^0 \tensor K_{\flat}} \Delta^1 \tensor K_\flat = \mathcal{Q}
       \] 
       Let $\sigma:\Delta^3 \xlongrightarrow{} K$ be a $3$-simplex exhibiting $e$ as an invertible morphism, such that $\sigma(0 \xlongrightarrow{} 2)$ and $\sigma(1 \xlongrightarrow{} 3)$ are degenerate edges and such that $d_2(\sigma)=s_0(e)$. We consider a 4-simplex $\rho: \Delta^4 \xlongrightarrow{} \mathcal{Q}$ (equipped with the induced scaling) given by 
       \[
         \rho(i)=\begin{cases}
           (0,\sigma(i)), \enspace \text{ if }i< 4 \\
           (1,\sigma(3)), \enspace \text{else}.
         \end{cases}
       \]
       It then follows that we can use a morphism of type \ref{ms:wonky4} to scale the triangles $\{0,3,4\}$ and $\{0,1,4\}$ in $\rho$ which proves our claim. We can thus scale all triangles of the form $\on{id}\tensor e$ and obtain a factorization
       \[
         \mathcal{Q} \xlongrightarrow{\simeq} \mathcal{P} \xlongrightarrow{} \Delta^1 \tensor K_\sharp.
       \]
       where the first morphism is in the weakly saturated class of morphisms of type \ref{ms:wonky4}. To finish the proof it suffices to show that for every 2-simplex $\alpha: \Delta^2 \xlongrightarrow{} K$ we can scale every simplex in the image of the map
       \[
         \on{id} \tensor \alpha: \Delta^1 \tensor \Delta^2 \xlongrightarrow{} \mathcal{P}
       \]
       this follows immediately by using pushouts along morphisms of type \ref{ms:saturation} in \autoref{rem:saturation}

       We would like to point out that for the case $g \wedge_{\on{gray}}f$ the case is essentially the same except for the fact that we need to use the opposite of the morphism \ref{ms:wonky4} which is a priori not $\mathbf{MS}$-anodyne but is obviously a trivial cofibration.
    \end{itemize} 
    \item \textbf{The cofibration is of type:} \ref{cof:c2}
    In this case, the map $f \wedge_{\on{gr}} g$ is an isomorphism on the underlying simplicial sets and we only need that we can add the pertinent decorations using pushouts along trivial cofibrations. For every $\mathbf{MS}$-anodyne morphism $g: A \xlongrightarrow{} B$ morphism we define
    \[
      \Delta^1 \boxplus  B = (\Delta^1,\flat,\sharp) \tensor B \coprod_{\Delta^1 \otimes A} (\Delta^1,\sharp,\sharp) \otimes A
    \]
    and similarly $B \boxplus \Delta^1$
   \begin{itemize}
      \item[\ref{ms:innerhorn}]
      Since the map $g$ is minimally marked it follows that the marking on both marked-scaled simplicial sets coincide. We claim that the only case where $f \wedge_{\on{gr}} g$ is not an isomorphism of marked-scaled simplicial sets is precisely when $g:(\Lambda^m_i,\flat,\Delta^{\{i-1,i,i+1\}}) \xlongrightarrow{} (\Delta^m,\flat,\Delta^{\{i-1,i,i+1\}})$ with $m =2$. If $m\geq 4$ this is clear since the pushout product already contains all 2-simplices. Let $m=3$ then $\sigma: \Delta^2 \xlongrightarrow{} \Delta^1 \boxplus B$ such that $\sigma$ is scaled in the Cartesian product. Then by definition $\sigma$ must factors through $\Lambda^3_i$ and thus $\sigma$ is already scaled. 

      For the case $m=2$ one easily checks that $\Delta^1 \boxplus \Delta^2 \xlongrightarrow{} \Delta^1 \times \Delta^2$ is the in the weakly saturated class of morphisms of type \ref{ms:wonky4}.
      \item[\ref{ms:wonky4}]
      Inspection reveals that we need to scale the triangles $(0,0) \xlongrightarrow{} (0,3) \xlongrightarrow{} (1,4)$ and $(0,0) \xlongrightarrow{} (0,1) \xlongrightarrow{} (1,4)$. We show that we can scale the first triangle and leave the remaining case as an exercise. We extend this 2-simplex to 
      \[
        (0,0) \xlongrightarrow{} (0,3) \xlongrightarrow{} (1,3) \xlongrightarrow{} (1,4)
      \]
      it follows that this 3-simplex has every face already scaled in $\Delta^1 \boxplus \Delta^4$ except possible the face skipping the vertex $2$. It follows that we scale the $2$-nd face using a morphism of type \ref{ms:saturation}.

      \item[\ref{ms:2coCartesianmorphs}]
      It follows that we only need to check that the map $\Delta^1 \boxplus \Delta^m \xlongrightarrow{} \Delta^1 \times \Delta^m$ is anodyne when $m=2$. We further observe that we ca scaled every triangle using morphisms of type \ref{ms:saturation} except possibly the triangle $\sigma:(0,1) \xlongrightarrow{} (0,2) \xlongrightarrow{} (1,2)$. However, we note that we have a $3$-simplex $\rho:\Delta^3 \xlongrightarrow{} \Delta^1 \boxplus \Delta^4$ such that $\rho(0)=(0,0)$ and $d_0(\rho)=\sigma$. The claim follows from \autoref{lem:0saturations}.

      We finish the proof by noting that the marking in both marked-scaled simplicial sets coincide.
      \item[\ref{mb:2coCartliftsExist}]
      The map is an isomorphism on the underlying scaled simplicial sets. The only edge that remains unmarked in $\Delta^1 \boxplus \Delta^1$ is precisely $(0,1) \xlongrightarrow{} (1,1)$ which can be scaled using a morphism of type $ii)$ in \autoref{lem:0saturations}.
      \item[\ref{ms:composeacrossthin}]
      The map is an isomorphism on the underlying scaled simplicial sets. The only edge we need to mark is the edge $(0,0) \xlongrightarrow{} (1,2)$. We can do this using a morphism of type  \ref{ms:composeacrossthin}.
      \item[\ref{ms:equivalences}]
      The map is again an isomorphism on the underlying scaled simplicial sets. To mark the remaining edges which are of the form $(0,x) \xlongrightarrow{} (1,y)$ we consider a 2-simplex $\sigma:(0,x) \xlongrightarrow{} (1,x) \xlongrightarrow{} (1,y)$ we mark $d_1(\sigma)$ use a morphism of type \ref{ms:composeacrossthin}.

    \end{itemize} 
    \item \textbf{The cofibration is of type:} \ref{cof:c3}
    As in the previous case, the map $f \wedge_{\on{gr}} g$ is an isomorphism on the underlying simplicial sets, so we only need to add the decorations. For every $\mathbf{MS}$-anodyne morphism $g: A \xlongrightarrow{} B$ morphism we define
    \[
      \Delta^2 \boxplus B = \Delta^2_{\sharp} \otimes A \coprod\limits_{\Delta^2_{\flat} \otimes A} \Delta^2_\flat \otimes B
    \]
    and similarly $B \boxplus \Delta^2$.
   \begin{itemize}
      \item[\ref{ms:innerhorn}]
      Let $g:\Lambda^m_i \xlongrightarrow{} \Delta^m$ it follows that the map is an isomorphism of marked-scaled simplicial sets for $m\geq 3$. For the case $m=2$ we only need to show that we can add the triangle $(0,0) \xlongrightarrow{} (1,0) \xlongrightarrow{} (2,2)$. To do so, we extend this to a $3$-simplex
      \[
        (0,0) \xlongrightarrow{} (1,0)\xlongrightarrow{} (2,0)\xlongrightarrow{}(2,2)
      \]
      which allows us to take a pushout along a morphism of type \ref{ms:saturation}.
      \item[\ref{ms:wonky4}]
      The map is an isomorphism.
      \item[\ref{ms:2coCartesianmorphs}]
      Let $g:\Lambda^m_0 \xlongrightarrow{} \Delta^m$. Again, we see that the map in question is an isomorphism unless $m=2$. We need to add the triangles $\sigma_1:(0,0) \xlongrightarrow{} (1,1) \xlongrightarrow{} (2,2)$ and $\sigma_2:(0,1) \xlongrightarrow{} (1,1) \xlongrightarrow{} (2,2)$. First, let us consider the 3-simplex
      \[
        (0,0)\xlongrightarrow{} (1,0) \xlongrightarrow{}(1,1) \xlongrightarrow{}(2,2).
      \]
      Then it follows that every simplex except the face skipping $(1,0)$ is already scaled. Consequently we can scale $\sigma_1$ using a morphism of type \ref{ms:saturation}. Now, we look at the $3$-simplex
      \[
        (0,0) \xlongrightarrow{} (0,1)\xlongrightarrow{}(1,1)\xlongrightarrow{}(2,2)
      \]
      using we already scaled $\sigma_1$ it follows that every face in this 3-simplex is scaled except possible the $0$-th face. We can now scale $\sigma_2$ using a morphism of type $i)$ in \autoref{lem:0saturations}.
      \item[\ref{mb:2coCartliftsExist}]
      This follows from \autoref{prop:msCart}.
      \item[\ref{ms:composeacrossthin}]
      We only need to scale the simplex $(0,0) \xlongrightarrow{} (1,2) \xlongrightarrow{} (2,2)$ to this end we consider the 3-simplex
      \[
        (0,0) \xlongrightarrow{}(0,2) \xlongrightarrow{} (1,2) \xlongrightarrow{} (2,2)
      \]
      and conclude that we can scale the desired triangle again using a morphism of type \ref{ms:saturation}.
      \item[\ref{ms:equivalences}] 
      We need to show that we can scale every 2-simplex in $\Delta^2 \boxplus K$. Let $x \xlongrightarrow{} y \xlongrightarrow{} z$ be a 2-simplex in $K$. To finish the proof it is enough to show that we can scale $(0,x) \xlongrightarrow{} (1,y) \xlongrightarrow{} (2,z)$. We consider a 3-simplex
      \[
        (0,x) \xlongrightarrow{} (1,x) \xlongrightarrow{} (1,y) \xlongrightarrow{} (2,z)
      \]
      and observe that every face is scaled except the face skipping the vertex $(1,x)$ and thus we can scale the triangle in question using a morphism of type \ref{ms:saturation}.
    \end{itemize} 
 \end{itemize}
\end{proof}

\begin{proposition}\label{prop:msgrayleftquillen}
  The functor $\btensor: \on{Set}_\Delta^{\mathbf{ms}}\times \on{Set}_\Delta^{\mathbf{ms}} \xlongrightarrow{} \on{Set}_\Delta^{\mathbf{ms}}$ is a left Quillen bifunctor.
\end{proposition}
\begin{proof}
  By construction the functor $\btensor$ preserves colimits separately in each variable. Therefore, the only thing to show is that given cofibrations $i:A \xlongrightarrow{} B$ and $j:C \xlongrightarrow{} D$ the pushout-product
  \[
    i \wedge_{\on{gray}} j = A \btensor D \coprod\limits_{A \btensor C} B \btensor C \xlongrightarrow{} B \btensor D
  \]
  is a cofibration which is a trivial cofibration whenever $i$ or $j$ is a trivial cofibration. It is obvious that $i \wedge_{\on{gray}} j$ is a cofibration. Let us suppose $i$ is a trivial cofibration and consider a commutative diagram 
  \[
    \begin{tikzcd}
      A \arrow[r] \arrow[d,"\simeq"] & B \arrow[d,"\simeq"] \\
      \bcat{A} \arrow[r] & \bcat{B}
    \end{tikzcd}
  \]
  such that the vertical arrows are $\mathbf{MS}$-anodyne and the marked-scaled simplicial sets at the bottom are $\infty$-bicategories. We obtain a commutative diagram
  \[
    \begin{tikzcd}
      A \btensor D \coprod\limits_{A \btensor C} B \btensor C \arrow[r] \arrow[d,"\simeq"] & B \btensor D \arrow[d,"\simeq"] \\
     A \btensor D \coprod\limits_{A \btensor C} \bcat{B} \btensor C \arrow[d,"\simeq"] \arrow[r] & \bcat{B} \btensor D \\
      \bcat{A} \btensor D \coprod\limits_{\bcat{A} \btensor C} \bcat{B} \btensor C \arrow[ur] & \\
    \end{tikzcd}
  \]
  where the vertical morphisms are equivalences due to \autoref{prop:grpushout} and due to the fact that the pushouts involved are homotopy pushouts.  Using 2-out-of-3 we see that it is enough to show that the bottom diagonal morphisms is a weak equivalence. We consider the composite
  \[
    \bcat{A} \btensor D \xlongrightarrow{} \bcat{A} \btensor D \coprod\limits_{\bcat{A} \btensor C} \bcat{B} \btensor C \xlongrightarrow{} \bcat{B} \btensor D
  \]
  and conclude that our problem reduces to proving the following:
  \begin{itemize}
    \item[$\ast)$] For every $X \in \on{Set}_{\Delta}^{\mathbf{ms}}$ and every equivalence of $\infty$-bicategories $f:\bcat{A} \xlongrightarrow{} \bcat{B}$ the induced map
    \[
       \bcat{A}\btensor X \xlongrightarrow{} \bcat{B} \btensor X
     \] 
     is a weak equivalence.
  \end{itemize}
  We pick a homotopy inverse $g: \bcat{B}\xlongrightarrow{} \bcat{A}$ and claim that the obvious maps $\bcat{A} \btensor X \xlongrightarrow{} \bcat{A}\btensor X$ and $\bcat{B} \btensor X \xlongrightarrow{} \bcat{B} \btensor X$ are weak equivalences. This will imply the result by 2-out-of-3. Let $\bcat{A} \times (\Delta^1)^{\sharp} \xlongrightarrow{} \bcat{A}$ be a homotopy between the identity on $\bcat{A}$ and $g \circ f$. Then it follows from \autoref{prop:grpushout} that the map $\bcat{A}  \xlongrightarrow{}  \bcat{A} \times (\Delta^1)^{\sharp}$ induced by $\Delta^{0} \xlongrightarrow{} (\Delta^1)^\sharp$ is a trivial cofibration. Using 2-out-of-3 we deduce that 
  \[
    \bcat{A} \times (\Delta^1)^{\sharp} \btensor X \xlongrightarrow{} \bcat{A} \btensor X
  \]
  is also a weak equivalence. Let $\bcat{A} \times (\Delta^1)^{\sharp} \xlongrightarrow{} \bcat{A}$ be the obvious collapse map. Then again, it follows from 2-out-of-3 that the maps
  \[
    (\bcat{A}\times \Delta^{\{1\}}) \btensor X \xlongrightarrow{} (\bcat{A} \times (\Delta^1)^{\sharp}) \btensor X \xlongrightarrow{} \bcat{A}\btensor X
  \]
  are also a weak equivalences which proves our claim. The other case is totally analogous and so is the computation where $j$ is a trivial cofibration. The result is now proved.
\end{proof}

\begin{definition}\label{def:globularGray}
  Let $X,Y \in \on{Set}_\Delta^{\mathbf{ms}}$ we define the \emph{globular Gray product} $X \odot Y \in \on{Set}_{\Delta}^{\mathbf{ms}}$ by declaring its underlying scaled simplicial set to be that of the $X \tensor Y$ and by declaring and edge to be marked if and only if it is marked in $X$.
\end{definition}

\begin{remark}\label{rem:bodot}
  Following the same procedure as in \autoref{prop:bothgraycoincide} we define a colimit preserving version of the globular Gray product which we denote by $\bodot$.
\end{remark}

\begin{remark}
  We note that for every $X,Y \in \on{Set}_\Delta^{\mathbf{ms}}$ there is a natural map 
  \[
   \eta_{X,Y}: X \tensor Y  \xlongrightarrow{} X \odot Y
  \]
  which is always a cofibration and similarly for $\bodot$.
\end{remark}

\begin{proposition}\label{prop:globularcoincide}
  Let $X,Y \in \on{Set}_\Delta^{\mathbf{ms}}$ then we have a pushout diagram
  \[
    \begin{tikzcd}
      X \btensor Y \arrow[r,"\simeq"] \arrow[d] & X \tensor Y \arrow[d] \\
      X \bodot Y \arrow[r,"\simeq"] & X \odot Y
    \end{tikzcd}
  \]
  In particular, the obvious natural map $ X \bodot Y  \xlongrightarrow{} X \odot Y$ is a trivial cofibration.
\end{proposition}
\begin{proof}
  Immediate from \autoref{prop:bothgraycoincide}.
\end{proof}

\begin{proposition}\label{prop:globularpushout}
  Let $f: X \xlongrightarrow{} Y$ be cofibration and let $g:A \xlongrightarrow{} B$ be a $\mathbf{MS}$-anodyne morphism. Then the pushout-product maps
  \[
    f \wedge_{\on{gl}} g: X \odot B \coprod\limits_{X \odot A} Y \odot A \xlongrightarrow{} Y \odot B.
  \]
  \[
    g \wedge_{\on{gl}} f: B \odot X \coprod\limits_{A \odot X} A \odot Y \xlongrightarrow{} B \odot Y
  \]
  is a trivial cofibration and similarly for $\bodot$.
\end{proposition}
\begin{proof}
 We consider a pushout diagram
 \[
   \begin{tikzcd}
     X \tensor B \coprod\limits_{X \tensor A} Y \tensor A  \arrow[r,"\eta_{X,Y}"] \arrow[d,"\simeq"] & X \odot B \coprod\limits_{X \odot A} Y \odot A  \arrow[d,"\simeq"] \\
   Y \tensor B \arrow[r] & L(f,g)
   \end{tikzcd}
 \]
 where the left-most vertical morphism is a weak equivalence due to \autoref{prop:grpushout}. Therefore, by 2-out-of-3 it will enough to show that $\alpha_{f,g}:L(f,g) \xlongrightarrow{} X \odot Y$ is a weak equivalence. Note that $\alpha_{f,g}$ is an isomorphism on the underlying scaled-simplicial sets so it will be enough to show that we can add the remaining marked edges in a suitable way. As routinary, we verify case by case.
 \begin{itemize}
   \item \textbf{The cofibration is of type} \ref{cof:c1}.
   \begin{itemize}
     \item[\ref{ms:innerhorn}] Since the marking on both marked-scaled simplicial sets is the minimal one it follows that $\alpha_{f,g}$ is an isomorphism (resp. $\alpha_{g,f}$).
     \item[\ref{ms:wonky4}] Same as above.
     \item[\ref{ms:2coCartesianmorphs}] We verify the check $\alpha_{g,f}$. Note since $g:\Lambda^m_0 \xlongrightarrow{} \Delta^m$ and the edge $0 \xlongrightarrow{} 1$ is marked in $\Delta^m$ it follows that an edge in $\Delta^m \odot \Delta^n$ is marked if and only its component in $\Delta^m$ is either degenerate or $0 \xlongrightarrow{} 1$. However, all such edges factor through $\Lambda^m_0 \odot \Delta^n$ and thus $\alpha_{g,f}$ is an isomorphism.
     \item[\ref{mb:2coCartliftsExist}] The map $\alpha_{f,g}$ is always an isomorphim. We also see that $\alpha_{g,f}$ is an isomorphism unless $f:\partial \Delta^1 \xlongrightarrow{} \Delta^1$. In this case the only edges to mark are $(0,0) \xlongrightarrow{} (1,1)$ and $(1,0) \xlongrightarrow{} (1,1)$. This can be done using morphisms of type \ref{ms:composeacrossthin} and of type $ii)$ in \autoref{lem:0saturations}.
     \item[\ref{ms:composeacrossthin}] The only relevant case is $\alpha_{g,f}$ with $f:\partial \Delta^1 \xlongrightarrow{} \Delta^1$. In this case we are missing the marking on $(0,0) \xlongrightarrow{} (2,1)$ which can be marked with a morphism of type \ref{ms:composeacrossthin} after considering the 2-simplex $(0,0) \xlongrightarrow{} (2,0) \xlongrightarrow{} (2,1)$.
     \item[\ref{ms:equivalences}] As before we need to check $\alpha_{g,f}$ with $f:\partial \Delta^1 \xlongrightarrow{} \Delta^1$. The argument is essentially the same as in the previous case and left as an exercise.
   \end{itemize}
   \item \textbf{The cofibration is of type} \ref{cof:c2}.
      \begin{itemize}
        \item[\ref{ms:innerhorn}] The only case to check is $\alpha_{f,g}$ with $f:\partial \Delta^1 \xlongrightarrow{} \Delta^1$. In this case we need to mark $(0,0) \xlongrightarrow{} (1,2)$ which we do usin the 2-simplex $(0,0) \xlongrightarrow{} (1,0) \xlongrightarrow{} (1,2)$ and a morphism of type \ref{ms:composeacrossthin}.
        \item[\ref{ms:wonky4}] The map is an isomorphism.
        \item[\ref{ms:2coCartesianmorphs}] The only case to check is $\alpha_{f,g}$ with $f:\partial \Delta^1 \xlongrightarrow{} \Delta^1$. We only need to mark the edge $(0,1) \xlongrightarrow{} (1,2)$. To this end we consider the 2-simplex $(0,1) \xlongrightarrow{} (1,1) \xlongrightarrow{} (1,2)$ and use a morphism of type \ref{ms:composeacrossthin}.
        \item[\ref{mb:2coCartliftsExist}] The map is an isomorphism.
        \item[\ref{ms:equivalences}] The map is an isomorphism.
      \end{itemize}
    \item \textbf{The cofibration is of type} \ref{cof:c3}.

    In this case the maps $\alpha_{f,g}$ and $\alpha_{g,f}$ are always isomorphisms.
 \end{itemize}
\end{proof}

\begin{proposition}\label{prop:globularleftquillen}
  The functor $\bodot: \on{Set}_\Delta^{\mathbf{ms}}\times \on{Set}_\Delta^{\mathbf{ms}} \xlongrightarrow{} \on{Set}_\Delta^{\mathbf{ms}}$ is a left Quillen bifunctor.
\end{proposition}
\begin{proof}
  The proof is essentially the same as the proof of \autoref{prop:msgrayleftquillen} and left to the reader.
\end{proof}

\begin{definition}\label{def:graymapping}
  Let $X$ be a marked-scaled simplicial set and let $\bcat{D}$ be an $\infty$-bicategory. We define marked-scaled simplicial sets via the universal properties
  \[
   \Hom_{\on{Set}_\Delta^{\mathbf{ms}}}(A, \on{Fun}^{\mathbf{gr}}(X,\bcat{D}))\isom \Hom_{\on{Set}_\Delta^{\mathbf{ms}}}(A \tensor X , \bcat{D} )
  \]
  \[
    \Hom_{\on{Set}_\Delta^{\mathbf{ms}}}(A, \on{Fun}^{\mathbf{opgr}}(X,\bcat{D}))\isom \Hom_{\on{Set}\Delta^{\mathbf{ms}}}(X \tensor A, \bcat{D} )
  \]
 which are $\infty$-bicategories in virtue of \autoref{prop:grpushout}.
\end{definition}

\begin{remark}[Warning]
   Let $X$ be a marked-scaled simplicial set. Then the functor category $\on{Fun}^{\mathbf{gr}}(X,\bcat{D})$ does not coincide with the category of functors and partially lax natural transformations studied in section 4.1 of \cite{GHL_LaxLim}. This is due to the fact that the objects of $\on{Fun}^{\mathbf{gr}}(X,\bcat{D})$ are functors $F:X \xlongrightarrow{} \bcat{D}$ sending the marked edges in $X$ to marked edges in $\bcat{D}$ whereas in $\cite{GHL_LaxLim}$ the marking is used to model the level of laxness of the natural transformation.

   One could produce a version of the Gray tensor product for marked-scaled simplicial sets equipped which an additional collection of marked edges. We will not follow that line of thought here for the sake of simplicity.
\end{remark}

\begin{definition}\label{def:globfunctorcat}
  Let $X$ be a marked-scaled simplicial set and let $\bcat{D}$ be an $\infty$-bicategory. We define scaled simplicial sets via the universal properties
  \[
   \Hom_{\on{Set}_\Delta^{\mathbf{ms}}}(A, \on{Fun}^{\mathbf{gl}}(X,\bcat{D}))\isom \Hom_{\on{Set}_\Delta^{\mathbf{ms}}}(A \odot X, \bcat{D} )
  \]
  \[
    \Hom_{\on{Set}_\Delta^{\mathbf{ms}}}(A, \on{Fun}^{\mathbf{opgl}}(X,\bcat{D}))\isom \Hom_{\on{Set}_\Delta^{\mathbf{ms}}}(X \odot A, \bcat{D} )
  \]
 which are $\infty$-bicategories in virtue of \autoref{prop:globularpushout}.
\end{definition}

\begin{remark}\label{rem:graymodels}
  We know from \autoref{prop:msgrayleftquillen} and \autoref{prop:globularleftquillen} that the functors $\btensor$ and $\bodot$ are left Quillen bifunctors. In particular they descend to functors of $(\infty,1)$-categories. From this point on we will identify $\tensor $ with $\btensor$ and similarly for $\odot$ and $\bodot$ since they represent the same functor a the level of $(\infty,1)$-categories. In this abuse of notation, we are tacitly assuming that whenever we need to make use of an specific model-categorical incarnation of the Gray tensor product we will pick the variant that makes the most sense (see for example \autoref{prop:coendleftquillen}, \autoref{prop:enhancedcoendleftquillen}). We would like to remind the reader that the constructions given in \autoref{def:graymapping} and \autoref{def:globfunctorcat} are insensitive to the choice of model for the Gray tensor product (resp. globular Gray tensor product).
\end{remark}

\subsubsection{The universal property of the Gray tensor product}
\begin{definition}\label{def:invertible2morph}
  Let $\bcat{C}$ be an $\infty$-bicategory. We define a subcollection of the thin triangles $M_{\bcat{C}}\subseteq T_{\bcat{C}}$ called the invertible 2-morphisms which consist in in those thin 2-simplices $\sigma:\Delta^2 \xlongrightarrow{} \bcat{C}$ such that the edge $\sigma(0\xlongrightarrow{} 1)$ or the edge $\sigma(1\xlongrightarrow{} 2)$ are equivalences in $\bcat{C}$.
\end{definition}

\begin{definition}\label{def:laxfun}
  Let $\bcat{C}$ and $\bcat{A}$ be a pair of $\infty$-bicategories. We define a oplax normalised functor from $\bcat{C}$ to $\bcat{A}$ to be a map of marked-scaled simplicial sets $f:(\bcat{C},E_{\bcat{C}},M_{\bcat{C}}) \xlongrightarrow{} (\bcat{A},E_{\bcat{A}},T_{\bcat{A}})$. Similarly we define  
    \[
       \Fun^{\mathbf{oplax}}(\bcat{C},\bcat{D}):= \Fun (\bcat{C},E_{\bcat{C}},M_{\bcat{C}}),(\bcat{A},E_{\bcat{A}},T_{\bcat{A}}) )
    \] 
    which is an $\infty$-bicategory since $\bcat{A}$ is fibrant.
\end{definition}

\begin{remark}\label{rem:iotacd}
  Let $\bcat{C}$ and $\bcat{D}$ be a pair of $\infty$-bicategories and observe that the scaling $M_{\bcat{C}\times \bcat{D}}\subset T_{\bcat{C}\tensor \bcat{D}}$ where the later denotes the scaling in the Gray tensor product \autoref{def:gray}. In particular this induces a morphism of marked-scaled simplicial sets
  \[
    \iota_{\bcat{C},\bcat{D}}:(\bcat{C}\times \bcat{D},E_{\bcat{C}} \times E_{\bcat{D}},M_{\bcat{C}\times \bcat{D}}) \xlongrightarrow{}  (\bcat{C}\times \bcat{D},E_{\bcat{C}}\times E_{\bcat{D}}, T_{\bcat{C}\tensor \bcat{D}})
  \]
\end{remark}

The next result can be found as Theorem 3.7 in \cite{GHL_Gray}.

\begin{theorem}\label{thm:univgray}
  Let $\bcat{C},\bcat{D}$ and $\bcat{A}$ be  $\infty$-bicategories the restriction along the map $\iota_{\bcat{C},\bcat{D}}$ (see \autoref{rem:iotacd}) induces a fully faithful functor
  \[
    \iota^*: \Fun^{\mathbf{oplax}}(\bcat{C}\times \bcat{D},\bcat{A}) \xlongrightarrow{} \on{Fun}(\bcat{C}\tensor \bcat{D},\bcat{A})
  \]
  whose essential image is is given by those oplax normalised functors $F:\bcat{C}\times \bcat{D} \xlongrightarrow{} \bcat{A}$ satisfying the following conditions:
  \begin{enumerate}
    \item For all objects $x \in \bcat{C}$ and $y \in \bcat{D}$  the restriction of $F$ to $\{x\} \times \bcat{D}$ and the restriction to $\bcat{C}\times \{y\}$ define functors of $\infty$-bicategories $F_x: \bcat{D} \xlongrightarrow{} \bcat{A}$ and $F_y:\bcat{C}\xlongrightarrow{} \bcat{A}$ respectively.
    \item For all arrows $f:x \xlongrightarrow{} x'$ in $\bcat{C}$ and $g:y \xlongrightarrow{} y'$ in $\bcat{D}$ the functor $F$ maps the triangle
    \[
       \begin{tikzcd}
         (x,y) \arrow[r,"f \times \on{id}"] \arrow[dr,swap,"f\times g"] & (x',y) \arrow[d,"\on{id}\times g"] \\
         &  (x',y')
       \end{tikzcd}
     \] 
     to a thin simplex in $\bcat{A}$.
  \end{enumerate}
\end{theorem}

\subsection{Complete Segal objects in \texorpdfstring{$(\infty,1)$-}-categories}

\begin{definition}\label{def:segalobj}
  Let $\iCat_\infty$ be the $(\infty,1)$-category of $(\infty,1)$-categories and let $\on{Fun}(\Delta^\op,\iCat_\infty)$ denote the category of simplicial objects in $\iCat_\infty$. We say that $X_\bullet: \Delta^\op \xlongrightarrow{} \iCat_\infty$ is a Segal object $\iCat_\infty$ if the following holds:
  \begin{itemize}
    \item[$\mathbf{S1})$] The value $X_\bullet(0)=X_0$ factors through the $(\infty,1)$-category of spaces $\scr{S}$.
    \item[$\mathbf{S2})$] Let  $\on{P(n)}$ be the full subcategory of $(\Delta_{/n})^{\op}$ spanned by objects of the form $u:[0] \xlongrightarrow{} [n]$ where $0<u(0)<n$ and $f:[1] \xlongrightarrow{} [n]$ where $f(1)=f(0)+1$. Then we have
    \[
      X_n\isom\lim_{\on{P(n)}}X_\bullet \circ \pi_n.
    \]
    where $\pi_n: \on{P}(n) \xlongrightarrow{} \Delta^\op$ is the obvious projection functor.
  \end{itemize}
  We call a Segal object $X_\bullet: \Delta^\op \xlongrightarrow{} \scr{S} \subset \iCat_\infty$ which factors through the $(\infty,1)$-category of spaces a Segal space.
\end{definition}

\begin{remark}\label{rem:underlyingsegal}
  Let $X_\bullet \in \on{Fun}(\Delta^\op,\iCat_\infty)$ be a Segal object. Then it follows that after taking the underlying groupoid at each level we obtain an associated Segal space $X_\bullet^{\simeq} \subseteq X_\bullet$.
\end{remark}

\begin{definition}\label{def:segalhomcat}
  Let $Y_\bullet:\Delta^\op \xlongrightarrow{} \scr{S}$ be a Segal space. We define the \emph{homotopy category} of $Y_\bullet$ denoted by $\on{ho}(Y_\bullet)$ as follows:
  \begin{itemize}
    \item[i)] The objects of $\on{ho}(Y_\bullet)$ are given by the set of objects of $Y_0$.
    \item[ii)] Given a pair of objects $x,y \in \on{ho}(Y_\bullet)$ we set $\on{Map}_{\on{ho}(Y_\bullet)}(x,y)$ to be the connected components of the space $\pi_0(\on{Map}_{Y_\bullet}(x,y))$ where we define $\on{Map}_{Y_\bullet}(x,y)$ as the pullback
    \[
      \begin{tikzcd}
        \on{Map}_{Y_\bullet}(x,y) \arrow[r] \arrow[d] & Y_1 \arrow[d,"d_0 \times d_1"] \\
      \Delta^0 \arrow[r,"(x{,}y)"] & Y_0 \times Y_0
      \end{tikzcd}
    \]
  \item[iii)] The composition functor is induced by the map
  \[
    \begin{tikzcd}
       Y_1 \times_{Y_0} Y_1 \arrow[r,"\simeq"] & Y_2 \arrow[r,"d_1"]& Y_1.
    \end{tikzcd}
  \]
  where the first invertible map is provided by $\mathbf{S2})$ in \autoref{def:segalobj}.
  \end{itemize}
\end{definition}

\begin{definition}\label{def:equivgroupoid}
  Let $X_\bullet: \Delta^\op \xlongrightarrow{} \iCat_\infty$ be a Segal object with underlying Segal space $X_\bullet^{\simeq}$. We say that an object $[f] \in X_1$ is an equivalence if it defines an isomorphism in $\on{ho}(X^{\simeq}_{\bullet})$. More generally we say that an object $[g] \in X_n$ is an equivalence if for every $e:[1] \xlongrightarrow{} [n]$ we have that $e([g]) \in X_1$ is an equivalence.

  By restricting to those objects that are equivalences with obtain a simplicial space 
  \[
    X^{\on{eq}}_\bullet: \Delta^\op \xlongrightarrow{}\scr{S}.
  \]
\end{definition}

\begin{definition}\label{def:completesegalobj}
  We say that a Segal object in $(\infty,1)$-categories $X_\bullet: \Delta^\op \xlongrightarrow{} \iCat_\infty$ is a complete Segal object if the simplicial space of equivalences $X_\bullet^{\on{eq}}$ (see \autoref{def:equivgroupoid}) is constant. In other words, for every $\alpha:[n] \xlongrightarrow{} [m]$ we have that the induced morphism
  \[
    X^{\on{eq}}_m \xlongrightarrow{\simeq} X^{\on{eq}}_n
  \]
  is a homotopy equivalence. We denote by $\CSeg$ the full subcategory of $\Fun(\Delta^\op,\iCat_\infty)$ spanned by the complete Segal objects.
\end{definition}

\subsection{A comparison of models}
In this section, we prove a comparison of two models of the theory of $(\infty,2)$-categories:  We show that $\infty$-bicategories (i.e. fibrant scaled simplicial sets) are equivalent to complete Segal objects in $(\infty,1)$-categories. The proof of this fact will utilize the comparison of several models for $(\infty,2)$-categories given by Lurie in \cite{LurieGoodwillie}. Let us point out that this two models are known to be equivalent (see \cite{LurieGoodwillie}) but in this section we provide a different equivalence of which, to our knowledge, no proof is available.

\begin{definition}\label{def:bicatinfty}
  Let $\on{Set}_\Delta^{\mathbf{ms}}$ denoted the $\on{Set}_\Delta^+$-enriched model category of marked-scaled simplicial sets and let $\left(\on{Set}_\Delta^{\mathbf{ms}}\right)^{\on{o}}$ denote the full enriched subcategory spanned by the $\infty$-bicategories. We denote $\bcat{B}\!\on{icat}_\infty=\Nsc\left(\left(\on{Set}_\Delta^{\mathbf{ms}}\right)^{\on{o}}\right)$ and call it the $\infty$-bicategory of $\infty$-bicategories.
\end{definition}

\begin{definition}\label{def:leq}
  Let $\on{Set}_\Delta^+$ denote the category of marked simplicial sets. We have a colimit preserving functor
  \[
    L_\sharp: \on{Set}_\Delta^+ \xlongrightarrow{} \on{Set}_\Delta^{\mathbf{ms}}, (X,E_X) \mapsto L_\sharp(X,E_X)=(X,E_X,\sharp)
  \]
  whose right adjoint we denote $(-)^{\leq 1}:\on{Set}_\Delta^{\mathbf{ms}} \xlongrightarrow{} \on{Set}_\Delta^{+}$. One can easily check that the adjunction 
  \[
    L_\sharp:\on{Set}_\Delta^+ \llra \on{Set}_\Delta^{\mathbf{ms}}:(-)^{\leq 1}
  \]
  is a Quillen pair.
\end{definition}

\begin{definition}
  Let $\on{Fun}(\Delta^\op,\on{Set}_\Delta^{+})$ denote the (ordinary) category of simplicial objects in marked simplicial sets. We construct a functor 
  \[
    \mathbf{Gl}: \on{Set}_\Delta^{\mathbf{ms}} \xlongrightarrow{} \on{Fun}(\Delta^\op,\on{Set}_\Delta^{+})
  \]
  which sends a marked-scaled simplicial set $(X,E_X,T_X)$ to the functor defined by 
  \[
    \mathbf{Gl}(X)_n:=\on{Fun}^{\mathbf{opgl}}((\Delta^n,\flat,\sharp),X)^{\leq 1}.
  \]
\end{definition}

\begin{remark}\label{rem:globcoend}
 Let $\mathbf{Gl}: \on{Set}_\Delta^{\mathbf{ms}} \xlongrightarrow{} \on{Fun}(\Delta^\op,\on{Set}_\Delta^{+})$ as above. Then $\mathbf{Gl}$ as a left adjoint
 \[
   \int^{\mathbf{opgl}}:\on{Fun}(\Delta^\op,\on{Set}_\Delta^{+}) \xlongrightarrow{} \on{Set}_\Delta^{\mathbf{ms}}, \enspace F \mapsto \int  
L_\sharp(F) \odot L_\sharp(\Delta)
 \]
 
 where the symbol $\int$ denotes \emph{the coend} of the functors $L_\sharp(F)$ obtained from postcomposing $F$ with $L_\sharp$ and the functor $L_\sharp(\Delta^{\bullet}): \Delta \xlongrightarrow{} \on{Set}_{\Delta}^{\mathbf{ms}}$ given by postcomposing the canonical comsimplical object $\Delta^\bullet: \Delta \xlongrightarrow{} \on{Set}_\Delta^+$ that maps $[n] \xlongrightarrow{} (\Delta^n,\flat)$ with the functor $L_\sharp$.
\end{remark}

\begin{proposition}\label{prop:coendleftquillen}
  The adjoint functors
  \[
    \mathbf{Gl}: \on{Set}_\Delta^{\mathbf{ms}} \llra \on{Fun}(\Delta^\op,\on{Set}_\Delta^{+}): \int^{\mathbf{opgl}}
  \]
  are a Quillen pair where the left-handside is endowed with the model structure of marked-scaled simplicial sets given in \autoref{thm:markedscaledmodel} and the right-hand side is equipped with the complete Segal model structure given in \cite[Proposition 1.5.4]{LurieGoodwillie}.
\end{proposition}
\begin{proof}
  Since $L_\sharp:\on{Set}_\Delta^+ \xlongrightarrow{} \on{Set}_\Delta^{\mathbf{ms}}$ is a left Quillen functor it follows from \autoref{prop:globularleftquillen} that the composite
  \[
    \on{Set}_\Delta^+ \times \on{Set}_\Delta^+ \xlongrightarrow{} \on{Set}_\Delta^{\mathbf{ms}}\times \on{Set}_\Delta^{\mathbf{ms}} \xlongrightarrow{\odot} \on{Set}_\Delta^{\mathbf{ms}}
  \]
  is a left-Quillen bifunctor (see \autoref{rem:graymodels}). It follows from \cite[Proposition A.2.9.26]{HTT} that the coend functor
  \[
    \on{Fun}(\Delta,\on{Set}_\Delta^+) \times \on{Fun}(\Delta^\op,\on{Set}_\Delta^+) \xlongrightarrow{} \on{Set}_\Delta^{\mathbf{ms}}
  \]
  where the left-handside is equipped with the Reedy model structure on each factor. Since the canonical cosimplicial object $\Delta^{\bullet}: \Delta \xlongrightarrow{} \on{Set}_\Delta^+$ is Reedy cofibrant it follows that we have a left Quillen functor
  \[
  \int^{\mathbf{opgl}}: \on{Fun}(\Delta^\op,\on{Set}_\Delta^+) \xlongrightarrow{} \on{Set}_\Delta^{\mathbf{ms}}
  \]
  To finish the proof we need to show that $\int^{\mathbf{opgl}}$ sends the maps below to weak equivalences in $\on{Set}_\Delta^{\mathbf{ms}}$.
  \begin{itemize}
    \item The map $i_n:\on{Sp}(n) \xlongrightarrow{} \Delta^n$ given by the inclusion of the spine of a simplex. 
    \item The map $t: \on{Iso} \xlongrightarrow{} \Delta^0$ where $\on{Iso}$ is the free walking isomorphism.  
  \end{itemize} 
  A direction computation reveals that $\int^{\mathbf{opgl}}i_n:(\on{Sp}(n),\flat,\sharp) \xlongrightarrow{} (\Delta^n,\flat,\sharp)$ and similarly for similarly for $\int^{\mathbf{opgl}}t$ so the claim follows. 
\end{proof}

Our goal in this section is to show that $\mathbf{Gl}$ is a right Quillen equilvanece. We will use to our advantage several models of the theory of $(\infty,2)$-categories and the comparison done by Lurie in \cite{LurieGoodwillie}.

\begin{definition}
  Let $S$ be a set and consider the functor $H_S: \Delta^\op \xlongrightarrow{} \on{Set}$ sending $n \mapsto \Hom_{\on{Set}}([n],S)$. We define $\Delta_S$ to be the category of elements of the presheaf $H_S$. Given an object $(n,c: [n] \xlongrightarrow{} S) \in \Delta_S$ we will use the notation $(n,c: [n] \xlongrightarrow{} S)=[s_0,s_1,\dots,s_n]$ and $(0,c:[0] \xlongrightarrow{} S)=[s]$.
\end{definition}

\begin{definition}
  Let $S \in \on{Set}$ and let $\on{Fun}^{\on{pre}}(\Delta^\op_S,\on{Set}_\Delta^+)$ denote the full subcategory on those functors $X:\Delta^\op_S \xlongrightarrow{} \on{Set}_\Delta^+$ such that for every $[s] \in \Delta_S$ we have $X[s]=\Delta^0$. We observe that this definition extends to a functor 
  \[
    \on{Set}^\op \xlongrightarrow{} \on{Cat}; \enspace S \mapsto \on{Fun}^{\on{pre}}(\Delta_S^\op,\on{Set}_\Delta^+).
  \]
  Taking the (contravariant) Grothendieck construction we obtain a category $\on{Seg}_{\on{Set}_\Delta^+}$ which we call the category of $\on{Set}_\Delta^+$-enriched Segal precategories. Given a Segal precategory $X: \Delta^\op_S \xlongrightarrow{} \on{Set}_\Delta^+$ we will call the set $S$, \emph{the set of objects} of $X$.
\end{definition}

\begin{definition}
  We say that a $\on{Set}_\Delta^+$-enriched Segal precategory $X:\Delta_S^\op \xlongrightarrow{} \on{Set}_\Delta^+$ is an $\on{Set}_\Delta^+$-enriched Segal category if for every $[s_0,s_1,\dots,s_n]$ the map 
  \[
     X[s_0,s_1,\dots,s_n] \xlongrightarrow{} X[s_0,s_1] \times X[s_1,s_2]\times \cdots \times X[s_{n-1},s_n]
  \]
  exhibits $X[s_0,s_1,\dots,s_n]$ as the homotopy product of the objects $\{X[s_i,s_{i+1}]\}_{i=0}^{n-1}$.
\end{definition}

\begin{definition}
   We say that a $\on{Set}_\Delta^+$-enriched Segal precategory $X$ is locally fibrant if for every $[s_0,\dots,s_n]$ the value $X[s_0,\dots,s_n]$ is a fibrant object in $\on{Set}_\Delta^+$.
\end{definition}

\begin{definition}
  We define a functor $G:\on{Cat}_\Delta^+ \xlongrightarrow{} \on{Seg}_{\on{Set}_\Delta^+}$ mapping a $\on{Set}_\Delta^+$-enriched category $\scr{C}$ to the $\on{Set}_{\Delta}^+$-enriched Segal category $G(\scr{C})$, whose set of objects is given by the objects of $\scr{C}$, via the formula 
  \[
    G(\scr{C})[s_0,\dots,s_n]=\scr{C}(s_0,s_1) \times\cdots \times \scr{C}(s_{n-1},s_n)
  \]
\end{definition}
The next result can be found in \cite[Theorem 2.2.16]{LurieGoodwillie}
\begin{theorem}
  The functor $G:\on{Cat}_\Delta^+ \xlongrightarrow{} \on{Seg}_{\on{Set}_\Delta^+}$ is a right Quillen equivalence where $\on{Cat}_\Delta^+$ is equipped with the usual model structure on $\on{Set}_\Delta^+$-enriched categories and $\on{Seg}_{\on{Set}_\Delta^+}$ is equipped with the projective model structure on $\on{Set}_\Delta^+$-enriched Segal precategories found in Theorem 2.2.16 \cite{LurieGoodwillie}.
\end{theorem}

\begin{lemma}\label{lem:presegalequiv}
  Let $X_i:\Delta^\op_{S_i} \xlongrightarrow{} \on{Set}_\Delta^+$ be $\on{Set}_\Delta^+$-enriched Segal precategories with $i=1,2$ and suppose that we are given a morphism $X_0 \xlongrightarrow{} X_1$ such that the associated map $f:S_1 \xlongrightarrow{} S_2$ is surjective and such that for every $(n,c:[n] \xlongrightarrow{} S_1)$ the map
  \[
    X_0[s_0,\dots,s_n] \xlongrightarrow{} X_1[f(s_0),\dots,f(s_n)]
  \]
  is a weak equivalence. Then the map $X_0 \xlongrightarrow{} X_1$ is a weak equivalence in the projective model structure on $\on{Set}_\Delta^+$-enriched Segal precategories. 
\end{lemma}
\begin{proof}
  The result follows from Proposition 2.2.9 and Corollary 2.2.12 in \cite{LurieGoodwillie}.
\end{proof}

\begin{definition}
  Given $Y: \Delta^\op \xlongrightarrow{} \on{Set}_\Delta^+$ we let $S=\Hom_{\on{Set}_\Delta^+}(\Delta^0,Y_0)$ and define a $\on{Set}_\Delta^+$-enriched Segal precategory given by defining $ D(Y)[s_0,\dots, s]$ to be the pullback
  \[
     \begin{tikzcd}
       D(Y)[s_0,\dots,s] \arrow[r] \arrow[d] & Y(n) \arrow[d,"\prod_{i=0}^n d_i"] \\
     \Delta^0 \arrow[r,"\prod_{i=0}^{n}s_i"] & Y_0^{n}
     \end{tikzcd}
  \]
  This construction extends to a functor $D: \on{Fun}(\Delta^\op,\on{Set}_\Delta^+) \xlongrightarrow{} \on{Seg}_{\on{Set}_\Delta^+}$ which is a right Quillen equivalence between the complete Segal model structure and the projective model structure by Proposition 2.3.1 and Proposition 2.3.9 in \cite{LurieGoodwillie}.
\end{definition}

\begin{definition}
  We define a functor $\on{N}^{\mathbf{ms}}:\on{Cat}_\Delta^{+} \xlongrightarrow{} \on{Set}_\Delta^{\mathbf{ms}}$ as follows:
  \begin{enumerate}
    \item Given $\scr{C} \in \on{Cat}_\Delta^{+}$ the underlying scaled simplicial set of $\on{N}^{\mathbf{ms}}(\scr{C})$ coincides with $\Nsc(\scr{C})$ (see \autoref{def:rigidification}).
    \item We declare an edge in $\on{N}^{\mathbf{ms}}(\scr{C})$ is the associated morphism $f \in \scr{C}(x,y)$ satisfies the following condition:
    \begin{itemize}
      \item There exists a morphism $g \in \scr{C}(y,x)$ such that $g\circ f$ is connected to the identity on $x$ via a zig-zag of marked morphisms and similarly for $f \circ g$.
    \end{itemize}
  \end{enumerate}
\end{definition}

\begin{proposition}
  The functor $\on{N}^{\mathbf{ms}}:\on{Cat}_\Delta^{+} \xlongrightarrow{} \on{Set}_\Delta^{\mathbf{ms}}$ preserves weak equivalences between fibrant objects and induces an equivalence upon passage to homotopy categories.
\end{proposition}
\begin{proof}
  We claim that if $\scr{C}$ is fibrant then $\Nsm(\scr{C})$ is a fibrant marked-scaled simplicial set. This ammounts to showing that the marked morphisms in $\Nsm(\scr{C})$ are precisely the equivalences on $\Nsc(\scr{C})$. Let $f \in \scr{C}(x,y)$ be an equivalence. Then, since $\scr{C}$ is fibrant there exists equivalences in the associated mapping categories $g\circ f \isom \on{id}_x$ and $f \circ g \isom \on{id}_y$ which translate to the corresponding thin simplices expressing $f$ as an equivalence in $\Nms(\scr{C})$. The converse is clear. Therefore given a weak equivalence between fibrant objects we obtain a map between fibrant marked-scaled simplicial sets whose underlying map of scaled simplicial sets is a bicategorical equivalence. This shows the first claim.

  We finish the proof by consider the following commutative diagram
  \[
    \begin{tikzcd}
      \on{Cat}_\Delta^+ \arrow[r,"\Nms"] \arrow[dr,swap,"\Nsc"] & \on{Set}_\Delta^{\mathbf{ms}} \arrow[d,"U"] \\
      & \on{Set}_\Delta^{\mathbf{sc}}
    \end{tikzcd}
  \]
  where $U$ is the right adjoint to the left Quillen equivalence given in \autoref{thm:scms}. The claim follows by 2-out-of-3 since $\Nsc$ is also a right Quillen functor.
\end{proof}

\begin{remark}\label{rem:onecomps}
  Let $\bcat{D}$ be an $\infty$-bicategory and recall that by \cite[Proposition 4.1.8]{LurieGoodwillie} it follows that the pullback
  \[
    \begin{tikzcd}
      \Map_{\bcat{D}}(x,y) \arrow[r] \arrow[d] & \on{Fun}^{\mathbf{opgl}}(\Delta^1,\bcat{D}) \arrow[d,"\on{ev}_0 \times \on{ev}_1"] \\
    \Delta^0 \arrow[r,"x \times y"] & \bcat{D}\times \bcat{D}
    \end{tikzcd}
  \]
  is a model for the mapping $(\infty,1)$-category.
\end{remark}

\begin{theorem}\label{thm:compari}
  The functor $\mathbf{Gl}:\on{Set}_\Delta^{\mathbf{ms}} \xlongrightarrow{} \on{Fun}(\Delta^\op,\on{Set}_\Delta^+)$ is a right Quillen equilvalence where the left-handside is endowed with the model structure of marked-scaled simplicial sets given in \autoref{thm:markedscaledmodel} and the right-hand side is equipped with the complete Segal model structure given in \cite[Proposition 1.5.4]{LurieGoodwillie}.
\end{theorem}
\begin{proof}
 We will show that $\mathbf{Gl}$ induces an equivalence upon passage to homotopy categories. First, we consider a non-commutative diagram
 \[
   \begin{tikzcd}
     \on{Cat}_\Delta^+ \arrow[r,"\Nms"] \arrow[d,"G"] & \on{Set}_{\Delta}^{\mathbf{ms}} \arrow[d,"\mathbf{Gl}"] \\
     \on{Seg}_{\on{Set}_\Delta^+} & \arrow[l, "D"] \on{Fun}(\Delta^\op,\on{Set}_\Delta^+). 
   \end{tikzcd}
 \]
 We will show that the functor $T=D \circ \mathbf{Gl} \circ \Nms$ becomes naturally equivalent to $G$ after localization. The result will then follow by 2-out-of-3.

 Let $\on{Un}^{\mathbf{sc}}_*: \on{Set}_\Delta^+ \xlongrightarrow{} \on{Set}_\Delta^+$ be the Unstraightening functor over the point constructed in Section 3.5 in \cite{LurieGoodwillie}. We make some crucial observations
 \begin{itemize}
   \item[i)] The functor $\on{Un}^{\mathbf{sc}}_*$ is a right Quillen equivalence (see \cite[Corollary 3.6.3]{LurieGoodwillie}).
   \item[ii)] Its left adjoint satisfies $\on{St}_*^{\mathbf{sc}}(\Delta^0)=\Delta^0$.
   \item[iii)] For every $\scr{C} \in \on{Cat}_\Delta^+$ and every pair of objects $x,y \in \scr{C}$ we have a canonical isomorphism
   \[
      \on{Un}^{\mathbf{sc}}_*\left(\scr{C}(x,y)\right) \isom \Map_{\Nsc(\scr{C})}(x,y).
    \] 
 \end{itemize}
 Using $i)$ and $ii)$ we obtain an endofunctor $\mathcal{U}: \on{Cat}_\Delta^+ \xlongrightarrow{} \on{Cat}_\Delta^+$ which sends a $\on{Set}_\Delta^+$-enriched category $\scr{C}$ to $\mathcal{U}(\scr{C})$ which has the same set of objects as $\scr{C}$ but its mapping spaces are given by $\mathcal{U}(\scr{C})(x,y)=\on{Un}^{\mathbf{sc}}_*(\scr{C}(x,y))$. The functor comes equiped with a natural transformation from the identity functor
 \[
   \Xi_{\scr{C}}: \scr{C} \xlongrightarrow{} \mathcal{U}(\scr{C})
 \]
 which is an equivalence of fibrant $\on{Set}_\Delta^+$-enriched categories if $\scr{C}$ is fibrant. We see that we can safely replace $G$ by $G \circ \mathcal{U}=\hat{G}$. 

For every $\on{Set}_\Delta^+$-enriched category $\scr{A}$ we have a map (natural in $\scr{A}$) from $T(\scr{A}) \xlongrightarrow{} \hat{G}(A)$ whose component at $[s_0,\dots,s_n]$ is induced by taking the corresponding pullback of the map of the
\[
  Y_n \xlongrightarrow{} \colim\limits_{\on{P}(n)}Y \circ \pi_n
\]
(see \autoref{def:segalobj}) where $Y_\bullet= \mathbf{Gl}\circ \Nms(\scr{A})$. If $\scr{A}$ is fibrant the map above is a trivial fibration and so it follows that the map 
\[
  T(\scr{A})[s_0,\dots,s_n]  \xlongrightarrow{} \hat{G}(\scr{A})[s_0,\dots,s_n]
\]
is again a trivial fibration. The result now follows from \autoref{lem:presegalequiv}.
\end{proof}

\section{The main combinatorial arguments}
In this section, we prove the main combinatorial results underlying the proofs of our main theorems \autoref{thm:sqefullyfaithful} and \autoref{thm:maineglob}. This section is of technical character and may be safely skipped by those readers who are not willing to delve deeply into the intricacies of the combinatorics of the theory of marked-scaled simplicial sets.

Throughout this section we fix a fibrant marked-scaled simplicial set $(\bcat{C},E_{\bcat{C}},T_{\bcat{C}})$ which we will simply denote by $\bcat{C}$.

\begin{definition}\label{def:overlineCC}
  We will denote by $\overline{\bcat{C}}$ the marked-scaled simplicial set given by $(\bcat{C},E_{\bcat{C}},M_{\bcat{C}})$ where $M_{\bcat{C}} \subset T_{\bcat{C}}$ is the collection of triangles consisting in the invertible 2-morphisms (see \autoref{def:invertible2morph}).
\end{definition}

\begin{remark}
  We will make heavy use of the constructions of the Gray tensor product of marked-scaled simplicial sets (see \autoref{def:gray}) which we denote by $\tensor$ and its globular counterpart $\odot$, which was given in \autoref{def:globularGray}.
\end{remark}

\begin{definition}\label{def:indexStensor}
       We define $\mathcal{S}_{\tensor}$ to be the (ordinary) category whose objects are given by maps $f:\Delta^n_\flat \tensor \Delta^{k}_\sharp \xlongrightarrow{} \bcat{C}$ and whose morphisms are given by commutative diagrams
  \[
     \begin{tikzcd}
        \Delta^{n}_\flat \tensor \Delta^{k}_{\sharp} \arrow[rr,"\alpha \tensor \beta"] \arrow[dr,swap,"f"] && \Delta^s_{\flat} \tensor \Delta^{\ell}_\sharp \arrow[dl,"g"] \\
        & \bcat{C} &
     \end{tikzcd}
  \]
  where $\alpha \tensor \beta$ is a product of monotone maps. 

  We similar define a globular version of $\mathcal{S}_\tensor$ denoted by $\mathcal{S}_\odot$  whose objects are given by maps $f:\Delta^n_\flat \odot \Delta^{k}_\sharp \xlongrightarrow{} \overline{\bcat{C}}$ and whose morphisms are given by commutative diagrams
  \[
     \begin{tikzcd}
        \Delta^{n}_\flat \odot \Delta^{k}_{\sharp} \arrow[rr,"\alpha \odot \beta"] \arrow[dr,swap,"f"] && \Delta^s_{\flat} \odot \Delta^{\ell}_\sharp \arrow[dl,"g"] \\
        & \overline{\bcat{C}} &
     \end{tikzcd}
  \]
  where $\alpha \odot \beta$ is a product of monotone maps.
  \end{definition}

  \begin{definition}\label{def:indexfunctortensor}
    We define a functor $T_\tensor: \mathcal{S}_{\tensor} \xlongrightarrow{} \on{Set}_\Delta^{\mathbf{ms}}$ which sends an object $f: \Delta^n_\flat \tensor \Delta^k_\sharp \xlongrightarrow{} \bcat{C}$ to the marked-scaled simplicial set $T_\tensor(f)$ whose underlying simplicial set is given by $\Delta^n \times \Delta^k$ and where an edge (resp. triangle) is marked (resp. thin) if its image under $f$ is marked (resp. thin) in $\bcat{C}$.

    We similarly define a functor $T_\odot: \mathcal{S}_{\odot} \xlongrightarrow{} \on{Set}_\Delta^{\mathbf{ms}}$ in a totally analogous way as before but where in this case we are using $\overline{\bcat{C}}$ instead of $\bcat{C}$.
  \end{definition}

  \begin{definition}\label{def:indexfunctorcolimittensor}
    Let $\mathbb{X}^\tensor_{\bcat{C}}$ (resp. $\mathbb{X}_{\bcat{C}}^{\odot}$) denote the (ordinary) colimit of the functor $T_\tensor$ (resp. $T_\odot$) given in \autoref{def:indexfunctortensor}. Then it follows that we have canonical maps
    \[
      \bcat{C} \xlongrightarrow{s_\tensor}  \mathbb{X}^\tensor_{\bcat{C}} \xlongrightarrow{\pi_\tensor} \bcat{C}, \enspace \enspace  \overline{\bcat{C}} \xlongrightarrow{s_\odot}  \mathbb{X}_{\bcat{C}}^{\odot} \xlongrightarrow{\pi_\odot} \overline{\bcat{C}}
    \]
    such that $\pi_\tensor \circ s_\tensor =\on{id}$ and similarly for the globular versions. 

    The map $s_\tensor$ is induced by the universal property of the colimit after restricting $T_\tensor$ along the inclusion of the subcategory $\mathcal{S}^{0}_\tensor$ consisting in those maps $\Delta^n_\flat \tensor \Delta^0 \xlongrightarrow{} \bcat{C}$ and similarly for $s_\odot$. The definition of $\pi_\tensor$ (resp. $\pi_\odot$) is clear since the marked-scaled simplicial set $\bcat{C}$ (resp. $\overline{\bcat{C}}$) defines as a cone for the functor $T_\tensor$ (resp. $T_\odot$) and thus the map is induced via the universal property of the colimit.
  \end{definition} 

The goal of this section is to show that the maps
\[
    \bcat{C} \xlongrightarrow{s_\tensor}  \mathbb{X}^\tensor_{\bcat{C}} , \enspace \enspace \overline{\bcat{C}} \xlongrightarrow{s_\odot}  \mathbb{X}_{\bcat{C}}^{\odot}
  \]
are weak equivalences of marked-scaled simplicial sets. We will only explicitly verify the statement for $s_\odot$ which is more delicate than the case of $s_\tensor$ due to the fact that one has to take care of respecting the bigger collection of marked edges associated to the globular Gray product. However, the argument that we will use for $s_\odot$ translates perfectly to the case of $s_\tensor$ and we will comment along the way how to adapt our proofs to the case of $s_\tensor$. 

\subsection{Path extensions}
In this section, we collect some useful lemmata which will play a central role in the proof of \autoref{thm:bigcombinatorial}.
\begin{definition}\label{def:pathextension}
  Let $\Delta^n \times \Delta^k$ and consider a simplex of maximal dimension (which we will call a path) $\gamma: \Delta^{n+k} \xlongrightarrow{} \Delta^n \times \Delta^k$. For every $i \in \Delta^{n+k}$ we denote the value of $\gamma$ at the point $i$ as $\gamma(i)=(a_i,b_i)$ We can then define a map of  marked-scaled simplicial sets 
  \[
    E_\gamma: \Delta^{n+k}_\flat \odot \Delta^k_\sharp  \xlongrightarrow{} \Delta^n_\flat \odot \Delta^k_\sharp , \enspace (i,j) \mapsto (a_i,\max(b_i,j)).
  \]
  If we are a given a map $\sigma: \Delta^n_\flat \odot \Delta^k \xlongrightarrow{} \overline{\bcat{C}}$ of marked-scaled simplicial sets, we call the composite $\sigma \circ E_\gamma$ the $\gamma$-extension of $\sigma$.
\end{definition}

\begin{definition}\label{def:pathorder}
  Let $\gamma_1,\gamma_2$ be two simplices of maximal dimension in $\Delta^n \times \Delta^k$ where we denote $\gamma_i(\ell)=(a^{i}_\ell,b^{i}_\ell)$ for $i=1,2$. Let $s \in \Delta^{n}\times \Delta^k$ be the first element where $\gamma_1(s) \neq \gamma_2(s)$. We say that $\gamma_1 < \gamma_2$ if $a^1_{s} < a^2_s$.
\end{definition}

\begin{lemma}\label{lem:extensionsstable}
   Let $\gamma$ be a simplex of maximal dimension inside of $\Delta^n \odot \Delta^k$ and consider the map 
  \[
      E_\gamma: \Delta^{n+k}_\flat \odot \Delta^k_\sharp  \xlongrightarrow{} \Delta^n_\flat \odot \Delta^k_\sharp
   \] 
   given in \autoref{def:pathextension}. Given a simplex of maximal dimension $\varphi:\Delta^{n+2k} \xlongrightarrow{} \Delta^{n+k}_\flat \odot \Delta^k_\sharp$ then we have a commutative diagram
   \[
     \begin{tikzcd}
       \Delta^{n+2k}_\flat \odot \Delta^{k}_\sharp \arrow[r,"E_{\varphi}"] \arrow[d,"s \odot \on{id}"] & \Delta^{n+k}_\flat \odot \Delta^k_\sharp \arrow[d,"E_\gamma"] \\
       \Delta^{n+k}_\flat \odot \Delta^{k}_\sharp \arrow[r,"E_{\gamma'}"] & \Delta^n_\flat \odot \Delta^k_\sharp
     \end{tikzcd}
   \]
   where $s:[n+2k] \xlongrightarrow{} [n+k]$ is a surjective map and $\gamma'$ is a simplex of maximal dimension in $\Delta^n_\flat \odot \Delta^k_\sharp$ such that $\gamma' \leq \gamma$ with respect to the order in \autoref{def:pathorder}.
\end{lemma}
\begin{proof}
  We consider the composite $\theta=E_\gamma \circ \varphi: \Delta^{n+2k} \xlongrightarrow{} \Delta^{n} \odot \Delta^k$ and make the following claim:
  \begin{itemize}
     \item[$\star)$] There are precisely $k$ vertices in $\Delta^{n+2k}$ such that $\theta(i)=\theta(i+1)$.
   \end{itemize}  

  We will check that the $\star)$ holds inductively, using the order on the paths $\varphi$ in $\Delta^{n+k}_\flat \odot \Delta{k}_\sharp$. Let $\varphi_m$ be the biggest path in our order. Then one checks easily that $\theta$ degenerates precisely in the last $k$ vertices. Now let us suppose that the claim holds for some path $\varphi$ and let $\psi$ be the predecesor of $\varphi$ in the order. Observe that since $\psi$ is the predecesor of $\varphi$ then boths paths differ only on one vertex $s \in [2n+k]$. Let us fix some notation by denoting $\varphi(j)=(x_j^\varphi, y_j^\varphi)$ and similarly for $\psi(j)$. Our assumptions imply that we have 2-simplices of the form
    \[
     \sigma_1 :\varphi(s-1) \xlongrightarrow{} \varphi(s) \xlongrightarrow{}  \varphi(s+1) \enspace =  \enspace (x_{s-1}^\varphi,y_{s-1}^\varphi) \xlongrightarrow{} (x_{s-1}^\varphi + 1,y_{s-1}^\varphi) \xlongrightarrow{} (x_{s-1}^\varphi +1,y_{s-1}^\varphi +1),
    \]
     \[
     \sigma_2: \psi(s-1) \xlongrightarrow{} \psi(s) \xlongrightarrow{}  \psi(s+1) \enspace =  \enspace (x_{s-1}^\varphi,y_{s-1}^\varphi) \xlongrightarrow{} (x_{s-1}^\varphi ,y_{s-1}^\varphi +1) \xlongrightarrow{} (x_{s-1}^\varphi +1,y_{s-1}^\varphi +1).
    \]
    We similary fix the notation $\gamma(x_j^\varphi,y_j^\varphi)=(a_j^\varphi,b_j^\varphi)$. We will show that $\sigma_1$ and $\sigma_2$ have the same number of edges which degenerate under $E_\gamma$ which are of the form $d_i(\sigma_j)$ with $i\in \{0,2\}$ for $j=1,2$. If this holds, we can conclude that $\psi$ degenerates precisely on $k$ vertices using an inductive argument since we are assuming that the claim holds for $\varphi$.

     First, we claim that the edge 
    \[
      \varphi(s-1) \xlongrightarrow{} \varphi(s+1)
    \]
    cannot be degenerate under $E_\gamma$. Indeed, if that were true then we would have that $a^\varphi_{s-1}=a^\varphi_s$ and that $\max(b^\varphi_{s-1},y^\varphi_{s-1})=\max(b^\varphi_{s},y^\varphi_{s-1}+1)$. We see then that $b^\varphi_{s-1} > y^\varphi_{s-1}$ which then implies that $b^\varphi_{s-1}=b^\varphi_{s}$ which is a contradicition. Note that this proof also implies that $\psi(s-1) \xlongrightarrow{} \psi(s+1)$ cannot degenerate under $E_\gamma$. We finish the proof of $\star)$ by checking it case by case:
 
  \begin{itemize}
    \item[i)] The edge $\varphi(s-1) \xlongrightarrow{} \varphi(s)$ degenerates under $E_\gamma$.

    The argument given before shows that in this case the path $\varphi(s) \xlongrightarrow{} \varphi(s+1)$ cannot degenerate under $E_\gamma$. Our assumptions imply that  $a_{s-1}=a_s$ and that $b_s=b_{s-1}+1$. We claim that $\psi(s) \xlongrightarrow{} \psi(s+1)$ must be degenerate. If that holds then $\psi(s-1) \xlongrightarrow{} \psi(s)$ cannot degenerate and the claim will hold. We note that our assumptions imply that $\max(b_{s-1},y_{s-1})=\max(b_s,y_{s-1})$ since $b_s > b_{s-1}$ it follows that we must have $y_{s-1}> b_{s-1}$ which implies that $y_{s-1}+1 > b_{s-1}+1=b_s$. The claim now follows.
  
    \item[ii)]  The edge $\varphi(s) \xlongrightarrow{} \varphi(s+1)$ degenerates under $E_\gamma$.

    Again, as above we know that $\varphi(s-1) \xlongrightarrow{} (\varphi(s)$ cannot degenerate. In this case we will show that either the edge $\psi(s-1) \xlongrightarrow{} \psi(s)$ or the edge $\psi(s) \xlongrightarrow{} \psi(s+1)$ degenerate under $E_\gamma$.

    We know by our assumptions that $\max(b_s, y_{s-1})=\max(b_s, y_{s-1} +1)$ which implies that we must have that $b_s > y_{s-1}$. We consider two subcases,  depending wheter $a_{s-1}= a_s$ or not. Let us first suppose that $a_{s-1}<a_{s}$ which in turn implies that $b_s=b_{s-1}$. Then it follows that $b_{s-1}> y_{s-1}$ which implies that $b_{s-1}\geq y_{s-1} +1$ and so $\psi(s-1) \xlongrightarrow{} \psi(s)$ will degenerate under $E_\gamma$. Finally we assume that $a_s=a_{s-1}$ which implies that $b_s=b_{s-1} +1$. We know that $b_s \geq y_{s-1} +1$. If $b_s -1= b_{s-1} \geq y_{s-1} +1$ then $\psi(s-1) \xlongrightarrow{} \psi(s)$ degenerates as before. If $b_s=y_{s-1} +1$ then $\psi(s) \xlongrightarrow{} \psi(s+1)$ degenerates under $E_\gamma$.
   
   \item[iii)] No edge in $\sigma_1$ degenerates under $E_\gamma$
   In this case we wish to show that no edge in $\sigma_2$ degenerates under $E_\gamma$. To do so we assume that a $d_i(\sigma_2)$ degenerates with $i \in \{0,2\}$. Then a similar argument as above shows that a certain edge in $\sigma_1$ must degenerate under $E_\gamma$ which is a contradicition.
  \end{itemize}

  We have now established the claim $\star)$, which gives us an epi-mono factorization of $\theta$
  \[
    \Delta^{n+2k} \xlongrightarrow{s} \Delta^{n+k} \xlongrightarrow{\gamma'} \Delta^n_\flat \odot \Delta^k_\sharp.
  \]
  It is immediate from the definition of $E_\gamma$ that our new path satisfies $\gamma'\leq \gamma$. 

  Given $(i,j) \in \Delta^{n+2k}_\flat \odot \Delta^k_\sharp$ we have that
  \[
    E_\gamma \circ E_\varphi ((i,j))=E_\gamma((x_i^\varphi,\max(y_i^\varphi,j)))=(a_i^\varphi, \max(b_i^\varphi,\max(y_i^\varphi,j)),
  \]
  where we are borrowing the notation from the proof of the claim $\star)$. The same claim implies that we have that $\gamma'(s(i))=(a_i^\varphi,\max(b_i^\varphi,y_i^\varphi))$. We finally see that
  \[
    E_{\gamma'}((s(i),j))=(a_i^\varphi,\max(\max(b_i^\varphi,y_i^\varphi),j))= E_\gamma \circ E_\varphi ((i,j)).
  \]
and thus the result follows.
\end{proof}

\begin{lemma}\label{lem:postextension}
   Let $\gamma$ be a simplex of maximal dimension inside of $\Delta^n_\flat \odot \Delta^k_\sharp$ and consider the map 
  \[
      E_\gamma: \Delta^{n+k}_\flat \odot \Delta^k_\sharp  \xlongrightarrow{} \Delta^n_\flat \odot \Delta^k_\sharp
   \] 
   given in \autoref{def:pathextension}. Then the following holds:
   \begin{enumerate}
     \item Given a face map $d_r:[n] \xlongrightarrow{}  [n+1]$ then we have a commutative diagram
      \[
        \begin{tikzcd}
          \Delta^{n+k}_\flat \odot \Delta^k \arrow[d,"d_\kappa \odot \on{id}"] \arrow[r,"E_\gamma"] & \Delta^n_\flat \odot \Delta^k_\sharp \arrow[d,"d_r \odot \on{id}"] \\
          \Delta^{n+k+1}_\flat \odot \Delta^k_\sharp \arrow[r,"E_{\gamma'}"] & \Delta^{n+1}_\flat \odot \Delta^k_\sharp
        \end{tikzcd}
      \]
      where $\gamma': \Delta^{n+k+1} \xlongrightarrow{} \Delta^{n+1}_\flat \odot \Delta^k_\sharp$ is the unique simplex of maximal dimension such that $\gamma' \circ d_\kappa=(d_r \odot \on{id})\circ \gamma$.
     \item Given a face map $d_s:[k] \xlongrightarrow{}  [k+1]$ then we have a commutative diagram
     \[
        \begin{tikzcd}
          \Delta^{n+k}_\flat \odot \Delta^k_\sharp \arrow[d,"d_{\kappa}\odot d_{s}"] \arrow[r,"E_\gamma"] & \Delta^n_\flat \odot \Delta^k_\sharp \arrow[d,"\on{id} \odot d_s"] \\
          \Delta^{n+k+1}_\flat \odot \Delta^{k+1}_\sharp \arrow[r,"E_{\gamma'}"] & \Delta^{n}_\flat \odot \Delta^{k+1}_\sharp
        \end{tikzcd}
      \]
      where $\gamma': \Delta^{n+k+1} \xlongrightarrow{} \Delta^{n+1}_\flat \odot \Delta^k_\sharp$ is the unique simplex of maximal dimension such that $\gamma' \circ d_\kappa=(\on{id}\odot d_s)\circ   \gamma$.
      \item Given a degeneracy map $s_u:[n] \xlongrightarrow{} [n-1]$ then we have a commutative diagram
      \[
        \begin{tikzcd}
           \Delta^{n+k}_\flat \odot \Delta^k_\sharp \arrow[d,"s_\kappa \odot \on{id}"] \arrow[r,"E_\gamma"] & \Delta^n_\flat \odot \Delta^k_\sharp \ \arrow[d,"s_u \odot \on{id}"] \\
         \Delta^{n+k-1}_\flat \odot \Delta^k_\sharp \arrow[r,"E_{\gamma'}"] & \Delta^{n-1}_\flat \odot \Delta^{k}_\sharp
        \end{tikzcd}
      \]
      where $\gamma'$ is the unique simplex of maximal dimension in $\Delta^{n-1}_\flat \odot \Delta^{k}_\sharp$ such that $\gamma' \circ s_\kappa=(s_u \circ \on{id}) \circ \gamma$.
      \item Given a degeneracy map $s_\ell: [k] \xlongrightarrow{} [k-1]$ then we have a commutative diagram
      \[
        \begin{tikzcd}
           \Delta^{n+k}_\flat \odot \Delta^k_\sharp \arrow[d,"s_\kappa \odot s_\ell"] \arrow[r,"E_\gamma"] & \Delta^n_\flat \odot \Delta^k_\sharp \ \arrow[d,"\on{id} \odot s_\ell"] \\
         \Delta^{n+k-1}_\flat \odot \Delta^{k-1}_\sharp \arrow[r,"E_{\gamma'}"] & \Delta^{n}_\flat \odot \Delta^{k-1}_\sharp
        \end{tikzcd}
      \]
      where $\gamma'$ is the unique simplex of maximal dimension in $\Delta^{n-1}_\flat \odot \Delta^{k}_\sharp$ such that $\gamma' \circ s_\kappa=(\on{id}\odot s_\ell) \circ \gamma$.
   \end{enumerate}
\end{lemma}
\begin{proof}
  We proceed case by case.
  \begin{enumerate}
    \item Observe that the composite $\varphi:\Delta^{n+k} \xlongrightarrow{\gamma}  \Delta^n_\flat \odot \Delta^k_\sharp \xlongrightarrow{d_r \odot \on{id}} \Delta^{n+1}_\flat \odot \Delta^k_\sharp$ is an injective morphism. Therefore there exists a simplex of maximal dimension $\gamma': \Delta^{n+k+1} \xlongrightarrow{} \Delta^{n+1}_\flat \odot \Delta^k_\sharp$ and  vertex $\kappa \in [n+k]$ such that $\gamma' \circ d_\kappa= \varphi$. The claim follows easily.
    \item Given $(i,j) \in  \Delta^{n+k}_\flat \odot \Delta^k_\sharp$ we denote $\gamma(i)=(a_i,b_i)$ and check that we have
    \[
      E_{\gamma'}((d_\kappa(i),d_s(j)))=(a_i,\max(d_s(b_i),d_s(j)))=(a_i,d_s(\max(b_i,j)))=((\on{id}\odot d_s)\circ  E_\gamma )((i,j)).
    \]
    \item The proof is dual to $1$ and left as an exercise.
    \item The proof is dual to $2$ and left as an exercise.
  \end{enumerate}
\end{proof}

\begin{lemma}\label{lem:facesofextensions}
  Let $\gamma$ be a simplex of maximal dimension inside of $\Delta^n \odot \Delta^k$ and consider the map 
  \[
      E_\gamma: \Delta^{n+k}_\flat \odot \Delta^k_\sharp  \xlongrightarrow{} \Delta^n_\flat \odot \Delta^k_\sharp
   \] 
   given in \autoref{def:pathextension}. Then the following holds:
   \begin{enumerate}
     \item If $\gamma$ is the biggest simplex in the order given in \autoref{def:pathorder} then we have that the composite
     \[
        \Delta^n_\flat \odot \Delta^k_\sharp \xlongrightarrow{\alpha \odot \on{id}} \Delta^{n+k}_\flat \odot \Delta^k_\sharp  \xlongrightarrow{E_\gamma} \Delta^n_\flat \odot \Delta^k_\sharp
      \] 
      where $\alpha$ is induced by the inclusion of the set $[0,\dots,n] \subset [n+k]$, equals the identity.
     \item Let $\gamma_1 < \gamma_2$ be a pair of paths such that $\gamma_2$ is the succesor of $\gamma_1$. Given the obvious face map $d_u: \Delta^{n+k-1} \xlongrightarrow{} \Delta^{n+k}$ such that $\gamma_1 \circ d_u= \gamma_2 \circ d_u$ we have that the maps
     \[
       \Delta^{n+k-1}_\flat \odot \Delta^k_\sharp \xlongrightarrow{d_u \times \on{id}} \Delta^{n+k}_\flat \odot \Delta^k_\sharp \xlongrightarrow{E_{\gamma_i}} \Delta^n_\flat \odot \Delta^k_\sharp
     \]
     coincide for $i=1,2$.

    \item Let $0\leq i \leq n+k$ and denote $\gamma(i)=(a_i,b_i)$.  Let us suppose that one of the following conditions hold:
    \begin{itemize}
      \item[i)] We have that $i=0$ and that $\gamma(1)=(1,0)$.
      \item[ii)] We have that $0<i<n+k$ and $b_{i-1}= b_i=b_{i+1}$.
      \item[iii)] We have $i=n+k$ and $\gamma(n+k-1)=(n-1,k)$. 
    \end{itemize}
     Then we have a commutative diagram
     \[
        \begin{tikzcd}
          \Delta^{n+k-1}_\flat \odot \Delta^{k}_\sharp \arrow[r,"d_s \odot \on{id}"] \arrow[d,swap,"E_{\hat{\gamma}}"] & \Delta^{n+k}_\flat \odot \Delta^{k}_\sharp \arrow[d,"E_\gamma"] \\
          \Delta^{n-1}_\flat\odot \Delta^{k}_\sharp \arrow[r,"d_i \odot \on{id}"] & \Delta^{n}_\flat \odot \Delta^k_\sharp
        \end{tikzcd}
      \]
      where $\hat{\gamma}$ is the unique simplex of maximal dimension in $\Delta^{n-1}_\flat \odot \Delta^{k}_\sharp$ such that $(d_i \odot \on{id}) \circ \hat{\gamma}=\gamma \circ d_s$.

      \item Let $0\leq i \leq n+k$ and denote $\gamma(i)=(a_i,b_i)$.  Let us suppose that one of the following conditions hold:
      \begin{itemize}
        \item[i]  We have that $i=0$ and that $\gamma(1)=(0,1)$.
        \item[ii)] We have that $0<i<n+k$ and $a_{i-1}= a_i=a_{i+1}$.
      \item[iii)] We have $i=n+k$ and $\gamma(n+k-1)=(n,k-1)$.  
      \end{itemize}
      Then we have a commutative diagram
      \[
        \begin{tikzcd}
           \Delta^{n+k-1}_\flat \odot \Delta^{k-1}_\sharp \arrow[r,"d_s \odot d_{b_i}"] \arrow[d,swap,"E_{\hat{\gamma}}"] & \Delta^{n+k}_\flat \odot \Delta^{k}_\sharp \arrow[d,"E_\gamma"] \\
          \Delta^{n}_\flat\odot \Delta^{k-1}_\sharp \arrow[r,"\on{id} \odot d_{b_i}"] & \Delta^{n}_\flat \odot \Delta^k_\sharp
        \end{tikzcd}
      \]
      where $\hat{\gamma}$ is the unique simplex of maximal dimension in $\Delta^{n}_\flat \odot \Delta^{k}_\flat$ such that $(\on{id}\odot d_{b_i}) \circ \hat{\gamma}=\gamma \circ d_s$.

     \item Let $d_0: \Delta^{n+k-1} \xlongrightarrow{} \Delta^{n+k+1}$ and assume that $\gamma(1)=(0,1)$ then we have a commutative diagram
     \[\begin{tikzcd}
  & {\Delta^{n+k-1}_\flat \odot \Delta^{k-1}_\sharp} && {\Delta^{n+k}_\flat \odot \Delta^k_\sharp} \\
  {\Delta^{n+k-1}_\flat \odot \Delta^{k}_\sharp} &&&& {} \\
  & {\Delta^{n+k}_\flat \odot \Delta^{k}_\sharp} & {} & {\Delta^{n}_\flat \odot \Delta^{k}_\sharp}
  \arrow["{d_0 \odot d_0}", from=1-2, to=1-4]
  \arrow["{\operatorname{id}\odot s_0}"{pos=0.4}, from=2-1, to=1-2]
  \arrow["{d_0 \odot \operatorname{id}}"', from=2-1, to=3-2]
  \arrow["{E_{\gamma}}"', from=3-2, to=3-4]
  \arrow["{E_{\gamma}}", from=1-4, to=3-4].
\end{tikzcd}\]
    which combined with 4 above yields another commutative diagram
    \[\begin{tikzcd}
  & {\Delta^{n+k-1}_\flat \odot \Delta^{k-1}_\sharp} && {\Delta^{n}_\flat \odot \Delta^{k-1}_\sharp} \\
  {\Delta^{n+k-1}_\flat \odot \Delta^{k}_\sharp} &&&& {} \\
  & {\Delta^{n+k}_\flat \odot \Delta^{k}_\sharp} & {} & {\Delta^{n}_\flat \odot \Delta^{k}_\sharp}
  \arrow["{E_{\hat{\gamma}}}", from=1-2, to=1-4]
  \arrow["{\operatorname{id}\odot s_0}"{pos=0.4}, from=2-1, to=1-2]
  \arrow["{d_0 \odot \operatorname{id}}"', from=2-1, to=3-2]
  \arrow["{E_{\gamma}}"', from=3-2, to=3-4]
  \arrow["{\operatorname{id}\odot d_0}", from=1-4, to=3-4].
\end{tikzcd}\]
    \item Let $d_0: \Delta^{k-1} \xlongrightarrow{} \Delta^k$ and let $r \in [n+k]$ be the biggest index such that we have $\gamma(r)=(a_r,b_r)$ with $b_r=0$. Then there exists a path $\hat{\gamma}$ in $\Delta^n_\flat \odot \Delta^{k-1}_\sharp$ and a commutative diagram
    \[\begin{tikzcd}
  & {\Delta^{n+k-1}_\flat \odot \Delta^{k-1}_\sharp } && {\Delta^{n}\odot \Delta^{k-1}} \\
  {\Delta^{n+k}_\flat \odot\Delta^{k-1}}_\sharp \\
  & {\Delta^{n+k}_\flat\odot\Delta^k_\sharp} && {\Delta^n_\flat \odot \Delta^k_\sharp}
  \arrow["{s_{r}\odot\operatorname{id}}", from=2-1, to=1-2]
  \arrow["{\operatorname{id}\odot d_0}"', from=2-1, to=3-2]
  \arrow["{E_{\gamma}}"', from=3-2, to=3-4]
  \arrow["{E_{\tilde{\gamma}}}", from=1-2, to=1-4]
  \arrow["{\operatorname{id}\odot d_0}", from=1-4, to=3-4].
\end{tikzcd}\]
   \end{enumerate}
\end{lemma}
\begin{proof}
  As usual we check each case separately.
  \begin{enumerate}
    \item This is immediate from the definitions.
    \item This is immediate from the definitions.
    \item We note that the conditions in the statement guarantee the existence of $\hat{\gamma}$. The result follows from the proof of 1 in \autoref{lem:postextension}.
    \item The result follows from the proof of 2 in \autoref{lem:postextension}.
    \item We check that $E_{\gamma}(i+1,s_0(j)+1)=E_\gamma(i+1,j)$. Note that $s_0(j)+1=j$ unless $j=0$ in which case we need to check $E_\gamma(i+1,1)=E_\gamma(i+1,0)$. Since $\gamma(1)=(0,1)$ then if we denote $\gamma_{i+1}=(a_{i+1},b_{i+1})$ we have that $b_{i+1}\geq 1$ which implies the claim.
    \item We define the path $\tilde{\gamma}$ as follows
    \[
      \tilde{\gamma}: \Delta^{n+k-1} \xlongrightarrow{} \Delta^{n+k-1}_\flat \odot \Delta^{k-1}_\sharp, \enspace s \mapsto \begin{cases}
        (a_s,0), \enspace \text{ if } s \leq r,\\
        (a_{s+1},b_{s+1}-1) \text{ if } s > r.
      \end{cases}
    \]
    Then we check that $i \leq r$ we have that $E_\gamma(i,j+1)=(a_i,j+1)$ and that $E_{\tilde{\gamma}}(i,j)=(a_i,j)$ and so the claim holds in this case. If $i = r +1$ we need to check that $E_\gamma(r+1,j+1)=E_\gamma(r,j+1)$. We observe that $a_{r+1}=a_{r}$ since $\gamma$ is a simplex of maximal dimension. Finally we note that $j+1 \geq b_{r+1}$ and so the claim holds.

    Finally if $i> r+1$ we have that $E_{\tilde{\gamma}}(i-1,j)=(a_{i},\max(b_i-1,j))$. We see that $E_\gamma(i,j+1)=(a_i, \max(b_i,j+1))=(a_i,d_0(\max(b_i-1,j))$. The proof is now complete. \qedhere  
  \end{enumerate}
\end{proof}

\begin{lemma}\label{lem:noboundaries}
  Let $\gamma$ be a simplex of maximal dimension inside of $\Delta^n_\flat \odot \Delta^k_\sharp$ and consider the map 
  \[
      E_\gamma: \Delta^{n+k}_\flat \odot \Delta^k_\sharp  \xlongrightarrow{} \Delta^n_\flat \odot \Delta^k_\sharp
   \] 
   given in \autoref{def:pathextension}. Let $0<r\leq k$ and define $\epsilon_r$ to be the first index in $[n+k]$ such that $\gamma(\epsilon_r)=(a_{\epsilon_r},b_{\epsilon_r})$ where $b_{\epsilon_r}=r$. We further define $u_r$ to be the first index in $[n+k]$ such that we have $\gamma(u_r)=(a_{u_r},b_{u_r})$ with $a_{u_r}=a_{\epsilon_r}$ and note that $0\leq u_r <\epsilon_r$. Given $\alpha:[\ell] \xlongrightarrow{} [n+k]$ and $\beta:[m] \xlongrightarrow{} [k]$ such that:
   \begin{enumerate}
     \item The maps $\alpha,\beta$ are injective.
      \item The index $\epsilon_r$ is not in the image of $\alpha$.
     \item The index $0$ is in the image of $\alpha$. 
     \item The index $n+k$ is in the image of $\alpha$ unless $\epsilon_r=n+k$.
     \item There exists an index $u_r \leq j < \epsilon_r$ such that $j$ is in the image of $\alpha$.
     \item The indices $0,r$ are in the image of $\beta$.
     \item The last $k-r$ indices are not in the image of $\beta$.
   \end{enumerate}
   Then it follows that we have a commutative diagram
   \[
     \begin{tikzcd}
       \Delta^{\ell+m}_\flat \odot \Delta^{m}_\sharp \arrow[r,"s \odot \on{id}"] \arrow[d,"E_{\tau}"] & \Delta^{\ell+p}_\flat \odot \Delta^m_\sharp \arrow[r,"\hat{\alpha} \odot \beta"]  & \Delta^{n+k}_\flat \odot \Delta^{k}_\sharp \arrow[d,"E_\gamma"] \arrow[d] \\
       \Delta^{\ell}_\flat \odot \Delta^{m}_\sharp \arrow[r,"\alpha \odot \beta"] & \Delta^{n+k}_\flat \odot \Delta^{k}_\sharp \arrow[r,"E_\gamma"] & \Delta^n_\flat \odot \Delta^k_\sharp 
     \end{tikzcd}
   \]
   where
   \begin{itemize}
      \item The map $\hat{\alpha}$ has  $\epsilon_r$ in its image and restricts to $\alpha$.
      \item The map $s$ is surjective.
      \item $\tau$ is a path in $\Delta^\ell_\flat \odot \Delta^m_\sharp$.
    \end{itemize} 
\end{lemma}
\begin{proof}
Let $u_r \leq \omega_r< \epsilon_r$ be the biggest index in condition $5$ and define $v_r=\alpha^{-1}(\omega_r)$. We define a path $\tau$ in $\Delta^\ell_\flat \odot \Delta^m_\sharp$ as follows:
\[
  \tau: \Delta^{\ell+m} \xlongrightarrow{}\Delta^\ell_\flat \odot \Delta^m_\sharp, \enspace  i \mapsto \begin{cases}
    (i,0), \enspace \text{ if }\enspace  0\leq i \leq v_r \\
    (v_r,i-v_r), \enspace \text{ if }\enspace  v_r  < i \leq v_r + m \\
    (i-m, m), \enspace \text{ if } \enspace m+v_r < i \leq m+l
  \end{cases}
\]
We further define another path
\[
  \tilde{\tau}:\Delta^{n+2k}_\flat \xlongrightarrow{} \Delta^{n+k}_\flat \odot \Delta^k_\sharp, \enspace i \mapsto \begin{cases}
    (i,0), \enspace \text{ if }\enspace i\leq \omega_r, \\
    (\omega_r,i-\omega_r), \enspace \text{ if } \enspace \omega_r < i \leq \omega_r + r, \\
    (i- r, r), \enspace \text{ if } \enspace \omega_r + r < i \leq n+k+r ,\\
    (n+k,i-n-k ), \enspace \text{ if } \enspace n+k+r < i \leq n+2k
  \end{cases}
\]
and note that we have a commutative diagram,
\[
  \begin{tikzcd}
    \Delta^{\ell+m}_\flat \odot \Delta^m_\sharp \arrow[r,"\iota \odot \beta"] \arrow[d,"E_{\tau}"] & \Delta^{n+2k}_\flat \odot \Delta^k_\sharp \arrow[d,"E_{\tilde{\tau}}"] \\
    \Delta^\ell_\flat \odot \Delta^m_\sharp \arrow[r,"\alpha \odot \beta"] & \Delta^{n+k}_\flat \odot \Delta^k_\sharp 
  \end{tikzcd}
\]
where $\iota$ is the injective map described below
\[
  \iota: \Delta^{\ell+m}_\flat \xlongrightarrow{} \Delta^{n+2k}_\flat, \enspace i \mapsto \begin{cases}
    \alpha(i), \enspace \text{ if } \enspace 0\leq i \leq v_r \\
    \omega_r+\beta(i-v_r), \enspace \text{ if } \enspace v_r < i \leq v_r+m \\
  \alpha(i-m) + r, \enspace \text{ if } \enspace m+v_r < i \leq m+l.

  \end{cases}
\]

Following the proof of \autoref{lem:extensionsstable} we note that the composite $\theta: \Delta^{n+2k}_\flat \xlongrightarrow{\tilde{\tau}} \Delta^{n+k}_\flat \odot \Delta^k_{\sharp} \xlongrightarrow{E_\gamma} \Delta^n_\flat \odot \Delta^k_\sharp$ factors as
\[
   \Delta^{n+2k}_\flat \xlongrightarrow{\nu} \Delta^{n+k}_\flat \xlongrightarrow{\gamma}  \Delta^n_\flat \odot \Delta^k_\sharp
\]
where $\nu$ is a surjective map and thus we obtain a commutative diagram
\[
  \begin{tikzcd}
     \Delta^{\ell+m}_\flat \odot \Delta^m_\sharp \arrow[r,"\iota \odot \beta"] \arrow[d,"E_{\tau}"] & \Delta^{n+2k}_\flat \odot \Delta^k_\sharp \arrow[d,"E_{\tilde{\tau}}"] \arrow[r,"\nu \odot \on{id}"] & \Delta^{n+k}_\flat \odot \Delta^k_\sharp  \arrow[d,"E_\gamma"] \\
    \Delta^\ell_\flat \odot \Delta^m_\sharp \arrow[r,"\alpha \odot \beta"] & \Delta^{n+k}_\flat \odot \Delta^k_\sharp  \arrow[r,"E_\gamma"] & \Delta^n_\flat \odot \Delta^k_\sharp
  \end{tikzcd}
\]
To finish the proof we need to show that $\nu \circ \iota =\hat{\alpha} \circ s$. Let $\mu=\nu \circ \iota$ and note that the image of $\nu \circ \iota$ is precisely the image of the map
\[
 \mu: \Delta^{\ell+m}_\flat \xlongrightarrow{} \Delta^n_\flat \odot \Delta^k_\sharp, \enspace i \mapsto \begin{cases}
     (a_{\alpha(i)},b_{\alpha(i)}), \enspace \text{ if } \enspace 0\leq i \leq v_r \\
   (a_{\omega_r},\max(b_{\omega_r},\beta(i-v_r))), \enspace \text{ if } \enspace v_r < i \leq v_r+m \\
  (a_{\alpha(i-m)},\max(r,b_{\alpha(i-m)})), \enspace \text{ if } \enspace m+v_r < i \leq m+l.
  \end{cases}
\]
where we denote $\gamma(j)=(a_j,b_j)$.

Note that by construction if $m+v_r < i \leq m+l$ then $i-m> v_r$ and it follows that $b_{\alpha(i-m)}\geq r$ so in particular, if $i> m + v_r$ we have that $\mu(i)=(a_{\alpha(i-m)},b_{\alpha(i-m)})$. Let $0\leq j\leq m$ such that $b_{\omega_r} \leq \beta(j) \leq b_{\epsilon_r}=r$. Then if we set $i=j+v_r$ it follows that 
\[
  \mu(i)=(a_{\omega_r},\beta(j))=\gamma(s), \enspace \text{ for some } \omega_r <s <\epsilon_r
\]
If we consider $0\leq j \leq m$ such that $0\leq \beta(j) \leq b_{\omega_r}$ it follows that if $i=v_r+j$ we get $\mu(i)=(a_{\omega_r},b_{\omega_r})$. It follows that $\mu$ has an epi-mono factorization
\[
  \Delta^{\ell+m}_\flat \xlongrightarrow{s} \Delta^{\ell+p}_\flat \xlongrightarrow{\hat{\alpha}} \Delta^{n+k}_{\flat}
\]
and so, the result follows.
\end{proof}
\subsection{A technical result}
\begin{theorem}\label{thm:bigcombinatorial}
  The maps of marked-scaled simplicial sets
  \[
    \bcat{C} \xlongrightarrow{s_\tensor}  \mathbb{X}^\tensor_{\bcat{C}} , \enspace \enspace \overline{\bcat{C}} \xlongrightarrow{s_\odot}  \mathbb{X}_{\bcat{C}}^{\odot}
  \]
  are trivial cofibrations in the model structure of marked-scaled simplicial sets given in \autoref{thm:markedscaledmodel}.
\end{theorem}
\begin{proof}
  We first study the case of $s_\odot$ and later we show that our proof translates mutatis mutandis to the case of $s_\tensor$.

  For $i \geq -1$ be define a marked-scaled simplicial set $F_i \subset \mathbb{X}_{\bcat{C}}^{\odot}$ (with the induced decorations) as follows:
 \begin{itemize}
   \item We declare $F_{-1}=  \overline{\bcat{C}}$ and define inductively $F_i$ for $i\geq 0$ to contain the simplices in $F_{i-1}$ in addition to those simplices contained in $\Delta^{\ell}_\flat \odot \Delta^s_\sharp $ according to the colimit given in \autoref{def:indexfunctorcolimittensor} (and with the corresponding decorations given by the colimit) corresponding to a map $\Phi: \Delta^{\ell}_\flat \odot \Delta^{s}_\sharp \xlongrightarrow{} \bcat{C}$ such that $\Phi$ can be factored as
   \[
      \Delta^{\ell}_\flat \odot \Delta^{s}_\sharp \xlongrightarrow{\alpha \odot \beta} \Delta^{i+r}_\flat \odot \Delta^{r}_\sharp \xlongrightarrow{E_\gamma} \Delta^i_\flat \odot \Delta^{r}_\sharp \xlongrightarrow{} \overline{\bcat{C}}
   \]
   where $E_\gamma$ is a certain $\gamma$-extension (see \autoref{def:pathextension}). Note that as consequence of 1 in \autoref{lem:facesofextensions} we have already those factors in the colimit corresponding to maps that factor as
   \[
     \Delta^{\ell}_\flat \odot \Delta^{s}_\sharp \xlongrightarrow{\alpha \odot \beta} \Delta^i_\flat \odot \Delta^{r}_\sharp \xlongrightarrow{} \overline{\bcat{C}}.
   \]
 \end{itemize}
 This definition yields a filtration $\colim\limits_{i \geq 0} F_i \isom \mathbb{X}_{\bcat{C}}^{\odot}$ which reduces the problem to showing that each step in the filtration is a trivial cofibration.

 Our next order of  business is to filter again the map $F_{i-1} \xlongrightarrow{} F_{i}$. Given some index $i$, we define for $j\geq 0$ a simplicial subset  $F_i^{j} \subset F_i$ as follows:
\begin{itemize}
  \item We declare $F_i^0=F_{i-1}$ and define inductively $F_i^j$ to contain every simplex of $F_i^{j-1}$ in addition to those simplices in $\Delta^{\ell}_\flat \odot \Delta^{s}_\sharp$ corresponding to a map $\Phi: \Delta^\ell_\flat \odot \Delta^s_\sharp \xlongrightarrow{} \bcat{C}$ such that $\Phi$ factors as
  \[
      \Delta^{\ell}_\flat \odot \Delta^{s}_\sharp \xlongrightarrow{\alpha \odot \beta} \Delta^{i+j}_\flat \odot \Delta^{j}_\sharp \xlongrightarrow{E_\gamma} \Delta^i_\flat \odot \Delta^{j}_\sharp \xlongrightarrow{} \overline{\bcat{C}}
   \]
   where $E_\gamma$ is a $\gamma$-extension.
\end{itemize}
 Again we get that $\colim\limits_{j\geq 1} F_i^{j}=F_i$ so it we turn our attention into showing that each step in this subfiltration is a trivial cofibration.

 Let us pick a total order of the set $M_\Psi=\{ \Psi_0 < \Psi_1 < \cdots < \cdots\}$ of the maps $\Psi_\alpha: \Delta^{i}_\flat \odot \Delta^j_\sharp \xlongrightarrow{} \overline{\bcat{C}}$ which do not factor through the previous steps of the filtration and consider yet another filtration
 \[
   F^{j-1}_i=P^{i,j}_{-1} \xlongrightarrow{} P^{i,j}_0 \xlongrightarrow{} \cdots \xlongrightarrow{}  P^{i,j}_k \xlongrightarrow{}  \cdots \xlongrightarrow{} F^{j}_i
 \]
 by adding to $P^{i,j}_{k-1}$ the simplices that factor through a $\gamma$-extension of $\Psi_k$. To ease the notation we will drop the superscripts of $P^{i,j}_{k}$ whenever the indices $i,j$ are clear from context.

  Now we fix a certain map $\Psi_\alpha: \Delta^{i}_\flat \odot \Delta^j_\sharp \xlongrightarrow{} \bcat{C}$ and let $\Gamma(\Psi_\alpha):=\{\gamma^{\alpha}_0 < \gamma^{\alpha}_1 < \cdots < \gamma^{\alpha}_m\}$ be the set of paths in $\Delta^{i}_\flat \odot \Delta^j_\sharp$ ordered as in \autoref{def:pathorder}. We define $K^\alpha_{-1}=P^{i,j}_{\alpha-1}$ and define inductively $K^{\alpha}_\nu$ for $0\leq \nu \leq m$ by attaching to $K^{\alpha}_{\nu-1}$ all of the simplices which factor through the $\gamma$-extension of $\Psi_\alpha$ determined by $\gamma_\nu$.  We get a filtration
  \[
   P^{i,j}_{\alpha-1}=K^{\alpha}_{-1}\xlongrightarrow{} K^{\alpha}_0 \xlongrightarrow{} \cdots \xlongrightarrow{} K^{\alpha}_{m-1} \xlongrightarrow{} K^{\alpha}_m=P^{i,j}_\alpha.
 \]
 Now, let us pick some index $0\leq \nu \leq m$ and denote $f_\nu=E_{\gamma_\nu} \circ \Phi_\alpha$. We denote $T_\odot(f_\nu)=U_\nu$ and look at the pullback diagram
 \[
   \begin{tikzcd}
     \partial U_\nu \arrow[d] \arrow[r] &  U_{\nu} \arrow[d] \\
   K^{\alpha}_{\nu-1} \arrow[r] & K^\alpha_\nu.
   \end{tikzcd}
 \]
 We claim that this diagram is also a pushout. Indeed given a map 
 \[
   \Delta^{a}_\flat \odot \Delta^{b}_\sharp \xlongrightarrow{(u,v)} \Delta^{i+j}_\flat \odot \Delta^j_\sharp \xlongrightarrow{E_{\gamma_\nu}} \Delta^{i}_\flat \odot \Delta^j_\sharp \xlongrightarrow{\Psi_\alpha} \overline{\bcat{C}}
 \]
 we can take an epi-mono factorization of both $u$ and $v$ and observe that the identifications provided by the colimit defining $\mathbb{X}_{\bcat{C}}^{\odot}$ yield the result. Therefore we can focus our attention into showing that the map $\partial U_\nu \xlongrightarrow{} U_\nu$ is a trivial cofibration. Before embarking upon the proof of this fact we make some important comments:
 \begin{itemize}
   \item We note that this argument up to this point can be changed to an argument for $s_{\tensor}$ by simply replacing the the symbols ”$\odot$“ with ”$\tensor$“.
   \item We will use usual arguments with anodyne morphisms to show that the map $\partial U_\nu \xlongrightarrow{} U_\nu$ is a trivial cofibration. More precisely we will filter this map and show that each step of the filtration can be obtained by taking a pushout along a trivial cofibration of the form
   \[
    \iota: (A,E_A,T_A) \xlongrightarrow{} (\Delta^n,E_n,T_n).
   \]
   This will done in such a way that whenever we need to use the marking in those simplicial sets to show that $\iota$ is a trivial cofibration, the marked edge in question will be marked in both $\mathbb{X}_{\bcat{C}}^{\odot}$ and in $\mathbb{X}^\tensor_{\bcat{C}}$. In particular, our proof for $s_\odot$ will apply with cosmetic changes to $s_\tensor$.
 \end{itemize}
 We are ready to commence the main part of the proof, but before that some more comments:
 \begin{itemize}
   \item[i)] We can use \autoref{lem:facesofextensions} to see some of the faces (where by this we mean, the simplicial subsets of $U_\nu$ obtained by removing certains columns and/or rows) of $U_\nu$ which factor through $\partial U_\nu$. 
   \item[ii)] Suppose that we have a certain face of $U_\nu$ which factors through $\partial U_\nu$ which is associated to a map
   \[
     \xi:\Delta^a_\flat \odot \Delta^b_\sharp \xlongrightarrow{} \Delta^{u+v}_\flat \odot \Delta^v_\sharp \xlongrightarrow{} \bcat{C}
   \]
   The for every simplex of maximal dimension $\rho$ in $\Delta^a_\flat \odot \Delta^b_\sharp$ it follows from \autoref{lem:extensionsstable} and \autoref{lem:postextension} that the factors in the colimit associated to the composites $E_\rho \circ \xi$ will also factor through previous steps in the filtration.
   \item[iii)] Let us consider the parameter $r$ given in \autoref{lem:noboundaries}. If the face of $U_\nu$ given by skipping the $r$-th column factors through $\partial U_\nu$ then it follows that from $ii)$ above and \autoref{lem:noboundaries} that $\partial U_\nu =U_\nu$. 
   \item[iv)] Let $r$ as above and suppose that there is a face of $U_\nu$ which factors through $\partial U_\nu$ and that skips certain rows and columns (according to \autoref{lem:noboundaries}), one from those being the $r$-th column. Then the \autoref{lem:noboundaries} implies that face of $U_\nu$ which skips the same rows and columns except the $r$-th column factors through $\partial U_\nu$.
 \end{itemize}
  Let $P_{U}=\{\varphi_0 < \cdots < \cdots \varphi_\ell \}$ be the set of simplices of maximal dimension inside of $\Delta^{i+j}_\flat \odot \Delta^j_\sharp$ (i.e the set of paths) ordered as in \autoref{def:pathorder}. By declaring $\partial U_\nu = R^\nu_{\ell+1}$ and defining inductively $R^\nu_i$ by attaching to $R^\nu_{i+1}$ the simplex of maximal dimension determined by $\varphi_{i}$ we obtain a filtration
 \[
   \partial U_\nu= R^\nu_{\ell+1}\xlongrightarrow{} R^\nu_\ell \xlongrightarrow{} \cdots \xlongrightarrow{}  R^\nu_1 \xlongrightarrow{} R^\nu_0=U_\nu.
 \]
 We consider a pullback diagram
 \[
   \begin{tikzcd}
     \partial R_\nu^{s} \arrow[r] \arrow[d] & \Delta^{i+2j} \arrow[d,"\varphi_s"] \\
     R^\nu_{s+1} \arrow[r] & R^\nu_s 
   \end{tikzcd}
 \]
 where we view every simplicial set on the top row as having the obvious induced decorations. We will show that the top horizontal morphism is a trivial cofibration. Let us assume without loss of generality that this map is not the identity map. We need to consider several cases.
 \begin{itemize}
   \item[A1)] We have that $s=\ell$.

   In this case, we note that $\partial R_\nu^\ell$ contains the face skipping the vertex $0$ and the face skipping the last $j$ vertices inside of $\Delta^{i+2j}$ (see point $i)$ above). We observe that we cannot have the face skipping the vertex $i+j$ unless $\partial R_{\nu}^{\ell}=R_\nu^\ell$. Morever, given  $\Delta^{I} \subset \partial R_\nu^{\ell}$ such that $i+j \in I$. Then it follows that $\Delta^{J} \subset \partial R_{\nu}^\ell$ with $J=I \cup \{i+j\}$. Indeed, if the vertex $i+j$ were the reason why we are skipping a row or a column then it follows that either $0 \notin I$ or $I$ does not contain the last $j$ vertices of $\Delta^{i+2j}$. 

   If $\partial R_\nu^\ell$ contains all the decorations of $\Delta^{i+2j}$ then the result follows already from \autoref{lem:refinedinnerpivot} where:
   \begin{itemize}
      \item The pivot point is $i+j$.
      \item The set $U=\{0\}$ and the set $V=\{i+j+1, \dots, i+2j\}$. 
    \end{itemize} 
    To finish the proof we show that the decorations of $\partial R_\nu^\ell$ satisfy the hypothesis in the aforementioned lemma.

   Let us suppose that we are given a marked edge in $\Delta^{i+2j}$ which does not factor through $\partial R_\nu^\ell$. In particular this edge is of the form $0 \xlongrightarrow{} x$ where $x$ is one of the last $j$ vertices. Since every $\gamma$-extension maps the last column to a degenerate edge we see that every edge in the 2-simplex, $0 \xlongrightarrow{} i+j \xlongrightarrow{} x$, is marked in $\Delta^{i+2j}$ so that the marking in $\partial R_\nu^\ell$ satisfies condition 4 in \autoref{lem:refinedinnerpivot}. Now, let us suppose that we are given a $2$-simplex $0<y< z$ which is thin in $\Delta^{i+2j}$ and does not factor through $\partial R_\nu^\ell$. Note that we must have that $z> i+j$. If $y=i+j$ then it follows that condition 5 in \autoref{lem:refinedinnerpivot} holds. If $y\neq i+j$ we look at the 3-simplex determined by the vertices $\{0,i+j,y,z\}$. We wish to show that every face of this simplex is thin in $\Delta^{i+2j}$. Our assumptions guarantee that we have a map 
   \[
     (\Delta^3,E, T) \xlongrightarrow{} \overline{\bcat{C}}
   \] 
   where $E$ contains the edges $\Delta^{\{2,3\}}$, and at least one of the edges $\Delta^{\{0,1\}}$ or $\Delta^{\{1,3\}}$ in  $\Delta^3$. The case of $\Delta^{\{2,3\}}$ follows from the definition of $E_\gamma$, and the last two cases are due to the definition of the thin simplices in $\overline{\bcat{C}}$.
 
   We see that $T$ contains every 2-dimensional face except possibl the face skipping the final vertex. However, since $\bcat{C}$ is an $(\infty,2)$-category this map factors through $(\Delta^3,E',\sharp)$ where $E'=E$ if $\Delta^{\{0,1\}}$ is marked or $E=E \cup \{\Delta^{\{1,2\}}\}$ otherwise. We conclude that the face skipping the final vertex is sent to a thin simplex in $\overline{\bcat{C}}$ and therefore we see that condition 5 in our lemma holds.

    \item[A2)] There exists $0<r<j$ such that $\varphi_s$ is of the form
    \[
      (0,0) \xlongrightarrow{} (0,1) \xlongrightarrow{} \cdots \xlongrightarrow{} (0,r)\xlongrightarrow{} (1,r) \xlongrightarrow{} \cdots \xlongrightarrow{} (i+j,r) \xlongrightarrow{} \cdots \xlongrightarrow{} (i+j,j).
    \]
    In this situation we have that $\partial R^s_\nu$ contains the face skipping the vertex $0$ and the face skipping the vertex $r$ (since that simplex factors through some $R^{s'}_\nu$ where the path $\varphi_{s'}$ is the succesor of $\varphi_s$). Moreover  we also have the face that skips the last $j-r+1$ vertices in $\Delta^{i+2j}$, which corresponds to the vertices in the path $\{(i+j,r), \cdots ,(i+j,j)\}$. We make the following claims

    \begin{enumerate}
      \item The simplicial set $\partial R^s_\nu$ contains every edge for every $s,\nu$ and every $2$-simplex unless $\nu=0$ and $i\leq 1$.
      
      The claim about the edges is obvious so we can focus on the 2-simplices. Suppose that $\nu >0$. Then it follows from 2 in \autoref{lem:facesofextensions} that $\partial R^{s}_\nu$ contains the face that skips a vertex $r<\alpha<r+i+j-1$ and so the claim holds. Similarly if $\nu=0$ and $i>1$ we can use 3 in \autoref{lem:facesofextensions} to obtain a similar vertex $r<\beta<r+i+j-1$.

      \item Given a thin simplex $0<r< x$ in  $\Delta^{i+2j}$ which does not factor through $\partial R^s_\nu$ then  it follows that the every 2-dimensional face in $3$-simplex $\theta$ given by the vertices $\{0,r,r+i+j-1,x\}$ is thin.

      We note that by construction $x \geq r+i+j-1$ so we might assume without loss of generality that $x > r+i+j-1$. We look at the composite
      \[
        \Delta^3 \xlongrightarrow{\theta} \Delta^{i+j}_\flat \odot \Delta^j_\sharp \xlongrightarrow{E_{\gamma_\nu}} \Delta^{i}_\flat \odot \Delta^j_\sharp \xlongrightarrow{\Psi_\alpha} \overline{\bcat{C} }
      \]
      is degenerate on the last edge so the claim holds.
    \end{enumerate}

     We consider a factorization
    \[
      \partial R^s_\nu \xlongrightarrow{}Q^s_\nu \xlongrightarrow{} \Delta^{i+2j}
    \]
    where $Q^s_\nu$ is obtained by attaching the simplex of $\Delta^{i+2j}$ that skips the last $j-r$ vertices. First we will show that the final map satisfies the conditions of \autoref{lem:refinedinnerpivot} with the following parameters:
    \begin{itemize}
      \item The pivot point is $r+i+j-1$ corresponding to the vertex $(i+j,r)$ in the path.
      \item We have $U=\{r\}$ corresponding to the vertex $(0,r)$ and $V$ is preciselly the set consisting in the last $j-r$ vertices in $\Delta^{i+2j}$.
    \end{itemize}
    Conditions 1-3 in the lemma are clearly satisfied which reduces our problem to verify that the decorations in $Q^s_\nu$ satisfy conditions 4 and 5. This follows directly from the claims proved above.

    To finish the proof of this case we verify that the map $\partial R^s_\nu \xlongrightarrow{} Q^s_\nu$ is also \textbf{MS}-anodyne. To this end we consider a pullback diagram
    \[
      \begin{tikzcd}
        \partial Q^s_\nu \arrow[r] \arrow[d] & \Delta^{[0,r+i+j-1]} \arrow[d] \\
        \partial R^s_\nu \arrow[r] & Q^s_\nu
      \end{tikzcd}
    \]
    and wish to show that the top horizontal morphism satisfies the conditions of \autoref{lem:refinedinnerpivot}. Note that in $\partial Q^s_\nu$ we have the faces that skip the vertex $0$, $r$ and $r+i+j-1$ respectively. Let $\epsilon_r$ be the first index in $[i+j]$ such that $\gamma_\nu(\epsilon_r)=(a_{\epsilon_r},b_{\epsilon_r})$ with $b_{\epsilon_r}=r$. We claim that $\partial Q^s_\nu$ satisfies the hypothesis of \autoref{lem:refinedinnerpivot} with the parameters:
    \begin{itemize}
      \item The pivot point is given by the vertex $\omega \in \Delta^{[r+i+j-1]}$ corresponding to  $(\epsilon_r,r)$.
      \item Let $u$ be the smallest index such that $\gamma_\nu(u)=(a_u,b_u)$ with $a_{u}=a_{\epsilon_r}$. We define $U=\{ \Xi \}$ where $\Xi$ is the vertex corresponding to $(u,r)$ in the path $\varphi_s$. Note that $r\leq \Xi <\omega<r+i+j-1$ and in some cases we can have $\Xi=r$.
      \item We define $V=\{r+i+j-1\}$.
    \end{itemize}
    Let us verify the conditions of \autoref{lem:refinedinnerpivot}. First we check condition 2, which essentially amounts to showing that we have the face that skips the vertex $\Xi$. To see this we can assume that $u\neq 0$ since otherwise we have $\Xi = r $ and the claim already holds. If $u>0$ it follows that $\gamma_\nu(u-1)=(a_{\epsilon_r}-1,b_u)$ and so the face skipping $\Xi$ factors through a previous step of the filtration by 2 in \autoref{lem:facesofextensions}. To check condition 3 in our lemma it will suffice to check that the 2-simplex $\sigma:\Xi<\omega <r+i+j-1$ gets mapped to a thin simplex in $\Delta^{i}_\flat \odot \Delta^j_\sharp$ under $E_{\gamma_\nu}$. We check that 
    \[
      E_{\gamma_\nu}(u,r)=(a_{\epsilon_r},\max(b_u,r)), \enspace E_{\gamma_\nu}(\epsilon_r,r)=(a_{\epsilon_r},\max(b_{\epsilon_r},r)),
    \]
    since $b_u < r$ we conclude that $ E_{\gamma_\nu}(u,r)=E_{\gamma_\nu}(\epsilon_r,r)$ so $E_\gamma(\sigma)$ is degenerate. We deal now with 4 and 5. We note that 4 holds since $\partial R_\nu^s$ contains all edges already. It follows that $\partial Q^\nu_s$ contains every 2-simplex unless $\Xi=r$. If $\Xi=r$ then any thin simplex which does not factor through $\partial Q^\nu_s$ must be of the form $0<r<r+i+j-1$. Then since  $E_{\gamma_\nu}(u,r)=E_{\gamma_\nu}(\epsilon_r,r)$ it follows that  every 2-dimensional face of the 3-simplex $0<r<\omega<r+i+j-1$ is thin thus condition 5 holds. We are left to show that condition 1 in our lemma holds.

    Let $\rho:\Delta^I \xlongrightarrow{} \partial Q^s_\nu$ such that $\omega \notin I$ and we wish to show that we have a simplex $\hat{\rho}:\Delta^{J} \xlongrightarrow{} \partial Q^s_\nu$ given by $J=I \cup \{\omega \}$. We can assume without loss of generality that the vertices $\Xi,0,r,r+i+j-1,$ are in $I$ since otherwise the claim is already true. Now, given that $I$ contains those vertices it follows that if $\rho$ factors through $ \partial Q^s_\nu$ then we must have a factorization
    \[
      \rho:\Delta^{I} \xlongrightarrow{} \Delta^{\ell}_\flat \odot \Delta^m_\flat \xlongrightarrow{\alpha \odot \beta} \Delta^{i+j}_\flat \odot \Delta^j_\flat \xlongrightarrow{E_{\gamma_\nu}} \Delta^{i}_\flat \odot \Delta^j_\sharp
      \]
    such that $E_{\gamma_{\nu}} \circ (\alpha \odot \beta)$ factors through a previous step in the filtration. However, now we can apply \autoref{lem:noboundaries} to show that $\hat{\rho}$ already factors through $\partial Q^s_\nu$.

    \item[A3)] There exists some $0<t<i+j$ such that $\varphi_s$ contains an edge of the form $(t,u) \xlongrightarrow{} (t,u+1)$.

    Let $y$ be biggest vertex in $\Delta^{i+2j}$ such that $\varphi_s(y)=(a_y,b_y)$ with $a_y=t$ and let $p$ be the smallest vertex in $\Delta^{i+2j}$ such that $\varphi^s(p)=(a^\varphi_p,b^\varphi_p)$ with $a_p=t$. Finally we set $x$ to be the smallest index in $\Delta^{i+2j}$ such that $\varphi^s(x)=(a^\varphi_x,b^{\varphi}_x)$ with $b^{\varphi}_x=b^{\varphi}_p$. 

    It follows that the face that skips the vertex $x$ factors through $\partial R_\nu^s$. This is because either $x=0$ or this face factors through the image of some path $\varphi_{s'}\geq \varphi_s$. Similarly, it follows that the face that skips the vertex $y$ factors through $\partial R_\nu^s$.

    Let $\sigma:\Delta^I \xlongrightarrow{} \partial R^s_\nu$ such that $p \notin I$. We claim that the simplex $\hat{\sigma}:\Delta^{J} \xlongrightarrow{} \partial R^s_\nu$ with $J=I \cup \{p\}$ also factors. We can assume without loss of generality that $x,y \in I$ since otherwise the claim is true. If $\sigma$ factors through $\partial R^s_\nu$ because it lands in the image of a bigger path $\varphi_{s'}$ then $\sigma$ is contained in the face that skips a certain vertex $\omega \neq p$ and the claim also holds.  Finally if $\sigma$ factors because it skips certains rows and columns it follows that since $x,y \in I$ we cannot skip the columns and rows containing $p$ so $\hat{\sigma}$ factors through $\partial R^s_\nu$.

    We will use \autoref{lem:refinedinnerpivot} with the following parameters:
    \begin{itemize}
      \item The pivot is the point $p$.
      \item $U=\{x\}$ and $V=\{y\}$ 
    \end{itemize}
    It follows from our previous discussion that conditions 1-3 in our lemma hold. We will finish the proof of this case by studying the decorations.

    Let $\xi=\gamma_\nu(t)=(a_t^{\gamma},b_{t}^\gamma)$ and suppose that $a^{\gamma}_t\neq i$. Then it follows that we have the face that skips a certain vertex $z$ with $z>y$. Indeed, there are two possibilities: Either the face that skips the vertex $z$ factors through a previous step of the filtration with $\nu' < \nu$ or we have that $\gamma_\nu(i+j-1)=(i-1,j)$ in which case the face skipping the last vertex factors through a step in the filtration with $i'< i$. In particular, $\partial R^s_\nu$ already contains all edges. We also have all edges if $x \neq 0$ or if $\varphi^s(y)=(t,b^{\varphi}_y)\neq (t,j)$. 

    Let us assume that $a^{\gamma}_t = i$, $x=0$, $b^{\varphi}_y=j$. We consider a 3-simplex $0<p<y<i+2j$ which we call $\zeta$. If $\zeta$ factors through $\partial R^s_\nu$ then we already have all edges. Otherwise it follows by the previous discussion that $d_1(\zeta)$ does not factor through $\partial R^s_\nu$. We consider a pullback diagram
    \[
     \begin{tikzcd}
        T_3 \arrow[r] \arrow[d]& \Delta^3 \arrow[d,"\zeta"] \\
      \partial R^s_\nu \arrow[r] & R^s_\nu
     \end{tikzcd}
    \]
    where the decorations are induced by the map $\zeta$.

    Note that the edge $y \xlongrightarrow{} i+2j$ is marked since its image under the map  $\Psi_{\alpha} \circ E_{\gamma_\nu}$ is degenerate. We further observe that by construction $\zeta$ gets mapped under $\Psi_{\alpha} \circ E_{\gamma_\nu}$ to degenerate 3-simplex which is degenerate on a thin 2-simplex which shows that all faces in $\zeta$ are thin. Moreover, the same argument shows that if $x \xlongrightarrow{} y$ is marked in $ R^s_\nu$ then the edge $x \xlongrightarrow{} i+2j$ must also be marked.  Note that $T_3$ contains precisely the face that skips the vertex $0$ and the face that skips the vertex $2$. It is clear that if we disregard the marking, the top horizontal map is in the weakly saturated class of morphisms of type  \ref{ms:innerhorn} in \autoref{def:msanodyne}. Therefore we obtain a factorization
    \[
      T_3 \xlongrightarrow{\simeq} \hat{T}_3 \xlongrightarrow{} \Delta^3.
    \]
    We note that $\hat{T}_3$ and $\Delta^3$ only possibly differ in the marking of the edge $0 \xlongrightarrow{} 2$. However if $0 \xlongrightarrow{} 2$ is marked then the edges $0 \xlongrightarrow{} 3$ and $2 \xlongrightarrow{} 3$ are already marked in $T_3$. In particular, we do not have to use that the edge $1 \xlongrightarrow{} 2$ is marked to obtain that the last map is a trivial cofibration. 

    Using a totally analogous argument we can add to $\partial R^s_\nu$ every missing simplex in $R^s_\nu$ of the form $0<p<y<z$ with $y<z<i+2j$ and thus we obtain a factorization
    \[
      \partial R^s_\nu \xlongrightarrow{\simeq}M^s_\nu \xlongrightarrow{} R^s_\nu.
    \]
    We consider a 2-simplex determined by vertices $\{0,a,y\}$ which does not factor through $M^s_\nu$. Then it follows that the 4-simplex $\phi$ determined by the simplices $\{0,a,p,y,i+2j\}$ does not factor through $M^s_\nu$. We consider a pullback diagram 
    \[
     \begin{tikzcd}
        S_4 \arrow[r] \arrow[d] & \Delta^4 \arrow[d,"\phi"] \\
      M^s_\nu \arrow[r] & R^s_\nu
     \end{tikzcd}
    \]
    Note that by construction $S_4$ is the union of the face that skips the vertex $0$, the face that skips the vertex $a$ and the face that skips the vertex $y$. It follows that we can construct a factorization
    \[
      S_4 \xlongrightarrow{\simeq } \hat{S}_4 \xlongrightarrow{} \Delta^4
    \]
    where the first map is in the weakly saturated class of morphisms of type  \ref{ms:innerhorn} in \autoref{def:msanodyne} and $\hat{S}_4$ coincides with $\Delta^4$ except possible in the scaling of the triangle $\{0,a,y\}$. If this latter triangle is scaled in $R^s_\nu$ we note the 3-simplex $\tau$ given by the vertices $\{0,a,y,i+2j\}$ gets send under $\Psi_\alpha \circ E_{\gamma_\nu}$ to a degenerate simplex in $\bcat{C}$ which shows that $\tau$ is thin scaled in $R^s_\nu$. Now in $\hat{S}_4$ every triangle in $\tau$ is thin except possible the face skipping the last vertex but since the last edge is marked it follows that $\hat{S}_4 \xlongrightarrow{} \Delta^4$ is a trivial cofibration. 

    We can repeat this process and add every missing 4-simplex of the form $\{0,a,y,p,i+2j\}$ where $\{0,a,y\}$ does not factor through $M^s_\nu$. We thus obtian a factorization
    \[
      M^s_\nu \xlongrightarrow{\simeq} N^s_\nu \xlongrightarrow{}R^s_\nu
    \]
    where the first morphism is a trivial cofibration. Since every simplex that we added contained the pivot point $p$ conditions 1-3 in \autoref{lem:refinedinnerpivot} still hold. We note that conditions 4-5 hold since $N^s_\nu$ contains all the decorations.

    We have established the claim in the cases where $a^{\gamma}_t = i$, $x=0$, $b^{\varphi}_y=j$. We deal with the remaining cases.

    \begin{itemize}
      \item We have that $x\neq 0$ and $a^{\gamma}_t\neq i$ or $b^{\varphi}_y \neq j$. 

      Then $\partial R^\nu_s$ contains all triangles and so the conditions of \autoref{lem:refinedinnerpivot} holds.

      \item We have that $x \neq 0$ ,  $a^{\gamma}_t= i$ and $b^{\varphi}_y = j$.

      Then the only possibly missing 2-simplex is of the form $0<x<y$. The proof in this case is similar as before: We note that the 4-simplex $0<x<p<y<i+2j$ gets mapped under $\Psi_\alpha \circ E_{\gamma_\nu}$ to a degenerate simplex.  Since we can add this simplex using pushouts along anodyne maps and the 3-simplex $0<x<y<i+2j$ has all faces thin except possible $0<x<y$ we conclude that since $y \xlongrightarrow{} i+2j$ must be marked we can scale the remaining triangle.

      \item We have that $x=0$, $b^{\varphi}_y=j$ and that  $a^{\gamma}_t\neq i$.

       Let $t<\ell\leq i+2j$ be the first index such that $E_{\gamma_\nu}(\varphi_s(\ell))=(i,j)$. Then it follows that the face that skips the vertex $\ell$ factors through $\partial R^s_\nu$. Therefore the only possible missing scaled 2-simplex is of the form $\{0,y,\ell\}$. If $\ell \neq i+2j$  then we consider a 4-simplex $\rho$ given by $\{0<p<y<\ell<i+2j\}$ and note that $E_{\gamma_\nu}(\rho)=s_2(E_{\gamma_\nu}(d_3(\rho)))$. It is clear that we can add $\rho$ to $\partial R^s_\nu$ using pushout along inner horns. Finally we look at $d_1(\rho)$ and conclude that every face except possible the face missing the last vertex must be thin. Since the last edge is marked by construction the claim follows. If $\ell=i+2j=t+1$ then we must have $\gamma_\nu(t)=(a^{\gamma}_t,j)$ in which case the simplex $\theta$ given by $\{0,p,y,\ell\}$ satisfies $E_{\gamma_\nu}(\theta)=s_1(E_{\gamma_\nu}(d_2(\theta)))$ and it follows that condition $5$  in \autoref{lem:refinedinnerpivot} is satisfied. 
      
      Finally let $\ell=i+2j$ but $\ell \neq t+1$. Then we look at $\gamma_\nu(t+1)=(a^\gamma_{t+1},b^{\gamma}_{t+1})$. If $a^\gamma_{t+1} > a^{\gamma}_t$ it follows that we have a vertex $y<\mu <\ell$ such that the face skipping the vertex $\mu$ factors through $\partial R^{s}_\nu$ and therefore we already have all triangles. If $a^\gamma_{t+1}=a^\gamma_{t}$ but $b^\gamma_{t+1}> b^\gamma_{t}$ we consider the simplex $\Xi$ given by $\{0,p,y,y+1,\ell\}$ and note that $E_{\gamma_\nu}(y \xlongrightarrow{} y+1)$ is degenerate. Again, we can add $\Xi$ using inner hornx. By looking at $\{0,y,y+1,\ell\}$ we see that every triangle is already scaled except possible at the face skipping the vertex $y+1$ and so the result follows.

      \item We have that $x=0$, $b^{\varphi}_y \neq j$ but $a^{\gamma}_t=i$.

      Then it follows the only scaled triangle that could be missing would be of the form $0<y<\ell$ where $E_{\gamma_\nu}(\ell)=(a^{\gamma}_\ell,b^{\gamma}_\ell)$ with $a^{\gamma}_\ell=i$. One checks that condition 5 in \autoref{lem:refinedinnerpivot} already holds.

    \end{itemize}

    \item[A4)] We have that $s=0$.
     
    Let $\epsilon_j$ be the first index in $i+j$ such that $\gamma_{\nu}(\epsilon_j)=(a^{\gamma}_{\epsilon_j},b^{\gamma}_{\epsilon_j})$ with $b^{\gamma}_{\epsilon_j}=j$. Let us further define an index $u_j$ to be the first index in $i+j$ such that $\gamma_\nu(u_j)=(a^{\gamma}_{u_j},b^{\gamma}_{u_j})$ such that $a^{\gamma}_{u_j}=a^{\gamma}_{\epsilon_j}$.

    We observe that $\partial R^{s}_\nu$ contains the faces skipping the vertex $0$, the vertex $j$ and the vertex $x$ such that $\varphi(x)=(a^{\varphi}_x,b^{\varphi}_x)$ with $a^{\varphi}_x=u_j$ and $b^{\varphi}_x=j$. We call $x<p$ the vertex corresponding $\varphi(p)=(a^{\varphi}_p,b^{\varphi}_p)$ with $a^{\varphi}_p=\epsilon_j$ and $b^{\varphi}_p=j$.   We divide the proof into cases.

    \begin{itemize}
      \item We have $\epsilon_r \neq i+j$.

      Then it follows that we also have the face that skips the vertex $i+2j$. Moreover for every for every $x<a<p$ it follows that the triangle $\sigma$ given by $a<p<i+2j$ satisfies $E_{\gamma_\nu}(\sigma)=s_0(E_{\gamma_\nu})(d_0(\sigma))$. It follows from \autoref{lem:noboundaries} that we can apply \autoref{lem:refinedinnerpivot} where $p$ is our pivot point $U={x}$ and $V=\{i+2j\}$

      \item We have $\epsilon_r=i+j$.

      Observe that for every $x \geq t$ the triangle $0 \xlongrightarrow{} t \xlongrightarrow{} n$ gets always mapped by $E_{\gamma_\nu}$ to triangle which is degenerate on the last edge. We can use \autoref{lem:noboundaries} to see that the conditions of \autoref{lem:refouterpivot} hold where $U=\{0\}$ and $V=\{x\}$. 
    \end{itemize}
 \end{itemize}
Every case is now proved and moreover, the times that we used \autoref{lem:refouterpivot} we considered marked edges which are marked in both $\mathbb{X}_{\bcat{C}}^{\odot}$ and in $\mathbb{X}^\tensor_{\bcat{C}}$. Our theorem is now proved.
\end{proof}

\begin{lemma}[The pivot trick]\label{lem:refinedinnerpivot}
  Let $\Delta^n_\dagger=(\Delta^n,E_n,T_n)$ be a marked-scaled simplex and let $Z^n_\dagger \subset \Delta^n_\dagger$ be a simplicial subset. Suppose that the following conditions hold:
  \begin{enumerate}
    \item There exists a vertex $s \in \Delta^n_\dagger$ such that for every (decorated) simplex $\Delta^{I} \xlongrightarrow{} Z^n_\dagger$ (where $I \subset[n]$) such that $s \notin I$ then it follows that the simplex $\Delta^J$ with $J=I\cup \{s\}$ factors through $Z^n_\dagger$. We call $s$ the pivot point.
    \item There exists a pair of subsets $U,V \subset [n]$ with the following properties:
    \begin{itemize}
       \item We have $U,V \neq \emptyset$ , $U \cap V = \emptyset$ and $s\notin U \cup V$.
       \item For every $u\in U$ and $v \in V$ then we have $u<s<v$. 
       \item The face skipping every vertex of $U$ and the face skipping every vertex of $V$ factor through $Z^n_\dagger$. 
     \end{itemize} 
    \item For every $u\leq i<s<j\leq v$ where $u \in U$ and $v \in V$ then the triangle $\Delta^{\{i,s,j\}}$ is thin in $\Delta^n_\dagger$.
    \item Given $u \in U$ and $v \in V$ and a marked edge $e: u \xlongrightarrow{} v$ in $\Delta^n_\dagger$ which does not factor through $Z^n_\dagger$ then the edges $u \xlongrightarrow{} s$ and $s \xlongrightarrow{} v$ are also marked.
    \item Given a thin simplex $x \xlongrightarrow{} y \xlongrightarrow{} z$ in $\Delta^n_\dagger$ which does not factor through $Z^n_\dagger$ then it follows that every 2-dimensional face of the 3-simplex determined by the vertices $\{x,y,z,s\}$ is thin.
  \end{enumerate}
  Then the map $\iota: Z^n_\dagger \xlongrightarrow{} \Delta^n_\dagger$ is \textbf{MS}-anodyne.
\end{lemma}
\begin{proof}
 Observe that condition $1$ implies that we might assume without loss of generality that $Z^n_\dagger$ does not contain the face skipping the vertex $s$ since otherwise $Z^n_\dagger=\Delta^n_\dagger$. We first consider the case where $Z^n_\dagger$ already contains all of the decorations.

  Let $\mathcal{Y}_0$ be the set of simplices $\sigma$ in $\Delta^n_\dagger$ which are of minimal dimension among those simplices satisfying the following:
  \begin{itemize}
    \item[i)] The simplex $\sigma$ does not factor through $Z^n_\dagger$.
    \item[ii)] The simplex $\sigma$ contains the vertex $j$.
  \end{itemize}
  Note that by construction every simplex $\sigma \in \mathcal{Y}_0$ has the same dimension. We define $Y_0$ as the simplicial subset of $\Delta^n_\dagger$ containing $Z^n_\dagger$ in addition to every simplex $\sigma \in \mathcal{Y}_0$. We define inductively $\mathcal{Y}_\ell$ as the set of simplices of minimal dimension among those which do not factor through $Y_{\ell-1}$ and contain the vertex $s$. We consequently define $Y_s$ by attaching the simplices of $\mathcal{Y}_s$ to $Y_{s-1}$ thus yielding a filtration
  \[
    Z^n_\dagger=Y_{-1} \xlongrightarrow{} Y_0 \xlongrightarrow{} \cdots \xlongrightarrow{} Y_\ell \xlongrightarrow{}\Lambda^n_s \xlongrightarrow{} \Delta^n_\dagger.
  \]
  It will be enough to show that each step of the filtration (the last being obvious) is in the weakly saturated class of morphisms of type \ref{ms:innerhorn} in \autoref{def:msanodyne}. 

   It follows by minimality of the dimension that given $\sigma \in \mathcal{Y}_\ell$ then every face $d_i(\sigma)$ where $i \neq s$ must factor through $Y_{\ell-1}$. Moreover since the dimension of the simplices in $\mathcal{Y}_{\ell}$ is strictly bigger to that of the simplices in $\mathcal{Y}_{\ell-1}$ it follows that $d_s(\sigma)$ cannot factor through $Y_{\ell-1}$ if $\ell\geq 1$. If $\ell=0$ then it follows from condition 1 that $d_s(\sigma)$ cannot factor through $Z^n_\dagger$.

    Moreover $2$ in our statement guarantees that $\sigma$ contains at least a vertex $u \in U$ and $v \in V$ with $u<s<v$ which finally shows that we have pullback diagram (where the decorations are those induced from $\Delta^n_\dagger$)
   \[
     \begin{tikzcd}
       \Lambda^{m}_i \arrow[r] \arrow[d] & \Delta^m \arrow[d] \\
       Y_{\ell-1} \arrow[r]  & Y_{\ell}
     \end{tikzcd}
   \]
   where $0<i<m$ and the triangle $\{i-1,i,i+1\}$ is thin by condition $3$ in the statement. We finish the proof by noting that if $\sigma_1, \sigma_2 \in \mathcal{Y}_\ell$ satisfy $d_s(\sigma_1)=d_s(\sigma_2)$ then $\sigma_1=\sigma_2$ which shows that we can add every $\sigma \in \mathcal{Y}_s$ using pushouts along inner horns.

   To finish the proof we will need to take care of the decorations. We borrow the notation of the previous case and we will show that the filtration before will take care of the decorations under our assumptions. Let us suppose that that there is a marked edge $e$ which does not factor through $Z^n_\dagger$. Condition 2 implies that $e$ is of the form $u \xlongrightarrow{} v$ with $u \in U$ and $v \in V$ and it follows that the simplex $\sigma: u<s<v$ belongs to $\mathcal{Y}_0$. Using condition 4 in our statement it follows that we add $\sigma$ using a pushout along a morphism of type
   \[
      (\Lambda^2_1,\sharp,\sharp)\xlongrightarrow{} (\Delta^2,\sharp,\sharp)
    \] 
    which is in the weakly saturated class of morphisms of type \ref{ms:innerhorn} and \ref{ms:composeacrossthin} in \autoref{def:msanodyne}.

    Finally let us suppose that we have a 2-simplex $\rho:x<y<z$ which does not factor through $Z^n_\dagger$. Note  that we have at least pair of vertices in $\rho$ which belong to $U$ and $V$ respectively. First let us suppose that $x \in U$, $y=s$ and $z \in V$. Then it follows from condition 1 that $d_1(\rho)$ cannot factor through $Z^n_\dagger$. In particular, $\rho \in \mathcal{Y}_0$ and we are done. If $y \neq s$ then we consider a $3$-simplex $\theta$ given by the vertices $\{x,y,z,s\}$. It follows that $\theta \in \mathcal{Y}_\ell$ for $\ell\leq 1$. Since there exists a pair of vertices $a,b \in \{x,y,z,s\}$ such that $a\in U$ and $b \in V$ it follows from condition 5 in our statement that we can add $\theta$ using a pushout along a morphism of the form
    \[
      (\Lambda^3_i,E,\sharp) \xlongrightarrow{} (\Delta^3,E,\sharp)
    \]
    which is in the weakly saturated class of morphisms of type \ref{ms:innerhorn} and \ref{ms:wonky4} in \autoref{def:msanodyne}. 
\end{proof}

\begin{lemma}[The outer pivot trick]\label{lem:refouterpivot}
   Let $\Delta^n_\dagger=(\Delta^n,E_n,T_n)$ be a marked-scaled simplex and let $X^n_\dagger \subset \Delta^n_\dagger$ be a simplicial subset. Suppose that the following conditions hold:
   \begin{enumerate}
     \item Given a simplex $\Delta^{I} \xlongrightarrow{} X^n_\dagger$ such that $n \notin I$ then it follows that $\Delta^{J} \xlongrightarrow{} \Delta^n_\dagger$ where $J=I \cup \{n\}$ also factors through $X^n_\dagger$.
     \item There exists a pair of subsets $U,V \subset [n]$ with the following properties:
     \begin{itemize}
        \item $U=\{0\}$, such that $0,n \notin  V$.
        \item The face skipping every vertex of $U$ and the face skipping every vertex of $V$ factor through $X^n_\dagger$.  
      \end{itemize} 
      \item For every $v \leq x < n$ where $v \in V$ it follows that the triangle $0<x<n$ is thin in $\Delta^n_\dagger$.
      \item For every $v\leq x<n$ where $v \in V$ it follows that the edge $x \xlongrightarrow{} n$ is marked in $\Delta^n_\dagger$.
      \item Given a marked edge $0 \xlongrightarrow{} v$ which does not factor through $X^n_\dagger$ then it follows that $0 \xlongrightarrow{} n$ is marked in $\Delta^n_\dagger$.
      \item Given a thin triangle $0 \xlongrightarrow{} v \xlongrightarrow{} x$ which does not factor through $X^n_\dagger$ then it follows that the triangle given by $0 \xlongrightarrow{} v \xlongrightarrow{} x \xlongrightarrow{} n$ is fully thin scaled.
   \end{enumerate}
   Then map $\iota : X^n_\dagger \xlongrightarrow{} \Delta^n_\dagger$ is a trivial cofibration of marked-scaled simplicial sets.
\end{lemma}
\begin{proof}
  The proof is essentially the same as in \autoref{lem:refinedinnerpivot}. We define the analogous filtration where we use the vertex $n$ as the pivot point 
  \[
     X^n_\dagger \xlongrightarrow{} Y_0 \xlongrightarrow{} Y_1 \cdots \xlongrightarrow{} \Lambda^n_n \xlongrightarrow{} \Delta^n_\dagger
   \] 
   where each step in the filtration is in the weakly saturated class of morphisms of type:
   \begin{itemize}
     \item $(\Lambda^n_n,\Delta^{\{n-1, n\}},\Delta^{\{0,n-1,n\}}) \xlongrightarrow{} (\Delta^n,\Delta^{\{n-1, n\}},\Delta^{\{0,n-1,n\}})$.
     \item $(\Lambda^2_2,\sharp,\sharp) \xlongrightarrow{} (\Delta^2,\sharp,\sharp)$.
     \item $(\Delta^3,\Delta^{\{2, 3\}},U_3) \xlongrightarrow{} (\Delta^3,\Delta^{\{2, 3\}},\sharp)$ we $U_3$ is the collection of all triangles except the face skipping the vertex $3$.
   \end{itemize}
   We leave the details to the reader.
\end{proof}

 \section{Lax Kan extensions and the enhancement construction}
 In this section, we develop the main conceptual ingredients of this paper. We commence by reviewing the theory of fibrations of $(\infty,2)$-categories which will later be deployed to construct a lax version of the right Kan extension functor. We will apply this new technology to the study of complete Segal objects in $(\infty,1)$-categories (resp. double $(\infty,1)$-categories) to obtain enhancements of the usual theory of Segal objects (resp. double $(\infty,1)$-categories) which will play a key role in the final sections of this document.
 \subsection{Preliminaries on fibrations of \texorpdfstring{$(\infty,2)$-}-categories}
 \begin{definition}
   Let $r: \scr{X} \xlongrightarrow{} \scr{S}$ be a functor of $(\infty,1)$-categories. We say that $p$ is a $0$-fibration if it is a cocartesian fibration. Similarly we say that $p$ is a 1-fibration if it is a cartesian fibration.
 \end{definition}

 \begin{definition}\label{def:ienriched}
   Let $p:\bcat{X} \xlongrightarrow{} \bcat{S}$ be a functor of $(\infty,2)$-categories. We say that $p$ is enriched in $i$-fibrations for $i \in \{0,1\}$ if for every pair of objects $a,b \in \bcat{X}$ the induced morphism on mapping $(\infty,1)$-categories
   \[
     \bcat{X}(a,b) \xlongrightarrow{} \bcat{S}(p(a),p(b))
   \]
   is an $i$-fibration and for every triple of objects $a,b,c \in \bcat{X}$ the composition functor
   \[
     \bcat{X}(a,b) \times \bcat{X}(b,c) \xlongrightarrow{} \bcat{X}(a,c)
   \]
   preserves $i$-cartesian edges. Given another functor $q: \bcat{Y} \xlongrightarrow{} \bcat{S}$ enriched in $i$-fibrations and a morphism over $\bcat{S}$
   \[
     \begin{tikzcd}
       \bcat{X} \arrow[rr,"f"] \arrow[dr,swap,"p"] & & \bcat{Y} \arrow[dl,"q"] \\
       & \bcat{S} &
     \end{tikzcd}
   \]
   we say that $f$ is a morphism of $i$-enriched fibrations if it preserves $i$-cartesian edges on the mapping $(\infty,1)$-categories.
 \end{definition}

 \begin{definition}
     Let $p:\bcat{X} \xlongrightarrow{} \bcat{S}$ be a functor of $(\infty,2)$-categories. We say that an edge $e:a \xlongrightarrow{} b$ in $\bcat{X}$ is $0$-cartesian if for every $c \in \bcat{X}$ the following diagram
     \[
       \begin{tikzcd}
         \bcat{X}(b,c) \arrow[r,"-\circ e"] \arrow[d] & \bcat{X}(a,c) \arrow[d] \\
         \bcat{S}(p(b),p(c)) \arrow[r,"-\circ p(e)"] & \bcat{S}(p(a),p(c))
       \end{tikzcd}
     \]
     induced by precomposition by $e$ is a pullback diagram.  We define similarly, using postcomposition, the notion of a 1-cartesian edge.
 \end{definition}

 \begin{definition}\label{def:enough}
     Let $p:\bcat{X} \xlongrightarrow{} \bcat{S}$ be a functor of $(\infty,2)$-categories. We say that $p$ has \emph{enough} 0-cartesian morphisms if we can always produce a 0-cartesian edge solving the lifting problem below
     \[
       \begin{tikzcd}
         \Delta^{\{0\}} \arrow[d] \arrow[r] & \bcat{X} \arrow[d,"p"] \\
         \Delta^1 \arrow[ur,dotted] \arrow[r] & \bcat{S}.
       \end{tikzcd}
     \]
     We dually define the notion of $p$ having enough $1$-cartesian edges, using the inclusion of the terminal vertex $\Delta^{\{1\}} \xlongrightarrow{} \Delta^1$.
 \end{definition}

 \begin{definition}\label{def:ijfib}
     Let $p:\bcat{X} \xlongrightarrow{} \bcat{S}$ be a functor of $(\infty,2)$-categories. We say that $p$ is a $(i,j)$-fibration if the following conditions hold:
     \begin{enumerate}
       \item The functor $p$ has enough $i$-cartesian morphisms (see \autoref{def:enough}).
       \item The functor $p$ is enriched in $j$-fibrations (see \autoref{def:ienriched}).
     \end{enumerate}
 \end{definition}

 \begin{remark}\label{rem:fibtensor}
   Let $\bcat{D}$ be an $(\infty,2)$-category and let $p:\bcat{X} \xlongrightarrow{} \bcat{S}$ be $(i,j)$-fibration then we have an $(i,j)$-fibration $\pi:\bcat{X}\times \bcat{D} \xlongrightarrow{} \bcat{S}$ characterised as follows:
   \begin{itemize}
     \item The map $\pi$ is given by the composite
     \[
        \bcat{X}\times \bcat{D} \xlongrightarrow{} \bcat{X} \xlongrightarrow{p} \bcat{S}.
      \] 
      \item Given $a,b \in \bcat{X}$ and $s,t \in \bcat{D}$ an edge  $e:\Delta^1 \xlongrightarrow{} \bcat{X}(a,b)\times \bcat{D}(s,t)$ is $j$-cartesian if and only if its image in $ \bcat{X}(a,b)$ is $j$-cartesian and its image in $\bcat{D}(s,t)$ is an equivalence.
      \item An edge $u:\Delta^1 \xlongrightarrow{} \bcat{X} \times \bcat{D}$ is $i$-cartesian if its image in $\bcat{X}$ is $i$-cartesian and its image in $\bcat{D}$ is an equivalence.
   \end{itemize}
 \end{remark}

 \begin{definition}
   Let $(\bcat{S},E)$ denote the pair consisting of an $(\infty,2)$-category $\bcat{S}$ together with a collection of edges $E$ containing all the equivalences in $\bcat{S}$. We say that a functor over $\bcat{S}$ 
    \[
     \begin{tikzcd}
       \bcat{X} \arrow[rr,"f"] \arrow[dr,swap,"p"] & & \bcat{Y} \arrow[dl,"q"] \\
       & \bcat{S} &
     \end{tikzcd}
   \]
   where $p,q$ are $(0,j)$-fibrations is an $E$lax morphism of $(0,j)$-fibrations if:
   \begin{itemize}
     \item The map $f$ is enriched in $j$-fibrations.
     \item The map $f$ preserves $0$-cartesian edges lying over the eges in $E$. 
   \end{itemize}
   Given $p,q$ as above we an $(\infty,1)$-category of E-lax morphisms over $\bcat{S}$ denoted by $\on{Map}_{\bcat{S}}^{0,j}(X,Y)^{\on{Elax}}$ via the universal property that identifies every functor of $(\infty,1)$-categories 
   \[
     \scr{C} \xlongrightarrow{} \on{Map}_{\bcat{S}}^{0,j}(\bcat{X},\bcat{Y})^{\on{Elax}}
   \]
   with an $E$lax morphism over $\bcat{S}$
   \[
     \begin{tikzcd}
       \bcat{X}\times \scr{C}\arrow[rr] \arrow[dr,swap,"p"] & & \bcat{Y} \arrow[dl,"q"] \\
       & \bcat{S} &
     \end{tikzcd}
   \]
   (see \autoref{rem:fibtensor}). We denote by $\bcat{F}\!\on{ib}_{0,j}(\bcat{S})^{\on{Elax}}$ the resulting $(\infty,2)$-category. We similarly define in the case of $(1,j)$-fibrations the $(\infty,2)$-category $\bcat{F}\!\on{ib}_{1,j}(\bcat{S})^{\on{Eoplax}}$
 \end{definition}

 \begin{definition}\label{def:elaxnatural}
    Let $(\bcat{C},E)$ denote the pair consisting of an $(\infty,2)$-category $\bcat{C}$ together with a collection of edges $E$ containing all the equivalences in $\bcat{C}$. Given another $(\infty,2)$-category $\bcat{A}$, we define the $(\infty,2)$-category $\on{Fun}(\bcat{C},\bcat{A})^{\on{Elax}}$ of functors and $E$lax natural transformations as the subcategory of $\on{Fun}^{\mathbf{gr}}((\bcat{C},\flat),\bcat{A})$ (see \autoref{def:graymapping}) consisting in those lax natural transformations  
    \[
      \Delta^1 \tensor \bcat{C} \xlongrightarrow{} A
    \]
   such that for every edge $e:\Delta^1 \xlongrightarrow{} \bcat{C}$ where $e \in E$ then it follows that the restriction 
   \[
       \Delta^1 \tensor \Delta^1 \xlongrightarrow{ \on{id} \times e}  \Delta^1 \tensor \bcat{C}  \xlongrightarrow{} A  
  \]   
   defines a commutative square in $\bcat{A}$. We dually define the $(\infty,2)$-category $\on{Fun}(\bcat{C},\bcat{A})^{\on{Eoplax}}$.
 \end{definition}

 \begin{remark}\label{rem:Eisequiv}
   We observe that if $E$ is precisely the collection of equivalences we have that the definition above coincides with the construction given in \autoref{def:graymapping}. In this case we will use the notation $\on{Fun}(\bcat{C},\bcat{A})^{\on{lax}}$ and $\on{Fun}(\bcat{C},\bcat{A})^{\on{oplax}}$.
 \end{remark}

 \begin{remark}
   Note that \autoref{def:elaxnatural} is precisely the category of partially lax natural transformations studied in \cite{GHL_LaxLim}.
 \end{remark}
 
 We conclude the section with a version of the straightening-unstraightening equivalence for $E$lax transformations found in \cite{Agh}.
 \begin{theorem}\label{thm:laxGroth}
    Let $(\bcat{S},E)$ denote the pair consisting of an $(\infty,2)$-category $\bcat{S}$ together with a collection of edges $E$ containing all the equivalences in $\bcat{S}$. Then there exists a natural equivalences of $(\infty,2)$-categories
    \[
      \on{St}_{\bcat{S}}^{\on{Elax}}: \bcat{F}\!\on{ib}_{0,1}(\bcat{S})^{\on{Elax}} \llra \on{Fun}\left(\bcat{S},\bcat{C}\!\on{at}_{(\infty,2)}\right)^{\on{Elax}}: \on{Un}_{\bcat{S}}^{\on{Elax}},
    \]
    \[
      \on{St}_{\bcat{S}}^{\on{Eoplax}}: \bcat{F}\!\on{ib}_{1,0}(\bcat{S})^{\on{Eoplax}} \llra \on{Fun}\left(\bcat{S}^{\op},\bcat{C}\!\on{at}_{(\infty,2)}\right)^{\on{Eoplax}}: \on{Un}_{\bcat{S}}^{\on{Eoplax}},
    \]
    and similarly for the rest of the variances.
 \end{theorem}
 \subsubsection{A model for the theory of \texorpdfstring{$(0,1)$-}-fibrations}
 We briefly recall an specific model for the theory $(0,1)$-fibrations which will useful for computational purpouses in the next section. The results here exposed can be found in \cite{AbGroth}.

 \begin{definition}
  A \emph{marked biscaled} simplicial set ($\bS$ simplicial set) is given by the following data
  \begin{itemize}
    \item A simplicial set $X$.
    \item A collection of edges  $E_X \in X_1$ containing all degenerate edges.
    \item A collection of triangles $T_X \in X_2$ containing all degenerate triangles. We will refer to the elements of this collection as \emph{thin triangles}.
    \item A collection of triangles $C_X \in X_2$ such that $T_X \subseteq C_X$. We will refer to the elements of this collection as \emph{lean triangles}.
  \end{itemize}
  We will denote such objects as triples $(X,E_X, T_X \subseteq C_X)$. A map $(X,E_X, T_X \subseteq C_X) \xlongrightarrow{} (Y,E_Y,T_Y \subseteq C_Y)$ is given by a map of simplicial sets $f:X \xlongrightarrow{} Y$ compatible with the collections of edges and triangles above. We denote by $\mbsSet$ the category of mb simplicial sets.
\end{definition}

\begin{notation}
  Let $(X,E_X, T_X \subseteq C_X)$ be a mb simplicial set. Suppose that the collection $E_X$ consist only of degenerate edges. Then we fix the notation $(X,E_X, T_X \subseteq C_X)=(X,\flat,T_X \subseteq E_X)$ and do similarly for the collection $T_X$. If $C_X$ consists only of degenerate triangles we fix the notation $(X,E_X, T_X \subseteq C_X)=(X,E_X, \flat)$. In an analogous fashion we wil use the symbol “$\sharp$“ to denote a collection containing all edges (resp. all triangles). Finally suppose that $T_X=C_X$ then we will employ the notation $(X,E_X,T_X)$.
\end{notation}

\begin{remark}
  We will often abuse notation when defining the collections $E_X$ (resp. $T_X$, resp. $C_X$) and just specified its non-degenerate edges (resp. triangles).
\end{remark}

  \begin{definition}\label{def:mbsanodyne}
  The set of \emph{generating mb anodyne maps} \(\bS\) is the set of maps of mb simplicial sets consisting of:
  \begin{enumerate}
    \myitem{(A1)}\label{mb:innerhorn} The inner horn inclusions 
    \[
    \bigl(\Lambda^n_i,\flat,\flat \subset \{\Delta^{\{i-1,i,i+1\}}\}\bigr)\xlongrightarrow{} \bigl(\Delta^n,\flat,\flat \subset \{\Delta^{\{i-1,i,i+1\}}\}\bigr)
    \quad , \quad n \geq 2 \quad , \quad 0 < i < n ;
    \]
    These maps force left-degenerate lean-scaled triangles to represent Cartesian edges of the mapping category.
    \myitem{(A2)}\label{mb:wonky4} The map 
    \[
    (\Delta^4,\flat,\flat \subset T)\xlongrightarrow{} (\Delta^4,\flat,\flat \subset T\cup \{\Delta^{\{0,3,4\}}, \ \Delta^{\{0,1,4\}}\}),
    \]
    where we define
    \[
    T\overset{\text{def}}{=}\{\Delta^{\{0,2,4\}}, \ \Delta^{\{ 1,2,3\}}, \ \Delta^{\{0,1,3\}}, \ \Delta^{\{1,3,4\}}, \ \Delta^{\{0,1,2\}}\};
    \]
    \myitem{(A3)}\label{mb:2coCartesianmorphs} The set of maps
    \[
    \Bigl(\Lambda^n_0,\{\Delta^{\{0,1\}}\}, \{ \Delta^{\{0,1,n\}} \}\Bigr) \xlongrightarrow{} \Bigl(\Delta^n,\{\Delta^{\{0,1\}}\}, \{ \Delta^{\{0,1,n\}} \}\Bigr) \quad , \quad n \geq 2.
    \]
    This forces the marked morphisms to be $p$-coCartesian with respect to the given thin triangles. 
    \myitem{(A4)}\label{mb:2CartliftsExist} The inclusion of the initial vertex
    \[
    \Bigl(\Delta^{0},\sharp,\sharp \Bigr) \xlongrightarrow{} \Bigl(\Delta^1,\sharp,\sharp \Bigr).
    \]
    This requires $p$-coCartesian lifts of morphisms in the base to exist.
    \myitem{(S1)}\label{mb:composeacrossthin} The map
    \[
    \Bigl(\Delta^2,\{\Delta^{\{0,1\}}, \Delta^{\{1,2\}}\},\sharp \Bigr) \xlongrightarrow{} \Bigl(\Delta^2,\sharp,\sharp \Bigr),
    \]
    requiring that $p$-coCartesian morphisms compose across thin triangles.
    \myitem{(S2)}\label{mb:coCartoverThin} The map
    \[
    \Bigl(\Delta^2,\flat,\flat \subset \sharp \Bigr) \xlongrightarrow{} \Bigl( \Delta^2,\flat,\sharp\Bigr),
    \]
    which requires that lean triangles over thin triangles are, themselves, thin.
    \myitem{(E)}\label{mb:equivalences} For every Kan complex $K$, the map
    \[
    \Bigl( K,\flat,\sharp  \Bigr) \xlongrightarrow{} \Bigl(K,\sharp, \sharp\Bigr).
    \]
    Which requires that every equivalence is a marked morphism.
  \end{enumerate}
  A map of $\bS$ simplicial sets is said to be \bS-anodyne if it belongs to the weakly saturated closure of \bS.
\end{definition}

\begin{definition}
  A map of $\bS$ simplicial sets is said to be an \bS-\emph{fibration} if it has the right lifting property against the class of \bS-anodyne maps.
\end{definition}

\begin{definition}
  Given a scaled simplicial set $(S,T_S)$ we define the category $\left(\on{Set}^{\mathbf{mb}}_\Delta\right)_{/S}$ of $\bS$ simplicial sets over $(S,\sharp,T_S \subset \sharp)$ as follows:
  \begin{itemize}
    \item The objects are maps $p:(X,E_X,T_X \subset C_X) \xlongrightarrow{} (S,\sharp,T_S \subset \sharp)$.
    \item A morphism from $p:(X,E_X,T_X \subset C_X) \xlongrightarrow{} (S,\sharp,T_S \subset \sharp)$ to $q:(Y,E_Y,T_Y \subset C_Y) \xlongrightarrow{} (S,\sharp,T_S \subset \sharp)$ is given by a map $f: (X,E_X,T_X \subseteq C_X) \xlongrightarrow{} (Y,E_Y,T_Y \subset C_Y) $ such that $q \circ f=p$.
  \end{itemize}
  An object of $\left(\on{Set}^{\mathbf{mb}}_\Delta\right)_{/S}$ is said to be a $(0,1)$-fibration if the corresponding map of $\bS$ simplicial sets is an $\bS$-fibration.
\end{definition}

\begin{theorem}\label{thm:model}
  Let $S$ be a scaled simplicial set. Then there exists a left proper combinatorial simplicial model structure on $\left(\on{Set}^{\mathbf{mb}}_\Delta\right)_{/S}$, which is characterized uniquely by the following properties:
  \begin{itemize}
    \item[C)] A morphism $f:X \xlongrightarrow{} Y$ in $\left(\on{Set}^{\mathbf{mb}}_\Delta\right)_{/S}$ is a cofibration if and only if $f$ induces a monomorphism on the underlying simplicial sets.
    \item[F)] An object $p:X \xlongrightarrow{} S$ in $\left(\on{Set}^{\mathbf{mb}}_\Delta\right)_{/S}$ is fibrant if and only if it is a $(0,1)$-fibration.
  \end{itemize}
\end{theorem}

\begin{remark}
  Let us suppose that $(S,T_S)$ is actually given by a fibrant scaled simplicial. Given an object $p:X \xlongrightarrow{} S$  in $\left(\on{Set}^{\mathbf{mb}}_\Delta\right)_{/S}$ it follows that $p$ defines a fibrant object in the model structure given in \autoref{thm:model} if and only if $X$ is an $\infty$-bicategory and the map $p$ it is a $(0,1)$-fibration in the sense of \autoref{def:ijfib}.
\end{remark}

 \subsection{Right lax Kan extensions}
 Let $(\bcat{C},E)$ be an $(\infty,2)$-category equipped with a collection of edges containing the equivalences and let $f:\bcat{C} \xlongrightarrow{} \bcat{D}$ be a functor. The goal of this section is to produce a right adjoint to the restriction functor
 \[
   f^*: \on{Fun}\left(\bcat{D},\bcat{C}\!\on{at}_{(\infty,2)}\right) \xlongrightarrow{} \on{Fun}\left(\bcat{C},\bcat{C}\!\on{at}_{(\infty,2)}\right)^{\on{Elax}}
 \]
 building the first step towards a general theory of partially lax Kan extensions which will appear in \cite{Kan}. For the rest of this section we fix once and for all the data of $(\bcat{C},E)$ and $f:\bcat{C}\xlongrightarrow{} \bcat{D}$. Throughout this section we will the model structure given in \autoref{thm:model} to deal with $(0,1)$-fibrations.

 \begin{remark}\label{rem:nicemarking}
   Let $(\bcat{C},E)$ be as above. We say that $E$ is saturated if given a \emph{thin} 2-simplex $\sigma:\Delta^2 \xlongrightarrow{} \bcat{C}$ such that:
   \begin{itemize}
     \item The restriction of $\sigma$ to $\Delta^{\{0,2\}}$ belongs to $E$.
     \item The restriction of $\sigma$ to $\Delta^{\{0,1\}}$ \emph{or} the restriction of $\sigma$ to $\Delta^{\{1,2\}}$ is an equivalence in $\bcat{C}$.
   \end{itemize}
   Then it follows that the remaining edge of $\sigma$ belongs to $E$. We denote by $E^{\on{sat}}$ the smallest saturated marking such that $E \subseteq E^{\on{sat}}$.

   Finally we observe that it is easy to see that we have an equivalence of $(\infty,2)$-categories
   \[
     \on{Fun}\left(\bcat{C},\bcat{C}\!\on{at}_{(\infty,2)}\right)^{\on{Elax}} \xlongrightarrow{\isom} \on{Fun}\left(\bcat{C},\bcat{C}\!\on{at}_{(\infty,2)}\right)^{\on{E^{\on{sat}}lax}}
   \]
   so from this point on we will assume without loss of generality that $E$ is saturated.
 \end{remark}

 \begin{definition}\label{def:freeId}
   Let $\bcat{D}$ be an $(\infty,2)$-category (incarnated as a fibrant scaled simplicial set). We denote by $\mathbb{F}(\bcat{D})=\on{Fun}(\Delta^1,\bcat{D})^{\on{oplax}}$ (see \autoref{rem:Eisequiv}) and observe that we have maps
    \[
     \on{ev}_i: \mathbb{F}(\bcat{D}) \xlongrightarrow{} \bcat{D}
    \]
    given by evaluation at the vertex $i$ for $i \in \{0,1\}$.  It follows from Theorem 2.2.6 in \cite{GHL_Cart} that evaluation at $0$ yields a $(1,0)$-fibration. One can similarly show that the map $\on{ev}_1$ defines a $(0,1)$-fibration. 
 \end{definition}

 \begin{remark}\label{rem:scalingfree}
   Let us consider the $(1,0)$-fibration $\on{ev}_0: \mathbb{F}(\bcat{D}) \xlongrightarrow{} \bcat{D}$. We view $\mathbb{F}(\bcat{D})$ as a $\bS$ simplicial set as follows:
   \begin{itemize}
     \item[i)] An edge $e$ in $\mathbb{F}(\bcat{D})$ is declared marked if  $\ev_1(e)$ is an equivalence in $\bcat{D}$ and the associated square commutes.
     \item[ii)] A triangle $\sigma$ in $\mathbb{F}(\bcat{D})$ is declared lean if $\ev_1(\sigma)$ is thin in $\bcat{D}$.
     \item[iii)] A triangle $\rho$ in $\mathbb{F}(\bcat{D})$ is declared thin if $\ev_i(\sigma)$ is thin $\bcat{D}$ for $i=0,1$.
   \end{itemize}
 \end{remark}

 \begin{definition}\label{def:fibersfree}
   We set the following notation
   \[
       \begin{tikzcd}
         \bcat{D}_{d \upslash} \arrow[r] \arrow[d] & \mathbb{F}(\bcat{D}) \arrow[d,"\ev_0"]  \\
         \Delta^0 \arrow[r,"d"] & \bcat{D}
       \end{tikzcd} \enspace \enspace \enspace 
       \begin{tikzcd}
         \bcat{C}_{d \upslash} \arrow[r] \arrow[d] &  \bcat{D}_{d \upslash} \arrow[d,"\ev_1"]  \\
         \bcat{C} \arrow[r,"f"] & \bcat{D}.
       \end{tikzcd} 
     \]
     where all of the diagrams above are pullbacks.
 \end{definition}

 \begin{definition}\label{def:cosimpKan}
   Let $p:(L,E_L,T_L \subseteq C_L) \xlongrightarrow{} (\bcat{D},\sharp,T_{\bcat{D}}\subset \sharp )$ be an object of $\left(\on{Set}^{\mathbf{mb}}_\Delta\right)_{/\bcat{D}}$. We define another $\bS$ simplicial set $R(p):(R(L),E_{R(L)},T_{R(L)}\subseteq C_{R(L)}) \xlongrightarrow{} (\bcat{C},\sharp,T_{\bcat{C}}\subseteq \sharp)$ as follows:
   \begin{itemize}
     \item[i)] The underlying simplicial set of $R(L)$ is defined via the diagram 
     \[\begin{tikzcd}
  {R(L)} && {\bcat{C}} \\
  {Q(L)} & {\mathbb{F}(\bcat{D})} & {\bcat{D}} \\
  L & {\bcat{D}}
  \arrow[from=1-1, to=2-1]
  \arrow[from=2-1, to=2-2]
  \arrow[from=2-1, to=3-1]
  \arrow["p"', from=3-1, to=3-2]
  \arrow["{\ev_0}", from=2-2, to=3-2]
  \arrow["{\ev_1}"', from=2-2, to=2-3]
  \arrow["{R(p)}", from=1-1, to=1-3]
  \arrow["f", from=1-3, to=2-3]
\end{tikzcd}\]
where all the squares depicted are pullback diagrams. 
    \item[ii)] The collection $E_{R(L)}$ consists in those edges $e$ in $R(L)$ such that:
     \begin{enumerate}
        \item The associated square in $\bcat{D}$ commutes.
        \item The image of $e$ in $L$ is marked.
      \end{enumerate} 
      \item[iii)] The collection $C_{R(L)}$ consists in those triangles $\sigma$ such the image of $\sigma$ in $L$ is lean.
      \item[iv)] The collection $T_{R(L)}$ consists in those lean triangles $\rho$ such that $R(p)(\rho)$ is thin in $\bcat{C}$. 
   \end{itemize}
   This construction yields a functor $R: \left(\on{Set}^{\mathbf{mb}}_\Delta\right)_{/\bcat{D}} \xlongrightarrow{} \left(\on{Set}^{\mathbf{mb}}_\Delta\right)_{/\bcat{C}}$.
 \end{definition}

 \begin{remark}\label{rem:hatE}
   Let $R(p):(R(L),E_{R(L)},T_{R(L)}\subseteq C_{R(L)}) \xlongrightarrow{} (\bcat{C},\sharp,T_{\bcat{C}}\subseteq \sharp)$ as above. It will useful to consider an alternative marking $\hat{E}_{R(L)}$ given by those edges which belong to $E_{R(L)}$ whose image under $R(p)$ belong to the marked edges in $\bcat{C}$. This gives a variant of our functor $R$ which we denote $\hat{R}$
 \end{remark}

 \begin{definition}
   Let $\left(\on{Set}^{\mathbf{mb}}_\Delta\right)^{\on{Elax}}_{/\bcat{C}}$ by the category of $\bS$ simplicial sets over $(\bcat{C},\sharp,T_{\bcat{C}}\subset \sharp)$ and morphims those maps of simplicial sets over $\bcat{C}$ that preserve marked edges over $E$ and every lean and thin triangle.
 \end{definition}

 \begin{definition}\label{def:fshriekmodel}
   We define a functor $f_{*}: \left(\on{Set}^{\mathbf{mb}}_\Delta\right)^{\on{Elax}}_{/\bcat{C}} \xlongrightarrow{} \left(\on{Set}^{\mathbf{mb}}_\Delta\right)_{/\bcat{D}}$ by means of the universal property
   \[
     \Hom_{\left(\on{Set}^{\mathbf{mb}}_\Delta\right)_{/\bcat{D}}}\left(L, f_{*}X\right) \isom \Hom_{\left(\on{Set}^{\mathbf{mb}}_\Delta\right)^{\on{Elax}}_{/\bcat{C}}}(R(L),X).
   \]
 \end{definition}

 \begin{remark}\label{rem:fhat}
   If we consider the alternative collection of marked edges discussed in \autoref{rem:hatE} it follows that we have isomorphisms
   \[
     \Hom_{\left(\on{Set}^{\mathbf{mb}}_\Delta\right)^{\on{Elax}}_{/\bcat{C}}}(R(L),X) \isom \Hom_{\left(\on{Set}^{\mathbf{mb}}_\Delta\right)_{/\bcat{C}}}(\hat{R}(L),X)
   \]
   which provide an alternative description of $f_*X$.
 \end{remark}

 \begin{lemma}\label{lem:kansatflower}
  Let $q:K \xlongrightarrow{} \bcat{D}$ be a functor where $K$ is a Kan complex and let $p:\bcat{X}\xlongrightarrow{} \bcat{C}$ be $(0,1)$-fibration. Let $K^\flat=(K,\flat,\sharp)$ and $K^\sharp=(K,\sharp,\sharp)$ be $\bS$ simplicial sets. Then there exists a solution to the lifting problem
    \[
       \begin{tikzcd}
         \hat{R}(K^\flat) \arrow[r,"\varphi"] \arrow[d] & \bcat{X} \arrow[d,"p"] \\
         \hat{R}(K^\sharp) \arrow[r] \arrow[ur,dotted] & \bcat{C}
       \end{tikzcd}
    \]   
 \end{lemma}
 \begin{proof}
   Let $e: \Delta^1\xlongrightarrow{}\hat{R}(K^\flat)$ be an edge which is marked in $\hat{R}(K^\sharp)$ we wish to show that $\varphi(e)$ is a $0$-cartesian edge of $\bcat{X}$. We identify $e$ with a commutative diagram in $\bcat{D}$
   \[
     \begin{tikzcd}
       p(k_0) \arrow[r,"p(u)"] \arrow[d] & p(k_1) \arrow[d] \\
       f(c_0) \arrow[r,"f(\alpha)"] & f(c_1)
     \end{tikzcd}
   \]
   where $\alpha$ is a marked edge in $\bcat{C}$. By picking an inverse of $u$ which we call $v$, we extend this diagram to 
   \[
     \begin{tikzcd}
      p(k_1) \arrow[r,"p(v)"] \arrow[d] & p(k_0) \arrow[r,"p(u)"] \arrow[d] & p(k_1) \arrow[d] \\
      f(c_0) \arrow[r,"="]& f(c_0) \arrow[r,"f(\alpha)"] & f(c_1)
     \end{tikzcd}
   \]
   This defines a thin 2-simplex $\rho:\Delta^2 \xlongrightarrow{} \hat{R}(K^\flat)$ with the following properties:
   \begin{enumerate}
     \item The edge $d_2(\rho)$ is an equivalence in $\hat{R}(K^\flat)$.
     \item The edge $d_1(\rho)$ is marked in $\hat{R}(K^\flat)$.
     \item The edge $d_0(\rho)=e$
   \end{enumerate}
   In particular it follows that $\varphi(\rho)$ is a thin 2-simplex such that every edge $0$-cartesian except possibly the edge $d_0(\varphi(\rho))$. It follows that this edge must be $0$-cartesian and the claim follows.
 \end{proof}

 \begin{proposition}
   Let $\bcat{X} \xlongrightarrow{} \bcat{C}$ be a $(0,1)$-fibration then $f_{*}\bcat{X} \xlongrightarrow{} \bcat{D}$ is again a $(0,1)$-fibration. Moreover, the functor $f_{*}$ descends to a morphism of $(\infty,2)$-categories
   \[
     f_*:\bcat{F}\!\on{ib}_{0,1}(\bcat{C})^{\on{Elax}} \xlongrightarrow{} \bcat{F}\!\on{ib}_{0,1}(\bcat{D}).
   \]
   which we call the $E$lax right Kan extension functor.
 \end{proposition}
 \begin{proof}
   First, let us assume that we have shown that $f_*\bcat{X}$ is a $(0,1)$-fibration and let us show that the map $f_*$ descends to a morphism of $(\infty,2)$-categories. Let $\scr{K}$ be an $(\infty,1)$-category and consider an $E$lax morphism over $\bcat{C}$ 
   \[
     \begin{tikzcd}
       \bcat{X}\times \scr{K}\arrow[rr] \arrow[dr,swap,"p"] & & \bcat{Y} \arrow[dl,"q"] \\
       & \bcat{C} &
     \end{tikzcd}
   \]
   We wish to show that we can define a morphism of $(0,1)$-fibrations over $\bcat{D}$,
   \[
     \begin{tikzcd}
       f_*\bcat{X}\times \scr{K}\arrow[rr] \arrow[dr,swap,"f_*(p)"] & & f_*\bcat{Y} \arrow[dl,"f_*(q)"] \\
       & \bcat{D} &
     \end{tikzcd}
   \]
   Suppose that we are given an $\mathbf{MB}$ simplicial set $A$ over $\bcat{D}$ and maps $R(A) \xlongrightarrow{} \bcat{X}$, $A \xlongrightarrow{} \scr{K}$. Then using that we always have a morphism $R(A) \xlongrightarrow{} A$ we can construct a map over $\bcat{C}$
   \[
     R(A) \xlongrightarrow{} \bcat{X} \times \scr{K} \xlongrightarrow{} \bcat{Y}
   \]
   which yields the definition of the map $ f_*\bcat{X}\times \scr{K} \xlongrightarrow{} f_* \bcat{Y}$. It is easy to see that this map preserves the marked edges in $f_* \bcat{X} \times \scr{K}$.

   We show now that $f_*\bcat{X} \xlongrightarrow{} \bcat{D}$ is a $(0,1)$-fibration by verifying that it has the right lifting property against the class of \bS-anodyne morphisms (see \autoref{def:mbsanodyne}). 

    Before perfoming the necessary verifications case by case we make several observations:
    \begin{enumerate}
       \item We note that the map $R(A) \xlongrightarrow{} A$ is always a $(1,0)$-fibration when equipped with the following decorations:
    \begin{itemize}
       \item An edge in $R(A)$ is 1-cartesian if the associated square in $\bcat{D}$ commutes and its image in $\bcat{C}$ is an equivalence.
       \item A triangle in $R(A)$ is lean if its image in $\bcat{C}$ is thin.
       \item A triangle in $R(A)$ is thin if it is lean and its image in $A$ is thin. We denote this collection of triangles $\mathcal{E}_A$.
     \end{itemize} 
     \item As a consequence of 1 above together with Lemma 3.55 and Proposition 3.79 in\footnote{For this is crucial to note that the proof of Proposition 3.79 only uses scaled anodyne morphisms in the filtration.} \cite{AGS_CartII} we see that if $A=(\Delta^n,\flat)\xrightarrow{\sigma} \bcat{D}$ we have a weak equivalence of scaled simplicial sets
     \[
       (P(\Delta^n),\mathcal{E}_{P(\Delta^n)})=\left(R(\Delta^{[0,n-1]}),\mathcal{E}_{\Delta^[0,n-1]}\right) \coprod_{\Delta^{[0,n-1]} \times \bcat{C}_{d\upslash}}\Delta^n \times \bcat{C}_{d\upslash} \xlongrightarrow{} (R(\Delta^n),\mathcal{E}_{\Delta^n}).
     \]
     where $d=\sigma(n)$. Here we are viewing $\Delta^n \times \bcat{C}_{d\upslash}$ as having usual scaling in the Cartesian product (and similarly for $\Delta^{[0,n-1]}\times \bcat{C}_{d\upslash}$).
     \item Given $(\Lambda^n_i,\flat)\xrightarrow{\sigma} \bcat{D}$ with $i<n$ we have a weak equivalence of scaled simplicial sets
     \[
        (P(\Lambda^n_i),\mathcal{E}_{P(\Lambda^n_i)})=\left(R(\Delta^{[0,n-1]}),\mathcal{E}_{\Delta^[0,n-1]}\right) \coprod_{\Delta^{[0,n-1]} \times \bcat{C}_{d\upslash}}\Lambda^n_i \times \bcat{C}_{d\upslash} \xlongrightarrow{} (R(\Lambda^n_i),\mathcal{E}_{\Lambda^n_i}).
     \]
     where $d=\sigma(n)$ and the scaling on the cartesian products is as above. To see this, we note that $R(\Lambda^n_i)$ splits as a colimit of its $n-1$-dimensional faces and use the previous point to each of the faces except the $n$-th face.
     \item More generally, given $\sigma: \Delta^n_\dagger=(\Delta^n,T_n) \xlongrightarrow{} \bcat{D}$ and $\sigma_{|\Lambda^n_i}:(\Lambda^n_i)_\dagger=(\Lambda^n_i,T_{(\Lambda^n_i)_\dagger}) \xlongrightarrow{} \bcat{D}$ (with $i<n$) where the scaling of $\Lambda^n_i$ is inherited from $T_n$ we define
     \[
       (P(\Delta^n_\dagger),\mathcal{E}_{P(\Delta^n_\dagger)})=\left(R(\Delta^{[0,n-1]}_\dagger),\mathcal{E}_{\Delta^{[0,n-1]}_\dagger}\right) \coprod_{(\Delta^{[0,n-1]}_{\dagger}\times \bcat{C}_{d\upslash})}(\Delta^n,T) \times \bcat{C}_{d\upslash}
     \]
     where $\Delta^{[0,n-1]}_\dagger$ is the obvious restriction. We claim that we have weak equivalences of scaled simplicial sets
     \[
        (P(\Delta^n_\dagger),\mathcal{E}_{P(\Delta^n_\dagger)}) \xlongrightarrow{} (R(\Delta^n_\dagger),\mathcal{E}_{\Delta^n_\dagger}), \enspace \enspace (P((\Lambda^n_i)_\dagger),\mathcal{E}_{P((\Lambda^n_i)_\dagger)})\xlongrightarrow{} (R((\Lambda^n_i)_\dagger),\mathcal{E}_{(\Lambda^n_i)_\dagger}).
     \]
     We check this for the first map since the remaining case follows in an exactly identical fashion. We fix the notation $\mathbf{P}^\dagger= (P(\Delta^n_\dagger),\mathcal{E}_{P(\Delta^n_\dagger)})$ and consider a pushout diagram
     \[
       \begin{tikzcd}
          (P(\Delta^n),\flat,\mathcal{E}_{P(\Delta^n)}) \arrow[r] \arrow[d] & (R(\Delta^n),\flat,\mathcal{E}_{\Delta^n}) \arrow[d] \\
          \mathbf{P}^\dagger \arrow[r] & \mathbf{T}^\dagger
       \end{tikzcd}
     \]
     since the left-most vertical morphism is a cofibration and the top horizontal morphism is a weak equivalence by 3 (resp. 4 for the horn case) it follows that it will be enough to show that the induced morphism
     \[
       \mathbf{T}^{\dagger} \xlongrightarrow{} \mathbf{R}^\dagger=(R(\Delta^n_\dagger),\mathcal{E}_{\Delta^n_\dagger})
     \]
     is a weak equivalence. Let $\rho:\Delta^2 \xlongrightarrow{} \mathbf{R}^\dagger$ be a thin simplex which does not factor through $\mathbf{T}^\dagger$. Then it follows that the image of $\rho$ in $\Delta^n$ which we denote $\overline{\rho}$ is of the form $i \xlongrightarrow{} j \xlongrightarrow{} n$. We consider a morphism
     \[
       \kappa:\Delta^2 \xlongrightarrow{v} \Delta^n \times \bcat{C}_{d\upslash} \xlongrightarrow{} \mathbf{R}^\dagger
     \]
     where the value of $\kappa$ at $\Delta^n$ is given by $\overline{\rho}$ and the value of $\kappa$ at $ \bcat{C}_{d\upslash}$ is degenerate on the object $\rho(2)$. Let $M=(\Delta^3,E_M,T_M \subseteq C_M)$ be a $\bS$ simplicial set defined as follows:
     \begin{itemize}
       \item An edge is marked if it is of the form $i \xlongrightarrow{} 3$ with $i \neq 0$.
       \item Every triangle is lean and every triangle is thin except $\Delta^{\{0,2,3\}}$ and $\Delta^{\{1,2,3\}}$. 
     \end{itemize}
     We consider a map $\iota: M \xlongrightarrow{} \Delta^n$ given by $s_0(\overline{\rho})$. We further define a subsimplicial subset (with the inherited decorations) $N \subset M$ consisting in the following faces:
     \begin{itemize}
       \item The face missing the vertex $0$.
       \item The edge $0 \xlongrightarrow{} 3$.
     \end{itemize}
      We look at a commutative diagram
     \[
       \begin{tikzcd}
         N \arrow[d] \arrow[r] & \mathbf{R}^\dagger \arrow[d] \\
         M \arrow[r,"\iota"] \arrow[ur,"\alpha",dotted] & \Delta^n_\dagger
       \end{tikzcd}
     \]
     where the top horizontal morphism maps the face missing the vertex $0$ to $\kappa$ and the edge $0 \xlongrightarrow{} 3$ to $d_1(\rho)$. Since the left-most vertical morphism is anodyne (with respect to the model structure of $(1,0)$-fibrations) it follows that the dotted arrow $\alpha$ exists. By construction, every triangle of $\alpha$ is thin in $\mathbf{R}^\dagger$ excepts possibly the face skipping the vertex $1$. It follows that we can use a scaled anodyne morphism of type \ref{i:saturation} in \autoref{def:scanodyne} to scale $d_1(\alpha)$.

     To finish the proof of the claim we consider a morphism,
     \[
       e:\Delta^1 \xlongrightarrow{} \Delta^n \times \bcat{C}_{d\upslash} \xlongrightarrow{} \mathbf{R}^\dagger
     \]
     whose component at $\Delta^n$ is given by $d_0(\overline{\rho})$ and whose component at $\bcat{C}_{d \upslash}$ is degenerate on $\rho(2)$. By construction this map yields a $1$-cartesian edge in $\mathbf{R}^\dagger$. We construct a commutative diagram
      \[
       \begin{tikzcd}
         \Lambda^2_2 \arrow[r,"\Xi"] \arrow[d] & \mathbf{R}^\dagger \arrow[d] \\
         \Delta^2 \arrow[r] \arrow[ur,dotted,"\tau_{e}"] & \Delta^n_\dagger
       \end{tikzcd}
     \]
     where the bottom horizontal morphism is $s_0d_0(\overline{\rho})$ and the map $\Xi$ maps $0\xlongrightarrow{} 2$ to $d_0(\rho)$ and $\Xi$ maps $1 \xlongrightarrow{} 2$ to $e$. We note that the solution of this lifting problem $\tau_e$ is a thin simplex. Finally we consider a lifting problem
     \[
       \begin{tikzcd}
         \Lambda^3_3 \arrow[d] \arrow[r] & \mathbf{R}^\dagger \arrow[d] \\
         \Delta^3 \arrow[r] \arrow[ur,"\beta",dotted] & \Delta^n_\dagger
       \end{tikzcd}
     \]
     where the morphism $\Delta^3 \xlongrightarrow{} \Delta^n$ is given by $s_1(\overline{\rho})$ and the top horizontal morphism is defined as follows:
     \begin{itemize}
       \item We map the $0$-face via $\tau_{e}$ (see the case of marked edges for a definition of $\tau_e$).
       \item We map the $1$-face via $d_1(\alpha)$.
       \item We map the $2$-face via $\rho$. 
     \end{itemize}
     It is easy to see that the faces of $\beta$ are all scaled except possibly $d_2(\beta)=\rho$. In particular, we can again scale this face using a scaled anodyne morphism of type \ref{i:saturation} in \autoref{def:scanodyne}. The claim now follows.

     \item It is easy to see that we can upgrade the equivalences above to weak equivalences in the model structure of $(0,1)$-fibrations over $\bcat{C}$
     \[
       (P(\Delta^n_\dagger),\flat,\mathcal{E}_{P(\Delta^n_\dagger)}) \xlongrightarrow{} (R(\Delta^n_\dagger),\flat,\mathcal{E}_{\Delta^n_\dagger}),
     \]
     \[
        (P((\Lambda^n_i)_\dagger),\flat,\mathcal{E}_{P((\Lambda^n_i)_\dagger)}) \xlongrightarrow{} (R((\Lambda^n_i)_\dagger),\flat,\mathcal{E}_{(\Lambda^n_i)_\dagger})
     \]

     \item Let $Q=(\Delta^n,E_Q,T_Q \subseteq C_Q) \xrightarrow{\sigma} \bcat{D}$ and let $(Q_i=\Lambda^n_i,E_{Q_i},T_{Q_i}\subseteq C_{Q_i}) \xrightarrow{\sigma_i} \bcat{D}$ with $i<n$ be a pair of $\bS$ simplicial sets over $\bcat{D}$ where the decorations of $Q_i$ are induced from $Q$. Let us assume that $E_Q$ satisfies the following property:
     \begin{itemize}
       \item[$\ast)$] Given a non-degenerate edge $i \xlongrightarrow{} j$ in $E_Q$ such that $j \neq n$ then it follows that $i \xlongrightarrow{} j \xlongrightarrow{} n$ belongs to $T_Q$. 
     \end{itemize}
     Let $(R(\Delta^n,\hat{E}_{R(Q)},T_{R(Q)}\subseteq C_{R(Q)})$ (resp. $(R(\Lambda^n_i),\hat{E}_{R(Q_i)},T_{R(Q_i)}\subseteq C_{R(Q_i)})$) as in \autoref{def:cosimpKan} and \autoref{rem:hatE}. We look at the $\bS$ simplicial sets as defined un point 5 $(P(Q),\flat,\mathcal{E}_{P(Q)})$ (resp. $(P(Q_i),\flat,\mathcal{E}_{P(Q_i)})$) and define new decorations for $P(Q)$ (resp. $P(Q_i)$) as follows:
     \begin{itemize}
       \item[i)] We decorate the factor of the pushout corresponding to $R(\Delta^{[0,n-1]})$ with the decorations induced from $(R(\Delta^n),\hat{E}_{R(Q)},T_{R(Q)}\subseteq C_{R(Q)})$.
       \item[ii)] We decorate the factor $\Delta^n \times \bcat{C}_{d\upslash}$ by declaring an edge to be marked if its image in $\Delta^n$ belongs to $E_Q$, its image in $\bcat{C}_{d\upslash}$ defines a commutative square in $\bcat{D}$ and its projection to $\bcat{C}$ is marked. A triangle is declared lean if its image in $\Delta^n$ belongs to $C_Q$. Finally we declare a triangle to be thin if it is lean and its image in $\bcat{C}_{d \upslash}$ is thin.
       \item[iii)] We denote the resulting $\bS$ simplicial set by $\mathbf{U}(Q)$ (resp. $\mathbf{U}(Q_i)$).
     \end{itemize}
    We claim that we have an induced morphism $\mathbf{U}(Q) \xlongrightarrow{} (R(\Delta^n),\hat{E}_{R(Q)},T_{R(Q)}\subseteq C_{R(Q)})$ (and similarly for the horn). To check this it suffices to verify that the map 
    \[
      \Phi:\Delta^n \times \bcat{C}_{d\upslash} \xlongrightarrow{} R(\Delta^n)
    \]
    is compatible with the decorations. The case of the triangles follows immediately. Let $e: \Delta^1  \xlongrightarrow{} \Delta^n \times \bcat{C}_{d \upslash}$ and let us denote the value of $e$ in $\Delta^n$ by $i \xlongrightarrow{} j$. We identify the image of $e$ in $\bcat{C}_{d \upslash}$ with a diagram in $\bcat{D}$
    \[\begin{tikzcd}
  {\sigma(n)} && {\sigma(n)} \\
  {f(c_0)} && {f(c_1)}
  \arrow["{\operatorname{id}}", from=1-1, to=1-3]
  \arrow["u"', from=1-1, to=2-1]
  \arrow["v", from=1-3, to=2-3]
  \arrow["{f(\alpha)}"', from=2-1, to=2-3]
  \arrow[""{name=0, anchor=center, inner sep=0}, curve={height=6pt}, draw=none, from=1-1, to=2-3]
  \arrow["\simeq"{pos=0.4}, shorten <=5pt, shorten >=10pt, Rightarrow, from=0, to=1-3]
\end{tikzcd}\]
where $\theta$ is invertible and $\alpha$ is marked in $\bcat{C}$. Finally let $\sigma(e(0) \xlongrightarrow{} e(1))=s$. We consider two cases $e(1)=n$ and $e(1)\neq n$. If $e(1)=n$ the value of $\Phi$ can be identified with a commutative diagram 
\[\begin{tikzcd}
  {\sigma(i)} && {\sigma(n)} \\
  {f(c_0)} && {f(c_1)}
  \arrow["s", from=1-1, to=1-3]
  \arrow["{u\circ s}"', from=1-1, to=2-1]
  \arrow["v", from=1-3, to=2-3]
  \arrow["{f(\alpha)}"', from=2-1, to=2-3]
  \arrow[""{name=0, anchor=center, inner sep=0}, curve={height=6pt}, draw=none, from=1-1, to=2-3]
  \arrow["\simeq"{pos=0.4}, shorten <=5pt, shorten >=11pt, Rightarrow, from=0, to=1-3]
\end{tikzcd}\]
which defines a marked edge in the target. Let us assume that $e(1)\neq n$. If $e(0)=e(1)$ the claim follows in a similar way as before so let us assume that $e(0)\neq e(1)$. We consider a 2-simplex  $\varphi:\Delta^2 \xlongrightarrow{} \Delta^n$ given by $i \xlongrightarrow{} j \xlongrightarrow{} n$.  Note that by condition $\ast)$ above it follows that $\varphi$ belongs to $T_Q$. We fix some notation for the image of $\varphi$ in $\bcat{D}$ under $\sigma$
\[\begin{tikzcd}
  & {\sigma(j)} \\
  {\sigma(i)} && {\sigma(n)}
  \arrow["s", from=2-1, to=1-2]
  \arrow["t", from=1-2, to=2-3]
  \arrow[""{name=0, anchor=center, inner sep=0}, "w"', from=2-1, to=2-3]
  \arrow["\simeq"', shorten <=5pt, shorten >=3pt, Rightarrow, from=0, to=1-2]
\end{tikzcd}\]
Finally we can now identify the value of $\Phi(e)$ with a diagram
\[\begin{tikzcd}
  {\sigma(i)} && {\sigma(j)} \\
  {f(c_0)} && {f(c_1)}
  \arrow["s", from=1-1, to=1-3]
  \arrow["{u\circ w}"', from=1-1, to=2-1]
  \arrow["{v\circ t}"', swap, from=1-3, to=2-3]
  \arrow["{f(\alpha)}"', from=2-1, to=2-3]
  \arrow[""{name=0, anchor=center, inner sep=0}, shift right=2, draw=none, from=1-1, to=2-3]
  \arrow["\theta"{pos=0.4}, shorten <=5pt, shorten >=10pt, Rightarrow, from=0, to=1-3]
\end{tikzcd}\]
where $\theta$ can be expresed as a composite of invertible two morphisms 
\[
  f(\alpha) \circ u \circ w \xRightarrow{\simeq} f(\alpha)\circ u \circ t \circ s \xRightarrow{\simeq} v \circ t \circ s
\]
which shows that $\Phi$ is compatible with our new decorations.

Now that the claim is established we can construct a factorization of the map
\[
  \mathbf{U}(Q) \xlongrightarrow{\simeq} \mathbf{V}(Q) \xlongrightarrow{\gamma} (R(\Delta^n),\hat{E}_{R(Q)},T_{R(Q)}\subseteq C_{R(Q)})
\]
 (resp. for $Q_i$) where the first morphism is obtained via the homotopy pushout
\[
  \begin{tikzcd}
    (P(Q),\flat,\mathcal{E}_{P(Q)}) \arrow[d] \arrow[r,"\simeq"] & (R(\Delta^n),\flat,\mathcal{E}_{Q}) \arrow[d] \\
    \mathbf{U}(Q) \arrow[r,"\simeq"] & \mathbf{V}(Q)
  \end{tikzcd}
\]
here we are using point $5$ to guarantee that the top (and hence the bottom) horizontal morphism is a weak equivalence in the model structure of $(0,1)$-fibrations over $\bcat{C}$. We claim that the map $\gamma$ is a weak equivalence. To show this we will add the remaining decorations using pushout along $\bS$-anodyne morphisms (see \autoref{def:mbsanodyne}). The argument for adding the decorations on the triangles is very similar to that of $4$ and therefor left as an exercise.  Let $e: \Delta^1 \xlongrightarrow{} R(\Delta^n)$ which belongs to $\hat{E}_{R(Q)}$ but does not factor through $\mathbf{V}(Q)$. This implies that the image of $e$ in $\Delta^n$ is of the form $i \xlongrightarrow{} n$ with $i \neq n$. Then one can easily construct a thin 2-simplex $\rho$ in $\mathbf{V}(Q)$ such that:
\begin{itemize}
  \item The edge $\rho(1 \xlongrightarrow{} 2)$ is marked and lies in the image of $\Phi$.
  \item The edge $\rho(0 \xlongrightarrow{} 1)$ is marked and its image in $\Delta^n$ is degenerate in $i$.
  \item The edge $\rho(0 \xlongrightarrow{} 2)$ equals $e$.
\end{itemize}
The claim now follows by taking a pushout along of morphism of type \ref{mb:composeacrossthin} in \autoref{def:mbsanodyne}.
     \end{enumerate} 
   
   Equipped with observations 1-6 we are left to prove the following claim.

   \begin{itemize}
     \item[\textbf{Claim}:] Given a $\bS$-anodyne morphism  $i:A \xlongrightarrow{} B$ over $\bcat{D}$ then the induced cofibration 
     \[
       \hat{R}(A) \xlongrightarrow{} \hat{R}(B)
     \]
     is a trivial cofibration in the model structure of $(0,1)$-fibrations over $\bcat{C}$.

     The case where the morphism $i$ is of type \ref{mb:equivalences} in \autoref{def:mbsanodyne} was done in \autoref{lem:kansatflower}. If $i$ is of type \ref{mb:coCartoverThin} then the induced map is an isomorphism. For the remaining cases we use 1-6 to see that we have a commutative diagram
     \[
       \begin{tikzcd}
         \mathbf{U}(A) \arrow[r,"\simeq"] \arrow[d,"\simeq"] & \hat{R}(A) \arrow[d] \\
       \mathbf{U}(B) \arrow[r,"\simeq"] & \hat{R}(B)
       \end{tikzcd}
     \]
     where the horizontal maps are weak equivalences by the previous discussion. The fact that the left-most vertical morphism is a weak equivalence follows by the description of $\mathbf{U}(-)$ as a homotopy pushout. We conclude by 2-out-of-3. \qedhere
   \end{itemize}
 \end{proof}

 \begin{remark}\label{rem:othervariances}
   We can define an alternative version of the functor $R(-)$ by reversing exchanging the roles of $\ev_0$ and $\ev_1$ in \autoref{def:cosimpKan}. An analogous argument to that of the previous proposition shows that there exists a functor of $(\infty,2)$-categories (note the abuse of notation)
   \[
     f_{*} : \bcat{F}\!\on{ib}_{1,0}(\bcat{C})^{\on{Eoplax}} \xlongrightarrow{} \bcat{F}\!\on{ib}_{1,0}(\bcat{D})
   \]
   It is also possible to construct versions for the other variances but that requires working with dual versions of $\mathbb{F}(\bcat{D})$ given by $\Fun(\Delta^1,\bcat{D})^{\on{lax}}$.
 \end{remark}

\begin{remark}\label{rem:limitformula}
  Under the Grothendieck construction of \autoref{thm:laxGroth} we can identify $f_*$ with a functor
  \[
    f_*: \on{Fun}(\bcat{C},\bcat{C}\!\on{at}_{(\infty,2)})^{\on{Elax}} \xlongrightarrow{} \on{Fun}(\bcat{D},\bcat{C}\!\on{at}_{(\infty,2)})^{\on{Elax}}.
  \]
  Let $G:\bcat{C} \xlongrightarrow{} \bcat{C}\!\on{at}_{(\infty,2)}$ with associated $(0,1)$-fibration $\bcat{X} \xlongrightarrow{} \bcat{C}$. We wish to determine the value of $f_*G(d)$ which is in turn given by the fibre
  \[
    f_*\bcat{X}\times_{\bcat{D}}\{d\} \isom \on{Map}^{(0,1)}_{\bcat{C}}(\bcat{C}_{d\upslash},\bcat{X})^{\on{Elax}}.
  \]
  We can interpret the right-hand side using \cite{Agh} as the Elax limit of the composite
  \[
    \bcat{C}_{d\upslash} \xlongrightarrow{} \bcat{C} \xlongrightarrow{G} \bcat{C}\!\on{at}_{(\infty,2)}
  \]
  which gives the expected limit formula for the right lax Kan extension. Here the marking in $\bcat{C}_{d\upslash}$ consists in edges of the form
 \[\begin{tikzcd}
  & d \\
  {} \\
  {f(c_0)} && {f(c_1)}
  \arrow[from=1-2, to=3-1]
  \arrow["{f(\alpha)}"', from=3-1, to=3-3]
  \arrow[""{name=0, anchor=center, inner sep=0}, from=1-2, to=3-3]
  \arrow["\simeq"', shorten <=20pt, shorten >=12pt, Rightarrow, from=3-1, to=0]
\end{tikzcd}\]
where the associated triangle in $\bcat{D}$ commutes and the edge $\alpha:c_0 \xlongrightarrow{} c_1$ is marked in $\bcat{C}$.

    In the Eoplax case the right lax Kan extension of a funtor $H: \bcat{C}^\op \xlongrightarrow{}  \bcat{C}\!\on{at}_{(\infty,2)}$ is given by the Eoplax limit of the functor
  \[
    (\bcat{C}_{\upslash d})^\op \xlongrightarrow{} \bcat{C}^\op \xlongrightarrow{H} \bcat{C}\!\on{at}_{(\infty,2)}.
  \]
\end{remark}

\begin{definition}
 Given a pair of functors of $(\infty,2)$-categories $ L: \bcat{A} \llra \bcat{B}: R$ we say that $L$ is left adjoint to $R$ (denoted by $L \dashv R$) if there is a natural equivalences of functors
 \[
   \bcat{A}(-,R(-)) \isom \bcat{B}(L(-),-) : \bcat{A}^\op \times \bcat{B} \xlongrightarrow{} \bcat{C}\!\on{at}_{(\infty,1).}
 \]
 Equivalently, we have that $L \dashv R$ if we have unit counit natural transformations $\epsilon: \on{id}_{\bcat{A}} \xRightarrow{} RL$ and $\eta: LR \xRightarrow{} \on{id}_{\bcat{B}}$ satisfying the usual triangular identities.
\end{definition}

Our next order of bussiness is to construct unit and counit natural transformations for the functors
\[
     f^* :  \bcat{F}\!\on{ib}_{0,1}(\bcat{D}) \llra \bcat{F}\!\on{ib}_{0,1}(\bcat{C})^{\on{Elax}}:f_*
   \] 
   where $f^*$ is the pullback functor. Let $p:L \xlongrightarrow{} \bcat{C} \xlongrightarrow{} \bcat{D}$ be a $\bS$ simplicial set over $\bcat{D}$ such that $p$ factors through $\bcat{C}$. Note that exists a map $L \xlongrightarrow{} R(L)$ over $\bcat{C}$ induced by the projection
   \[
     \Delta^1 \otimes L \xlongrightarrow{} L
   \]
   which preserves the decorations. This in turn induces for every $(0,1)$-fibration $p:\bcat{X} \xlongrightarrow{} \bcat{C}$ an Elax morphism
   \[
     \begin{tikzcd}
       f^*f_* \bcat{X} \arrow[rr,"\eta_{\bcat{X}}"] \arrow[dr,"f^*R(p)",swap] & & \bcat{X} \arrow[dl,"p"] \\
       & \bcat{C} & 
     \end{tikzcd}
   \]
   This map $\eta_{\bcat{X}}$ is clearly natural on $\bcat{X}$ thus inducing a natural transformation $\eta: f^* f_* \Rightarrow \on{id}$.

   The definition of the unit $\epsilon: \on{id} \Rightarrow f_* f^*$ is slightly more involved and therefore we will proceed in two steps.
   \begin{enumerate}
     \item Given $\bcat{Y} \xlongrightarrow{} \bcat{D}$ and let $i_*\bcat{Y} \xlongrightarrow{} \bcat{D}$ be the value of the right lax Kan extension with respect to the map $(\bcat{D},\sharp) \xlongrightarrow{} (\bcat{D},\sharp)$. Given a $\bS$ simplicial $p: L \xlongrightarrow{} \bcat{D}$ it follows that a map $L \xlongrightarrow{} i_* \bcat{Y}$ is precisely given by a map of $\bS$-simplicial sets $Q(L) \xlongrightarrow{} \bcat{Y}$ over $\bcat{D}$ (respecting some decorations) (see \autoref{def:cosimpKan}) where we view $Q(L)$ as a $\bS$-simplicial over $\bcat{D}$ via the map
     \[
      \pi: Q(L) \xlongrightarrow{} \mathbb{F}(\bcat{D})\xlongrightarrow{\ev_1}\bcat{D}
     \]
     In an almost identical way as in the definition of $\eta$ we can consider an obvious map $L \xlongrightarrow{} Q(L)$ over $\bcat{D}$ which yields a morphism of $(0,1)$-fibrations $\Xi_{\bcat{Y}}:i_*\bcat{Y} \xlongrightarrow{} \bcat{Y}$. Using the dual version of Theorem 3.17\footnote{note that the proof of the theorem is totally combinatorial and does not use fibrancy of $\bcat{X}$.} \cite{AScof} we conclude that the map $L \xlongrightarrow{} Q(L)$ is $\bS$ anodyne which in turn implies that $\Xi_{\bcat{Y}}$ must be a trivial fibration.  Moreover, since the map is natural in $\bcat{Y}$ we obtain a natural equivalence $\Xi:i_* \xRightarrow{\simeq} \on{id}$. 
     \item To finally give the definition of $\epsilon$ we note that by 1 above it will be enough to construct a natural transformation $\hat{\epsilon}:i_* \xRightarrow{} f_*f^*$. 

     Given a $\bS$ simplicial $p: L \xlongrightarrow{} \bcat{D}$ and a map $L \xlongrightarrow{} i_* \bcat{Y}$ it follows that we have a commutative diagram
     \[
      \begin{tikzcd}
         R(L)=Q(L)\times_{\bcat{D}} \bcat{C} \arrow[r] \arrow[d] & \bcat{Y}\times_{\bcat{D}}\bcat{C}=f^*\bcat{Y} \arrow[d] \\
       Q(L) \arrow[r] & \bcat{Y}
      \end{tikzcd}
     \]
     where we identify the top horizontal morphism with a morphism $L \xlongrightarrow{} f_*f^*\bcat{Y}$. This gives a map (which is natural in $\bcat{Y}$) $i_* \bcat{Y} \xlongrightarrow{} f_*f^*\bcat{Y}$ and provides the definition of $\hat{\epsilon}: i_* \Rightarrow f_*f^*$.
   \end{enumerate}

\begin{theorem}\label{thm:rightlaxkan}
  Let $(\bcat{C},E)$ be an $(\infty,2)$-category equipped with a collection of edges containing the equivalences and let $f:\bcat{C} \xlongrightarrow{} \bcat{D}$. Then there exists an adjunction of $(\infty,2)$-categories.
  \[
    f^*: \Fun(\bcat{D},\bcat{C}\!\on{at}_{(\infty,2)}) \llra   \Fun(\bcat{C},\bcat{C}\!\on{at}_{(\infty,2)})^{\on{Elax}}: f_*
  \]
  and similarly for the Eoplax case.
\end{theorem}
\begin{proof}
  Working on the fibrational side of the picture it will be enough to show that the unit and counit natural transformations that we defined above satisfy the triangular identities.
  \begin{enumerate}
    \item We consider $f^*\circ i_* \xRightarrow{}f^*\circ f_* \circ f^* \xRightarrow{}f^*$ and wish to show that it is equivalent to $f^*\Xi$. Unraveling the definition we see that the composite is precisely given by $f^*\Xi$. 
    \item Finally, we look at $\tau:i_* \circ f_*  \xRightarrow{} f_* \circ f^* \circ f_* \xRightarrow{} f_*$ and we wish to show that this composite is equivalent to $\Xi f_*$. 

    Recall, that map of $\bS$ simplicial sets $L \xlongrightarrow{} i_*f_* \bcat{X}$ over $\bcat{D}$ is given by a map (of $\bS$ simplicial sets) over $\bcat{C}$ 
    \[
      \hat{R}(Q(L)) \xlongrightarrow{} \bcat{X}.
    \]
    Given $q:L \xlongrightarrow{} \bcat{D}$ let us unpack the definition of $\hat{R}(Q(L))$.
    \begin{itemize}
      \item A simplex $\Delta^n \xlongrightarrow{} \hat{R}(Q(L))$ consists in two maps of scaled simplicial sets 
      \[
        \alpha_i:   \Delta^1 \otimes \Delta^n  \xlongrightarrow{}\bcat{D}, \enspace i=1,2
      \]
      satisfying the following conditions:
      \begin{itemize}
        \item[i)] The restriction of $\alpha_1$ to $ \Delta^{\{0\}} \times \Delta^n$ factors through $p:L \xlongrightarrow{} \bcat{D}$.
        \item[ii)] The restriction of $\alpha_2$ to $ \Delta^{\{1\}} \times \Delta^n$ factors through $f:\bcat{C} \xlongrightarrow{} \bcat{D}$.
        \item[iii)] The of $\alpha_1$ to $ \Delta^{\{1\}} \times \Delta^n$ coincides with the restriction of $\alpha_2$ to $\Delta^{\{0\}} \times \Delta^n$
      \end{itemize}
      We depict a 2-simplex in $\hat{R}(Q(L))$ diagrammatically as follows:
      \[\begin{tikzcd}
  {p(l_0)} & {p(l_1)} & {p(l_2)} \\
  {d_0} & {d_1} & {d_2} \\
  {f(c_0)} & {f(c_1)} & {f(c_2)}
  \arrow[from=1-1, to=2-1]
  \arrow[from=2-1, to=3-1]
  \arrow[from=1-1, to=1-2]
  \arrow[from=1-2, to=2-2]
  \arrow[from=1-2, to=1-3]
  \arrow[from=1-3, to=2-3]
  \arrow[from=2-1, to=2-2]
  \arrow[from=2-2, to=3-2]
  \arrow[from=3-1, to=3-2]
  \arrow[from=2-2, to=2-3]
  \arrow[from=2-3, to=3-3]
  \arrow[from=3-2, to=3-3]
  \arrow[""{name=0, anchor=center, inner sep=0}, curve={height=12pt}, draw=none, from=1-1, to=2-2]
  \arrow[""{name=1, anchor=center, inner sep=0}, curve={height=12pt}, draw=none, from=1-2, to=2-3]
  \arrow[""{name=2, anchor=center, inner sep=0}, curve={height=12pt}, draw=none, from=2-1, to=3-2]
  \arrow[""{name=3, anchor=center, inner sep=0}, curve={height=12pt}, draw=none, from=2-2, to=3-3]
  \arrow[shorten <=4pt, shorten >=8pt, Rightarrow, from=0, to=1-2]
  \arrow[shorten <=4pt, shorten >=8pt, Rightarrow, from=1, to=1-3]
  \arrow[shorten <=4pt, shorten >=8pt, Rightarrow, from=2, to=2-2]
  \arrow[shorten <=4pt, shorten >=8pt, Rightarrow, from=3, to=2-3]
\end{tikzcd}\]
\item An edge in $\hat{R}(Q(L))$ is marked if its image in $\bcat{C}$ is marked, its image in $\bcat{L}$ is marked and both squares in $\bcat{D}$ commute. A triangle is lean precisely when its image in $L$ is lean. Finally a triangle is thin if it is lean and its image in $\bcat{C}$ is thin.
\item In other words, a simplex $\Delta^n \xlongrightarrow{} \hat{R}(Q(L))$ is the data of a map $ \Lambda^2_1 \tensor \Delta^n$ satisfying the conditions above.
    \end{itemize}
   First, we identify the pair of maps (over $\bcat{C}$) $\varphi_i:\hat{R}(L) \xlongrightarrow{} \hat{R}(Q(L))$ for $i=0,1$ inducing the maps $\Xi_{f_*\bcat{X}}$ and $\tau_{\bcat{X}}$ respectively. After unraveling the definitions we see that $\varphi_i$ is induced by the maps
   \[
     \Delta^n \tensor \Lambda^2_1 \xlongrightarrow{\on{id}\times q_i} \Delta^n \tensor \Delta^1
   \]
   where $q_i$ is the map that collapses the edge $i \xlongrightarrow{} i+1$.

   Given $p:L \xlongrightarrow{} \bcat{D}$ as before we define an $\bS$ simplicial $T(L) \xlongrightarrow{} \hat{R}(Q(L)) $ as follows:
   \begin{itemize}
     \item A simplex $\Delta^n \xlongrightarrow{} T(L)$ is given by a map 
     \[
       \Delta^2_{\sharp} \tensor \Delta^n \xlongrightarrow{} \bcat{D}, 
     \]
     whose restriction to $ \Lambda^2_1 \tensor \Delta^n$ defines a simplex in $\hat{R}(Q(L))$.
     \item The decorations on $T(L)$ are induced from those of $\hat{R}(Q(L))$.
   \end{itemize}
   This definition yields a functor:
   \[
     \Gamma_f: \bcat{F}\!\on{ib}_{0,1}(\bcat{C})^{\on{Elax}} \xlongrightarrow{} \bcat{F}\!\on{ib}_{0,1}(\bcat{D})
   \]
   It follows from \autoref{prop:grpushout} that the canonical map $\phi:T(L) \xlongrightarrow{} \hat{R}(Q(L)) $ is a trivial fibration of $\bS$ simplicial sets over $\bcat{C}$ which gives an equivalence of functors $\pi: i_*f_* \xRightarrow{} \Gamma_f$. We can pick a section of $\phi$ denoted by $\kappa: \hat{R}(Q(L)) \xlongrightarrow{} T(L)$ such that the composite
   \[
     \hat{R}(L) \xlongrightarrow{\varphi_i} \hat{R}(Q(L))  \xlongrightarrow{\kappa} T(L), \enspace i=0,1
   \]
   is induced by the map 
   \[
      s_i  \times \on{id}:  \Delta^2_{\sharp} \tensor \Delta^n \xlongrightarrow{}  \Delta^1 \tensor \Delta^n
   \]
   where $i=0,1$ and  $s_i:\Delta^2_\sharp \xlongrightarrow{} \Delta^1$ is the obvious degeneracy map. Again we obtain an equivalence of functors $s:\Gamma_f \xRightarrow{} i_*f_*$. To finish the proof it suffices to show that $\Xi f_* \circ s \simeq \tau \circ s$.

   Finally, let $\nu: T(L) \xlongrightarrow{} \hat{R}(L)$ be the map induced by the face operator $d_1:\Delta^2_\sharp \xlongrightarrow{} \Delta^1$ which in turn gives us a natural transformation $t:f_* \Rightarrow \Gamma_f$. Direct inspection reveals that
   \[
     \Xi f_*\circ s \circ t = \tau \circ s \circ t = \on{id}.
   \]
    Since $\Xi f_*$ and $s$ are equivalences it follows that $t$ is also an equivalence. We conclude that $\Xi f_* \simeq \tau$ and thus the triangular identities hold. \qedhere
  \end{enumerate}
\end{proof}
  
  \subsection{Enhanced double \texorpdfstring{$(\infty,1)$-}-categories and Segal objects}
  \begin{definition}\label{def:verticalfragment}
    Let $\scr{X}_\bullet:\Delta^\op \xlongrightarrow{} \bcat{C}\!\on{at}_{(\infty,1)}$ be a functor. We define the \emph{vertical fragment} of $\scr{X}_\bullet$ which we denote as $\scr{X}_{\bullet}^{\mathbf{v}}$ by declaring for every $[n] \in \Delta^\op$ the category $\scr{X}_{n}^{\mathbf{v}} \subset \scr{X}_n$ to be the subcategory spanned by those edges whose image under any map $[n] \xlongrightarrow{} [0]$ in $\Delta^\op$ factors through $\scr{X}_0^{\simeq}$, the underlying $\infty$-groupoid of $\scr{X}_0$. In particular, it follows that $\scr{X}^{\mathbf{v}}_{0}=\scr{X}^{\simeq}_{0}$. 
  \end{definition}

  \begin{remark}
    We observe that if $\scr{X}_\bullet:\Delta^\op \xlongrightarrow{} \bcat{C}\!\on{at}_{(\infty,1)}$ satisfies the Segal condition then it follows that $\scr{X}_{\bullet}^{\mathbf{v}}$ also satisfies the Segal condition.
  \end{remark}

  \begin{definition}\label{def:dblinftycat}
    Let $\scr{X}_\bullet:\Delta^\op \xlongrightarrow{} \bcat{C}\!\on{at}_{(\infty,1)}$ be a functor. We say that $\scr{X}_\bullet$ is a \emph{double} $(\infty,1)$-category if it satisfies the Segal condition and its vertical fragment (see \autoref{def:verticalfragment}) $\scr{X}_\bullet^{\mathbf{v}}$ is a complete Segal object (see \autoref{def:completesegalobj}) in $\bcat{C}\!\on{at}_{(\infty,1)}$.

    We denote by $\on{Dbl}(\bcat{C}\!\on{at}_{(\infty,1)})$ the full subcategory of the $(\infty,2)$-category $\Fun(\Delta^\op,\bcat{C}\!\on{at}_{(\infty,1)})$ spanned by the double $(\infty,1)$-categories.
  \end{definition}

  \begin{remark}
    Observe that given a double $(\infty,1)$-category $\scr{X}_\bullet$ then it follows that $\scr{X}_\bullet$ is a complete Segal object if and only if $\scr{X}_0$ is an $\infty$-groupoid.
  \end{remark}

  \begin{definition}\label{def:deltaidle}
    Let $\Delta$ be the simplex category. We say that a morphism $\alpha: [k] \xlongrightarrow{} [n]$ is \emph{idle} if at least one of the following conditions hold:
    \begin{itemize}
      \item[i)]  The map $\alpha$ is surjective.
      \item[ii)] The map $\alpha$ is the inclusion of a subinterval of $[n]$. Here by subinterval we mean a subset $I \subseteq [n]$ such that for every $x \in [n]$ such that  $\min(I) \leq x \leq \max(I)$ then $x \in I$.
    \end{itemize}
    We denote by $\mathcal{I}$ the collection composites of idle morphisms in $\Delta$.
  \end{definition}

  \begin{remark}
   In this section we will have to deal with the comma categories $\Delta_{/[n]}$ where $[n] \in \Delta$. To ease the notation we will simply denote the object $[n]$ by $n$ and similarly we will write $\Delta_{/n}$ for the comma category above.
  \end{remark}

  \begin{lemma}\label{lem:iclarified}
    A morphism $\varphi: k \xlongrightarrow{} n$ belongs to $\mathcal{I}$ if and only if it can be factored as a surjection followed by an interval inclusion. 
  \end{lemma}
  \begin{proof}
    One direction is obvious. Let $i:k \xlongrightarrow{} l$ be an interval inclusion and let $p:l \xlongrightarrow{} n$ be a surjection. It is easy to see that it will suffice to show that $\phi=p \circ i$ can be factored as a surjection followed by an interval inclusion. Let $r,s \in n$ with $r < s$ which are in the image of $\phi$ and let $r < x < s$ Then $r=p\circ i(a)$ and $s=p \circ i (b)$. It follows that the restriction of $p$ to $i(a)< i(b)$ surjects onto the interval $[r,s]$. Since $i$ is an interval inclusion it follows that $a<b$ surjects onto $i(a)<i(b)$ and the claim holds.
  \end{proof}

  \begin{remark}
    As above, usually denote the objects of $\Fun(\Delta^\op,\bcat{C}\!\on{at}_{(\infty,1)})$ using the lower-bullet notation $X_{\bullet}$ where we reserve the use of caligraphic letters to denote double $(\infty,1)$-categories (resp. complete Segal objects).
  \end{remark}

  \begin{definition}\label{def:laxsegcond}
    Let $X_\bullet$ be an object of $\Fun(\Delta^\op,\bcat{C}\!\on{at}_{(\infty,1)})$. We say that $X_\bullet$ satisfies the \emph{oplax} Segal condition if:
    \begin{itemize}
      \item[\textbf{OS1)}] For every $n\geq 2$ the canonical map
      \[
         L_n:X_n \xlongrightarrow{} X_1 \times_{X_0}X_1 \times_{X_0}\cdots \times_{X_0} X_1:=\on{S}(X)_n
       \] 
      admits a fully faithful right adjoint. 
      \item[\textbf{OS2)}] For every $u:m \xlongrightarrow{} n$ in $\Delta$ the natural transformation $X(u)\circ R_n \xRightarrow{}R_m\circ L_m \circ X(u) \circ R_n$ induced by the unit of the adjunction $L_m \dashv R_m$ is a pointwise weak equivalence. We will use the convention that $L_m$ (resp $R_m$) is the identity functor for $m=0,1$.
    \end{itemize}
    By reversing the role of the adjoints we obtain the dual notion of the $\emph{lax}$ Segal condition.
  \end{definition}

  \begin{definition}\label{def:segalnucleus}
    Let $X_\bullet$ be an object of $\Fun(\Delta^\op,\bcat{C}\!\on{at}_{(\infty,1)})$ which satisfies the oplax Segal condition. Given a morphism $u:n \xlongrightarrow{} m$ in $\Delta$ we define $\on{S}(X)(u)=L_n \circ X(u) \circ R_m$. Then it follows from condition \textbf{OS2)} in \autoref{def:laxsegcond} that  the $(\infty,1)$-categories, $\{\on{S}(X)_n\}_{n \geq 0}$ (where $\on{S}(X)_n=X_n$ if $n=0,1$) assemble into a functor $\on{S}(X)_\bullet: \Delta^\op \xlongrightarrow{} \bcat{C}\!\on{at}_{(\infty,1)}$ which we call the \emph{nucleus} of $X_\bullet$. 
  \end{definition}

  \begin{remark}\label{rem:degfactorsthroughessentialimage}
    Note that condition \textbf{OS2} in \autoref{def:laxsegcond} guarantees that given a surjective map $s_i: 2 \xlongrightarrow{} 1$ the image of the map $X_1 \xlongrightarrow{} X_2$ factors through the essential image of $R_2$.
  \end{remark}

  \begin{lemma}\label{lem:canonicalnat}
     Let $X_\bullet$ be an object of $\Fun(\Delta^\op,\bcat{C}\!\on{at}_{(\infty,1)})$ which satisfies the oplax Segal condition. Then the right adjoints assemble into a natural transformation $R_\bullet: \on{S}(X)_\bullet \xRightarrow{} X_\bullet$ and the left-adjoints assemble into an $\mathcal{I}\!\on{oplax}$ natural transformation $L_\bullet: X_\bullet \xRightarrow{} \on{S}(X)_\bullet$ such that $R_\bullet$ and $L_\bullet$ are mates.
  \end{lemma}
  \begin{proof}
    The existence of $R_\bullet$ is clear from condition \textbf{OS2)} in \autoref{def:laxsegcond}. Then since $R_\bullet$ is a natural transformation which is a pointwise right adjoint we can use \cite[Theorem 5.3.5]{mates} to obtain its mate $L_\bullet$ which a priory is only an oplax natural transformation. To show that $L_\bullet$ is actually an $\mathcal{I}\!\on{oplax}$ natural transformation we will check that given $u:m \xlongrightarrow{} n$ where $u$ is either an interval inclusion or a surjective map then the associated squares actually commute.

    Observe that if $u$ is an interval inclusion there is an obvious map $\overline{u}:\on{S}(X)_n \xlongrightarrow{} \on{S}(X)_m$ which forgets certain factors in the limit defining $\on{S}(X)_n$. It is immediate to verify that $\overline{u} \circ L_n \isom L_m \circ X(u)$. Similarly, if $u$ is a surjection we can use the degeneracy maps $X_0 \xlongrightarrow{} X_1$ to produce some $\overline{u}$ satisfying the same relation as above.

    Unraveling the definitions and computing the mate of the square assocatied to $u$ we see that we need to show that the natural transformation $L_m \circ X(u) \xRightarrow{} L_m \circ X(u) \circ R_n \circ L_n$ (which is induced by the unit of the adjunction $L_n \dashv R_n$) is invertible. However, using the previous discussion we identify this natural transformation with $\overline{u} \circ L_n \xRightarrow{} \overline{u} \circ L_n \circ R_n \circ L_n$ since the counit is of the adjunction $L_n \dashv R_n$ is always invertible the claim now follows from the triangular identities.
  \end{proof}

  \begin{lemma}\label{lem:adjonadj}
    Let $X_\bullet$ be an object of $\Fun(\Delta^\op,\bcat{C}\!\on{at}_{(\infty,1)})$ which satisfies the oplax Segal condition. Then we have that:
    \begin{enumerate}
      \item We have a morphism of $\mathcal{I}\!\on{oplax}$ natural transformation $\Gamma:L_\bullet \circ R_\bullet \xlongrightarrow{} \on{id}$ where $L_\bullet \circ R_\bullet$ is the the obvious composite (see \autoref{lem:canonicalnat}) and $\on{id}$ is the identity natural transformation. Moreover, for every $n\geq 0$ the natural transformation $\Gamma_n$ is the counit of the adjunction $L_n \dashv R_n$.  
      \item We have a morphism of $\mathcal{I}\!\on{oplax}$ natural transformations $\Phi: \on{id} \xlongrightarrow{} R_\bullet \circ L_\bullet$ such that for every $n\geq 0$ the natural transformation $\Phi_n $ is the unit of the adjunction $L_n \dashv R_n$.
    \end{enumerate}
  \end{lemma}
  \begin{proof}
    We will show that the levelwise units and counits assemble into morphisms of ($\mathcal{I}\!\on{oplax}$) natural transformations $\Gamma:L_\bullet \circ R_\bullet \xlongrightarrow{} \on{id}$ and $\Phi: \on{id} \xlongrightarrow{} R_\bullet \circ L_\bullet$. We will verify the case of $\Phi$ and leave the analogous checks for $\Gamma$ to the reader.

    Let $(X \times \underline{\Delta}^{1})_\bullet$ denote the functor whose value at each $n\geq 0$ is given by $X_n \times \Delta^1$. We want to construct an $\mathcal{I}\!\on{oplax}$ natural transformation 
    \[
      (X \times \underline{\Delta}^{1})_\bullet \xRightarrow{} X_\bullet
    \]
    such that for every $n\geq 0$ the map $X_n \times \Delta^1 \xlongrightarrow{} \bcat{X}_n$ is precisely the unit of the adjunction $L_n \dashv R_n$. Unraveling the definitions it follows that it will suffice to show that for every map $u:m \xlongrightarrow{} n$ the composite
    \[
      X(u) \xRightarrow{} R_m \circ L_m \circ X(u) \xRightarrow{} R_m \circ \on{S}(X)(u) \circ L_n \xRightarrow{\alpha^{-1}} X(u)\circ R_n \circ L_n 
    \]
   is equivalent to the map induced by the unit of the adjunction $L_n \dashv R_n$. To see this we note that we have a commutative diagram of natural transformations
   \[
     \begin{tikzcd}
       X(u) \arrow[Rightarrow,r] \arrow[Rightarrow,d] & X(u)\circ R_n \circ L_n \arrow[Rightarrow,d,"\alpha"] \\
       R_m \circ L_m \circ X(u) \arrow[Rightarrow,r] & R_m \circ L_m \circ X(u) \circ R_n \circ L_n
     \end{tikzcd}
   \]
   where by definition $R_m \circ L_m \circ X(u) \circ R_n \circ L_n= R_m \circ \on{S}(X)(u) \circ L_n$. The claim now follows.
  \end{proof}

  \begin{remark}\label{rem:canonicalcone}
     Let $X_\bullet$ be an object of $\Fun(\Delta^\op,\bcat{C}\!\on{at}_{(\infty,1)})$ which satisfies the oplax Segal condition. We note that for every $n \in \Delta$ we have a canonical natural transformation 
     \[
       \overline{\Xi}_n: (\Delta_{/n})^\op \times \Delta^1 \xlongrightarrow{} \bcat{C}\!\on{at}_{(\infty,1)}
     \]
     from the constant functor at $X_n$ to the composite $(\Delta_{/n})^\op \xlongrightarrow{} \Delta^\op \xlongrightarrow{X_\bullet}  \bcat{C}\!\on{at}_{(\infty,1)} $. Composing this natural transformation with the $\mathcal{I}\!\on{oplax}$ natural transformation from $X_\bullet \xlongrightarrow{} \on{S}(X)_\bullet$ given in \autoref{lem:canonicalnat} we obtain an $\mathcal{I}\!\on{oplax}$ (with the obvious marking in the comma category) natural transformation 
     \[
       \Xi_n : (\Delta_{/n})^\op \times \Delta^1 \xlongrightarrow{} \bcat{C}\!\on{at}_{(\infty,1)}
     \]
     from the constant functor on $X_n$ to the composite $(\Delta_{/n})^\op \xlongrightarrow{} \Delta^\op \xlongrightarrow{\on{S}(X)_\bullet}  \bcat{C}\!\on{at}_{(\infty,1)}$.
  \end{remark}

  \begin{definition}\label{def:enhanced}
    Let $X_\bullet$ be an object of $\Fun(\Delta^\op,\bcat{C}\!\on{at}_{(\infty,1)})$ which satisfies the oplax Segal condition. We say that $X_\bullet$ is \emph{enhanced} for every $n \geq 0$ the natural transformation $\Xi_n$ given in \autoref{rem:canonicalcone} exhibits $X_n$ as the $\mathcal{I}\!\on{oplax}$ limit of the functor $(\Delta_{/n})^\op \xlongrightarrow{} \Delta^\op \xlongrightarrow{\on{S}(X)_\bullet}  \bcat{C}\!\on{at}_{(\infty,1)}$.
  \end{definition}

  \begin{definition}\label{def:enhanceddouble}
    Let $\bcat{X}_\bullet: \Delta^\op \xlongrightarrow{} \bcat{C}\!\on{at}_{(\infty,1)})$ be a functor. We say that $\bcat{X}_\bullet$ is an \emph{enhanced} double $(\infty,1)$-category if the following conditions are satisfied:
    \begin{enumerate}
      \item The functor $\bcat{X}_\bullet$ satisfies the oplax Segal condition (\autoref{def:laxsegcond}).
      \item The functor $\bcat{X}_\bullet$ is enhanced (\autoref{def:enhanced}).
      \item The nucleus of $\bcat{X}_\bullet$ (\autoref{def:segalnucleus}) defines a double $(\infty,1)$-category.
    \end{enumerate}
    We will denote by $\on{EDbl}(\bcat{C}\!\on{at}_{(\infty,1)})$ the full subcategory of the $(\infty,2)$-category $\Fun(\Delta^\op,\bcat{C}\!\on{at}_{(\infty,1)})$ spanned by the enhaced double $(\infty,1)$-categories.
  \end{definition}

  \begin{remark}\label{rem:nucleusnotation}
    We will tend to use blackboard characters $\bcat{X}_\bullet$ to denote enhaced double $(\infty,1)$-categories. We no confusion shall arise we will use the lighter notation $\scr{X}_\bullet=\on{S}(\bcat{X})_\bullet$ to refer to the nucleus of $\bcat{X}_\bullet$.
  \end{remark}

  \begin{lemma}\label{lem:verticalpreservesoplax}
    Let $\bcat{X}_\bullet$ be an enhanced double $(\infty,1)$-category. Then its vertical fragment $\bcat{X}^{\mathbf{v}}_\bullet$ satisfies the oplax Segal condition (\autoref{def:laxsegcond}) and is enhanced (\autoref{def:enhanced}).
  \end{lemma}
  \begin{proof}
    We observe that for every $n\geq 2$ we have that the unit of the adjunction $\on{id} \Rightarrow R_n \circ L_n$ satisfies that at each $x \in \bcat{X}_n$ the component $x \xlongrightarrow{} R_n\circ L_n (x)$ factors through $\bcat{X}_n^{\mathbf{v}}$. To see this we note that for every $0 \xlongrightarrow{} n$ we have a factorization
    \[
      \bcat{X}_n \xlongrightarrow{L_n} \bcat{X}_1 \times_{\bcat{X}_0}\cdots \times_{\bcat{X}_0} \bcat{X}_1 \xlongrightarrow{} \bcat{X}_1 \xlongrightarrow{}\bcat{X}_0
    \]
    where the final maps are given by the projection to one of the factors and the factorization $0 \xlongrightarrow{} 1 \xlongrightarrow{} n$ since $n\geq 2$. Therefore the claim follows from the triangular identities. The factorization above shows that $L_n \dashv R_n$ induces an adjunction in $L_n^{\mathbf{v}} \dashv R_n^{\mathbf{v}}$ where $R_n^{\mathbf{v}}$ is fully faithful. We conclude that $\bcat{X}^{\mathbf{v}}_\bullet$ satisfies the oplax Segal condition.

    To finish the proof we verify that $\bcat{X}^{\mathbf{v}}_\bullet$ is enhanced. Direct inspection reveals that $\on{S}(\bcat{X}^{\mathbf{v}})_\bullet$ coincides precisely with $\on{S}(\bcat{X})^{\mathbf{v}}_\bullet$. Let $K^{\mathbf{v}}$ denote the $\mathcal{I}$oplax limit of the composite
    \[
      V_n:(\Delta^{\op})^{\op}_{/n} \xlongrightarrow{} \Delta^\op \xlongrightarrow{S(\bcat{X})^{\mathbf{v}}_\bullet} \bcat{C}\!\on{at}_{(\infty,1)}
    \]
    and observe that the universal property of the partially lax limit implies that we have a map $K^{\mathbf{v}} \xlongrightarrow{} \bcat{X}_n$ which factors through $\bcat{X}_n^{\mathbf{v}}$. Moreover, one easily checks that the if $\alpha: a \xlongrightarrow{} n$  is an object of $(\Delta^{\op})^{\op}_{/n}$ the composite
    \[
      \bcat{X}_n^{\mathbf{v}} \xlongrightarrow{} \bcat{X}_n \xlongrightarrow{} \on{S}(\bcat{X})_a
    \]
    ,where the last morphism is given by the canonical cone that expresses $\bcat{X}_n$ as an $\mathcal{I}$oplax limit, factors through $\on{S}(\bcat{X})^{\mathbf{v}}_a$. This provides us with a map $\bcat{X}_n^{\mathbf{v}} \xlongrightarrow{} K^{\mathbf{v}}$. The proof follows after verifying that both maps are mutual inverses.
  \end{proof}

  \begin{proposition}\label{prop:enhancedprev}
    Let $\bcat{X}_\bullet: \Delta^\op \xlongrightarrow{} \bcat{C}\!\on{at}_{(\infty)}$ be a functor. Then the following conditions are equivalent:
    \begin{itemize}{}
      \item[i)] The functor $\bcat{X}_\bullet$ is an enhanced double $(\infty,1)$-category such that $\bcat{X}_0$ is an $\infty$-groupoid.
      \item[ii)] The functor $\bcat{X}_\bullet$ is an enhanced double $(\infty,1)$-category and its nucleus is a complete Segal object.
      \item[iii)] The following conditions hold:
      \begin{enumerate}
      \item The functor $\bcat{X}_\bullet$ satisfies the oplax Segal condition (\autoref{def:laxsegcond}).
      \item The functor $\bcat{X}_\bullet$ is enhanced (\autoref{def:enhanced}).
      \item The nucleus of $\bcat{X}_\bullet$ (\autoref{def:segalnucleus}) defines a complete Segal object in $\bcat{C}\!\on{at}_{(\infty,1)}$. 
      \end{enumerate}
    \end{itemize}
  \end{proposition}
  \begin{proof}
   Immediate from the definitions.
  \end{proof}

  \begin{definition}\label{def:enhancedsegalobject}
    We say that a functor $\bcat{X}_\bullet: \Delta^\op \xlongrightarrow{} \bcat{C}\!\on{at}_{(\infty,1)}$ is an \emph{enhanced} Segal object in $\bcat{C}\!\on{at}_{(\infty,1)}$ it satisfies any of the conditiosn in \autoref{prop:enhancedprev}. 
  \end{definition}

  \begin{proposition}\label{prop:vertofendbl}
    Let $\bcat{X}_\bullet$ be an enhanced double $(\infty,1)$-category. Then its vertical fragment (\autoref{def:verticalfragment}) $\bcat{X}_\bullet^{\mathbf{v}}$ is an enhanced Segal object.
  \end{proposition}
  \begin{proof}
    We know by \autoref{lem:verticalpreservesoplax} that $\bcat{X}_\bullet^{\mathbf{v}}$ satisfies the oplax Segal condition and is enhanced. We note that since the nucleus of $\bcat{X}_\bullet^{\mathbf{v}}$ is given by the vertical fragment of the nucleus of $\bcat{X}_\bullet$ it follows that condition 3 in \autoref{prop:enhancedprev} also holds.
  \end{proof}

  \begin{definition}\label{def:kdownstar}
     We consider the adjunction of $(\infty,2)$-categories
   \[
    K^*: \Fun(\Delta^\op,\bcat{C}\!\on{at}_{(\infty,1)}) \llra   \Fun(\Delta^\op,\bcat{C}\!\on{at}_{(\infty,1)})^{\on{\mathcal{I}\!oplax}}: K_*
  \]
  given by right Kan extending with respect to the canonical map $(\Delta,\mathcal{I}) \xlongrightarrow{} (\Delta,\sharp)$. 
  \end{definition}
 
  \begin{proposition}\label{prop:1dblcharac}
    Let $\scr{E}_{\bullet} \in \on{Dbl}(\bcat{C}\!\on{at}_{(\infty,1)})$ and let $\bcat{E}_\bullet=K_*(\scr{E})_\bullet$ then there exists a natural transformation $i_\bullet: \scr{E}_{\bullet} \Rightarrow{} \bcat{E}_{\bullet} $ such that:
    \begin{itemize}
      \item[i)] For $n\geq 0$ the map $i_n: \scr{E}_n \xlongrightarrow{} \bcat{E}_n$ is fully faithful and admits a left adjoint.
      \item[ii)] For $n=0,1$  the map $i_n: \scr{E}_n \xlongrightarrow{} \bcat{E}_n$ is an equivalence.
      \item[iii)] Given a degeneracy operator $s_i:[2] \xlongrightarrow{} [1]$ the induced functor $\bcat{E}_1 \xlongrightarrow{} \bcat{E}_2$ factors through $\scr{E}_2$.
    \end{itemize}
  \end{proposition}
  \begin{proof}
     We observe that $K^*$ acts on objects as the identity so we can use the unit map to obtain a natural transformation $i_\bullet:\scr{E}_{\bullet} \Rightarrow K_*K^* (\scr{E}_{\bullet})=K_* (\scr{E}_{\bullet})$. Dually we can view the counit as defining a natural transformation $r:K_*(\scr{E}_{\bullet})=K^*K_*(\scr{E}_{\bullet}) \Rightarrow \scr{E}_{\bullet}$. As a consequence we obtain for every $n \in \Delta$ a pair of functors
    \[
      i_n: \scr{E}_{n} \xlongrightarrow{} K_*(\scr{E})_n, \enspace \enspace r_n:K_*(\scr{E})_n \xlongrightarrow{} \scr{E}_{k},
    \]
    which we claim that form an adjoint pair $r_n \dashv i_n$. Chasing through the triangular identities one checks that $ r_n \circ i_n \simeq \on{id}$. To construct the natural transformation $\on{id} \Rightarrow{}  i_n \circ r_n $ we take the point of view of \autoref{rem:limitformula} to see that $K_*(\scr{E})_n$ is the $\mathcal{I}\!\on{oplax}$ limit of the diagram
    \[
      (\Delta_{/n})^\op \xlongrightarrow{\pi_n} \Delta^\op \xlongrightarrow{\scr{E}_{\bullet}} \bcat{C}\!\on{at}_{(\infty,1)}.
    \]
    From this point of view, we have that the map $i_n$ is given by the universal property of the limit. More precisely, it follows that the ordinary limit (i.e. the partially oplax limit with respect to the maximal marking) is given by $\scr{E}_n$. Then the canonical cone for the ordinary limit defines in particular a partially lax cone for $\scr{E}_{\bullet}\circ \pi_n$ and thus induces the desired map $i_n:\scr{E}_n \xlongrightarrow{} K_*(\scr{E})_n$. Similarly the value of $r_n$ is given by the canonical projection 
    \[
      K_*(\scr{E})_n \xlongrightarrow{} \scr{E}_{\bullet} \circ \pi_n (\on{id}_n)=\scr{E}_n.
    \]
    Given an object $u: i \xlongrightarrow{} n$ in $\Delta_{/n}$ let $\alpha_u: K_*(\scr{E})_n \xlongrightarrow{} \scr{E}_i$ denote the morphisms defining the canonical cone for $K_*(\scr{E})_n$. Using the notation we have that $\alpha_{\on{id}_n}=r_n$ and it is easy to verify that
    \[
      \alpha_u \circ i_n \circ r_n \simeq \scr{E}(u) \circ r_n.
    \]
    The fact that the maps $\alpha_u$ define a partially lax cone implies that we have 2-morphisms
    \[
      \theta_u: \alpha_u \Rightarrow{} \scr{E}(u) \circ r_n.
    \]
    These natural transformations assemble into a morphism of cones and thus induce by the universal property of the limit a natural transformation $\Xi:\on{id} \Rightarrow{}i_n \circ r_n$.

    Let us point out that the equivalence $r_n \circ i_n \isom \on{id}$ its a consequence of the fact that the have an equivalences
    \[
      \alpha_{u} \circ i_n \simeq \scr{E}(u)
    \]
    applying to the case where $u$ is the identity on $n$. To finish the proof we show that $\Xi$ defines the unit of an adjunction. By construction it is clear that $r_n*\Xi$ is equivalent to the identity natural transformation on $r_n$. Finally we observe that by construction of $i_n$ it follows that $\theta_u* i_n$ is equivalent to the identity on $\scr{E}(u)$ which shows that the remaing triangular identity holds. This proves the first claim in our statement.

    To prove the second claim it suffices to show that if $n=0,1$, the map $(\Delta_{/n}, \mathcal{I}) \xlongrightarrow{} (\Delta_{/n}, \sharp)$  is the weakly saturated class of morphisms of type 
    \begin{itemize}
      \item[*)]  $(\Delta^2,\{\Delta^{\{1,2\}},\Delta^{\{0,2\}}\},\sharp) \xlongrightarrow{} (\Delta^2,\sharp,\sharp)$.
    \end{itemize}
    We are denoting by $(\Delta_{/n}, \mathcal{I})$ the marking on $\Delta_{/n}$ where an edge is marked if and only if its image under the canonical projection to $\Delta$ belongs to $\mathcal{I}$. Then it follows that this map defines a weak equivalence of $(1,0)$-cartesian fibrations over $\Delta$ and it is final with respect to $\mathcal{I}\!\on{oplax}$ limits (see the final section in \cite{GHL_LaxLim}. This implies that $i_n$ is a weak equivalence whenever $k=0,1$.

    Finally, to show the final claim we see that the map $\bcat{E}_1 \xlongrightarrow{} \bcat{E}_2$ is induced by postcomposition with the degeneracy operator $(\Delta_{/2},\mathcal{I}) \xlongrightarrow{} (\Delta_{/1},\mathcal{I})$. However we saw that the marking on the target of this map can be saturated to contain all edges so it follows that $\bcat{E}_1 \xlongrightarrow{} \bcat{E}_2$ must factor through the ordinary limit which is $\scr{E}_2$.
  \end{proof}

  \begin{remark}
    We could have proved the previous proposition using the fibrational language. However, this proof illustrates how this theorem will generalise the moment we have a general notion of right lax Kan extension where the target category is a (sufficiently complete) $(\infty,2)$-category $\bcat{A}$ in which it makes sense to talk about double category objects and complete Segal objects.
  \end{remark}

  \begin{definition}
    Let $\on{Dbl}\left(\bcat{C}\!\on{at}_{(\infty,1)}\right)^{\mathcal{I}\!\on{oplax}} $ to be the full $(\infty,2)$-category spanned by the double $(\infty,1)$-categories, inside of $\Fun(\Delta^\op,\bcat{C}\!\on{at}_{(\infty,1)})^{\mathcal{I}\!\on{oplax}}$. We similarly define $\CSeglax$.
  \end{definition}

  \begin{proposition}\label{prop:enhancedkanfactors}
    Given the adjunction defined in \autoref{def:kdownstar}
   \[
    K^*: \Fun(\Delta^\op,\bcat{C}\!\on{at}_{(\infty,1)}) \llra   \Fun(\Delta^\op,\bcat{C}\!\on{at}_{(\infty,1)})^{\on{\mathcal{I}\!oplax}}: K_*
  \]
  then it follows that the restriction of $K_*$ to $\on{Dbl}(\bcat{C}\!\on{at}_{(\infty,1)})$, resp. to $\CSeglax$, factors through $\on{EDbl}(\bcat{C}\!\on{at}_{(\infty,1)})$, resp.  $\on{ESeg}_\Delta(\bcat{C}\!\on{at}_{(\infty,1)})$.
  \end{proposition}
  \begin{proof}
    Let $\scr{X}_\bullet \in \on{EDbl}(\bcat{C}\!\on{at}_{(\infty,1)})$ and let $K_*(\scr{X})_\bullet=\bcat{X}_\bullet$. We want to show that $\bcat{X}_\bullet$ satisfies the conditions of \autoref{def:enhanceddouble}.

     It follows from \autoref{prop:1dblcharac} that $\bcat{X}_\bullet$ satisfies the oplax Segal condition (see \autoref{def:laxsegcond}). By construction it is also clear that $\scr{X}$ is the nucleus of $\bcat{X}_\bullet$ which is a double $(\infty,1)$-category. We conclude by \autoref{rem:limitformula} that $\bcat{X}_\bullet$ is enhanced (\autoref{def:enhanced}). 

     Finally note that $K_*(\scr{X})_0 \isom \scr{X}_0$ which proves the claim for Segal objects due to \autoref{prop:enhancedprev}.
   \end{proof}

   \begin{theorem}\label{thm:enhanced}
      Let $K_*$ be the right Kan extension functor defined in \autoref{prop:enhancedkanfactors}. Then we have a commutative diagram 
      \[
        \begin{tikzcd}
          \CSeglax \arrow[d] \arrow[r,"\simeq"] & \on{ESeg}_\Delta(\bcat{C}\!\on{at}_{(\infty,1)}) \arrow[d] \\
          \on{Dbl}\left(\bcat{C}\!\on{at}_{(\infty,1)}\right)^{\mathcal{I}\!\on{oplax}} \arrow[r,"\simeq"] & \on{EDbl}(\bcat{C}\!\on{at}_{(\infty,1)})
        \end{tikzcd}
      \]
     where the vertical morphisms are fully faithful and the horizontal morphisms are equivalences of $(\infty,2)$-categories which are obtained by restricting the functor $K_*$.
   \end{theorem}
   \begin{proof}
    Note that the vertical morphisms are fully faithful. In particular, the top horizontal morphism is fully faithful if and only if the bottom horizontal morphism is. Moreover since $K_*(\scr{X})_0 \isom \scr{X}_0$ it follows that if the bottom horizontal morphism is essentially surjective so is the top one. We might therefore, reduce our attention to the case of double $(\infty,1)$-categories,

     We first show that $K_*$ is fully faithful. Let $\scr{X}_\bullet$  be a double $(\infty,1)$-category and let $\bcat{X}_\bullet$ be the image of $\scr{X}_\bullet$ under $K_*$ which is an enhanced double $(\infty,1)$-category by \autoref{prop:enhancedkanfactors}. Since $K_*$ is a right adjoint it will be enough to show that for any other double $(\infty,1)$-category $\scr{Y}_{\bullet}$, the map obtained by precomposing with the counit $L_\bullet : \bcat{X}_\bullet \xlongrightarrow{} \scr{X}_\bullet$ induces an equivalence on mapping $(\infty,1)$-categories
     \[
       \mathcal{L}^*:\on{Nat}_{\Delta^\op}(\scr{X}_\bullet,\scr{Y}_\bullet)^{\mathcal{I}\!\on{oplax}} \xlongrightarrow{} \on{Nat}_{\Delta^\op}(\bcat{X}_\bullet,\scr{Y}_{\bullet})^{\mathcal{I}\!\on{oplax}}.
     \]
     We note that the morphism induced by the unit $R_\bullet: \scr{X}_\bullet \xlongrightarrow{} \bcat{X}_\bullet$ also provides a map 
     \[
       \mathcal{R}^* : \on{Nat}_{\Delta^\op}(\bcat{X}_\bullet,\scr{Y}_\bullet)^{\mathcal{I}\!\on{oplax}} \xlongrightarrow{} \on{Nat}_{\Delta^\op}(\scr{X}_\bullet,\scr{Y}_\bullet)^{\mathcal{I}\!\on{oplax}}. 
     \]
     We can now use \autoref{lem:adjonadj} to see that $\mathcal{R}^* \circ \mathcal{L}^* \isom \on{id}$ and to further obtain a natural transformation $\on{id} \xRightarrow{} \mathcal{L}^*  \circ \mathcal{R}^*$. To finish the proof we need to show that this later natural transformation is a pointwise equivalence. This amount to verify that for every $\mathcal{I}\!\on{oplax}$ natural transformation $f: \bcat{X}_\bullet \xlongrightarrow{} \scr{Y}_\bullet$, every $n\geq 0$ and every $x \in \bcat{X}_n$ the image of the map $\epsilon_x : x \xlongrightarrow{} R_n \circ L_n(x)$ under $f_n$ is an equivalence in $\scr{Y}_n$. Since $\scr{Y}_\bullet$ is a double $(\infty,1)$-category it will suffice to check the map is an equivalence in $\scr{Y}_n$ after applying the functors $\scr{Y}_n \xlongrightarrow{} \scr{Y}_1$ associated to the maps $\varphi_i: 1 \xlongrightarrow{} n$ selecting the objects $i,i+1$ for $i=0,\dots,n-1$. Note that in particular the maps $\varphi_i$ are interval inclusions so it follows from \autoref{lem:adjonadj} that we have a commutative diagram
     \[
       \begin{tikzcd}
         \bcat{X}_n \times \Delta^1 \arrow[r,"\Phi_n"] \arrow[d,swap,"\bcat{X}(\varphi_i) \times \on{id}"] & \bcat{X}_n \arrow[d,"\bcat{X}(\varphi_i)"] \arrow[r,"f_n"] & \scr{Y}_n \arrow[d,"\scr{Y}(\varphi_i)"] \\
         \bcat{X}_1 \times \Delta^1 \arrow[r,"\Phi_1"] & \bcat{X}_1 \arrow[r,"f_1"] & \scr{Y}_1
       \end{tikzcd}
     \]
     and thus $f_n(\epsilon_x)$ must be an equivalence in $\scr{Y}_n$.

     To finish the proof we need to show that given an enhanced double $(\infty,1)$-category $\bcat{E}_\bullet$ there exists some double $(\infty,1)$-category $\scr{E}_\bullet$ such that $K_*(\scr{E})_\bullet \isom \bcat{E}_\bullet$. Let $\scr{E}_\bullet$ be the nucleus of $\bcat{E}_\bullet$. We can use $\mathcal{I}\!\on{oplax}$ natural transformation $\bcat{E}_\bullet \xRightarrow{} \scr{E}_\bullet$ and the adjunction $K^* \dashv K_*$ to obtain a natural transformation
     \[
       \Omega: \bcat{E}_\bullet \xRightarrow{} K_*(\scr{E})_\bullet.
     \]
     Since $\bcat{E}_\bullet$ is enhanced (\autoref{def:enhanced}) it follows that for each $n\geq 0$ the map $\bcat{E}_n \xlongrightarrow{} K_*(\scr{E})_n$ is induced by the identity map between the limit diagrams defining both $\bcat{E}_n$ and $K_*(\scr{E})_n$ so we conclude by \autoref{rem:limitformula} that $\Omega$ is a pointwise weak equivalence.
   \end{proof}

   \begin{definition}\label{def:nucmorph}
     Let $\mu:\bcat{E}_\bullet \xRightarrow{} \bcat{B}_\bullet$ be a morphism in $\EDbl$ and let $\scr{E}_\bullet$ and $\scr{B}_\bullet$ be their corresponding nuclei. We say that $\mu$ is \emph{nuclear} it given $n\geq 2$ and $x \in \bcat{E}_n$ which which factors through $\scr{E}_n$ then it follows that $\mu(x) \in \scr{B}_n$. We denote by $\EDblnuc$ the $(\infty,2)$-category obtained by only considering the nuclear morphisms in $\EDbl$. In a totally analogous way we define the $(\infty,2)$-category $\ESegnuc$.
   \end{definition}

  \begin{lemma}\label{lem:nuccharac}
    Let $\scr{E}_\bullet,\mathscr{L}_\bullet \in \on{Dbl}\left(\bcat{C}\!\on{at}_{(\infty,1)}\right)^{\mathcal{I}\!\on{oplax}}$. Then given an $\mathcal{I}\!\on{oplax}$ natural transformation $\Phi: \scr{E}_\bullet \Rightarrow \mathscr{L}_\bullet$ it follows that $\Phi$ is natural (non-lax) if and only if the induced map on enhanced double $(\infty,1)$-categories $K_*(\Phi)$ is nuclear (see \autoref{def:nucmorph}).
  \end{lemma}
  \begin{proof}
    As a consequence of \autoref{prop:1dblcharac} we see that condition $\star)$ is equivalent to the existence the natural transformation $\widetilde{\Phi}$ appearing in the left-most square in the commutative diagram
    \[
      \begin{tikzcd}
        \scr{E}_\bullet \arrow[d,"\widetilde{\Phi}"] \arrow[r,"i_{\scr{E}}"] & K_*(\scr{E})_\bullet \arrow[d,"K_*(\Phi)"] \arrow[r,"r_{\scr{E}}"] & \scr{E}_\bullet \arrow[d,"\Phi"] \\
        \mathscr{L}_\bullet \arrow[r,"\epsilon_{\mathscr{L}}"] & K_*(\mathscr{L})_\bullet \arrow[r,"\eta_{\mathscr{L}}"] & \mathscr{L}_\bullet
      \end{tikzcd}
    \]
    We know that $r \circ i$ is equivalent to the identity natural transformation so we conclude that $\Phi$ must be natural as well.
  \end{proof}

  \begin{corollary}\label{cor:enhanced}
    The commutative diagram  given in \autoref{thm:enhanced} restricts to a diagram
    \[
     \begin{tikzcd}
       \CSeg \arrow[r,"\simeq"] \arrow[d] & \ESegnuc \arrow[d] \\
       \on{Dbl}(\bcat{C}\!\on{at}_{(\infty,1)}) \arrow[r,"\simeq"] & \EDblnuc
     \end{tikzcd}
    \]
  where the vertical morphisms are fully faithful and the horizontal  are equivalences of $(\infty,2)$-categories.
  \end{corollary}
  \begin{proof}
    It follows from \autoref{thm:enhanced} that the restriction of the functor $K_*$ will remain essentially surjective. We also note that the vertical morphisms are again fully faithful so we can restrict our attention to showing that the bottom horizontal morphism is fully faithful.

     It is also a consequence of \autoref{thm:enhanced} that the functor $K_*$ restricted to $\on{Dbl}(\bcat{C}\!\on{at}_{(\infty,1)})$ induces a fully faithful functor on mapping $(\infty,1)$-categories. We conclude from \autoref{lem:nuccharac} that $K_*$ is also essentially surjective on mapping $(\infty,1)$-categories.
  \end{proof}

 \section{The enhanced square functor}
 \begin{definition}\label{def:sq}
   Let $\bcat{C}$ be a fibrant marked-scaled simplicial set. We define a functor
   \[
     \mathbf{Sq}_{\on{E}}(\bcat{C})_\bullet: \Delta^\op \xlongrightarrow{} \bcat{C}\!\on{at}_{(\infty,1)}, \enspace n \mapsto \Sqe(\bcat{C})_n= \Fun^{\mathbf{opgr}}(\Delta^n_\flat,\bcat{C})^{\leq 1}
   \]
   where $(-)^{\leq 1}$ is the underlying $(\infty,1)$-category functor given in \autoref{def:leq}. We call this functor the \emph{enhanced square functor}. 
 \end{definition}

\begin{definition}
  Let $\bcat{C}$ be a fibrant marked-scaled simplicial set. We define a functor
   \[
     \mathbf{Sq}(\bcat{C})_\bullet: \Delta^\op \xlongrightarrow{} \bcat{C}\!\on{at}_{(\infty,1)}, \enspace n \mapsto \mathbf{Sq}(\bcat{C})_n= \Fun^{\mathbf{opgr}}(\Delta^n_\sharp,\bcat{C})^{\leq 1}
   \]
   which we call the \emph{ordinary} square functor.
\end{definition}

 \begin{remark}\label{rem:osqdouble}
   It is easy to see that the ordinary square functor $\mathbf{Sq}(\bcat{C})_\bullet$ defines a double $(\infty,1)$-category.
 \end{remark}

 \begin{definition}
   We denote by $\bcat{B}\!\on{icat}_\infty$ the $\on{Set}_\Delta^+$-enriched category defined as follows:
   \begin{itemize}
     \item The objects are given by fibrant marked-scaled simplicial sets.
     \item Given a pair of objects $\bcat{C},\bcat{D}$ we have a marked simplicial set of maps between $\bcat{C}$ and $\bcat{D}$ given by $\Fun(\bcat{C},\bcat{D})^{\leq 1}$ (see \autoref{def:leq}).
   \end{itemize}
   We observe that since $\Fun(\bcat{C},\bcat{D})^{\leq 1}$ is always an $(\infty,1)$-category for every pair object it follows that $\bcat{B}\!\on{icat}_\infty$ is a fibrant $\on{Set}_\Delta^+$-enriched category which models the $(\infty,2)$-category of $(\infty,2)$-categories, functors and natural transformations.
 \end{definition}

 \begin{remark}\label{rem:enrichedmodelforfunct}
    Recall that we can model the $(\infty,2)$-category $\on{Fun}(\Delta^\op,\iCatt)$ as a fibrant $\on{Set}_\Delta^+$-enriched category as follows:
    \begin{itemize}
      \item The objects are $\on{Set}_\Delta^+$-enriched functors $X_\bullet:\Delta^\op \xlongrightarrow{} \on{Set}_\Delta^+$ which are cofibrant and fibrant for the injective model structure on $\on{Fun}(\Delta^\op,\on{Set}_\Delta^+)$. In particular for every $n\geq 0$ we have that $X_n$ is an $(\infty,1)$-category.
      \item Given a pair of objects $X_\bullet$ and $Y_\bullet$ we define $\on{Nat}_{\Delta^\op}(X_\bullet,Y_\bullet)$ as the (fibrant) marked simplicial set characterised by the universal property which identifies a map of marked simplicial sets $K \xlongrightarrow{} \on{Nat}_{\Delta^\op}(X_\bullet,Y_\bullet)$ with a natural transformation of $\on{Set}_\Delta^+$-enriched functors $(X \times \underline{K})_\bullet \xRightarrow{} Y_\bullet$ where
      \[
        (X \times \underline{K})_n= X_n \times K, \enspace \enspace \text{ For }n\geq 0.
      \]
    \end{itemize}
  \end{remark}

  \begin{remark}\label{rem:sqenriched}
   Let $\mathscr{L} \xlongrightarrow{} \Fun(\bcat{C},\bcat{D})^{\leq 1}$ then we can  define a natural transformation $(\Sqe(\bcat{C}) \times \underline{\mathscr{L}})_\bullet \xRightarrow{} \Sqe(\bcat{D})_\bullet$ by sending a simplex $(\sigma_1,\sigma_2): \Delta^k  \xlongrightarrow{} \Sqe(\bcat{C})_n \times \mathscr{L}$ to the composite
   \[
     \Delta^n_\flat \tensor \Delta^k_\sharp \xlongrightarrow{(\sigma_1,\sigma_2 \circ p_2)} \bcat{C}\times \mathscr{L} \xlongrightarrow{} \bcat{D}
   \]
   where $p_2$ denotes the projection map $ \Delta^n_\flat \tensor \Delta^k_\sharp \xlongrightarrow{} \Delta^k_\sharp$. It follows that we have a functor of $(\infty,2)$-categories
   \[
     \Sqe: \bcat{B}\!\on{icat}_\infty \xlongrightarrow{} \on{Fun}(\Delta^\op,\iCatt).
   \]

  Essentially the same discussion also shows that we have a functor of $(\infty,2)$-categories:
   \[
     \mathbf{Sq}: \bcat{B}\!\on{icat}_\infty \xlongrightarrow{} \on{Dbl}(\bcat{C}\!\on{at}_{(\infty,1)}).
   \]
 \end{remark}

The goal of this section is to establish the following two facts:
\begin{itemize}
  \item For every $(\infty,2)$-category $\bcat{C}$ the functor $\Sqe(\bcat{C})_\bullet$ defines an enhanced double $(\infty,1)$-category.
  \item The functor $\Sqe$ is a fully faithful functor of $(\infty,2)$-categories.
\end{itemize}

\begin{definition}\label{def:expandedDelta}
 Let $D^n$ be the poset of pairs $(i,S)$ where $S\subseteq [n]$ with $S \neq 0$ and $i \in S$. We declare $(i,S) < (j,T)$ if $i \leq j$ and $S \subseteq T$. We define a marked-scaled simplicial set $\mathcal{D}^n \in \on{Set}_\Delta^{\mathbf{ms}}$ as follows:
 \begin{enumerate}
   \item The underlying simplicial set of $\mathcal{D}^n$ is the simplicial subset of $D^n$ consisting in those simplices 
   \[
      (i_0,S_0) \xlongrightarrow{} (i_1,S_1) \xlongrightarrow{} \cdots (i_n,S_n)
    \] 
    such that $i_{j} \in S_0$ for $0 < j \leq n$.
  \item An edge $(i,S) \xlongrightarrow{} (j,T)$ is declared marked if $i=j$.
  \item A  2-simplex $(i,S) \xlongrightarrow{} (j,T) \xlongrightarrow{} (\ell, W)$ is declared to be thin if at least one of the following conditions is satisfied:
  \begin{enumerate}
    \item The inclusion $S \xlongrightarrow{} T$ is idle (see \autoref{def:deltaidle}).
    \item We have that $j=\ell$. 
  \end{enumerate}
 \end{enumerate}
 It is easy to see that the categories $\mathcal{D}^n$ assemble into a cosimplicial object $\mathcal{D}^\bullet: \Delta \xlongrightarrow{} \on{Set}_\Delta^{\mathbf{ms}}$.
\end{definition}

\begin{definition}\label{def:alphanOO}
  Let $\Nsc(\OO^n)=\bcat{O}^n$ be the scaled nerve of the 2-category defined in \autoref{def:OI} where we view $\bcat{O}^n$ as a marked-scaled simplicial set by declaring every equivalence to be marked. We construct a map
  \[
    \alpha_n: \mathcal{D}^n \xlongrightarrow{} \bcat{O}^n
  \]
  which sends a $k$-simplex $\rho$ specified by $\{(i_j,S_j)\}_{j=0}^k$ to the $k$-simplex determined by the map $f_{\rho}:\OO^k \xlongrightarrow{} \OO^n$, which sends a vertex $u \in [k]$ to $i_u$ and sends $U \in \OO^k(u,v)$ where $U=\{u,v\}$ to $S_{u} \cap [i_u,i_v]$. We remind the reader that the functor $f_{\rho}$ is uniquely determined by this data since we have that given $u<y<v$ it follows that
  \[
    S_u \cap [i_u,i_v]  \subseteq \left( S_u \cap [i_u,i_y]\right) \cup \left( S_y \cap [i_y,i_v]\right).
  \]
  This defines $\alpha_n$ at the level of the underlying simplicial sets. We further check that $\alpha_n$ maps marked edges to degenerate edges. We finish the definition of $\alpha_n$ by showing that our map preserves thin 2-simplices. This amounts to showing that given $(i,S) \xlongrightarrow{} (j,T) \xlongrightarrow{} (\ell, W)$ satisfying either condition 1 or condition 2 in \autoref{def:expandedDelta} it follows that
  \[
    S \cap [i,\ell]= ( S \cap [i,j] ) \cup (T \cap [j,\ell]).
  \]
  If $S \xlongrightarrow{} T$ is idle it follows that if $x \in T \cap [j,\ell]$ then since by definition we have that $\ell \in S$ it follows that $x \in S \cap [i,\ell]$ which proves our claim. If $j=\ell$ the claim is obvious.

  One can easily verify that the maps $\{\alpha_n\}_{n \in \mathbb{N}}$ assemble into a natural transformation of cosimplicial objects $\alpha_\bullet: \mathcal{D}^\bullet \xRightarrow{} \bcat{O}^\bullet$.
\end{definition} 

\begin{lemma}\label{lem:expandeddeltaplus}
  Let $\mathcal{D}^n_+$ be the marked-scaled simplicial set obtained from $\mathcal{D}^n$ by marking (resp. scaling) an edge (resp. triangle) if its image in $\bcat{O}^n$ under $\alpha_n$ is marked (resp. scaled). Then the associated map $\mathcal{D}^n  \xlongrightarrow{} \mathcal{D}^n_+$ is a weak equivalence of marked-scaled simplicial sets.
\end{lemma}
\begin{proof}
  First we observe that $\mathcal{D}^n_+$ and $\mathcal{D}^n$ differ only in the scaling. Let $\sigma:(i,S) \xlongrightarrow{} (j,T) \xlongrightarrow{} (\ell, W)$ be a simplex such that $\alpha_n(\sigma)$ is thin in $\bcat{O}^n$. We note that this happens precisely if $S \cap [j,\ell]= T \cap [j,\ell]$ and we wish to show that we can scale $\sigma$ using pushouts along anodyne morphisms.

  We have a 3-simplex $\rho:(i,S) \xlongrightarrow{} (j,T) \xlongrightarrow{} (\ell,T) \xlongrightarrow{} (\ell,W)$ where every face is scaled except possibly the face skipping the vertex 2 and the face skipping the final vertex. Therefore, we see that if we can scale the face skipping the final vertex the result will follow after using a pushout along a morphism of type \ref{ms:wonky4} in \autoref{def:msanodyne}. A totally analogous argument, now using the 3-simplex $\theta:(i,S) \xlongrightarrow{} (j,S) \xlongrightarrow{} (j,T) \xlongrightarrow{} (\ell,T)$, shows that it will be enough to succesfully scale the simplex $(j,S) \xlongrightarrow{} (j,T) \xlongrightarrow{} (\ell,T)$. 

  We look now at a commutative diagram
  \[
    \begin{tikzcd}
      (j,S \cap [j,\ell]) \arrow[r] \arrow[d] & (j,T \cap [j,\ell]) \arrow[d] \arrow[r] & (\ell, T \cap [j,\ell]) \arrow[d] \\
      (j,S) \arrow[r] & (j,T) \arrow[r] & (\ell,T)
    \end{tikzcd}
  \]
 and note that the top horizontal row is given by a degenerate simplex. It follows that we can scale every simplex in this diagram, except the bottom horizontal row, using morphisms of type \ref{ms:wonky4} in \autoref{def:msanodyne}. To finisn the proof we look at the 3-simplex
 \[
   \Xi: (j,S \cap [j,\ell]) \xlongrightarrow{}  (j,S)  \xlongrightarrow{} (j,T) \xlongrightarrow{} (\ell,T)
 \]
 and note that we can scale the remaing to simplex using a pushout along a morphism of type $i)$ in \autoref{lem:0saturations}.
 \end{proof}

\begin{proposition}\label{prop:alphaessesurj}
  For every $n\geq 0$ we have that the map $\alpha_n: \mathcal{D}^n \xlongrightarrow{} \bcat{O}^n$ is a weak equivalence of marked-scaled simplicial sets.
\end{proposition}
\begin{proof}
  Note that in virtue of \autoref{lem:expandeddeltaplus} we can assume that the decorations of $\mathcal{D}^n$ are those induced from $\bcat{O}^n$. Let us consider an epi-mono factorization of $\alpha_n$ given by
  \[
    \mathcal{D}^n \xlongrightarrow{\pi_n} \mathcal{O}^n \xlongrightarrow{\iota_n} \bcat{O}^n.
  \]
  We will show that both $\pi_n$ and $\iota_n$ are weak equivalences. It is clear that we can express $\mathcal{D}^n$ as a homotopy colimit of its (decorated) simplices. Now, since both $\mathcal{D}^n$ and $\mathcal{O}^n$ can be expressed as homotopy colimits over the same diagram and $\pi_n$ is the map obtained from a natural transformation between the diagrams it will be enough to show that if we are given a (decorated) simplex $\tau: \Delta^k \xlongrightarrow{} \mathcal{D}^n$ such that its image under $\alpha_n$ factors as
  \[
    \Delta^k \xlongrightarrow{s} \Delta^m \xlongrightarrow{} \bcat{O}^n 
  \]
  where $s$ is surjective, then it follows that the map of marked-scaled simplicial sets $s:\Delta^k \xlongrightarrow{} \Delta^m$ (where the decorations of $\Delta^m$ are those induced from $\bcat{O}^n$) is a weak equivalence. However, in this situation it follows from \autoref{lem:expandeddeltaplus} that every edge that gets collapsed in $\bcat{O}^n$ is already marked in $\Delta^k$ and that every triangle containing one of those edges has to be thin. One can now easily check that the map  $s: \Delta^k \xlongrightarrow{} \Delta^m$ admits a homotopy inverse.

  To finish the proof, we will show that $\iota_n$ is a weak equivalence. Since the marking in $\mathcal{O}^n$ and in $\bcat{O}^n$ is given by degenerate edges it will be enough to show (after forgetting the marking) that the map of $\on{Set}_\Delta^+$-enriched categories $\mathfrak{C}^{\mathbf{sc}}[\mathcal{O}^n] \xlongrightarrow{}  \mathfrak{C}^{\mathbf{sc}}[\bcat{O}^n]$ is a weak equivalence.

  We can describe $\mathfrak{C}^{\mathbf{sc}}[\bcat{O}^n]$ as the a certain 2-category (viewed as a $\on{Set}_\Delta^+$-enriched category), together with a marking on the mapping categories, as follows:
  \begin{enumerate}
    \item Objects are given by elements $i \in [n]$.
    \item Given $i,j \in [n]$ with $i > j$ we define $\mathfrak{C}^{\mathbf{sc}}[\bcat{O}^n](i,j)$ to be the empty marked simplicial set. If $i \leq j$ we define $\mathfrak{C}^{\mathbf{sc}}[\bcat{O}^n](i,j)$ to be the poset defined as follows:
    \begin{itemize}
      \item The objects are given pairs $(S,\vec{U})$ where $S \subseteq [n]$ is of the form $S=\{i=i_0<i_1<i_2<\cdots i_{k-1}<i_{k}=j\}$ and $\vec{U}=\{U_\gamma\}_{\gamma=0}^{k-1}$ is a collection of non-empty subsets $U_{\gamma}\subseteq [n]$ such that $\min(U_\gamma)=i_{\gamma}$ and $\max(U_{\gamma})=i_{\gamma+1}$.
      \item We say that $(S,\vec{U}) \leq (T,\vec{V})$ if $S \subseteq T$ and for every $0\leq \gamma \leq k-1$ 
      \[
          U_{\gamma} \subset  \bigcup_{i_{\gamma}<x_{\epsilon}<i_{\gamma+1}}V_{\epsilon},
        \]  
        where $T=\{i=x_0<x_1 < \cdots <x_{t-1}<x_t=j\}$.
       \item We say that a map $(S,\vec{U}) \xlongrightarrow{} (T,\vec{V})$ is marked if for every $0\leq \gamma \leq k-1$  we have that
        \[
          U_{\gamma} =  \bigcup_{i_{\gamma}<x_{\epsilon}<i_{\gamma+1}}V_{\epsilon}.
        \]  
     \end{itemize} 
     \item We note that there is a map $\xi_n:\mathfrak{C}^{\mathbf{sc}}[\bcat{O}^n] \xlongrightarrow{} \OO^n$ which is the identity on objects and sends an element $(S,\vec{U}) \in \mathfrak{C}^{\mathbf{sc}}[\bcat{O}^n](i,j)$ in  to the union $\bigcup\limits_{\gamma }U_\gamma$ viewed as an object in $\OO^n(i,j)$. $\mathfrak{C}^{\mathbf{sc}}[\bcat{O}^n](i,j)$. One checks that the map 
     \[
      \xi_n^{i,j}: \mathfrak{C}^{\mathbf{sc}}[\bcat{O}^n](i,j) \xlongrightarrow{} \OO^n(i,j)
     \]
     admits a section $r_n^{i,j}$ which sends a subset $S \in \OO^n(i,j)$ to the pair $r_n^{i,j}(S)=(\{i,j\},S)$. We finally see that we have a natural transformation $r_n^{i,j}\circ \epsilon^{i,j}_n \xRightarrow{} \on {id}$  whose components are given by the obvious maps $(\{i,j\}, \xi^{i,j}_n(S,\vec{U})) \xlongrightarrow{} (S,\vec{U})$. Since the components of this natural transformation are marked edges we conclude that $\xi^n_{i,j}$ is a weak equivalence of marked simplicial sets and thus $\xi_n$ is weak equivalence of $\on{Set}_\Delta^+$-enriched categories.

  \end{enumerate}
Unpacking the definitions we see that we can identify $\mathfrak{C}^{\mathbf{sc}}[\mathcal{O}^n]$ with the $\on{Set}_\Delta^+$-enriched subcategory of $\mathfrak{C}^{\mathbf{sc}}[\bcat{O}^n]$ containing the same set of objects and where $\mathfrak{C}^{\mathbf{sc}}[\mathcal{O}^n](i,j)$ is the marked simplicial subset consisting in those simplices
\[
 \rho: (S^0,\vec{U}^0) \xlongrightarrow{} (S^1,\vec{U}^1) \xlongrightarrow{} \cdots \xlongrightarrow{} (S^k,\vec{U}^k)
\]
satisfying the following property:
\begin{itemize}
  \item[$\star)$] Let $S_0=\{i=i_0<i_1<i_2<\cdots i_{k-1}<i_{k}=j\}$ and let $i_{j}<i_{j+1}$ in $S_0$. Then for every $x \in S_k$ such that $i_{j}<x<i_{j+1}$ it follows that $x \in U^0_j$.
\end{itemize}
As a consequence of $3$ above, it follows that it will be enough to show that for every $i\leq j$ the obvious map
\[
  \hat{\xi}^{i,j}_n: \mathfrak{C}^{\mathbf{sc}}[\mathcal{O}^n](i,j) \xlongrightarrow{} \OO^n(i,j)
\]
is a weak equivalence. 

We note that the map $r_n^{i,j}$ in point 3, clearly factors through $\mathfrak{C}^{\mathbf{sc}}[\mathcal{O}^n](i,j)$ so it will be enough to show that the marked homotopy between $r_n^{i,j} \circ \xi^{i,j}_n \xRightarrow{} \on{id}$ descends. We consider a simplex
\[
 \rho: (S^0,\vec{U}^0) \xlongrightarrow{} (S^1,\vec{U}^1) \xlongrightarrow{} \cdots \xlongrightarrow{} (S^k,\vec{U}^k)
\]
and let $\xi^{i,j}_n(S^r,\vec{U}^r)=W^r$. We claim that every simplex in the diagram
\[
  \begin{tikzcd}
    (\{i,j\},W^0) \arrow[r] \arrow[d] &  (\{i,j\},W^1) \arrow[r] \arrow[d] & \cdots \arrow[r] \arrow[d] & (\{i,j\},W^k) \arrow[d] \\
    (S^0,\vec{U}^0) \arrow[r] & (S^1,\vec{U}^1) \arrow[r] & \cdots \arrow[r]& (S^k,\vec{U}^k)
  \end{tikzcd}
\]
factors through $\mathfrak{C}^{\mathbf{sc}}[\mathcal{O}^n](i,j)$ which follows immediately from condition $\star)$ above. This shows that the homotopy descends to $\mathfrak{C}^{\mathbf{sc}}[\mathcal{O}^n](i,j)$ and thus our result is proved.
\end{proof}

  \begin{lemma}\label{lem:redsubsets}
    Let $T_n \subset \Delta_{/n}$ be the full subcategory on objects such that the associated morphism $\varphi:a \xlongrightarrow{} n$ is a monomorphism and let us equipp $T_n$ with the marking induced from $(\Delta_{/n},\mathcal{I})$.  Then the morphism $\iota: (T_n,\mathcal{I}_{|T_n}) \xlongrightarrow{} (\Delta_{/n},\mathcal{I})$ defines a weak equivalence in the model structure of $(1,0)$-fibrations over $\Delta$.
  \end{lemma}
  \begin{proof}
    We will show that $\iota$ defines a weak equivalence in the model structure of $(1,0)$-fibrations over $\Delta_{/n}$. This is precisely the definition of being $(1,0)$-cofinal with respect to the markings described above. This in turn will imply that $\iota^\op$ is final with respect to $\mathcal{I}\!\on{oplax}$ limits. Invoking Theorem 3.31 in \cite{AScof} we are reduced to show the following:
    \begin{itemize}
      \item[i)] For every morphism $\varphi:a \xlongrightarrow{} n$ there exists an object in the pullback $(T_n)_{\varphi/}=T_n \times_{\Delta_{/n}}(\Delta_{/n})_{\varphi/}$ represented by a factorization $\varphi=u \circ v$ such that $v \in \mathcal{I}$ and $u \in T_n$.
      \item[ii)] Every such factorization defines an initial object the $(\infty,1)$-categorical localisation of $(T_n)_{\varphi/}$ at the collection of edges of $(T_n)_{\varphi/}$ which is given by those edges whose image in $T_n$ is marked.
    \end{itemize}
    If $\varphi \in T_n$ both conditions are immediate since $\iota$ is fully faithful. From now on, we can assume that $\varphi$ is not injective. Then the first condition is clearly satisfied by picking an epi-mono factorization of $\varphi$. Now suppose that we have a factorization $\varphi=u \circ v$ representing an object $(v,u) \in(T_n)_{\varphi/} $ as above and let us assume that $v$ is not surjective. Then by picking an epi-mono factorization of $v=\gamma \circ \beta$ we obtain a morphism between $(\beta,u \circ  \gamma) \xlongrightarrow{} (v,u)$ which is marked by \autoref{lem:iclarified}. It follows that both objects become equivalent in the localisation. We can safely assume that $v$ is surjective. In this situation every object of $(T_n)_{\varphi/}$ is given by a factorization $(r,s)$ where $r$ is not necessarily in $\mathcal{I}$.  Since $s$ is by assumption injective it follows that if we take an epi-mono factorization of $r=\tau \circ  \rho$ we obtain by uniqueness of the epi-mono factorization that $\rho=v$ and a unique morphism $(v,u) \xlongrightarrow{} (r,s)$. We learn that the object $(v,u)$ is already initial in $(T_n)_{\varphi/}$ and since localizations are (co)final the result follows. 
  \end{proof}

  \begin{proposition}\label{prop:Sqeenhanced}
  Let $\bcat{C} \in \bcat{B}\!\on{icat}_\infty$ be an $(\infty,2)$-category. Then we have an natural equivalence $\Sqe(\bcat{C})\xRightarrow{\simeq} K^*(\mathbf{Sq}(\bcat{C}))_\bullet$ where $K^*$ is the functor given in \autoref{def:kdownstar}. In particular, it follows that $\Sqe(\bcat{C})_\bullet$ is an enhanced double $(\infty,1)$-category.
  \end{proposition}
  \begin{proof}
    We consider an auxiliary functor $\mathbf{Sq}_{\bcat{O}}(\bcat{C})_\bullet$ given by $\mathbf{Sq}_{\bcat{O}}(\bcat{C})_n=\Fun^{\mathbf{opgr}}(\bcat{O}^n,\bcat{C})^{\leq 1}$. We note that we have a natural equivalence of functors $\Sqe(\bcat{C})_\bullet \xRightarrow{\simeq} \mathbf{Sq}_{\bcat{O}}(\bcat{C})_\bullet$ so it will suffice to construct an equivalence $\mathbf{Sq}_{\bcat{O}}(\bcat{C})\xRightarrow{\simeq} K^*(\mathbf{Sq}(\bcat{C}))_\bullet$

    We observe that we have maps $i_n:\Delta^n_\sharp \xlongrightarrow{} \bcat{O}^n$ which send and object $i \in [n]$ to the object $i \in \bcat{O}^n$ and send an edge $i < j$ to the subset $[i,j]=\{i,i+1,i+2,\dots,j\}$. We further consider the obvious map $\pi_n:\bcat{O}^n \xlongrightarrow{} \Delta^n_\sharp$ which is uniquely determined by mapping an object $i \in \bcat{O}^n$ to $i \in [n]$. We make some remarks:
    \begin{enumerate}
      \item The maps $\{i_n\}_{n \geq 0}$ assemble into an $\mathcal{I}$lax natural transformation $i: \Delta^\bullet_\sharp \xRightarrow{} \bcat{O}^\bullet$. Indeed, given a monotone morphism $f:[n] \xlongrightarrow{} [m]$ and $a,b \in [n]$ with $a<b$ we always have an inclusion of subsets $f([a,b]) \subset [f(a),f(b)]$ which is an equality precisly when $f$ belongs to $\mathcal{I}$.

      This yields an $\mathcal{I}$oplax natural transformation $\alpha^*:\mathbf{Sq}_{\bcat{O}}(\bcat{C})_\bullet \xRightarrow{} \mathbf{Sq}(\bcat{C})_\bullet$ for every $(\infty,2)$-category $\bcat{C}$.

      \item The maps $\{p_n\}_{n \geq 0}$ assemble into a natural transformation $p: \bcat{O}^\bullet \xRightarrow{} \Delta^\bullet_\sharp$ which yields a natural transformation $\mathbf{Sq}(\bcat{C})_\bullet \xRightarrow{} \mathbf{Sq}_{\bcat{O}}(\bcat{C})_\bullet$ which is pointwise fully faithful. The fully faithfulness follow from the fact that $\Delta^n_\sharp$ can be express as the localization of $\bcat{O}^n$ at all 2-morphisms. In order words, the map $p$ induces an equivalence of marked-scaled simplicial sets $(\bcat{O}^n,\flat,\sharp) \xlongrightarrow{} (\Delta^n,\flat,\sharp)$. 
    \end{enumerate}
    We can now see that since the functor $K_*$ was defined as a right adjoint then the $\mathcal{I}$oplax natural transformation $\alpha^*$ induces a natural transofmration $\Phi:\mathbf{Sq}_{\bcat{O}}(\bcat{C})_\bullet  \xRightarrow{} K_*(\mathbf{Sq}(\bcat{C})_\bullet)$. To finish the proof, we need to show that $\Phi$ is a pointwise equivalence. Let $n \geq 0$ and consider the map $T_n \xlongrightarrow{} \Delta_{/n}$ studied in \autoref{lem:redsubsets} then given the composite
    \[
      G: T_n^{\op} \xlongrightarrow{} (\Delta_{/n})^\op \xlongrightarrow{} \Delta^\op \xlongrightarrow{\mathbf{Sq}(\bcat{C})_n} \bcat{C}\!\on{at}_{(\infty,1)}
    \]
    Then it follows from \autoref{lem:redsubsets} that we have a factorization
    \[
       \mathbf{Sq}_{\bcat{O}}(\bcat{C})_n \xlongrightarrow{\varphi_n} \lim_{\mathcal{I}\!\on{oplax}}G \xlongrightarrow{\simeq} K_*(\mathbf{Sq}(\bcat{C}))_n
    \]
    Using the description of the $\mathcal{I}$oplax limit given in \cite{Berman} we can identify the map $\varphi$ as the map 
    \[
      \alpha_n^*: \Fun(\bcat{O}^n,\bcat{C})^{\leq 1} \xlongrightarrow{} \Fun(\mathcal{D}^n,\bcat{C})^{\leq 1}
    \]
    induced by restriction along the morphism $\alpha_n: \mathcal{D}^n \xlongrightarrow{} \bcat{O}^n$ given in \autoref{def:alphanOO}. The result now follows from \autoref{prop:alphaessesurj}.
  \end{proof}

  We wish now to show that our enhanced square functor is fully faithful. The strategy that we will deploy will be first to show that for every pair of $(\infty,2)$-categories $\bcat{C}$ and $\bcat{D}$, there exists a natural equivalence of $(\infty,1)$-categories
  \[
    \on{Nat}_{\Delta^\op}(\Sqe(\bcat{C})_\bullet,\Sqe(\bcat{D})_\bullet) \xlongrightarrow{\simeq} \Fun(\mathbb{X}^{\tensor}_{\bcat{C}},\bcat{D})^{\leq 1}
  \]
  where $\mathbb{X}^{\tensor}_{\bcat{C}}$ is a certain marked-scaled simplicial set. Moreover, we will show that the action of $\Sqe$ on mapping $(\infty,1)$-categories
  \[
    \Fun(\bcat{C},\bcat{D})^{\leq 1} \xlongrightarrow{} \on{Nat}_{\Delta^\op}(\Sqe(\bcat{C})_\bullet,\Sqe(\bcat{D})_\bullet)
  \]
is induced by restriction along a certain map $\eta_{\bcat{C}}:\mathbb{X}_{\bcat{C}}^\tensor \xlongrightarrow{} \bcat{C}$. 

We will conclude this section by showing that $\eta_{\bcat{C}}$ is a weak equivalence of marked-scaled simplicial sets which will imply our claim.

\begin{definition}\label{def:graycoend}
We have an adjunction of ordinary categories (see \autoref{rem:graymodels})
\[
  \int^{\mathbf{opgr}}_{E} : \Fun(\Delta^\op,\on{Set}_\Delta^{+})  \llra \on{Set}_\Delta^{\mathbf{ms}} : \Sqe
\]
where define the left adjoint view the coend formula
   \[
    \int^{\mathbf{opgr}}_{E} : \Fun(\Delta^\op,\on{Set}_\Delta^{+}) \xlongrightarrow{} \on{Set}_\Delta^{\mathbf{ms}}, \enspace X_\bullet \mapsto \int \Delta^\bullet_\flat \tensor X_\bullet,
  \]
and where the definition of $\Sqe(X)_\bullet$ for an arbitrary marked-scaled simplicial set is the obvious one.
 \end{definition} 

 \begin{proposition}\label{prop:enhancedcoendleftquillen}
   The adjoint functors 
   \[
  \int^{\mathbf{opgr}}_{E} : \Fun(\Delta^\op,\on{Set}_\Delta^{+})  \llra \on{Set}_\Delta^{\mathbf{ms}} : \Sqe
\]
define a Quillen adjunction where we view $\Fun(\Delta^\op,\on{Set}_\Delta^{+})$ as being equipped with the injective model structure and $ \on{Set}_\Delta^{\mathbf{ms}}$ with the model structure on marked-scaled simplicial sets provided in \autoref{thm:markedscaledmodel}.
 \end{proposition}
 \begin{proof}
   This follows from a totally analogous argument to that of \autoref{prop:coendleftquillen}.
 \end{proof}

 \begin{remark}\label{rem:etalc}
   Let $\mathscr{L}$ be an $(\infty,1)$-category. Then we observe that $\Sqe$ induces a map
   \[
     \Fun\left(\mathscr{L},\Fun(\bcat{C},\bcat{D})^{\leq 1}\right)^{\simeq} \xlongrightarrow{} \Fun(\mathscr{L},\on{Nat}_{\Delta^\op}(\Sqe(\bcat{C})_\bullet,\Sqe(\bcat{D})_\bullet))^{\simeq}
   \]
   whose target is equivalent to
   \[
    \Fun(\mathscr{L},\on{Nat}_{\Delta^\op}(\Sqe(\bcat{C})_\bullet,\Sqe(\bcat{D})_\bullet))^{\simeq}  \simeq \on{Nat}_{\Delta^\op}((\Sqe(\bcat{C})\times \underline{\mathscr{L}})_\bullet, \Sqe(\bcat{D})_\bullet)^{\simeq} .
   \]
   Here we are using the superscript “$\simeq$“ to denote the underlying $\infty$-groupoid of the functor $(\infty,1)$-category (resp. $(\infty,1)$-category of natural transformations). Let us remind the reader that the last equivalence follows readily from the definition of $\on{Nat}_{\Delta^\op}$ given in \autoref{rem:enrichedmodelforfunct}

    We observe that the Quillen adjunction in \autoref{prop:enhancedcoendleftquillen} provides us with an equivalence of $\infty$-groupoids
   \[
      \on{Nat}_{\Delta^\op}((\Sqe(\bcat{C})\times \underline{\mathscr{L}})_\bullet, \Sqe(\bcat{D})_\bullet)^\simeq \isom \Fun\left(\int^{\mathbf{opgr}}_{E}(\Sqe(\bcat{C})\times \underline{\mathscr{L}})_\bullet, \bcat{D}\right)^{\simeq},
   \]
   Since we have that $\Fun\left(\mathscr{L},\Fun(\bcat{C},\bcat{D})^{\leq 1}\right)^{\simeq} \isom \Fun(\bcat{C}\times \mathscr{L},\bcat{D})^{\simeq}$ we see after unraveling the definitions that the morphism above is induced by the map of marked-scaled simplicial sets
   \[
    \eta_{\bcat{C}}^{\mathscr{L}}: \int^{\mathbf{opgr}}_{E}(\Sqe(\bcat{C})\times \underline{\mathscr{L}})_\bullet \xlongrightarrow{} \bcat{C} \times \mathscr{L}.
   \]
  It is easy to see that at the level of the underlying simplicial sets the map $\eta_{\bcat{C}}^{\mathscr{L}}$ is simply the product of the counit of the adjunction $\eta_{\bcat{C}}$ in \autoref{prop:enhancedcoendleftquillen} with the identity functor on $\mathscr{L}$.
  \end{remark}

  \begin{definition}
    Let $\bcat{C}$ be an $(\infty,2)$-category and let $\mathscr{L}$ be an $(\infty,1)$-category. We define a marked scaled simplicial set $\mathbb{X}_{\bcat{C}}^{\mathscr{L}}$ whose underlying simplicial set coincides to that of $\int^{\mathbf{opgr}}_{E}(\Sqe(\bcat{C})\times \underline{\mathscr{L}})_\bullet $ by declaring an edge (resp. triangle) to be marked (resp. thin) if its image under the map defined in \autoref{rem:etalc}
    \[
    \eta_{\bcat{C}}^{\mathscr{L}}: \int^{\mathbf{opgr}}_{E}(\Sqe(\bcat{C})\times \underline{\mathscr{L}})_\bullet \xlongrightarrow{} \bcat{C} \times \mathscr{L}.
   \]
   is marked (resp. thin). We note that we have an isomorphism of marked-scaled simplicial sets $\mathbb{X}_{\bcat{C}}^{\mathscr{L}} \isom \mathbb{X}_{\bcat{C}}^{\Delta^{0}} \times \mathscr{L}$. In the specific case where $\mathscr{L}=\Delta^0$ we will denote $\mathbb{X}_{\bcat{C}}^{\Delta^{0}} $ by $\mathbb{X}^{\tensor}_{\bcat{C}}$.
  \end{definition}

    \begin{remark}
    For the rest of this section we will let $\mathscr{L}$ denote an $(\infty,1)$-category. 
  \end{remark}

  \begin{proposition}\label{prop:sqrigid}
    Let $\bcat{C},\bcat{D} \in \bcat{B}\!\on{icat}_{\infty}$. Then every morphism $\Phi:\Sqe(\bcat{C})_\bullet \xRightarrow{} \Sqe(\bcat{D})_\bullet$ is nuclear (see \autoref{def:nucmorph}). 
  \end{proposition}
  \begin{proof}
    First, let us assume that we know that $\Phi_1:\Sqe(\bcat{C})_1 \xlongrightarrow{} \Sqe(\bcat{D})_1 $ satisfies the following:
    \begin{itemize}
      \item[$\square)$] If $e:\Delta^1 \xlongrightarrow{} \Sqe(\bcat{C})_1$ is an edge represented by a commutative square. Then $\Phi_1(e)$ is also represented by a commutative square. 
    \end{itemize}
    Now we consider a an object $\sigma:\Delta^0 \xlongrightarrow{} \Sqe(\bcat{C})_2$ represented by a thin 2-simplex $f_\sigma:\Delta^2_\sharp \xlongrightarrow{} \bcat{C}$ and we wish to show that $\Phi_2(\sigma)$ is also represented by a thin 2-simplex in $\bcat{D}$. We consider an morphism $\nu$ in $\Sqe(\bcat{C})_2$ given by the map
    \[
      f_\nu:\Delta^2_\flat \tensor \Delta^1 \xlongrightarrow{m} \Delta^2_\flat \xlongrightarrow{f_\sigma} \bcat{C}
    \]
    where $m(i,j)=\max(i,j)$. We make some observations:
    \begin{enumerate}
      \item We have that $d_1(\nu)=\sigma$.
       \item Every square in $\nu$ commutes.
       \item The restriction  $f_\nu$ to $\Delta^{2}_\flat \tensor \Delta^{\{1\}}$ is degenerate.
       \item Given the map $i_2:[0] \xlongrightarrow{} [2]$ that selects the final vertex, it follows that $\Sqe(\bcat{C})(i_2)(\nu)=u$ is an equivalence in $\Sqe(\bcat{C})_0$.
     \end{enumerate} 
     Using $\square)$ and 2 above we conclude that $\Phi_2(\nu)$ consists only in commutative squares. Moreover 3, implies that bottom simplex in $\Phi_2(\nu)$is degenerate. Finally, it follows from from 4 that the $\Phi_0(u)$ is also an equivalence. A routinary argument using anodyne morphisms then shows that the associated map
     \[
       f_{\Phi_2(\nu)}: \Delta^2_\flat \tensor \Delta^1 \xlongrightarrow{} \bcat{D}
     \]
     has every simplex being thin except possible the restriction to $\Delta^2_\flat \tensor \Delta^{\{0\}}$. However, $4$ implies that this simplex must be thin aswell. We conclude that $\Phi_2(\sigma)$ is represented by a thin simplex as desired.

     Now we turn our attention to proving the claim $\square)$. First we show the claim for edges $e \in \Sqe(\bcat{C})_1$ which are of the form
     \[
       \Delta^1 \tensor \Delta^1 \xlongrightarrow{\tau_\epsilon} \Delta^1 \xlongrightarrow{\varphi} \bcat{C}, \enspace \enspace \epsilon \in \{-,+\}
     \]
     where $\tau_-(i,j)=\min(i,j)$ and $\tau_+(i,j)=\max(i,j)$. Let $\varphi:x \xlongrightarrow{} y$ in $\bcat{C}$ and consider the following diagram 
     \[\begin{tikzcd}
  & x & x \\
  x & x & y \\
  x & y & y
  \arrow[from=2-1, to=2-2]
  \arrow[from=2-1, to=3-1]
  \arrow["\varphi"', from=3-1, to=3-2]
  \arrow["\varphi", from=2-2, to=3-2]
  \arrow["\varphi", from=2-2, to=2-3]
  \arrow[from=2-3, to=3-3]
  \arrow[from=3-2, to=3-3]
  \arrow[from=1-2, to=2-2]
  \arrow[from=1-2, to=1-3]
  \arrow["\varphi"', from=1-3, to=2-3]
\end{tikzcd}\]
 Using a similar argument as before, we conclude that the claim holds for the class of edges described above. Suppose that we are given another edge represented $\gamma$ by a commutative diagram
 \[
  f_{\gamma}: \Delta^1 \times \Delta^1 \xlongrightarrow{} \bcat{C}
 \]
 such that the restriction of $f_\gamma$ to  $\Delta^1 \times \Delta^{\{\epsilon\}}$ with $\epsilon \in \{0,1\}$ is degenerate. Then we see that $f_\gamma$ represents an invertible 2-morphism in the mapping $(\infty,1)$-category $\bcat{C}(x,y)$. In particular, it follows that we can find edges in $\Sqe(\bcat{D})_2$ which exhibit $\Phi_1(\gamma)$ as representing an invertible morphism in $\bcat{D}(\Phi(x),\Phi(y))$.  For the final case we consider a general commutative square $f_\kappa: \Delta^1 \times \Delta^1 \xlongrightarrow{}\bcat{C}$ represented by a diagram
 \[\begin{tikzcd}
  {x_0} & {y_0} \\
  {x_1} & {y_1}
  \arrow["f", from=1-1, to=1-2]
  \arrow["u"', from=1-1, to=2-1]
  \arrow["g"', from=2-1, to=2-2]
  \arrow["v", from=1-2, to=2-2]
\end{tikzcd}\]
We consider a 3-simplex $\rho: \Delta^3 \xlongrightarrow{} \Sqe(\bcat{C})_1$ given by the diagram depicted below
\[\begin{tikzcd}
  {x_0} & {x_0} \\
  {x_0} & {y_0} \\
  {x_1} & {y_1} \\
  {y_1} & {y_1}
  \arrow["f", from=2-1, to=2-2]
  \arrow["u"', from=2-1, to=3-1]
  \arrow["g"', from=3-1, to=3-2]
  \arrow["v", from=2-2, to=3-2]
  \arrow[from=1-1, to=2-1]
  \arrow[from=1-1, to=1-2]
  \arrow["f", from=1-2, to=2-2]
  \arrow["g"', from=3-1, to=4-1]
  \arrow[from=3-2, to=4-2]
  \arrow[from=4-1, to=4-2]
\end{tikzcd}\]
 where 2-simplex is thin. It is easy to check that $\Phi_1(\rho)$ consists only in thin simplices which proves our result.
  \end{proof}

 \begin{lemma}\label{lem:sqedgessat}
     Let $\Xi:(\Sqe(\bcat{C}) \times \underline{\mathscr{L}}) \xRightarrow{} \Sqe(\bcat{D})_\bullet$ be a natural transformation. Given an edge $(u,v):\Delta^1 \xlongrightarrow{} \Sqe(\bcat{C})_1 \times \mathscr{L}$ such that $v$ is is an equivalence in $\mathscr{L}$, let us denote by $f_u:\Delta^1 \tensor \Delta^1 \xlongrightarrow{} \bcat{C}$ the map associated to the edge $u$. Then given another edge 
       \[
         \Delta^1 \xlongrightarrow{\tau} \Delta^1 \tensor \Delta^1 \xlongrightarrow{f_u} \bcat{C}
       \]
       such that $f_u \circ \tau$ is an equivalence in $\bcat{C}$ then it follows that $f_{\Xi(u)} \circ \tau$ is also an equivalence in $\bcat{D}$, where $f_{\Xi(u)}$ is the map associated to $\Xi(u,v)$.
  \end{lemma} 
  \begin{proof}
    First, we observe that without loss of generality one can assume that $\mathscr{L}$ is the walking isomorphism. Further inspection reveals that it will enough to show the case where $\mathscr{L}$ is the point. If $\tau$ selects a vertical or horizontal edge the claim follows easily. If $\tau$ selects the diagonal edge we consider a 2-simplex $\nu$ in $\Sqe(\bcat{C})_1$ given by restricting $f_u$ along the map
     \[
       p:\Delta^1 \tensor \Delta^2_\sharp \xlongrightarrow{} \Delta^1 \tensor \Delta^1, \enspace  \enspace  (i,j ) \mapsto p(i,j)=\begin{cases}
         \min(i,j), \enspace \text{ if } j\leq 1 \\
         (i,j-1), \enspace \text{ if } j =2.
       \end{cases}
     \]
     Let us make some remarks:
     \begin{itemize}
       \item[$i)$] The composite $f_u \circ p$ sends the edge $(1,0) \xlongrightarrow{} (1,2)$ to $f_u \circ \tau$.
       \item[$ii)$] We have that $d_0(\nu)=u$.
       \item[$iii)$] The edge $d_2(\nu)$ is represented by a commutative diagram. It follows from the proof of \autoref{prop:sqrigid} that is image in $\Sqe(\bcat{D})_1$ must be also commuting.
     \end{itemize}
     We can now look at $f_{\Xi(\nu)}: \Delta^1 \tensor \Delta^2_\sharp \xlongrightarrow{}\bcat{D}$ and observe the following:
     \begin{itemize}
         \item The image of the edge $(1,0) \xlongrightarrow{} (1,2)$ under $f_{\Xi(\nu)}$ is an equivalence in $\bcat{D}$.
         \item The image of the edge $(0,0) \xlongrightarrow{} (1,2)$ under $f_{\Xi(\nu)}$ is an equivalence in $\bcat{D}$ since it is the composite of a pair of equivalences.
         \item The image of the every triangle in the $3$-simplex given by $(0,0) \xlongrightarrow{} (0,1) \xlongrightarrow{} (1,1) \xlongrightarrow{} (1,2)$ under $f_{\Xi(\nu)}$ is thin. To see this we use $iii)$ above together with a routinary argument using anodyne morphisms.
       \end{itemize}  
       Using the remarks above we conclude that since the edge $(0,0) \xlongrightarrow{} (0,1)$ is also sent to an equivalence it follows that the image of the edge $(0,1) \xlongrightarrow{} (1,2)$ must be an equivalence in $\bcat{D}$ as well. This finishes the proof of $1$.
  \end{proof}

  \begin{lemma}\label{lem:sqetrisat}
    Let $\Xi:(\Sqe(\bcat{C}) \times \underline{\mathscr{L}}) \xRightarrow{} \Sqe(\bcat{D})_\bullet$ be a natural transformation. Given a triangle $(\sigma,\rho):\Delta^2 \xlongrightarrow{} \Sqe(\bcat{C})_2 \times \mathscr{L}$, let us denote by $f_{\sigma}:\Delta^2_\flat \tensor \Delta^2_\sharp \xlongrightarrow{} \bcat{C}$ the map associated to $\sigma$. Then given another triangle
       \[
         \Delta^2 \xlongrightarrow{\kappa} \Delta^2_\flat \tensor \Delta^2_\sharp \xlongrightarrow{f_{\sigma}} \bcat{C}
       \]
       such that $f_{\sigma} \circ \kappa$ is thin in $\bcat{C}$ then it follows that $f_{\Xi(\sigma)} \circ \kappa$ is also thin in $\bcat{D}$, where $f_{\Xi(\sigma)}$ is the map associated to $\Xi(\sigma,\rho)$.
  \end{lemma}
  \begin{proof}
    We can assume without loss of generality that $\mathscr{L}=\Delta^2_\sharp$. The proof will proceed by cases which are governed by the composite 
    \[
      \gamma:\Delta^2 \xlongrightarrow{\kappa} \Delta^2_\flat \tensor \Delta^2_\sharp \xlongrightarrow{} \Delta^2_\sharp
    \]
    where the second map is simply the projection onto the second factor.
    \begin{enumerate}
      \item The map $\gamma$ is degenerate on a point $j$.

      In this case we look at the 2-simplex $\phi$ given by restricting $f_\sigma$ along the map $\Delta^2_\flat \tensor \Delta^{\{j\}} \xlongrightarrow{} \Delta^2_\flat \tensor \Delta^2_\sharp$. We consider a 2-simplex $\hat{\sigma}$ in $\Sqe(\bcat{C})_2$ specified by the morphism
      \[
        f_{\hat{\sigma}}:\Delta^2_\flat \tensor \Delta^2_\sharp \xlongrightarrow{m} \Delta^2_\sharp \xlongrightarrow{\phi} \bcat{C}
      \]
      where the map $m$ is simply the maximum function. We consider the composite
      \[
        \Delta^2_\sharp \times \Delta^2_\sharp \xlongrightarrow{(\hat{\sigma},\on{id})} \Sqe(\bcat{C})_2 \times \Delta^2_\sharp \xlongrightarrow{} \Sqe(\bcat{D})_2
      \]
      which corresponds to a morphism
      \[
        F: \Delta^2_\flat \tensor ( \Delta^2_\sharp \times \Delta^2_\sharp) \xlongrightarrow{} \bcat{D}.
      \]
      We will show that every triangle in the image of $F$ is thin. To this end we make some observations:
      \begin{itemize}
        \item[i)]  Observe that the restriction of $F$ along the map $\iota: \Delta^{\{1,2\}}_\flat \tensor ( \Delta^2_\sharp \times \Delta^2_\sharp) \xlongrightarrow{} \Delta^2_\flat \tensor ( \Delta^2_\sharp \times \Delta^2_\sharp)$ factors as
        \[
            \Delta^{\{1,2\}}_\flat \tensor ( \Delta^2_\sharp \times \Delta^2_\sharp) \xlongrightarrow{\on{id}\times s_0 \times \on{id}}  \Delta^{\{1,2\}}_\flat \tensor (\Delta^1 \times \Delta^2_\sharp)
         \] 
         where $s_0:[2] \xlongrightarrow{} [1]$ is the degeneracy map which collapses the vertex $0$ and $1$. Consequently, in order to show that every triangle in the image of $F \circ \iota$ is thin it will suffice to show that every triangle in the image of
         \[
          \widetilde{F}:\Delta^{\{1,2\}}_\flat \tensor ( \Delta^{\{1,2\}}_\sharp \times \Delta^2_\sharp) \xlongrightarrow{} \Delta^2_\flat \tensor ( \Delta^2_\sharp \times \Delta^2_\sharp) \xlongrightarrow{F} \bcat{D}
         \]
         It follows from the proof of \autoref{prop:sqrigid} that every 2-simplex of the source of $\widetilde{F}$ which degenerates on a point in the factor $\Delta^2_\sharp$ must be sent to a thin simplex. Moreover, we see that every simplex of the source which either factors through $\Delta^{\{\epsilon\}}_\flat \tensor ( \Delta^{\{1,2\}}_\sharp \times \Delta^2_\sharp)$ with $\epsilon \in \{1,2\}$, through $\Delta^{\{1,2\}}_\flat \tensor ( \Delta^{\{2\}}_\sharp \times \Delta^2_\sharp)$ must be sent to a thin simplex. The first case is clear and the second follows since that restriction factors through the image of a degeneracy operator. The claim now follows after a routinary argument.
         \item[ii)]  Given the map $\alpha:\Delta^2_\flat \tensor ( \Delta^{\{1,2\}}_\sharp \times \Delta^2_\sharp) \xlongrightarrow{} \Delta^2_\flat \tensor ( \Delta^2_\sharp \times \Delta^2_\sharp)$ it follows that the restriction of $F$ along $\alpha$ factors through the map
         \[
           \Delta^2_\flat \tensor ( \Delta^{\{1,2\}}_\sharp \times \Delta^2_\sharp) \xlongrightarrow{s_0 \times \on{id}\times \on{id}} \Delta^{\{1,2\}}_\flat \tensor ( \Delta^{\{1,2\}}_\sharp \times \Delta^2_\sharp).
         \]
         We conclude from part $i)$ that every simplex in the image of $F \circ \alpha$ must be thin.
         \item[iii)] Let $\beta_\epsilon:\Delta^{\{0,\epsilon\}}_\flat \tensor ( \Delta^{\{0,\epsilon\}}_\sharp \times \Delta^2_\sharp)$ with $\epsilon \in \{1,2\}$. One verifies that also the restriction $F \circ \beta_{\epsilon}$ maps every triangle to a thin 2-simplex in $\bcat{D}$.
         \item[iv)] We also see that it follows from \autoref{prop:sqrigid} that every simplex in $\Delta^2_\flat \tensor ( \Delta^{2}_\sharp \times \Delta^2_\sharp)$ which degenerates on a point in the last factor must be sent to a thin simplex in $\bcat{D}$.
      \end{itemize}
       The claim follows after a routinary argument involving anodyne morphisms from observations $i)-iv)$.
       \item The map $\gamma$ is degenerate on an edge.

      To deal with this case we first establish the following claims:
      \begin{itemize}
        \item[$\diamond)$] Let $(e,\on{id}): \Delta^1 \times \Delta^2_\sharp \xlongrightarrow{} \Sqe(\bcat{C})_1 \times \Delta^2_\sharp $ such that $e$ is represented by a commutative square and consider the associated map 
        \[
          G: \Delta^1 \tensor( \Delta^1 \times \Delta^2_\sharp) \xlongrightarrow{} \bcat{D}
        \]
        Then it follows the the image of $G$ consists only in thin simplices.

        To show this claim we make again some observations:
        \begin{itemize}
          \item The map $G$ sends every triangle in  $\Delta^{\{\epsilon\}} \tensor( \Delta^1 \times \Delta^2_\sharp)$ with $\epsilon \in \{0,1\}$ to a thin simplex.
          \item The map $G$ sends every triangle in $\Delta^1 \tensor( \Delta^{\{\epsilon\}} \times \Delta^2_\sharp)$ with $\epsilon \in \{0,1\}$ to a thin simplex. This follows from the proof of $i)$ in case 1.
          \item The map $G$ sends every triangle which degenerates to a point in the last factor to a thin simplex. This follows from the proof of \autoref{prop:sqrigid}.
        \end{itemize}
        The claim $\diamond)$ now follows after another routinary argument using anodyne morphisms.
        \item[$\heartsuit)$] Let $\varphi_\epsilon: \Delta^2_\flat \tensor \Delta^1 \xlongrightarrow{} \Delta^1 \tensor \Delta^1$ with $\epsilon \in \{+,-\}$ and where we define
        \[
          \varphi_{-} (i,j)= \begin{cases}  
            (0,\min(i,j)), \enspace \text{ if } i \leq 1\enspace ,\\
            (i-1,j), \enspace \text{ if }i > 1 .
          \end{cases}         
          \enspace \enspace \enspace \varphi_+(i,j)=\begin{cases}
            (i,j), \enspace \text{ if } i \leq 1\enspace ,\\
            \max(i-1,j), \enspace \text{ if } i > 1.
          \end{cases}                          
        \]
        We further define 2-simplices $\alpha_{\epsilon}: \Delta^2 \xlongrightarrow{}\Delta^2_\flat \tensor \Delta^1$ with $\epsilon \in \{-,+\}$ given by
        \[
          \alpha_- : (0,0) \xlongrightarrow{} (1,0) \xlongrightarrow{}(2,1), \enspace \enspace \alpha_+ : (0,0) \xlongrightarrow{} (1,1) \xlongrightarrow{}(2,1)
        \]
        Given an edge $e: \Delta^1 \xlongrightarrow{} \Sqe(\bcat{C})_1$ consider another edge $\omega_\epsilon: \Delta^1 \xlongrightarrow{} \Sqe(\bcat{C})_2$ determined by the map
        \[
          f_{\omega_e} :\Delta^2_\flat \tensor \Delta^1 \xlongrightarrow{\varphi_{\epsilon}} \Delta^1 \tensor \Delta^1 \xlongrightarrow{f_e} \bcat{C}.
        \]
        We consider the map  $H_{\epsilon}:\Delta^2_\flat \tensor (\Delta^1 \tensor \Delta^2_\sharp) \xlongrightarrow{} \bcat{D}$ induced the composite
        \[
          \Delta^1 \times \Delta^2_\sharp \xlongrightarrow{(\omega_\epsilon,\on{id})} \Sqe(\bcat{C})_2 \times \Delta^2_\sharp \xlongrightarrow{\Xi} \Sqe(\bcat{D})_2.
        \]
        Then every simplex $\nu: \Delta^2 \xlongrightarrow{} \Delta^2_\flat \tensor (\Delta^1 \tensor \Delta^2_\sharp)$ whose projection onto the first 2 factors equal $\alpha_\epsilon$ is sent to a thin simplex in $\bcat{D}$ by $H_\epsilon$.

        We will only verify the case where $\epsilon=+$ and leave the remaining case as an (easier) exercise to the reader. Let $\nu$ be given by $(0,0,a) \xlongrightarrow{} (1,1,b) \xlongrightarrow{} (2,1,c)$ and consider a 3-simplex $\theta$ given by
        \[
          (0,0,a) \xlongrightarrow{} (1,0,a) \xlongrightarrow{} (1,1,b) \xlongrightarrow{} (2,1,c).
        \]
        It follows that the face skipping the vertex $0$ and the face skipping the vertex $3$ must be sent to a thin simplex in $\bcat{D}$ due to $\diamond)$. It then follows that it will suffice to show that the face skipping the vertex $2$ gets sent to a thin simplex. To this end we consider another 3-simplex
        \[
          (0,0,a) \xlongrightarrow{} (1,0,a) \xlongrightarrow{} (2,0,b) \xlongrightarrow{} (2,1,c).
        \]
        It is easy to verify that every face in this simplex (except possibly the face skipping the vertex $2$) is mapped under $H$ to a thin 2-simplex. We conclude that every face in this 3-simplex is thin in $\bcat{D}$ and thus the claim holds.

      \end{itemize}
      To prove the general case where $\gamma$ is degenerate on an edge, we can safely replace the original 2-simplex $\sigma:\Delta^2 \xlongrightarrow{} \Sqe(\bcat{C})_2$ with one of its faces. Note that if $f_\sigma \circ \kappa$ factors through a square $\Delta^1 \tensor \Delta^1$ then the result follows from $\diamond)$. We will assume that $\gamma$ is of the form $0 \xlongrightarrow{} 1 \xlongrightarrow{} 1$ and leave the remaining case as an (easier) exercise to the reader. We consider a map
      \[
        t_{-}: \Delta^3_\flat \tensor \Delta^1 \xlongrightarrow{} \Delta^2_\flat \tensor \Delta^1, \enspace \enspace t_{-}(i,j)=\begin{cases}
          (0,\min(i,j)), \enspace \text{ if } i \leq 1 \\
          (i-1,j), \enspace \text{ if } i > 1.
        \end{cases}
      \]
      which yields an edge  $s:\Delta^1 \xlongrightarrow{} \Sqe(\bcat{C})_3$. As before, we look at a functor
      \[
        L: \Delta^3_{\flat} \tensor (\Delta^1 \tensor \Delta^2_\sharp) \xlongrightarrow{} \bcat{D}
      \]
      and we wish to show that the image of the triangle $(1,0,a) \xlongrightarrow{} (2,1,b) \xlongrightarrow{} (3,1,c)$ under $L$ is thin in $\bcat{D}$. We extend this triangle to a 3-simplex
      \[
        (0,0,a) \xlongrightarrow{} (1,0,a)\xlongrightarrow{}(2,1,b)\xlongrightarrow{} (3,1,c)
      \]
      We observe that the first edge gets mapped under $L$ to an equivalence in $\bcat{D}$ and that the faces skipping the vertex $2$ and $3$ respectively get sent to a thin simplex as a consequence of $\heartsuit)$. Therefore it will be enough to show that the face skipping the vertex $1$ gets sent to a thin simplex. We consider another 3-simplex
      \[
        (0,0,a) \xlongrightarrow{} (0,1,a)\xlongrightarrow{} (2,1,b)\xlongrightarrow{} (3,1,c)
      \]
      In this simplex the face skipping the vertex $2$ and the face skipping the vertex $3$ get sent to a thin simplex due to $\diamond)$. The face skipping the vertex $0$ gets sent to a thin simplex due to point 1. The claim now follows.

     \item The simplex $\gamma$ does not degenerate.

     We break this final case into two subcases:
     \begin{itemize}
       \item The projection of $\kappa$ onto its first factor is degenerate on an edge.

       The only relevant case here is when $\kappa$ is of the form $(i,0) \xlongrightarrow{} (i,1) \xlongrightarrow{} (j,2)$ since the claim is already true if $\kappa$ is of the form $(i,0) \xlongrightarrow{} (j,1) \xlongrightarrow{} (j,2)$. Note that we can replace in this case $\sigma$ with one of its faces given by $f_{d(\sigma)}:\Delta^1 \tensor \Delta^2_\sharp  \xlongrightarrow{} \bcat{C}$. We define a map
       \[
         r: \Delta^2_\flat \tensor \Delta^2_\sharp \xlongrightarrow{}\Delta^1 \tensor \Delta^2_\sharp, \enspace r(i,j)=\begin{cases}
           (i,j), \enspace \text{ if } i \leq 1, \\
           (1,0), \enspace \text{ if } (i,j)=(2,0) \\
           (1,2), \enspace \text{ else. }
         \end{cases}
       \]
     which induces a morphism $W:\Delta^2_\flat \tensor (\Delta^2_\sharp \times \Delta^2) \xlongrightarrow{}\bcat{D}$. Our goal is to show that $W$ the image of a simplex of the form $(0,0,a) \xlongrightarrow{} (0,1,b) \xlongrightarrow{} (1,2,c)$ is thin in $\bcat{D}$. We look at a 3-simplex given by
     \[
       (0,0,a) \xlongrightarrow{} (0,1,b)  \xlongrightarrow{}1,2,c) \xlongrightarrow{} (2,2,c),
     \]
    and note that by $\heartsuit)$ every face of this 3-simplex gets mapped to a thin simplex in $\bcat{D}$ except possibly the skipping the vertex $2$ and the face skipping the vertex $3$. Moreover since the last vertex is an equivalence in $\bcat{D}$ it will suffice to show that the image $2$-nd face is thin in $\bcat{D}$. Finally let us consider another 3-simplex
    \[
        (0,0,a) \xlongrightarrow{} (0,1,b) \xlongrightarrow{}(2,1,c) \xlongrightarrow{}(2,2,c)
    \]
    It follows that every face in this simplex is thin in $\bcat{D}$  (using $\diamond)$) except possible the face skipping the vertex $2$ which implies our claim.
    \item The projection of $\kappa$ on its first factor does not degenerate.  The proof is almost the same as in the previous case so we only give a sketch. We define a map
    \[
      z: \Delta^3_\flat \tensor \Delta^2_\sharp \xlongrightarrow{} \Delta^2_\flat \tensor \Delta^2_\sharp, \enspace z(i,j)=\begin{cases}
        (i,j), \enspace \text{ if } i \leq 2\\
        (2,0), \enspace \text{ if } (i,j)=(3,0) \\
        (2,2), \enspace \text{ else }
      \end{cases}
    \]
    The proof now proceeds in almost the identical way as we did above with the map $W$.
     \end{itemize}
     The result now follows. \qedhere
    \end{enumerate}
  \end{proof}

  \begin{proposition}\label{prop:sqesaturation}
    The canonical map $\int^{\mathbf{opgr}}_{E}(\Sqe(\bcat{C})\times \underline{\mathscr{L}})_\bullet \xlongrightarrow{} \mathbb{X}^\tensor_{\bcat{C}} \times \mathscr{L}$ is a weak equivalence of marked-scaled simplicial sets.
  \end{proposition}
  \begin{proof}
    Let $\bcat{D}$ be a fibrant marked-scaled simplicial set. We need to show that the induced morphism of $\infty$-groupoids
    \[
      \varphi:\Fun(\mathbb{X}^\tensor_{\bcat{C}} \times \mathscr{L},\bcat{D})^{\simeq} \xlongrightarrow{} \Fun\left(\int^{\mathbf{opgr}}_{E}(\Sqe(\bcat{C})\times \underline{\mathscr{L}})_\bullet,\bcat{D}\right)^{\simeq} \isom \on{Nat}_{\Delta^\op}((\Sqe(\bcat{C})\times \underline{\mathscr{L}})_\bullet, \Sqe(\bcat{D})_\bullet)^\simeq 
    \]
    is a weak equivalence. First we note that by construction our map identifies the  left-most $\infty$-groupoid with a full subgroupoid of the space of natural transformations consisting in those natural transformations $\Xi: \Sqe(\bcat{C})\times \underline{\mathscr{L}})_\bullet \xRightarrow{} \Sqe(\bcat{D})$ satisfying the conditions studied in \autoref{lem:sqedgessat} and \autoref{lem:sqetrisat}. We conclude from both lemmas that these conditions are always satisfied and thus the result follows.
  \end{proof}

  \begin{definition}
      Let $\bcat{C}$ be an $\infty$-bicategory. We $\mathcal{S}_{\bcat{C},\tensor}$ to be the (ordinary) category whose objects are given by maps $f:\Delta^n_\flat \tensor \Delta^{k}_\sharp \xlongrightarrow{} \bcat{C}$ and whose morphisms are given by commutative diagrams
  \[
     \begin{tikzcd}
        \Delta^{n}_\flat \tensor \Delta^{k}_{\sharp} \arrow[rr,"\alpha \tensor \beta"] \arrow[dr,swap,"f"] && \Delta^s_{\flat} \tensor \Delta^{\ell}_\sharp \arrow[dl,"g"] \\
        & \bcat{C} &
     \end{tikzcd}
  \]
  where $\alpha \tensor \beta$ is a product of monotone maps. When the $\infty$-bicategory in question is clear from context we will use the simplified notation $\mathcal{S}_\tensor$.
  \end{definition}

  \begin{lemma}\label{lem:xxdescription}
    Let $\bcat{C}$ be an $\infty$-bicategory. We define a functor $T_{\bcat{C}}^{\tensor}:\mathcal{S}_\tensor \xlongrightarrow{} \on{Set}_\Delta^{\mathbf{ms}}$ sending an object $f:\Delta^n_\flat \tensor \Delta^k_\sharp$ to the marked-scaled simplicial set $T(f)$ whose underlying simplicial set is given by $\Delta^n \times \Delta^k$ and the decorations are induced from those of $\bcat{C}$ via the map $f$.  Then there exists an isomorphism of marked-scaled simplicial sets
    \[
      \colim_{\mathcal{S}_\tensor}T_{\bcat{C}}^\tensor \isom \mathbb{X}_{\bcat{C}}^\tensor.
    \]
  \end{lemma}
  \begin{proof}
    Let $\mathbb{T}_{\bcat{C}}^\tensor=\colim\limits_{\mathcal{S}_\tensor}T_{\bcat{C}}^\tensor$ and observe that this marked-scaled simplicial set comes equipped with a canonical map $p_{\bcat{C}}: \mathbb{T}_{\bcat{C}}^\tensor \xlongrightarrow{} \bcat{C}$. We will construct a map $\omega$ fitting into a commutative diagram
    \[\begin{tikzcd}
  {\mathbb{T}_{\bcat{C}}^\tensor} && {\mathbb{X}_{\bcat{C}}^{\tensor}} \\
  & {\bcat{C}}
  \arrow["\omega", from=1-1, to=1-3]
  \arrow["{p_{\bcat{C}}}"', from=1-1, to=2-2]
  \arrow["{\eta_{\bcat{C}}}", from=1-3, to=2-2]
\end{tikzcd}\]
   and show that $\omega$ is an isomorphism on the underlying simplicial sets. Since the decorations on both marked-scaled simplicial sets are induced via the vertical maps the result will follow. We give the construction of $\omega$ and leave the rest of the verifications to the reader. 

   If we define $\omega$ at the level of the underlying simplicial sets and show that $p_{\bcat{C}}= \eta_{\bcat{C}}\circ \omega$ then $\omega$ will automatically be compatible with the decorations. Now given $T(f)= \Delta^n_\flat \tensor \Delta^k_\sharp$ we observe that $f$ defines a $k$-simplex $\sigma_f:\Delta^k \xlongrightarrow{} \Fun^{\mathbf{opgr}}(\Delta^n_\flat,\bcat{C})^{\leq 1}$. We can therefore consider the map 
   \[
     \omega_f: \Delta^{n}_\flat \tensor \Delta^k_\sharp \xlongrightarrow{\on{id}\tensor \sigma_f} \Delta^n_\flat \tensor \Fun^{\mathbf{opgr}}(\Delta^n_\flat,\bcat{C})^{\leq 1} \xlongrightarrow{} \mathbb{X}_{\bcat{C}}^\tensor.
   \]
   This choices are compatible with the colimit defining $\mathbb{T}_{\bcat{C}}^\tensor$ and thus yield the desire map $\omega$ which clearly satisfies $p_{\bcat{C}}=\eta_{\bcat{C}}\circ \omega$.  
  \end{proof}

  \begin{theorem}\label{thm:sqefullyfaithful}
    The enhanced square functor $\Sqe: \bcat{B}\!\on{icat}_\infty \xlongrightarrow{} \on{EDbl}(\bcat{C}\!\on{at}_{(\infty,1)})$ is fully faithful.
  \end{theorem}
  \begin{proof}
    Let $\bcat{C},\bcat{D} \in \bcat{B}\!\on{icat}_\infty$ be a pair of $(\infty,2)$-categories. Combining \autoref{rem:etalc} and \autoref{prop:sqesaturation} we see that the map
    \[
      \Fun(\bcat{C},\bcat{D})^{\leq 1} \xlongrightarrow{} \on{Nat}_{\Delta^\op}(\Sqe(\bcat{C})_\bullet, \Sqe(\bcat{D})_\bullet) \isom \Fun(\mathbb{X}_{\bcat{C}}^{\tensor},\bcat{D})^{\leq 1}
    \]
    is induced by the map $\eta_{\bcat{C}}: \mathbb{X}_{\bcat{C}}^{\tensor} \xlongrightarrow{} \bcat{C}$. Therefore, we need to show that $\eta_{\bcat{C}}$ is an equivalence of marked-scaled simplicial sets. We observe that we have a section
    \[
      s_{\bcat{C}}^{\tensor}: \bcat{C} \xlongrightarrow{} \mathbb{X}_{\bcat{C}}^{\tensor}
    \]
    which sends a simplex $\sigma:\Delta^n_\flat \xlongrightarrow{} \bcat{C}$ to the factor in the colimit of \autoref{lem:xxdescription} corresponding to the map $\Delta^n_\flat \tensor \Delta^0 \xlongrightarrow{\sigma} \bcat{C}$. The map $s_{\bcat{C}}^{\tensor}$ is a trivial cofibration by \autoref{thm:bigcombinatorial} so we conclude by 2-out-of-3.
  \end{proof}

  \begin{corollary}\label{cor:ordinarysqe}
    The ordinary square functor $\mathbf{Sq}: \bcat{B}\!\on{icat}_\infty  \xlongrightarrow{} \on{Dbl}(\bcat{C}\!\on{at}_{(\infty,1)})$ is fully faithful. Moreover, the composite
    \[
        \bcat{B}\!\on{icat}_\infty  \xlongrightarrow{\mathbf{Sq}} \on{Dbl}(\bcat{C}\!\on{at}_{(\infty,1)})\xlongrightarrow{} \on{Dbl}\left(\bcat{C}\!\on{at}_{(\infty,1)}\right)^{\mathcal{I}\!\on{oplax}}
       \]   
     where the second functor is the obvious inclusion is also fully faithful.  
  \end{corollary}
  \begin{proof}
   Let $\widetilde{\mathbf{S}}\!\mathbf{q}$ be the composite in the statement. It follows from \autoref{prop:Sqeenhanced} that we have a commutative diagram
    \[
      \begin{tikzcd}
        \bcat{B}\!\on{icat}_\infty \arrow[r,"\widetilde{\mathbf{S}}\!\mathbf{q}"] \arrow[dr,swap,"\Sqe"] &  \on{Dbl}\left(\bcat{C}\!\on{at}_{(\infty,1)}\right)^{\mathcal{I}\!\on{oplax}} \arrow[d,"K_*"] \\
        & \on{EDbl}(\bcat{C}\!\on{at}_{(\infty,1)})
      \end{tikzcd}
    \]
    It then follows from \autoref{thm:enhanced} and \autoref{thm:sqefullyfaithful} that $\widetilde{\mathbf{S}}\!\mathbf{q}$ is fully faithful as well. 

    Let $\bcat{S}$ be the full subcategory of $\on{Dbl}(\bcat{C}\!\on{at}_{(\infty,1)})$ consisting in the essential image of $\mathbf{Sq}$. In order to show that $\mathbf{Sq}$ is fully faithful it will be enough to show that the induced functor $\bcat{S} \xlongrightarrow{}\on{Dbl}\left(\bcat{C}\!\on{at}_{(\infty,1)}\right)^{\mathcal{I}\!\on{oplax}} $ is fully faithful. In virtue of \autoref{lem:nuccharac} this amounts to showing that the image of the maps in $\bcat{S}$ under the functor $K_*$ are nuclear. This follows from \autoref{prop:sqrigid}.
  \end{proof}

 \section{The enhanced globular functor and the Gray tensor product}
 \begin{definition}
   Let $\scr{X}_\bullet, \scr{Y}_\bullet \in \CSeg$ be a pair of $(\infty,2)$-categories represented by complete Segal objects. We define an oplax normalised functor $f: \scr{X}_\bullet \xrightarrow{} \scr{Y}_\bullet$ to be the data of an $\mathcal{I}\!\on{oplax}$-natural transformation between the corresponding Segal objects. In particular, we can view the category $\CSeglax$ as the $(\infty,2)$-category whose objects are $(\infty,2)$-categories, its 1-morphisms are oplax normalised functors and its 2-morphisms are oplax natural transformations $F \xRightarrow{} G$ whose components at every $x \in \scr{X}_0$  are equivalences $F(x) \xlongrightarrow{\simeq} G(x)$. 
 \end{definition}

 \begin{lemma}\label{lem:bicatoplax}
   Let $(\bcat{C},E_{\bcat{C}},T_{\bcat{C}})$ and $(\bcat{A},E_{\bcat{A}},T_{\bcat{A}})$ be fibrant marked-scaled simplicial sets and recall the definition of the invertible 2-morphisms given in \autoref{def:invertible2morph}. We set $\overline{\bcat{C}}=(\bcat{C},E_{\bcat{C}},M_{\bcat{C}})$ and consider the fibrant marked-scaled simplicial set $\on{Fun}^{\mathbf{opgl}}(\overline{\bcat{C}},\bcat{A})$ (see \autoref{def:globfunctorcat}). Then there exists an isomorphism of marked-scaled simplicial sets
   \[
     \on{Fun}^{\mathbf{opgl}}(\overline{\bcat{C}},\bcat{A}) \xlongrightarrow{\isom} \on{Fun}^{\mathbf{opgl}}(\overline{\bcat{C}},\overline{\bcat{A}})
   \]
   where $\overline{\bcat{A}}=(\bcat{A},E_{\bcat{A}},M_{\bcat{A}})$.
 \end{lemma}
 \begin{proof}
   Let $K\in \on{Set}_\Delta^{\mathbf{ms}}$ and consider a map $\Psi:\overline{\bcat{C}} \odot K \xlongrightarrow{}\bcat{A}$. It will suffice to show that $\Psi$ factors through $\overline{\bcat{A}}$. Let $\sigma_1 : \Delta^2 \xlongrightarrow{} \overline{\bcat{C}}$ and $\sigma_2: \Delta^2 \xlongrightarrow{} K$ be a pair of maps defining a thin 2-simplex $\sigma$ in $\overline{\bcat{C}} \odot K$. Observe that by definition of the scaling in $\overline{\bcat{C}} \odot K$ \emph{both} $\sigma_1$ and $\sigma_2$ must be thin in their respective factors. In particular this implies that either $\Psi(\sigma(0\xlongrightarrow{} 1))$ or $\Psi(\sigma(1\xlongrightarrow{} 2))$ must be an equivalence in $\bcat{A}$ by the definition of the marking in the globular product. It follows that $\Psi(\sigma)$ is thin in $\overline{\bcat{A}}$.
 \end{proof}

 \begin{definition}\label{def:oplaxbicat}
   We define an $(\infty,2)$-category as the fibrant $\on{Set}_\Delta^+$-enriched category $\bcat{B}\!\on{icat}_\infty^{\mathbf{oplax},\odot}$ defined as follows:
   \begin{itemize}
     \item The objects of $\bcat{B}\!\on{icat}^{\mathbf{oplax},\odot}_{\infty}$ are given by fibrant marked-scaled simplicial sets, i.e. $\infty$-bicategories.
     \item Let $(\bcat{C},E_{\bcat{C}},T_{\bcat{C}})$ be an object of $\bcat{B}\!\on{icat}_\infty^{\mathbf{oplax},\odot}$ and recall the definition of the invertible 2-morphisms given in \autoref{def:invertible2morph}. We set $\overline{\bcat{C}}=(\bcat{C},E_{\bcat{C}},M_{\bcat{C}})$ and define the mapping $(\infty,1)$-category  between $\bcat{C}$ and $\bcat{A}$ to be underlying $(\infty,1)$-category of $\on{Fun}^{\mathbf{opgl}}(\overline{\bcat{C}},\bcat{A})$ (see \autoref{def:globfunctorcat}).
   \end{itemize}
   Note that the composition on mapping $(\infty,1)$-categories is well defined by \autoref{lem:bicatoplax}.
 \end{definition}

  The goal of this section is to define a functor $\EGlob: \bcat{B}\!\on{icat}^{\mathbf{oplax},\odot}_{\infty} \xlongrightarrow{} \ESeg$ and show that it yields an equivalence of $(\infty,2)$-categories. We will start by defining a fully faithful functor
  \[
    \overline{\EGlob}: \bcat{B}\!\on{icat}^{\mathbf{oplax},\odot}_{\infty}  \xlongrightarrow{} \on{Fun}(\Delta^\op,\bcat{C}\!\on{at}_{(\infty,1)})
  \]  
  and show that it factors through $\ESeg$. To improve readability, we will slightly abuse the notation and denote $\overline{\EGlob}$ by $\EGlob$. Given an $\infty$-bicategory $\bcat{C}$ the functor $\EGlob(\bcat{C})$ is defined as follows
 \[
    \EGlob(\bcat{C})_n= \on{Fun}^{\mathbf{opgl}}((\Delta^n,\flat,\flat),\bcat{C})^{\leq 1}
  \]
  where $(-)^{\leq 1}$ is the underlying $(\infty,1)$-category functor given in \autoref{def:leq}. 

  \begin{remark}\label{rem:n0isgroupoid}
    We observe that if $n=0$ and $\bcat{C}$ is an $\infty$-bicategory we have that $\EGlob(\bcat{C})_0$ is a space.
  \end{remark}

  \begin{lemma}\label{lem:Eglobfunct}
    Let $\Delta^n_\flat=(\Delta^n,\flat,\flat)$ and let $\bcat{C}$ be an $\infty$-bicategory. For every $n\geq 0$, we consider the canonical morphism,
    \[
     \Gamma^n: \Delta^n_\flat \odot \Fun^{\mathbf{opgl}}(\Delta^n_\flat,\bcat{C}) \xlongrightarrow{} \bcat{C}
    \]
    Then the following holds:
    \begin{enumerate}
       \item The map $\Gamma^n$ factors through $\overline{\bcat{C}}$.
       \item For every marked-scaled simplicial set $K$ and every $n\geq 0$ we have a morphism
    \[
      \Gamma^n_K: \Delta^n_\flat \odot \left(  \Fun^{\mathbf{opgl}}(\Delta^n_\flat,\bcat{C}) \times K \right) \xlongrightarrow{} \overline{\bcat{C}} \odot K
    \]
     \end{enumerate}  
  \end{lemma}
  \begin{proof}
    Follows after direct inspection.
  \end{proof}

  \begin{definition}\label{def:globcoend}
   We consider a functor (of ordinary categories) given by the coend formula
  \[
    \int^{\mathbf{opgl}}_{E} : \Fun(\Delta^\op,\on{Set}_\Delta^{+}) \xlongrightarrow{} \on{Set}_\Delta^{\mathbf{ms}}, \enspace X_\bullet \mapsto \int \Delta^\bullet_\flat \odot X_\bullet.
  \]
  which admits a right adjoint given by
  \[
     \int^{\mathbf{opgl}}_{E} : \Fun(\Delta^\op,\on{Set}_\Delta^{+}) \llra \on{Set}_\Delta^{\mathbf{ms}}: \EGlob.
  \]
  \end{definition}

  \begin{proposition}\label{prop:globQuillen}
    The adjunction given in \autoref{def:globcoend} 
    \[
     \int^{\mathbf{opgl}}_{E} : \Fun(\Delta^\op,\on{Set}_\Delta^{+}) \llra \on{Set}_\Delta^{\mathbf{ms}}: \EGlob
  \]
  is Quillen adjunction (see \autoref{rem:graymodels}) where we equipp the left-hand side with the injective model structure and the right-hand side with the model structure on marked-scaled simplicial sets given in \autoref{thm:markedscaledmodel}.
  \end{proposition}
  \begin{proof}
    This follows by essentially the same argument as in \autoref{rem:globcoend}.
  \end{proof}

\begin{definition}
  We   define a functor $\EGlob: \bcat{B}\!\on{icat}^{\mathbf{oplax},\odot}_{\infty} \xlongrightarrow{} \on{Fun}(\Delta^\op,\bcat{C}\!\on{at}_{(\infty,1)})$ of fibrant $\on{Set}_\Delta^+$-enriched categories as follows:
  \begin{itemize}
     \item The functor sends an $\infty$-bicategory $\bcat{C}$ to the functor given by $\EGlob(\bcat{C})_\bullet$ defined above. Note that by \autoref{prop:globQuillen} our functor yields a fibrant-cofibrant object in the injective model structure.
     \item Given a marked simplicial set $\mathscr{L}$ which we view as a marked-scaled simplicial set $(\mathscr{L},E_{\mathscr{L}},\sharp)$ and a morphism $\nu: \overline{\bcat{C}}\odot \mathscr{L} \xlongrightarrow{} \bcat{A}$ where $\bcat{A}$ is another $\infty$-bicategory we want to define a natural transformation $(\EGlob(\bcat{C}) \times \underline{\mathscr{L}})_\bullet \xRightarrow{} \EGlob(\bcat{A})_\bullet$. We can use the maps studied in \autoref{lem:Eglobfunct}
     \[
       \Delta^n_\flat \odot \left(  \Fun^{\mathbf{opgl}}(\Delta^n_\flat,\bcat{C}) \times \mathscr{L} \right) \xlongrightarrow{\Gamma^n_{\mathscr{L}}} \overline{\bcat{C}} \odot \mathscr{L} \xlongrightarrow{\nu} \bcat{A} 
     \]
     to define a map $\Fun^{\mathbf{opgl}}(\Delta^n_\flat,\bcat{C}) \times \mathscr{L} \xlongrightarrow{} \Fun^{\mathbf{opgl}}(\Delta^n_\flat,\bcat{A})$ which is natural in $\Delta$. We obtain the desired natural transformation after applying the underlying $(\infty,1)$-category functor $(-)^{\leq 1}$.
   \end{itemize} 
\end{definition}

\begin{proposition}\label{prop:eglobenhanced}
  The functor $\EGlob: \bcat{B}\!\on{icat}^{\mathbf{oplax},\odot}_{\infty} \xlongrightarrow{} \on{Fun}(\Delta^\op,\bcat{C}\!\on{at}_{(\infty,1)})$ factors through $\ESeg$.
\end{proposition}
\begin{proof}
  Observe that given an $(\infty,2)$-category it follows that the vertical fragment of $\Sqe(\bcat{C})_\bullet$ (see \autoref{def:verticalfragment}) is preciselly given by $\EGlob(\bcat{C})_\bullet$. The claim now follows from \autoref{prop:vertofendbl} and \autoref{prop:Sqeenhanced}.
\end{proof}

\begin{remark}\label{rem:gammaexplained}
  Let $\bcat{C},\bcat{A} \in \bcat{B}\!\on{icat}^{\mathbf{oplax},\odot}_{\infty}$ and let $\mathscr{L}$ be a marked simplicial set. Observe that we have a map on Hom-sets given by
  \[
    \Hom_{\on{Set}_\Delta^+}\left(\mathscr{L}, \Fun^{\mathbf{opgl}}(\overline{\bcat{C}},\bcat{A})^{\leq 1}\right) \xlongrightarrow{} \Hom_{\on{Set}_{\Delta}^+}\left(\mathscr{L}, \on{Nat}_{\Delta^\op}(\EGlob(\bcat{C})_\bullet, \EGlob(\bcat{A})_\bullet )\right)
  \]
  which is natural in $\mathscr{L}$ and corresponds to a map
  \[
    \Hom_{\on{Set}_\Delta^{\mathbf{ms}}}(\overline{\bcat{C}}\odot \mathscr{L}, \bcat{A}) \xlongrightarrow{} \Hom_{\on{Set}_\Delta^{\mathbf{ms}}}\left(\int_E^{\mathbf{opgl}}(\EGlob(\bcat{C}) \times \underline{\mathscr{L}})_\bullet, \bcat{A} \right).
  \]
  As a consequence of the Yoneda lemma we obtain a map $\int_E^{\mathbf{opgl}}(\EGlob(\bcat{C}) \times \underline{\mathscr{L}})_\bullet \xlongrightarrow{} \overline{\bcat{C}}\odot \mathscr{L}$ which is also natural in $\mathscr{L}$. To better describe the map let us remind the reader that (ordinary) coends can be expressed as colimits over the twisted arrow category (see \cite[Section 5.2.1]{HA} for a definition of $\on{Tw}(-)$) so we need to unravel the map
  \[
   \Gamma_\mathscr{L}: \colim\limits_{\on{Tw}(\Delta)}\left( \Delta^\bullet_\flat \odot (\Fun^{\mathbf{opgl}}(\Delta^\bullet_\flat,\bcat{C})^{\leq 1}\times \mathscr{L} ) \right) \xlongrightarrow{} \overline{\bcat{C}} \odot \mathscr{L}.
  \]
  This map is induced by the family of maps $\Gamma^f_{\mathscr{L}}$ for every $f:n \xlongrightarrow{} k$ described as follows
  \[
     \Delta^n_\flat \odot (\Fun^{\mathbf{opgl}}(\Delta^k_\flat,\bcat{C})^{\leq 1}\times \mathscr{L}) \xlongrightarrow{}  \Delta^n_\flat \odot (\Fun^{\mathbf{opgl}}(\Delta^n_\flat,\bcat{C})^{\leq 1}\times \mathscr{L}) \xlongrightarrow{\Gamma^n_{K}} \overline{\bcat{C}} \odot \mathscr{L}
   \] 
   where the first map is given by restriction along $f$ and the last map was given in \autoref{lem:Eglobfunct}.
\end{remark}

\begin{remark}\label{rem:globeta}
  We now perfom analogous analysis to that of \autoref{rem:etalc}. Let $\mathscr{L}$ be an $(\infty,1)$-category. Then we observe that $\EGlob$ induces a map
   \[
     \Fun\left(\mathscr{L},\Fun^{\mathbf{opgl}}(\overline{\bcat{C}},\bcat{D})^{\leq 1}\right)^{\simeq} \xlongrightarrow{} \Fun(\mathscr{L},\on{Nat}_{\Delta^\op}(\EGlob(\bcat{C})_\bullet,\EGlob(\bcat{D})_\bullet))^{\simeq}
   \]
   whose target is equivalent to
   \[
    \Fun(\mathscr{L},\on{Nat}_{\Delta^\op}(\EGlob(\bcat{C})_\bullet,\EGlob(\bcat{D})_\bullet))^{\simeq}  \simeq \on{Nat}_{\Delta^\op}((\EGlob(\bcat{C})\times \underline{\mathscr{L}})_\bullet, \EGlob(\bcat{D})_\bullet)^{\simeq} .
   \]
   Note that since we are working with the underlying $\infty$-groupoids it follows that we have an equivalence 
   \[
      \Fun\left(\mathscr{L},\Fun^{\mathbf{opgl}}(\overline{\bcat{C}},\bcat{D})^{\leq 1}\right)^{\simeq} \xlongrightarrow{\simeq}  \Fun(\overline{\bcat{C}} \odot \mathscr{L},\bcat{D})^{\simeq}.
   \]

   As a consequence of \autoref{prop:globQuillen} we conclude that the morphism above is induced by precomposition along $\Gamma_{\mathscr{L}}:\int_E^{\mathbf{opgl}}(\EGlob(\bcat{C}) \times \underline{\mathscr{L}})_\bullet \xlongrightarrow{} \overline{\bcat{C}}\odot \mathscr{L}$.
\end{remark}

\begin{definition}\label{def:xxglobular}
  Let $\bcat{C}$ be an $(\infty,2)$-category and $\mathscr{L}$ be an $(\infty,1)$-category. We define a marked-scaled simplicial set $\mathbb{Q}_{\bcat{C}}^{\mathscr{L}}$ whose underlying simplicial set is given by $\int_E^{\mathbf{opgl}}(\EGlob(\bcat{C}) \times \underline{\mathscr{L}})_\bullet$ and where an edge (resp. triangle) is declared to be marked (resp. thin) if its image under $\Gamma_{\mathscr{L}}$ is marked (resp. thin) in $\overline{\bcat{C}}\odot \mathscr{L}$. We observe that we have an isomorphism of marked-scaled simplicial sets $\mathbb{Q}_{\bcat{C}}^{\mathscr{L}} \isom \mathbb{Q}_{\bcat{C}}^{\Delta^0} \odot \mathscr{L}$.  If $\mathscr{L}=\Delta^0$ we will use the notation $\mathbb{X}_{\bcat{C}}^{\odot}$
\end{definition}

\begin{lemma}\label{lem:Glpreservesequiv}
  Let $\bcat{C},\bcat{A} \in \bcat{B}\!\on{icat}^{\mathbf{oplax},\odot}_{\infty}$ and consider a natural transformation $\Phi:(\EGlob(\bcat{C})\times \mathscr{L})_\bullet \xRightarrow{} \EGlob(\bcat{A})_\bullet$ where $\mathscr{L}$ marked simplicial set. Then we have the folowing
  \begin{itemize}
    \item[$\star)$] Given a morphism $e=(u,v): \Delta^1 \xlongrightarrow{} \EGlob(\bcat{C})_1 \times \mathscr{L}$ such that $u: \Delta^1 \xlongrightarrow{}  \EGlob(\bcat{C})_1$ is an equivalence between a pair of objects represented by equivalences in $\bcat{C}$ then it follows that $\Phi(e)$ is also an equivalence between a pair of objects represented by equivalences in $\bcat{A}$.
  \end{itemize}
\end{lemma}
\begin{proof}
We claim that it suffices to prove the statement in the specific situation where $u$ is degenerate on an object $f$ which is represented by an equivalence in $\bcat{C}$. Let us show the claim. First let us suppose that we know that $\Phi(e)$ is an equivalence in $\EGlob(\bcat{A})_1$. Then to show that $\Phi(e)$ is represented by a commutative square where all morphisms are equivalences it suffices to show that the source of the morphism $\Phi(e)$ is represented by an equivalence in $\bcat{A}$. This is clear since the rest will follow from 2-out-3. Now we note that $d_0(\Phi(e))=d_0(\Phi(e'))$ where $e'=(u',v)$ where $u'$ is degenerate on $d_0(u)$. It is immediate to see using 2-out-of-3 that to see that $\Phi(e)$ is an equivalence we can assume that $u$ is degenerate on $f$. 

To finish the proof we note that since $f$ is represented by an equivalence in $\bcat{C}$ it follows that there exists an object $x \in \EGlob(\bcat{C})_0$ such that $s(x) \isom f$, where $s(x)$ is the image of $x$ under the obvious degeneracy operator. Let $\alpha: s(x) \xlongrightarrow{} f$ be an equivalence and consider the map
 \[
   \eta: \Delta^1 \times \Delta^1 \xlongrightarrow{\alpha \times v:} \EGlob(\bcat{C})_1 \times \mathscr{L} \xlongrightarrow{} \EGlob(\bcat{A})_1.
 \]
 We see again by 2-out-of-3 that in order to show that the restriction of $\eta$ to $(1,0) \xlongrightarrow{} (1,1)$ is an equivalence in $\EGlob(\bcat{A})_1$ is its enough to show that the restriction of $\eta$ to $(0,0) \xlongrightarrow{} (0,1)$ maps to an equivalence. Let us denote $\varphi: \Delta^1 \xlongrightarrow{}  \EGlob(\bcat{C})_1 \times \mathscr{L}$ the map given precomposing $\alpha \times v$ with $(0,0) \xlongrightarrow{} (0,1)$. Using naturality of $\Phi$ we get a commutative diagram
 \[
   \begin{tikzcd}
     \EGlob(\bcat{C})_0 \times \mathscr{L} \arrow[r] \arrow[d,"s \times \on{id}"] & \EGlob(\bcat{A})_0 \arrow[d,"s"] \\
     \EGlob(\bcat{C})_1 \times \mathscr{L} \arrow[r] & \EGlob(\bcat{A})_1.
   \end{tikzcd}
 \]
 It is clear that $\varphi$ is in the image of the left-most vertical map. We know that $\EGlob(\bcat{A})_0$ is a space and that the image of the right-most vertical morphism consists only in equivalences represented by commutative squares $\Delta^1 \times \Delta^1 \xlongrightarrow{} \bcat{A}$ where all the morphisms involved are equivalences. Therefore the same holds for the image of $\varphi$ in $\EGlob(\bcat{A})$. 
\end{proof}

\begin{remark}\label{rem:funpreservinv2morph}
  Let $\Phi: \EGlob(\bcat{C})_\bullet \xRightarrow{} \EGlob(\bcat{A})_\bullet$.  Let $p$ be an object in $\EGlob(\bcat{C})_2$. Then if $s$ is represented by an invertible 2-morphism it follows that $\Phi(p)$ is also represented by an invertible morphism since $p$ is on the essential image of some degeneracy operator.
\end{remark}

\begin{lemma}\label{lem:Glpreserglobs}
  Let $\bcat{C},\bcat{A} \in \bcat{B}\!\on{icat}^{\mathbf{oplax},\odot}_{\infty}$ and consider a natural transformation $\Phi:(\EGlob(\bcat{C})\times \mathscr{L})_\bullet \xRightarrow{} \EGlob(\bcat{A})_\bullet$ where $\mathscr{L}$ is a marked simplicial set. Suppose we are given a 2-simplex
  \[
    \sigma=(\omega,\rho):\Delta^2 \xlongrightarrow{} \EGlob(\bcat{C})_2 \times \mathscr{L}
  \]
  where $\omega$ is determined by a diagram $f_\omega : \Delta^2_{\flat} \odot (\Delta^2,\flat,\sharp) \xlongrightarrow{} \bcat{C}$ and a map $\tau: \Delta^2 \xlongrightarrow{} \Delta^2_{\flat} \odot (\Delta^2,\flat,\sharp)$ such that its projection to $(\Delta^2,\flat,\sharp)$ is given by $0 \xlongrightarrow{} 1 \xlongrightarrow{} 2$ and such that $f_\omega \circ \tau$ defines an invertible 2-morphism in $\bcat{C}$ (see \autoref{def:invertible2morph}). Finally let us assume that the at least one of the following conditions below hold:
  \begin{itemize}
    \item[i)] The edge $(f_\omega \circ \tau )(1 \xlongrightarrow{} 2)$ is invertible in $\bcat{C}$. 
    \item[ii)] The edge $\rho(0 \xlongrightarrow{} 1)$ is marked in $\mathscr{L}$.
  \end{itemize}
  Then denoting $\Phi(\sigma)=(\Phi(\omega),\Phi(\rho))$ and similarly $f_{\Phi(\omega)}$ we have that $f_{\Phi(\omega)}\circ \tau$ is thin in $\bcat{A}$.
\end{lemma}
\begin{proof}
  We fix the notation for the composite $\tau_\flat:\Delta^2 \xlongrightarrow{} \Delta^2_{\flat} \odot (\Delta^2,\flat,\sharp) \xlongrightarrow{}  \Delta^2_\flat$. Let us suppose that $\tau_\flat(1)=\tau_\flat(2)$ and consider the 3-simplex
  \[
     (\tau_\flat(0),0) \xlongrightarrow{} (\tau_\flat(1),0) \xlongrightarrow{} (\tau_\flat(1),1) \xlongrightarrow{} (\tau_\flat(2),2).
   \] 
   Then every triangle in this 3-simplex is scaled except possibly the face skipping the vertex $(\tau_\flat(1),0)$. Therefore the claim holds by noting that we can scale the remaining face using a morphism of type \ref{ms:wonky4} in \autoref{def:msanodyne}. From now on we shall assume that $\tau_\flat(1) < \tau_\flat(2)$.

  Under this further assumption we can look at the 3-simplex
  \[
    (\tau_\flat(0),0) \xlongrightarrow{} (\tau_\flat(1),1) \xlongrightarrow{} (\tau_\flat(2),1) \xlongrightarrow{} (\tau_\flat(2),2)
  \]
  and note that the face skipping the first vertex is by definition scaled. One can use an analogous argument as the one given above to show that the face skipping the vertex $(\tau_\flat(1),1)$ can be scaled using a pushout along a $\mathbf{MS}$-anodyne morphism. By construction the final edge of this 3-simplex is marked. Consequently we learn that the 2-simplex $f_{\Phi(\omega)}\circ \tau$ is scaled in $\bcat{A}$ if and only if $f_{\Phi{\omega}} \circ \nu$ is scaled in $\bcat{A}$ where $\nu$ is given by $ (\tau_\flat(0),0) \xlongrightarrow{} (\tau_\flat(1),1) \xlongrightarrow{} (\tau_\flat(2),1)$. Let us point out that to check that $f_{\Phi(\omega)} \circ \nu$ is thin we can simply work with $d_2(\sigma)$.

  \begin{itemize}
    \item[i)] The edge $f_\omega \circ \tau (1 \xlongrightarrow{} 2)$ is invertible in $\bcat{C}$. 

    In this case it follows that the image of the  edge $e:(\tau_\flat(2),1) \xlongrightarrow{} (\tau_\flat(1),2)$ under $f_\omega \circ \tau$ is marked in $\bcat{C}$. It then follows from \autoref{lem:Glpreservesequiv} that $f_{\Phi(\omega)} \circ e$ is marked in $\bcat{A}$. If $\tau_\flat (0) = \tau_\flat (1)$ then the claim follows from \autoref{lem:Glpreservesequiv}. We will therefore assume that $\tau_\flat$ is precisely the triangle $0 \xlongrightarrow{} 1 \xlongrightarrow{} 2$.

    Let $\overline{\omega}=d_2(\omega)$. We consider a pair of edges
    \[
      \varphi_i: \Delta^1 \xlongrightarrow{} \EGlob(\bcat{C})_3 \times \mathscr{L}
    \]
    such that $\varphi_i=(\overline{\omega}_i,d_2(\rho))$ with $i=1,2$, where $\overline{\omega}_i$ is defined as the composite
    \[
      f_{\overline{\omega}_i}:\Delta^3_{\flat} \odot \Delta^1 \xlongrightarrow{E_i} \Delta^2_\flat \odot \Delta^1 \xlongrightarrow{f_{\overline{\omega}}} \bcat{C}
    \]
    and the maps $E_j$ are described by the formula
    \[
      E_j(s,t)=\begin{cases}
        (s,t), \enspace \text{ if }s\leq j \\
        (s-1,1), \enspace \text{ if } s>j.
      \end{cases}
    \]
    with $0\leq j \leq 2$. We make the following observations:
    \begin{enumerate}
      \item The composite map $\Delta^1 \xlongrightarrow{\varphi_2}  \EGlob(\bcat{C})_3 \times \mathscr{L} \xlongrightarrow{d_3 \times \on{id}}  \EGlob(\bcat{C})_2 \times \mathscr{L}$ is precisely $d_2(\sigma)$.
      \item The composites $\Delta^1 \xlongrightarrow{\varphi_i}  \EGlob(\bcat{C})_3 \times \mathscr{L} \xlongrightarrow{d_2 \times \on{id}}  \EGlob(\bcat{C})_2 \times \mathscr{L}$ coincide for $i=1,2$.
    \end{enumerate}
  Let us look at $\Phi(\varphi_1)$ which is represented by a map
  \[
    f_{\Phi(\overline{\omega}_1)}: \Delta^3_{\flat} \odot \Delta^1 \xlongrightarrow{} \bcat{A}.
  \]
  We now observe the $3$-simplex in $\Delta^3_{\flat} \odot \Delta^1$ given by
  \[
    (0,0) \xlongrightarrow{} (2,0) \xlongrightarrow{} (3,0) \xlongrightarrow{} (3,1).
  \]
  It follows by definition that the faces skipping the vertex $(0,0)$ and the face skipping the vertex $(2,0)$ are thin in $\bcat{A}$. It follows from \autoref{rem:funpreservinv2morph} that the face skipping the vertex $(3,1)$ is also thin. We conclude using a morphism of type \ref{ms:wonky4} in \autoref{def:msanodyne} that the triangle $(0,0) \xlongrightarrow{} (2,0) \xlongrightarrow{} (3,1)$ is also thin in $\bcat{A}$. Using a similar argument applied to the $3$-simplex
  \[
     (0,0) \xlongrightarrow{} (2,0) \xlongrightarrow{} (2,1) \xlongrightarrow{} (3,1)
  \]
  shows that the triangle $(0,0) \xlongrightarrow{} (2,0) \xlongrightarrow{} (3,1)$ is thin. Note that in order to see that the face skipping the vertex $0$ is thin we need to observe that as a consequence of our construction and \autoref{lem:Glpreservesequiv} it follows that the restriction along the map
  \[
    \Delta^1 \odot \Delta^1 \xlongrightarrow{\gamma \times \on{id}} \Delta^3_{\flat} \odot \Delta^1
  \]
  where $\gamma$ is either the edge $1 \xlongrightarrow{} 2$ or $2 \xlongrightarrow{} 3$ defines a commutative square in $\bcat{A}$ where all edges are equivalences. Using analogous arguments we show that the triangles
  \[
    (0,0) \xlongrightarrow{} (1,0) \xlongrightarrow{} (2,1), \enspace \enspace (0,0) \xlongrightarrow{} (1,1) \xlongrightarrow{} (2,1) 
  \]
  are thin, which finally allows us to show that the triangle $(0,0) \xlongrightarrow{} (1,1) \xlongrightarrow{} (3,1)$ is also thin. Using observation 2 above we learn that the image of the triangle $(0,0) \xlongrightarrow{} (1,1) \xlongrightarrow{} (3,1)$ under the map
 \[
    f_{\Phi(\overline{\omega}_2)}: \Delta^3_{\flat} \odot \Delta^1 \xlongrightarrow{} \bcat{A}.
  \]
  is also thin. From now on, we will focus our attention into $ f_{\Phi(\overline{\omega}_2)}$. We look at the 3-simplex
  \[
    (0,0) \xlongrightarrow{} (1,1) \xlongrightarrow{} (2,1) \xlongrightarrow{} (3,1)
  \]
  and note that due to observation 1, our result will follow if we show that the face skipping the last vertex is thin in $\bcat{A}$. By construction the edge $(2,1) \xlongrightarrow{} (3,1)$ gets mapped to an identity so it will suffice to show that every remaining face gets mapped to a thin simplex in $\bcat{A}$. By our previous discussion and by \autoref{rem:funpreservinv2morph}, we only need to verify this for the face skipping the vertex $(1,1)$. This follows using a totally analogous argument as before.

  \item[ii)] The edge $\rho(0\xlongrightarrow{} 1)$ is an equivalence in $\mathscr{L}$.
  
  Note that since by assumption $f_\omega \circ \tau$ defines an invertible 2-morphism in $\bcat{C}$ we might assume without loss of generality that the edge $\hat{e}=(f_\omega \circ \tau)(0 \xlongrightarrow{} 1)$ is an equivalence in $\bcat{C}$ since otherwise we can apply the argument above.

  Using the notations from the previous case we consider $\overline{\omega}=d_2(\omega)$ and the edges $\varphi_i=(\overline{\omega}_i,d_2(\rho))$ for $ i =1,2 $. We observe a similar relation as we did before, namely:
  \[
    f_{\Phi(\overline{\omega}_1)}: \Delta^3_{\flat} \odot \Delta^1 \xlongrightarrow{} \bcat{A}.
  \]
  and consider the 3-simplex $(0,0) \xlongrightarrow{} (2,0) \xlongrightarrow{} (3,0) \xlongrightarrow{} (3,1)$. Note that since $\hat{e}$ is an equivalence in $\bcat{C}$ it follows that the image of the edge $(0,0) \xlongrightarrow{} (2,0)$ under $f_{\Phi(\overline{\omega}_0)}$ is an equivalence in $\bcat{A}$. Using \autoref{rem:funpreservinv2morph} it follows that the face skipping the vertex $(3,1)$ gets mapped to a thin 2-simplex in $\bcat{A}$. One checks that every face in this 3-simplex is thin in $\bcat{A}$ except possibly the face skipping th evertex $(3,0)$ which implies that every face in this 3-simplex must be thin in $\bcat{A}$. The proof proceeds exactly in the same way as above. The only thin to remark is that the triangle $(2,0) \xlongrightarrow{} (2,1) \xlongrightarrow{} (3,1)$ gets mapped under $f_{\Phi(\overline{\omega}_1)}$ to a thin triangle in $\bcat{A}$ since this same triangle is sent to a thin triangle by $f_{\overline{\omega}_1}$ and the edge $l_0 \xlongrightarrow{} l_1$ is marked in $\mathscr{L}$. \qedhere
\end{itemize}
\end{proof}

\begin{proposition}\label{prop:eglobsat}
  Let $\bcat{C}$ be an $(\infty,2)$-category and let $\mathscr{L}$ be an $(\infty,1)$-category. Then canonical morphism of marked-scaled simplicial sets (see \autoref{def:xxglobular})
  \[
    \int^{\mathbf{opgl}}_E (\EGlob(\bcat{C}) \times \underline{\mathscr{L}})_\bullet \xlongrightarrow{\simeq} \mathbb{X}_{\bcat{C}}^{\odot} \odot \mathscr{L}
  \]
  is a weak equivalence of marked-scaled simplicial sets.
\end{proposition}
\begin{proof}
  The proof is essentially the same as the proof of \autoref{prop:sqesaturation} where we now use \autoref{lem:Glpreservesequiv} for the case of the marking and \autoref{lem:Glpreserglobs} for the case of the scaling.
\end{proof}

\begin{definition}
  Let $\bcat{C}$ be an $\infty$-bicategory. We $\mathcal{S}_{\bcat{C},\odot}$ to be the (ordinary) category whose objects are given by maps $f:\Delta^n_\flat \odot \Delta^{k}_\sharp \xlongrightarrow{} \overline{\bcat{C}}$ and whose morphisms are given by commutative diagrams
  \[
     \begin{tikzcd}
        \Delta^{n}_\flat \odot \Delta^{k}_{\sharp} \arrow[rr,"\alpha \odot \beta"] \arrow[dr,swap,"f"] && \Delta^s_{\flat} \odot \Delta^{\ell}_\sharp \arrow[dl,"g"] \\
        & \overline{\bcat{C}} &
     \end{tikzcd}
  \]
  where $\alpha \odot \beta$ is a product of monotone maps. When the $\infty$-bicategory in question is clear from context we will use the simplified notation $\mathcal{S}_\odot$.
\end{definition}

\begin{lemma}\label{lem:xxdescriptionglob}
    Let $\bcat{C}$ be an $\infty$-bicategory. We define a functor $T_{\bcat{C}}^{\odot}:\mathcal{S}_\odot \xlongrightarrow{} \on{Set}_\Delta^{\mathbf{ms}}$ sending an object $f:\Delta^n_\flat \tensor \Delta^k_\sharp$ to the marked-scaled simplicial set $T(f)$ whose underlying simplicial set is given by $\Delta^n \times \Delta^k$ and the decorations are induced from those of $\overline{\bcat{C}}$ via the map $f$.  Then there exists an isomorphism of marked-scaled simplicial sets
    \[
      \colim_{\mathcal{S}_\tensor}T_{\bcat{C}}^\odot \isom \mathbb{X}_{\bcat{C}}^\odot.
    \]
  \end{lemma}
  \begin{proof}
   The argument is mutatis mutandis that of \autoref{lem:xxdescription}.
  \end{proof}

 \begin{theorem}\label{thm:maineglob}
   The functor $\EGlob: \bcat{B}\!\on{icat}^{\mathbf{oplax},\odot}_\infty \xlongrightarrow{} \ESeg$ is an equivalence of $(\infty,2)$-categories.
 \end{theorem}
 \begin{proof}
   The proof of fully faithfulness proceeds in the same way as in \autoref{thm:sqefullyfaithful} by showing that the canonical map $\mathbb{X}_{\bcat{C}}^\odot \xlongrightarrow{\simeq} \overline{\bcat{C}}$ is a weak equivalence. This follows from \autoref{thm:bigcombinatorial} after considering the obvious section $s_{\bcat{C}}^\odot: \overline{\bcat{C}} \xlongrightarrow{} \mathbb{X}_{\bcat{C}}^\odot$.

   We finish the proof by showing that $\EGlob$ is essentially surjective. Let $\bcat{X}_\bullet \in \ESeg$. Then it follows from \autoref{thm:enhanced} that $\bcat{X}_\bullet \simeq K_*(\scr{X})_\bullet$ where $\scr{X}_\bullet$ is the nucleus (see \autoref{def:segalnucleus}) of $\bcat{X}_\bullet$. We use \autoref{thm:compari} to obtain an $(\infty,2)$-category and an equivalence $\mathbf{Gl}(\bcat{C})_\bullet \isom \scr{X}_\bullet$. We now observe that $\mathbf{Gl}(\bcat{C})_\bullet  \isom \mathbf{Sq}(\bcat{C})_\bullet^{\mathbf{v}}$. We conclude that
   \[
    \bcat{X}_\bullet \isom  K_*(\mathbf{Gl}(\bcat{C}))_\bullet \isom K_*(\mathbf{Sq}(\bcat{C}))^{\mathbf{v}}_* \isom \Sqe(\bcat{C})_\bullet^{\mathbf{v}} \isom \EGlob(\bcat{C})_\bullet
   \]
   where we are using the for every double $(\infty,1)$-category $\scr{Y}_\bullet$ it follows that $K_*(\scr{Y}^{\mathbf{v}})_\bullet \isom K_*(\scr{Y})_\bullet^{\mathbf{v}}$ as we saw in the proof of \autoref{lem:verticalpreservesoplax}.
 \end{proof}

 \begin{corollary}\label{thm:rearranged}
   There exists an equivalence of $(\infty,2)$-categories
   \[
     \bcat{B}\!\on{icat}_\infty^{\mathbf{opgl},\odot} \xlongrightarrow{\simeq}\on{CSeg}_\Delta(\bcat{C}\!\on{at}_{(\infty,1)})^{\mathcal{I}\!\on{oplax}}
   \]
 \end{corollary}
 \begin{proof}
   Combine \autoref{thm:maineglob} and \autoref{thm:enhanced}
 \end{proof}

 \begin{corollary}\label{cor:eglobnuc}
   Let $\bcat{B}\!\on{icat}^{\odot}_\infty$ be the $(\infty,2)$-category obtained from $\bcat{B}\!\on{icat}^{\mathbf{oplax},\odot}_\infty$ by considering only the non-lax functors. Then the funtor $\EGlob$ induces commutative diagram of $(\infty,2)$-categories
   \[
     \begin{tikzcd}
       \bcat{B}\!\on{icat}^{\odot}_\infty \arrow[d] \arrow[r,"\simeq"] & \on{ESeg}_\Delta^{\on{nuc}}\left(\bcat{C}\!\on{at}_{(\infty,1)}\right) \arrow[d] \\
       \bcat{B}\!\on{icat}^{\mathbf{oplax},\odot}_\infty \arrow[r,"\simeq"] & \ESeg
     \end{tikzcd}
   \]
   where the horizontal morphisms are equivalences of $(\infty,2)$-categories.
 \end{corollary}
 \begin{proof}
   Let $\EGlob^{\on{nuc}}$ denote the top horizontal functor. It is clear that a non-lax functor induces a nuclear morphism in the corresponding enhanced Segal objects. We know that this functor is essentially surjective by \autoref{thm:maineglob} and that its action on mapping $(\infty,1)$-categories is fully faithful. Therefore, we only need to show that it is essentially surjective on mapping $(\infty,1)$-categories. It immediate to verify that given an oplax functor $f:\overline{\bcat{C}} \xlongrightarrow{} \bcat{D}$ then it follows $F$ is non-lax precisely $\EGlob(f)$ is nuclear. The result now follows.
 \end{proof}

 \begin{definition}\label{def:seggray}[\cite{GaitsgoryRozenblyum} 0.34]
   Let $\scr{X}_\bullet$ and $\scr{Y}_\bullet$ be a pair of (complete) Segal objects in $\bcat{C}\!\on{at}_{(\infty,1)}$. We define the \emph{Gray tensor product} of $\scr{X}_\bullet$ and $\scr{Y}_\bullet$ denoted by $(\scr{X}\tensor \scr{Y})_\bullet \in \CSeg$ by means of the universal property
   \[
     \on{Nat}_{\Delta^\op}((\scr{X}\tensor \scr{Y})_\bullet, \scr{Z}_\bullet) \isom \on{Nat}_{\Delta^\op}(\scr{X}_\bullet \times \scr{Y}_\bullet, \scr{Z}_\bullet)^{\mathbf{gr}} \subset \on{Nat}_{\Delta^\op}(\scr{X}_\bullet \times \scr{Y}_\bullet, \scr{Z}_\bullet)^{\mathcal{I}\!\on{oplax}}
   \]
   where $\scr{Z}_\bullet \in \CSeg$ and $\on{Nat}_{\Delta^\op}(\scr{X}_\bullet \times \scr{Y}_\bullet, \scr{Z}_\bullet)^{\mathbf{gr}}$ is the full subcategory of $\on{Nat}_{\Delta^\op}(\scr{X}_\bullet \times \scr{Y}_\bullet, \scr{Z}_\bullet)^{\mathcal{I}\!\on{oplax}}$ consisting in those oplax functors $F$ which satisfy the following conditions below:
   \begin{itemize}
      \item For every object $x \in \scr{X}_0$ the restriction to $\{x\} \times \scr{Y}_\bullet \xlongrightarrow{} \scr{Z}_\bullet$ defines a functor and similarly for every object in $y \in \scr{Y}_0$.
      \item For every pair of objects $f \in \scr{X}_1$ and $g \in \scr{Y}_1$ let us consider an object $\sigma \in \scr{X}_2 \times \scr{Y}_2$ whose component in $\scr{X}$ is given by the composite of $f$ with the identity on the object $d_1(f)$ and whose component in $\scr{Y}_2$ is the composite of the identity on $d_0(g)$ with $g$. Then the morphism induced by $F$
      \[
        F_2(\sigma) \xlongrightarrow{}  F_1 \circ\left(\scr{X}(d_1) \times \scr{Y}(d_1)\right)(\sigma)
      \]
      is an equivalence in $\scr{Z}_1$.
    \end{itemize} 
 \end{definition}

 \begin{definition}\label{def:Eseggray}
    Let $\bcat{X}_\bullet$ and $\bcat{Y}_\bullet$ be a pair of enhanced Segal objects. We define the \emph{Gray tensor product} of $\bcat{X}_\bullet$ and $\bcat{Y}_\bullet$ denoted by $(\bcat{X}\tensor \bcat{Y})_\bullet \in \ESeg$ by means of the universal property 
    \[
     \on{Nat}_{\Delta^\op}((\bcat{X}\tensor \bcat{Y})_\bullet, \bcat{Z}_\bullet)^{\on{nuc}} \isom \on{Nat}_{\Delta^\op}(\bcat{X}_\bullet \times \bcat{Y}_\bullet, \bcat{Z}_\bullet)^{\mathbf{Egr}} \subset \on{Nat}_{\Delta^\op}(\bcat{X}_\bullet \times \bcat{Y}_\bullet, \bcat{Z}_\bullet)
   \]
   where $\bcat{Z}_\bullet \in \ESeg$ and $\on{Nat}_{\Delta^\op}(\bcat{X}_\bullet \times \bcat{Y}_\bullet, \bcat{Z}_\bullet)^{\mathbf{Egr}}$ is the full subcategory of $\on{Nat}_{\Delta^\op}(\bcat{X}_\bullet \times \bcat{Y}_\bullet, \bcat{Z}_\bullet)$ consisting in those oplax functors $F$ which satisfy the following conditions below:
   \begin{itemize}
      \item For every object $x \in \bcat{X}_0$ the restriction to $\{x\} \times \bcat{Y}_\bullet \xlongrightarrow{} \bcat{Z}_\bullet$ is nuclear and similarly for every object in $y \in \bcat{Y}_0$.
      \item For every pair of objects $f \in \bcat{X}_1$ and $g \in \bcat{Y}_1$ let us consider an object $\sigma \in \bcat{X}_2 \times \bcat{Y}_2$ whose component in $\bcat{X}$ is given by the composite of $f$ with the identity on the object $d_1(f)$ and whose component in $\bcat{Y}_2$ is the composite of the identity on $d_0(g)$ with $g$. Then the image $F(\sigma)$ belongs to the nucleus of $\bcat{Z}_\bullet$.
    \end{itemize} 
 \end{definition}

\begin{definition}
  Let $\bcat{C}, \bcat{D} \in \bcat{B}\!\on{icat}_\infty^{\mathbf{oplax},\odot}$ for every $(\infty,2)$-category $\bcat{A}$ we define 
  \[
     \Fun^{\mathbf{opgl}}\left(\overline{\bcat{C}\times \bcat{D}},\bcat{A}\right)^{\mathbf{gr},\leq 1} \xlongrightarrow{} \Fun^{\mathbf{opgl}}\left(\overline{\bcat{C}\times \bcat{D}},\bcat{A}\right)^{\leq 1}
   \] as the full subcategory on those oplax functors $F$ sending the collection of thin triangles in $\bcat{C}\tensor \bcat{D}$ to thin triangles in $\bcat{A}$.
\end{definition}

\begin{lemma}\label{lem:goodgrayuniv}
  For every triple of $(\infty,2)$-categories $\bcat{C},\bcat{D}$ and $\bcat{A}$ in $\bcat{B}\!\on{icat}_\infty^{\mathbf{oplax},\odot}$ there exists an equivalence of $(\infty,1)$-categories
  \[
     \Fun^{\mathbf{opgl}}\left(\overline{\bcat{C}\times \bcat{D}},\bcat{A}\right)^{\mathbf{gr},\leq 1} \xlongrightarrow{\simeq}  \Fun^{\mathbf{opgl}}\left(\bcat{C} \tensor \bcat{D},\bcat{A}\right)^{\leq 1}.
  \]
\end{lemma}
\begin{proof}
  We will proof that the map in question is an isomorphism. Since $\overline{\bcat{C}\times \bcat{D}}$ and $\bcat{C}\tensor \bcat{D}$ only differ in the scaling it will enough to show that given a map $F: \left(\overline{\bcat{C}\times \bcat{D}}\right) \odot \Delta^2_\sharp \xlongrightarrow{} \bcat{A}$ whose restriction to every $i \in \Delta^2_{\sharp}$ yields a map $F_i: \bcat{C} \tensor \bcat{D} \xlongrightarrow{} \bcat{A}$ then $F$ sends every thing triangle in $(\bcat{C}\tensor \bcat{D})\odot \Delta^2_\sharp$ to a thin triangle in $\bcat{A}$.

  Let $\rho: \Delta^2_\sharp \xlongrightarrow{} (\bcat{C}\tensor \bcat{D})\odot \Delta^2_\sharp$. We consider two cases:
  \begin{enumerate}
    \item The edge $\rho(1 \xlongrightarrow{} 2)$ is marked in $\bcat{C}\tensor \bcat{D}$. 

    Then it follows that $\rho$ defines a thin simplex in $\left(\overline{\bcat{C}\times \bcat{D}}\right) \odot \Delta^2_\sharp$ and the claim follows.
    \item The edge $\rho(0 \xlongrightarrow{} 1)$ is degenerate.

    We consider a 3-simplex which we depict as 
    \[
       \theta: (c_0,d_0,i) \xlongrightarrow{} (c_1,d_1,i) \xlongrightarrow{} (c_2,d_2,i) \xlongrightarrow{}(c_2,d_2,j)
     \] 
     such that $d_2(\theta)=\rho$. It follows that $F$ sends every face of this 3-simplex, (except possible the face skipping the vertex 2) to a thin simplex in $\bcat{A}$ and thus the claim holds. \qedhere
  \end{enumerate}
\end{proof}

\begin{theorem}\label{thm:maingray}
 The different versions of the Gray tensor product given in \autoref{thm:univgray}, \autoref{def:seggray} and $\autoref{def:Eseggray}$ are equivalent. More precisely, let $\bcat{C}\tensor_{\on{gr}}\bcat{D}$ denote a fibrant replacement of the Gray tensor product $\bcat{C}\tensor \bcat{D}$ (see \autoref{def:gray}) where $\bcat{C}$ and $\bcat{D}$ are fibrant marked-scaled simplicial sets. Then we have equivalences
 \[
   \EGlob(\bcat{C}\tensor_{\on{gr}}\bcat{D})_\bullet \simeq \left(\EGlob(\bcat{C})\tensor \EGlob(\bcat{D})\right)_\bullet, \enspace \enspace K_*((\mathbf{Gl}(\bcat{C})\tensor \mathbf{Gl}(\bcat{D}))_\bullet \isom (\EGlob(\bcat{C})\tensor \EGlob(\bcat{D}))_\bullet.
 \]
 which are natural in $\bcat{C}$ and $\bcat{D}$.
\end{theorem}
\begin{proof}
  Let $\bcat{A}$ be a fibrant marked-scaled simplicial set. In virtue of \autoref{thm:enhanced} and \autoref{thm:maineglob} it follows that the functors $K_*$ and $\EGlob$ induce commutative diagrams of $(\infty,1)$-categories

  \[
    \begin{tikzcd}
      \on{Nat}_{\Delta^\op}((\mathbf{Gl}(\bcat{C})\times \mathbf{Gl}(\bcat{D}))_\bullet,\mathbf{Gl}(\bcat{A})_\bullet)^{\mathbf{gr}} \arrow[r,"\simeq"] \arrow[d] & \on{Nat}_{\Delta^\op}((\EGlob(\bcat{C})\times \EGlob(\bcat{D}))_\bullet,\EGlob(\bcat{A})_\bullet)^{\mathbf{Egr}}  \arrow[d] \\
      \on{Nat}_{\Delta^\op}((\mathbf{Gl}(\bcat{C})\times \mathbf{Gl}(\bcat{D}))_\bullet,\mathbf{Gl}(\bcat{A})_\bullet)^{\mathcal{I}\!\on{oplax}} \arrow[r,"\simeq"] &  \on{Nat}_{\Delta^\op}((\EGlob(\bcat{C})\times \EGlob(\bcat{D}))_\bullet,\EGlob(\bcat{A})_\bullet)
    \end{tikzcd}
  \]
  \[
    \begin{tikzcd}
       \Fun^{\mathbf{opgl}}\left(\overline{\bcat{C}\times \bcat{D}},\bcat{A}\right)^{\mathbf{gr},\leq 1} \arrow[r,"\simeq"] \arrow[d] &  \on{Nat}_{\Delta^\op}((\EGlob(\bcat{C})\times \EGlob(\bcat{D}))_\bullet,\EGlob(\bcat{A})_\bullet)^{\mathbf{Egr}} \arrow[d] \\
    \Fun^{\mathbf{opgl}}(\overline{\bcat{C} \times \bcat{D}},\bcat{A})^{\leq 1} \arrow[r,"\simeq "] &  \on{Nat}_{\Delta^\op}((\EGlob(\bcat{C})\times \EGlob(\bcat{D}))_\bullet,\EGlob(\bcat{A})_\bullet)
    \end{tikzcd}
  \]
 where the horizontal morphisms are fully faithful functors.

We look at the correpresentable functor on $(\EGlob(\bcat{C})\tensor \EGlob(\bcat{D}))_\bullet$ which we denote as
\[
   \scr{G}^{\on{E}}_{\bcat{C},\bcat{D}}: \on{ESeg}_\Delta^{\on{nuc}}\left(\bcat{C}\!\on{at}_{(\infty,1)}\right) \xlongrightarrow{} \bcat{C}\!\on{at}_{(\infty,1)}.
 \] 
  It follows from the universal property of the Gray tensor that this functor is naturally equivalent to
\[
  \hat{\scr{G}}^{\on{E}}_{\bcat{C},\bcat{D}}: \on{ESeg}_\Delta^{\on{nuc}}\left(\bcat{C}\!\on{at}_{(\infty,1)}\right) \xlongrightarrow{} \bcat{C}\!\on{at}_{(\infty,1)}, \enspace \enspace \bcat{X}_\bullet  \mapsto \on{Nat}_{\Delta^\op}((\EGlob(\bcat{C})\times \EGlob(\bcat{D}))_\bullet,\EGlob(\bcat{A})_\bullet)^{\mathbf{Egr}}
\]
where $\EGlob(\bcat{A})_\bullet \isom \bcat{X}_\bullet$ is obtained by first mapping to $\bcat{B}\!\on{icat}_\infty^{\mathbf{oplax},\odot}$ using the inverse of $\EGlob$. From this point on we will identify both functors under the notation $\scr{G}^{\on{E}}_{\bcat{C},\bcat{D}}$.

We then consider the correpresentable functor on $K_*(\mathbf{Gl}(\bcat{C})\tensor \mathbf{Gl}(\bcat{D}))_\bullet$ which we denote as
\[
  \scr{H}_{\bcat{C},\bcat{D}}: \on{ESeg}_\Delta^{\on{nuc}}\left(\bcat{C}\!\on{at}_{(\infty,1)}\right) \xlongrightarrow{} \bcat{C}\!\on{at}_{(\infty,1)}.
\]
 Note that as a consequence of \autoref{cor:enhanced} it follows that this functor is naturally equivalent the functor which sends an enhanced Segal object $\bcat{X}_\bullet \isom \EGlob(\bcat{A})_\bullet$ to the $(\infty,1)$-category of natural transformations $\on{Nat}_{\Delta^\op}((\mathbf{Gl}(\bcat{C})\tensor \mathbf{Gl}(\bcat{D}))_\bullet,\mathbf{Gl}(\bcat{A}))$.

Finally, it follows that we can identify $\scr{H}_{\bcat{C},\bcat{D}}$ with the functor 
\[
  \hat{\scr{H}}_{\bcat{C},\bcat{D}}: \on{ESeg}_\Delta^{\on{nuc}}\left(\bcat{C}\!\on{at}_{(\infty,1)}\right) \xlongrightarrow{} \bcat{C}\!\on{at}_{(\infty,1)}, \enspace \enspace \bcat{X}_\bullet  \mapsto \on{Nat}_{\Delta^\op}((\mathbf{Gl}(\bcat{C})\times \mathbf{Gl}(\bcat{D}))_\bullet,\mathbf{Gl}(\bcat{A})_\bullet)^{\mathbf{gr}}.
\]
Our discussion above guarantees the existence of a natural equivalence of functors $\scr{H}_{\bcat{C},\bcat{D}}\isom \scr{G}^{\on{E}}_{\bcat{C},\bcat{D}}$ so we can conclude by Yoneda that we have an equivalence  
\[
  K_*((\mathbf{Gl}(\bcat{C})\tensor \mathbf{Gl}(\bcat{D}))_\bullet \xlongrightarrow{\simeq} (\EGlob(\bcat{C})\tensor \EGlob(\bcat{D}))_\bullet.
\]

To finish the proof the correspresentable functor on $\EGlob(\bcat{C}\tensor_{\on{gr}}\bcat{D})_{\bullet}$ which we denote by 
\[
  \scr{T}_{\bcat{C},\bcat{D}}:\on{ESeg}_\Delta^{\on{nuc}}\left(\bcat{C}\!\on{at}_{(\infty,1)}\right) \xlongrightarrow{} \bcat{C}\!\on{at}_{(\infty,1)}.
\]
Using \autoref{cor:eglobnuc} we obtain that functor is naturally equivalent to the functor 
\[
  \hat{\scr{T}}_{\bcat{C},\bcat{D}}:\on{ESeg}_\Delta^{\on{nuc}}\left(\bcat{C}\!\on{at}_{(\infty,1)}\right) \xlongrightarrow{} \bcat{C}\!\on{at}_{(\infty,1)}, \enspace \enspace \bcat{X}_\bullet \mapsto \on{Fun}^{\mathbf{opgl}}(\bcat{C}\tensor_{\on{gr}}\bcat{D},\bcat{A})^{\leq 1}
\]
where $\bcat{X}_\bullet\isom \EGlob(\bcat{A})_\bullet$. Since $\bcat{C}\tensor_{\on{gr}}\bcat{D}$ is defined as a fibrant replacement of $\bcat{C}\tensor \bcat{D}$ we see that the functor $\hat{\scr{T}}_{\bcat{C},\bcat{D}}$ is equivalent to the functor that sends $\bcat{X}_\bullet$ to $\Fun^{\mathbf{opgl}}(\bcat{C}\tensor \bcat{D},\bcat{A})^{\leq 1} \isom \Fun^{\mathbf{opgl}}(\overline{\bcat{C}\times \bcat{D}},\bcat{A})^{\mathbf{gray},\leq 1}$ where the last equivalence is given by \autoref{lem:goodgrayuniv}. We conclude from the commutative diagrams above that we have an equivalence of functor $\scr{T}_{\bcat{C},\bcat{D}}\isom \scr{G}^{\on{E}}_{\bcat{C},\bcat{D}}$ which yields an equivalence (by Yoneda) 
\[
  \EGlob(\bcat{C}\tensor_{\on{gr}}\bcat{D})_\bullet \xlongrightarrow{\simeq} \left(\EGlob(\bcat{C})\tensor \EGlob(\bcat{D})\right)_\bullet.
\]
It is clear that our constructions are natural in $\bcat{C}$ and $\bcat{D}$ so our theorem follows. 
\end{proof}

\begin{definition}
  Let $\scr{C}$ be an $(\infty,1)$-category and consider a monoidal structure $\tensor: \scr{C} \times \scr{C}\xlongrightarrow{} \scr{C}$. We say that the monoidal structure on $\scr{C}$ is \emph{closed} if for every $X \in \scr{C}$ the functors
  \[
    \mathblank \tensor X , X \tensor \mathblank : \scr{C} \xlongrightarrow{} \scr{C}
  \]
  admit right adjoints.
\end{definition}

\begin{remark}
  Given an $(\infty,2)$-category $\bcat{A}$ and a functor $\tensor: \bcat{A} \times \bcat{A} \xlongrightarrow{} \bcat{A}$. We will say that $\tensor$ defines a closed monoidal structure if the corresponding functor on the underlying $(\infty,1)$-categories does.
\end{remark}

\begin{corollary}\label{cor:GHRprop}
 The Gray tensor product in any of its incarnations studied in this section defines a closed monoidal structure. In particular, we have that
  \begin{itemize}
    \item The Gray tensor product is associative.
    \item The Gray tensor product commutes with $\infty$-categorical colimits in each variable.
  \end{itemize}
\end{corollary}
\begin{proof}
  By virtue of \autoref{thm:maingray} it is enough to prove these assertions in any of the models studied. In the setting of (marked) scaled simplicial sets this was originally established in \cite{GHL_Gray}. For completeness of this document we remind the reader that associativity follows immediately from the definition of the Gray tensor product of marked-scaled simplicial sets given in \autoref{def:gray} and the final assertion follows from \autoref{prop:msgrayleftquillen}.

  To finish the proof, we consider an $(\infty,2)$-category $\bcat{C}$ modelled by a fibrant marked-scaled simplicial set and wish to show that the functors
  \[
      \mathblank \tensor \bcat{C}, \bcat{C}\tensor \mathblank: \bcat{B}\!\on{icat}^{\leq 1}_\infty \xlongrightarrow{} \bcat{B}\!\on{icat}^{\leq 1}_\infty 
   \] 
   admit right adjoints (at level of $(\infty,1)$-categories). Note that for every triple of $(\infty,2)$-categories $\bcat{C},\bcat{D}$ and $\bcat{A}$  we have natural equivalences of $\infty$-groupoids
  \[
    \Fun(\bcat{D}\tensor \bcat{C},\bcat{A})^\simeq \xlongrightarrow{\simeq}\Fun(\bcat{D},\Fun^{\mathbf{gr}}((\bcat{C},\bcat{A}))^\simeq
  \]
  \[
    \Fun(\bcat{C}\tensor \bcat{D},\bcat{A})^\simeq  \xlongrightarrow{\simeq} \Fun(\bcat{D},\Fun^{\mathbf{opgr}}((\bcat{C},\bcat{A}))^\simeq
  \]
  We conclude that we have adjunctions $\mathblank \tensor \bcat{C} \dashv \Fun^{\mathbf{gr}}((\bcat{C},\bcat{A})$ and $\bcat{C} \tensor \mathblank \dashv \Fun^{\mathbf{opgr}}((\bcat{C},\bcat{A})$ and thus the result follows.
\end{proof}

 \newpage
 {}
 
\end{document}